\documentclass[letterpaper, 11pt,  reqno]{amsart}
\usepackage[margin=1.2in,marginparwidth=1.5cm, marginparsep=0.5cm]{geometry}
\pdfoutput=1

\usepackage{pdfpages,caption,geometry}
\usepackage{enumitem}
\usepackage{amsmath,amssymb,amscd,amsthm,amsxtra}

\usepackage{color}

\usepackage[implicit=true]{hyperref}

\usepackage{microtype}
\usepackage{xsavebox}

\usepackage{bbm}

\usepackage[textwidth=25mm]{todonotes}

\definecolor{gr}{rgb}   {0.,   0.69,   0.23 }
\definecolor{bl}{rgb}   {0.,   0.5,   1. }
\definecolor{mg}{rgb}   {0.85,  0.,    0.85}
\definecolor{yl}{rgb}   {0.8,  0.7,   0.}
\definecolor{or}{rgb}  {0.7,0.2,0.2}

\allowdisplaybreaks
\newtheorem{theorem}{Theorem}[section]
\newtheorem{lemma}[theorem]{Lemma}
\newtheorem{proposition}[theorem]{Proposition}
\newtheorem{remark}[theorem]{Remark}
\newtheorem{definition}[theorem]{Definition}

\DeclareMathOperator{\sgn}{sgn}

\newcommand{\I}{\hspace{0.5mm}\text{I}\hspace{0.5mm}}
\newcommand{\II}{\text{I \hspace{-2.8mm} I} }
\newcommand{\III}{\text{I \hspace{-2.8mm} I \hspace{-2.8mm} I} }

\newcommand{\IV}{\text{I \hspace{-2.8mm} V}}
\newcommand{\V}{\text{V}}
\newcommand{\VI}{\text{V \hspace{-2.8mm} I}}

\newcommand{\noi}{\noindent}
\newcommand{\Z}{\mathbb{Z}}
\newcommand{\R}{\mathbb{R}}
\newcommand{\C}{\mathbb{C}}
\newcommand{\T}{\mathbb{T}}

\newcommand{\Gdl}{\mathcal{G}_{\dl} }
\newcommand{\Tdl}{\mathcal{T}_{\dl} }
\newcommand{\Qdl}{\mathcal{Q}_{\dl}}
\newcommand{\Qd}{\wt{\mathcal{Q}}_{\dl}}
\newcommand{\Gd}{\wt{\mathcal{G}}_\dl}
\newcommand{\Ldl}{\mathfrak{L}_\dl}
\newcommand{\Kdl}{\mathfrak{K}_\dl}

\newcommand{\TT}{\mathcal{T}}

\newcommand{\too}{\longrightarrow}

\newcommand{\BO}{\text{\rm BO} }
\newcommand{\KDV}{\text{\rm KdV} }

\newcommand{\Ha}{\mathbb{H}_a}

\DeclareMathOperator{\Law}{Law}
\DeclareMathOperator{\ord}{rank}

\let\Re=\undefined\DeclareMathOperator*{\Re}{Re}
\let\Im=\undefined\DeclareMathOperator*{\Im}{Im}

\let\P= \undefined
\newcommand{\P}{\mathbf{P}}

\newcommand{\Pc}{\mathcal{P}}

\newcommand{\E}{\mathbb{E}}
\newcommand{\EE}{\mathcal{E}}

\renewcommand{\L}{\mathcal{L}}

\newcommand{\F}{\mathcal{F}}

\newcommand{\pf}{\mathfrak{p}}
\newcommand{\hf}{\mathfrak{h}}

\newcommand{\al}{\alpha}
\newcommand{\be}{\beta}
\newcommand{\dl}{\delta}

\newcommand{\eps}{\varepsilon}
\newcommand{\kk}{\kappa}
\newcommand{\g}{\gamma}

\newcommand{\ld}{\lambda}

\newcommand{\s}{\sigma}

\newcommand{\ft}{\widehat}
\newcommand{\wt}{\widetilde}
\newcommand{\cj}{\overline}
\newcommand{\dx}{\partial_x}
\newcommand{\dt}{\partial_t}
\newcommand{\dd}{\partial}

\newcommand{\ta}{\theta}

\renewcommand{\l}{\ell}
\renewcommand{\o}{\omega}
\renewcommand{\O}{\Omega}

\newcommand{\les}{\lesssim}
\newcommand{\ges}{\gtrsim}

\newcommand{\jb}[1]
{\langle #1 \rangle}

\renewcommand{\b}{\beta}
\newcommand{\ind}{\mathbf 1}

\usepackage{tikz}

\usepackage{marginnote}
\usepackage{scalerel} 

\usetikzlibrary{shapes.misc}
\usetikzlibrary{shapes.symbols}
\usetikzlibrary{decorations}
\usetikzlibrary{decorations.markings}

\tikzset{
	dot/.style={circle,fill=black,draw=black,inner sep=0pt,minimum size=0.5mm},
	>=stealth,
	}

\tikzset{
	dot2/.style={circle,fill=black,draw=black,inner sep=0pt,minimum size=0.2mm},
	>=stealth,
	}

\tikzset{
	ddot/.style={circle,fill=white,draw=black,inner sep=0pt,minimum size=0.8mm},
	>=stealth,
	}

\tikzset{decision/.style={ 
        draw,
        diamond,
        aspect=1.5
    }}

\tikzset{dia2/.style
={diamond,fill=white,draw=black,inner sep=0pt,minimum size=1mm},
	>=stealth,
	}

\tikzset{dia/.style
={star,fill=black,draw=black,inner sep=0pt,minimum size=1mm},
	>=stealth,
	}

\tikzset{dia/.style
={diamond,fill=black,draw=black,inner sep=0pt,minimum size=1.3mm},
	>=stealth,
	}

\makeatletter
\def\DeclareSymbol#1#2#3{\xsavebox{#1}{\tikz[baseline=#2,scale=0.15]{#3}}}
\def\<#1>{\xusebox{#1}}
\makeatother

\newsavebox{\peA}
\newsavebox{\pneA}
\newsavebox{\plA}
\newsavebox{\pgA}
\newsavebox{\pleA}
\newsavebox{\pgeA}
\newsavebox{\pezA}

\savebox{\peA}{\tikz \draw (0,0) node[shape=circle,draw,inner sep=0pt,minimum size=8.5pt] {\scriptsize  $=$};}
\savebox{\pneA}{\tikz \draw (0,0) node[shape=circle,draw,inner sep=0pt,minimum size=8.5pt] {\footnotesize $\neq$};}
\savebox{\plA}{\tikz \draw (0,0) node[shape=circle,draw,inner sep=0pt,minimum size=8.5pt] {\scriptsize $<$};}
\savebox{\pgA}{\tikz \draw (0,0) node[shape=circle,draw,inner sep=0pt,minimum size=8.5pt] {\scriptsize $>$};}
\savebox{\pleA}{\tikz \draw (0,0) node[shape=circle,draw,inner sep=0pt,minimum size=8.5pt] {\scriptsize $\leqslant$};}
\savebox{\pgeA}{\tikz \draw (0,0) node[shape=circle,draw,inner sep=0pt,minimum size=8.5pt] {\scriptsize $\geqslant$};}
\savebox{\pezA}{\tikz \draw (0,0) node[shape=circle,draw,
fill=white, 
inner sep=0pt,minimum size=8.5pt]{} ;}

\def \peB{\mathchoice
{\scalebox{.7}{{\usebox{\peA}}}}
{\scalebox{.7}{{\usebox{\peA}}}}
{\scalebox{.7}{{\usebox{\peA}}}}
{}
}

\def \pezB{\mathchoice
{\scalebox{.7}{{\usebox{\pezA}}}}
{\scalebox{.7}{{\usebox{\pezA}}}}
{\scalebox{.7}{{\usebox{\pezA}}}}
{}
}

\newcommand{\pe}{\mathbin{{\peB}}}

\newcommand{\pez}{\mathbin{{\pezB}}}

\usetikzlibrary{decorations.pathmorphing, patterns,shapes}

\DeclareSymbol{dot}{-2.7}{\draw (0,0) node[dot]{};}

\DeclareSymbol{1}{0}{\draw[white] (-.4,0) -- (.4,0); \draw (0,0)  -- (0,1.2) node[dot] {};}
\DeclareSymbol{2}{0}{\draw (-0.5,1.2) node[dot] {} -- (0,0) -- (0.5,1.2) node[dot] {};}
\DeclareSymbol{3}{0}{\draw (0,0) -- (0,1.2) node[dot] {}; \draw (-.7,1) node[dot] {} -- (0,0) -- (.7,1) node[dot] {};}
\DeclareSymbol{30}{-3}{\draw (0,0) -- (0,-1); \draw (0,0) -- (0,1.2) node[dot] {}; \draw (-.7,1) node[dot] {} -- (0,0) -- (.7,1) node[dot] {};}
\DeclareSymbol{31}{-3}{\draw (0,0) -- (0,-1) -- (1,0) node[dot] {}; \draw (0,0) -- (0,1.2) node[dot] {}; \draw (-.7,1) node[dot] {} -- (0,0) -- (.7,1) node[dot] {};}
\DeclareSymbol{32}{-3}{\draw (0,0) -- (0,-1) -- (1,0) node[dot] {}; \draw (0,0) -- (0,-1) -- (-1,0) node[dot] {}; \draw (0,0) -- (0,1.2) node[dot] {}; \draw (-.7,1) node[dot] {} -- (0,0) -- (.7,1) node[dot] {};}
\DeclareSymbol{20}{-3}{\draw (0,0) -- (0,-1);\draw (-.7,1) node[dot] {} -- (0,0) -- (.7,1) node[dot] {};}
\DeclareSymbol{22}{-3}{\draw (0,0.3) -- (0,-1) -- (1,0) node[dot] {}; \draw (0,0.3) -- (0,-1) -- (-1,0) node[dot] {};\draw (-.7,1) node[dot] {} -- (0,0.3) -- (.7,1) node[dot] {};}
\DeclareSymbol{31p}{-3}{\draw (0,0) -- (0,-1) -- (1,0) node[dot] {}; \draw (0,0) -- (0,1.2) node[dot] {}; \draw (-.7,1) node[dot] {} -- (0,0) -- (.7,1) node[dot] {}; 
\draw (0,-1) node{\scaleobj{0.5}{\pez}}; 
\draw (0,-1) node{\scaleobj{0.5}{\pe}}; }
\DeclareSymbol{32p}{-3}{\draw (0,0) -- (0,-1) -- (1,0) node[dot] {}; \draw (0,0) -- (0,-1) -- (-1,0) node[dot] {}; \draw (0,0) -- (0,1.2) node[dot] {}; \draw (-.7,1) node[dot] {} -- (0,0) -- (.7,1) node[dot] {}; 
\draw (0,-1) node{\scaleobj{0.5}{\pez}};
\draw (0,-1) node{\scaleobj{0.5}{\pe}};}
\DeclareSymbol{22p}{-3}{\draw (0,0.3) -- (0,-1) -- (1,0) node[dot] {}; \draw (0,0.3) -- (0,-1) -- (-1,0) node[dot] {};\draw (-.7,1) node[dot] {} -- (0,0.3) -- (.7,1) node[dot] {}; 
\draw (0,-1) node{\scaleobj{0.5}{\pez}};
\draw (0,-1) node{\scaleobj{0.5}{\pe}};}
\DeclareSymbol{21p}{-3}{\draw (0,0.3) -- (0,-1) -- (1,0) node[dot] {}; 
\draw (-.7,1) node[dot] {} -- (0,0.3) -- (.7,1) node[dot] {}; 
\draw (0,-1) node{\scaleobj{0.5}{\pez}};
\draw (0,-1) node{\scaleobj{0.5}{\pe}};}

\newcommand{\PP}{\mathbb{P}}
\newcommand{\M}{\mathcal{M}}
\newcommand{\A}{\mathcal{A}}
\def\e{\varepsilon}

\newcommand{\N}{\mathbb{N}}

\newcommand{\dr}{\theta}

\newcommand{\Dr}{\Theta}

\newtheorem*{ackno}{Acknowledgements}

\renewcommand{\H}{\mathcal{H}}

\newcommand{\dtv}{d_{\rm TV}}

\newcommand{\dkl}{d_{\rm KL}}

\def\sgn{\textup{sgn}}

\newcommand{\cX}{\mathcal{X}}

\newcommand{\GGM}{generalized Gibbs measure}
\newcommand{\GGMs}{generalized Gibbs measures}

\newcommand{\Id}{\textup{Id}}

\numberwithin{equation}{section}
\numberwithin{theorem}{section}

\makeatletter
\@namedef{subjclassname@2020}{\textup{2020} Mathematics Subject Classification}
\makeatother

\begin{document}
\baselineskip = 14pt

\title[Statistical equilibria for ILW]
{Deep-water  and shallow-water limits of statistical equilibria for the intermediate long wave equation}

\author[A.~Chapouto,
G.~Li, and T.~Oh]
{Andreia Chapouto, Guopeng Li, and Tadahiro Oh}

\address{Andreia Chapouto, Department of Mathematics, University of California, Los Angeles, CA 90095, USA
and
School of Mathematics\\
The University of Edinburgh\\
and The Maxwell Institute for the Mathematical Sciences\\
James Clerk Maxwell Building\\
The King's Buildings\\
Peter Guthrie Tait Road\\
Edinburgh\\
EH9 3FD\\
United Kingdom, and
CNRS, Laboratoire de math\'ematiques de Versailles, UVSQ, Universit\'e Paris-Saclay, CNRS, 45 avenue des 
\'Etats-Unis, 78035 Versailles Cedex, France, 
and School of Mathematics, Monash University, VIC 3800, Australia}

\email{andreia.chapouto@monash.edu}

\address{
Guopeng Li, 
School of Mathematics and Statistics, Beijing Institute of Technology, Beijing 100081, China\\
and
School of Mathematics\\
The University of Edinburgh\\
and The Maxwell Institute for the Mathematical Sciences\\
James Clerk Maxwell Building\\
The King's Buildings\\
Peter Guthrie Tait Road\\
Edinburgh\\
EH9 3FD\\
United Kingdom}

\email{guopeng.li@bit.edu.cn}


%
\address{
Tadahiro Oh, 
School of Mathematics\\
The University of Edinburgh\\
and The Maxwell Institute for the Mathematical Sciences\\
James Clerk Maxwell Building\\
The King's Buildings\\
Peter Guthrie Tait Road\\
Edinburgh\\
EH9 3FD\\
 United Kingdom,
 and
 School of Mathematics and Statistics, Beijing Institute of Technology, Beijing 100081, China
 }

\email{hiro.oh@ed.ac.uk}



\begin{abstract}

We study the construction of invariant measures
associated with higher order conservation laws
of  the intermediate long wave equation (ILW)
and their convergence properties in the deep-water and shallow-water limits.
By exploiting
its complete integrability,
we first carry out detailed analysis on the construction of appropriate conservation laws of ILW
at the $H^\frac k2$-level for each  $k \in \N$,
and
 establish their convergence to those
of the Benjamin-Ono equation (BO)
in the deep-water limit
and
to those of the Korteweg-de Vries equation (KdV)
in the  shallow-water limit.
In particular,  in the shallow-water limit,
we prove a rather striking
2-to-1 collapse of the conservation laws
 of ILW to those of KdV.
 Such a 2-to-1 collapse is novel in the literature
and, 
to our knowledge, this is the first construction of a complete family of 
shallow-water conservation laws with non-trivial shallow-water limits.
We then construct an infinite sequence of 
\GGMs~
for ILW
associated with these conservation laws
and prove their convergence to
the corresponding (invariant) \GGMs~ for BO and KdV
in the respective limits.
Finally, for $k \ge 3$,
we establish invariance of these measures under ILW dynamics,
and also convergence in the respective limits
 of the ILW dynamics at each equilibrium state
 to the corresponding invariant dynamics for BO and KdV
 constructed by Deng, Tzvetkov, and Visciglia (2010-2015) and Zhidkov (1996), respectively.
In particular, in the shallow-water limit,
we establish a 
2-to-1 collapse at the level of
the \GGMs~ as well as  the invariant ILW dynamics.
As a byproduct of our analysis,
we also prove  invariance of the \GGM~  associated
with the $H^2$-conservation law of KdV,
which seems to be missing in the literature.

\vspace{1cm}

\end{abstract}

%
\maketitle

\tableofcontents

\newpage

\section{Introduction}
\label{SEC:1}

\subsection{Intermediate long wave equation}
\label{SUBSEC:1.1}

We consider the intermediate long wave equation (ILW) on the one-dimensional torus $\T=\R/(2\pi\Z)$:
\begin{equation}\label{ILW}
\dt u - \Gdl\dx^2 u =\dx(u^2),
\qquad (t,x) \in \R\times\T.
\end{equation}

\noi
This equation was introduced in \cite{Joseph, Kubota}
as a model describing the propagation of an internal wave at the interface of a stratified fluid of
finite depth $\dl>0$,
and the unknown
 $u$ denotes the amplitude of the internal wave at the interface.
See also Remark 1.1 in \cite{LOZ}.
Note that the derivation in \cite{Joseph, Kubota} is formal;
see  \cite{BLS, CGK} for the rigorous derivation of \eqref{ILW}.
The operator~$\Gdl$ characterizes the phase speed and it is defined
as the following Fourier multiplier operator:
\begin{align}
\ft{\Gdl f}(n) =
\ft{\Gdl}(n) \ft f(n)
:=
-i\bigg(\coth(\dl n)  - \frac{1}{\dl n}\bigg) \ft{f}(n), \quad n\in\Z, 
\label{GG1}
\end{align}

\noi
 with the convention that $\coth(\dl n)  - \frac{1}{\dl n} =0 $ for $n=0$.

 The ILW equation \eqref{ILW}
 is an important physical model with further applications
 in the study of 
 atmospheric sciences,  oceanography, and quantum physics
 \cite{CMH78,  OB80, KB81, SH82,  BLL22}.
Furthermore,
it serves as an ``intermediate'' equation of finite depth $0 < \dl < \infty$,
providing a natural connection
between the Benjamin-Ono equation (BO), modeling fluid of infinite depth  ($\dl = \infty$),
and the  Korteweg-de Vries equation (KdV), modeling shallow water ($\dl = 0$).
From the analytical point of view,
 \eqref{ILW} is also of great  interest due to its rich structure;
 it is a dispersive equation, admitting soliton solutions.
  Moreover,  it is completely integrable with an infinite number of conservation laws \cite{KAS82, KSA, S89}.
See \cite{Saut2019, KleinSaut21} for an overview of the subject and the references therein.

Our main
goal in this paper is to advance our understanding
of convergence issues for ILW
from both the deterministic and statistical viewpoints.
In the deep-water limit ($\dl\to\infty$),
an elementary computation shows  that
 the operator $\Gdl$ in \eqref{GG1} converges to  the Hilbert transform $\H$
 with symbol\footnote{\label{FT:1}In this paper, we work with mean-zero functions.
In order to avoid ambiguity when $\H$ acts on a product
(which can be checked to have mean zero in our application),
we define $\H$ to be the Fourier multiplier operator
with multiplier
$-i \sgn(n)\cdot \ind_{n \ne 0}$.
Namely, we have $\H = \H \P_{\ne 0}$,
 where $\P_{\ne 0}$ denotes the projection onto non-zero frequencies.
 } $-i \sgn(n)$ (see Lemma~\ref{LEM:K1})
 and thus
ILW  \eqref{ILW} formally converges to~BO:
\begin{align}\label{BO}
\dt u - \H\dx^2 u = \dx(u^2).
\end{align}

\noi
Next, we consider
 the shallow-water limit ($\dl\to0$).
Recall that  ILW  \eqref{ILW} describes the motion
of the fluid interface in a stratified fluid of finite depth $\dl > 0$,
where $u$ denotes the amplitude of the internal wave at the interface.
As $\dl \to 0$, the entire fluid depth tends to $0$
and, in particular,
the amplitude $u$ of the internal wave at the interface is $O(\dl)$,  which also  tends to~$0$.
Hence, in order to observe any meaningful limiting behavior,
we need to magnify the fluid amplitude by a factor $\sim \frac 1\dl$.
This motivates us to consider the following scaling~\cite{ABFS}:
\begin{equation}
v(t,x)  = \tfrac1\dl u(\tfrac1\dl t, x),
\label{scale1}
\end{equation}

\noi
which leads to the following scaled ILW:
\begin{equation}
\dt v - \Gd \dx^2 v = \dx(v^2),
\quad
\text{where }\ \Gd := \frac1\dl \Gdl.
\label{sILW}
\end{equation}

\noi
By studying the limiting behavior
of  $\Gd\dx$ as $\dl \to 0$ (see Lemma~\ref{LEM:L1}), we see that
 the scaled ILW~\eqref{sILW} formally converges to
 the following KdV equation:
\begin{align}\label{kdv}
\dt v + \frac 13  \dx^3 v= \dx(v^2).
\end{align}

\noi
In the deep-water limit,
the Fourier multiplier of $\Gdl \dx$ converges
to that of $\H \dx$, uniformly in the frequency $n \in \Z$.
On the other hand,
such uniform convergence does not hold
in the shallow-water limit;
see Lemmas \ref{LEM:K1} and
\ref{LEM:L1}.
In this sense,
the shallow-water limit
is  {\it singular} (especially when compared to the deep-water limit).
As we see below, this singular nature of the shallow-water limit
also appears in
convergence of (suitably constructed) conservation laws of the scaled ILW
to those of KdV
(exhibiting regularity jumps)
and also in the statistical study of ILW.
See Remark~\ref{REM:sing}.

In recent years, there have been intensive research activities
on mathematical analysis of ILW;
see  \cite{ABFS, MST, MV15, MPV, IS23, CLOP, CFLOP, FLZ, LP} for the known well-posedness and ill-posedness results
for ILW.
See also \cite{ABFS, MPS, Li2022, LOZ, CLOP}
for results on convergence
of  ILW to BO in the deep-water limit ($\dl \to \infty$)
and to KdV in the shallow-water limit ($\dl \to 0$).\footnote{Strictly speaking,
in order to discuss convergence to KdV in the shallow-water limit,
we first need to apply the scaling  \eqref{scale1}
and consider  the scaled ILW \eqref{sILW}.
In this informal discussion, however, we suppress this point.}
In particular, in~\cite{LOZ},
the second and third authors with G.\,Zheng
initiated the study of convergence issues from the statistical viewpoint,
where
they constructed the Gibbs measure for ILW
and proved its convergence in the deep-water and shallow-water limits
to those for BO and KdV, respectively,
along with convergence results of  invariant Gibbs dynamics
(without uniqueness).
Our main goal in this paper is to
further the statistical study on  convergence  of ILW to BO and KdV
by exploiting the complete integrability of ILW.

The construction of invariant Gibbs measures
for Hamiltonian PDEs
was initiated in the seminal works \cite{LRS}
by Lebowitz, Rose, and Speer
and \cite{BO94, BO96} by
Bourgain.
This study lying at the intersection of dispersive PDEs
and probability theory has been particularly active
over the last two decades.
%
%
See  surveys papers \cite{OH6, BOP4}
for an overview of the subject
and the references therein.
Here, a Gibbs measure
is a probability measure
on
periodic functions (of low regularity)
associated with the Hamiltonian of a given equation.
See
\cite{LRS, BO94, OH2, Tz10, Deng15, R16, ORT, OST, CK, OST2}
for the construction of Gibbs measures
and their invariance for KdV- and BO-type equations.
For completely integrable PDEs
with infinitely many conservation laws
such as BO and KdV,
one may use higher order conservation laws to construct invariant measures
supported on Sobolev spaces of increasing smoothness.
By importing the terminology from statistical mechanics~\cite{RDYO}, 
we refer to such measures
as 
{\it \GGMs}. 
In \cite{Zhid1, Zhid2},
Zhidkov constructed an infinite sequence of invariant \GGMs~
associated with higher order conservation laws of
KdV.
In  a series of works
\cite{TV0, TV1, TV2,   DTV},
Tzvetkov and Visciglia with Deng
carried out 
an analogous program
for BO.
See also
\cite{Zhid3, NORS, GLV1, GLV2}
for related results on the one-dimensional (derivative) cubic  nonlinear Schr\"odinger equation.
In this work, we investigate
the construction
of invariant \GGMs~
associated with higher order conservation laws of ILW
and their convergence properties
in the deep-water and shallow-water limits
from both the static and dynamical viewpoints.
More precisely,
we establish the following results:

\begin{itemize}

\smallskip
\item[(i)]
For each $0 < \dl < \infty$,
we construct \GGMs,
supported on $L^2(\T)$,
associated
with the conservation laws of ILW
at the $H^\frac k2$-level
for each integer $k \ge 2$.

\smallskip
\item[(ii)]
We
 prove  convergence of these \GGMs~ to
the corresponding (invariant) 
 \GGMs~
 for BO and KdV
in the deep-water and shallow-water limits, respectively.
We also study equivalence\footnote{Namely, mutual absolute continuity.}\,/\,singularity properties
of these \GGMs.

\smallskip
\item[(iii)]
For $k \ge 3$,
we construct invariant ILW dynamics (with uniqueness) associated with these
 \GGMs~
supported on $H^{\frac {k-1}2-\eps}(\T)$, $\eps > 0$.\footnote{In the remaining part
of the paper, we use $\eps > 0$ to denote an arbitrarily small constant.
Various constants depend on $\eps$ but,
for simplicity of notation,  we suppress such $\eps$-dependence.}
Moreover, we  prove convergence in the respective limits
 of the ILW dynamics at each equilibrium state
 to the corresponding invariant dynamics for BO and KdV
 constructed by Deng, Tzvetkov, and Visciglia \cite{TV0, TV1, TV2, DTV}
 and Zhidkov \cite{Zhid1, Zhid2}, respectively.
This provides the first example of invariant ILW dynamics
(with uniqueness)
and their convergence  to invariant BO\,/\,KdV dynamics.

\end{itemize}

\noi
See Theorems \ref{THM:3} and \ref{THM:4}
for (i) and (ii), and
Theorems \ref{THM:5} and \ref{THM:6}
for (iii).
We point out that,  when $k$ is large,
the invariant dynamics is supported on smooth functions
and thus the statistical study of ILW carried out  in this work
is physically relevant.
For example, the (scaled) ILW invariant dynamics  constructed in 
Theorems \ref{THM:5} and \ref{THM:6}
enjoys a recurrence property on a phase space of smooth functions, 
described by the Poincar\'e recurrence theorem 
\cite[p.\,106]{Zhid2}; see also
\cite[Corollary~1.3]{TV1}.

In order to achieve these objectives,
we first need to carry out a detailed study
on the construction of  appropriate (polynomial) conservation laws of
the (scaled) ILW
via the
{\it  B\"acklund transform}
 (see~\eqref{BT1} in the deep-water regime\footnote{Hereafter, we use
 the terms ``deep-water regime'' and
 ``shallow-water regime''
 to refer to analysis on the ILW equation \eqref{ILW}
 and the scaled ILW equation \eqref{sILW}, respectively.}
 and \eqref{BTX1}  in the shallow-water regime).
 We then
 derive
a  suitable formulation for conservation laws
of the (scaled) ILW
in each of the deep-water and shallow-water regimes
such that
the (scaled) ILW conservation laws converge
to those of BO and KdV in the respective limits.
We remark the following two points.

\begin{itemize}

\smallskip
\item

In order to guarantee the desired convergence
of the (scaled) ILW conservation laws,
we need to use different
B\"acklund transforms
in the deep-water  and shallow-water regimes,
generating  two sets of the (scaled)
ILW conservation laws.

\smallskip
\item
The conservation laws of the (scaled) ILW
thus constructed
are  at the
$H^\frac k2$-level, $k \in \N$, just like the BO conservation laws.
On the other hand,
the KdV conservation laws are at the $H^\kk$-level, $\kk \in \N$.
Indeed,
we establish a
{\it  2-to-1 collapse of the conservation laws
 of the scaled ILW to those of KdV
in the shallow-water limit}
in the sense that,
for each $\kk \in \N$,
 {\it both}
the
$H^{\kk - \frac 12}$-level
and
$H^\kk$-level
 conservation laws of the  scaled ILW
converge
to the {\it same}
$H^\kk$-level conservation law of KdV;
see Figure \ref{FIG:1} below.
 Such a 2-to-1 collapse is novel in the literature,
explaining
the reason that
 KdV has half as many conservation laws as ILW and BO.
To our knowledge, this is the first construction of a complete family of 
 shallow-water conservation laws\footnote{In the following, 
we often refer to conservation laws for ILW (and for the scaled ILW, respectively)
as deep-water conservation laws
(and shallow-water conservation laws, respectively).}
 with non-trivial shallow-water limits.

\end{itemize}

\noi
The derivation of conservation laws
for  ILW \eqref{ILW} and the scaled ILW \eqref{sILW},
suitable to study their  deep-water and shallow-water limits, respectively,
involves rather intricate and lengthy algebraic\,/\,combinatorial  computations
and takes a non-trivial portion of this work.
Indeed,
one of the main novelties
of this work is
to provide the first rigorous justification for convergence of
(suitably constructed) conservation laws of the (scaled) ILW
to those of BO and KdV
in the respective limits,
in particular exhibiting  the 2-to-1 collapse  in the shallow-water limit;
see Theorems \ref{THM:1} and  \ref{THM:2}.
Such a 2-to-1 collapse in the shallow-water limit
also appears
at the level of measure convergence
as well as dynamical convergence,
which is
a completely new phenomenon in the statistical study of dispersive PDEs.


\begin{remark}\rm

In \cite{BO94}, Bourgain proved
invariance of the Gibbs measure for KdV
associated with the Hamiltonian
(= the $H^1$-level conservation law).
In \cite{Zhid1, Zhid2},
Zhidkov constructed  invariant \GGMs~
for KdV
associated with the $H^\kk$-level conservation laws
for $\kk \ge 3$.
See also \cite{QV, OH1, OH4, OH6, OQV, KMV}
for invariance of the white noise for KdV
(corresponding to the $\kk = 0$ case).
However,  the $\kk = 2$ case seems to be missing in the literature.
As a byproduct of our analysis,
we establish invariance of the 
\GGM~ associated
with the $H^2$-conservation law of KdV,
thus completing the picture for KdV;
see Theorem~\ref{THM:6}\,(ii).
See also a discussion in   \cite[Subsection 1.3]{TV0}.
In recent works \cite{Tz2, CF}, 
new kinds of  invariant measures
for BO and KdV have been constructed, 
based on more refined properties
of these equations
(namely, the  Birkhoff coordinates and the perturbation determinant, respectively).
It would be of interest to investigate such
construction for the (scaled) ILW.

\end{remark}

\begin{remark}\label{REM:sing}\rm
In the current work,
the singular nature of the
shallow-water regime
appears in various contexts.
First of all, we note that 2-to-1 collapse occurs in
the shallow-water convergence
of the shallow-water conservation laws (Theorem \ref{THM:2}),
the corresponding \GGMs~
(Theorem \ref{THM:4}),
and the invariant scaled ILW dynamics
(Theorem~\ref{THM:6}).
In particular, we observe regularity jumps
in all of these three settings.
Moreover, we show that,  in the deep-water regime, 
the \GGMs~
with different depth parameters
 (including the limiting one)
are all equivalent, 
while
the situation is much more complicated
 in the shallow-water regime;
(i) the \GGM~ with $0 < \dl < \infty$
and the limiting \GGM~(with $\dl = 0$)
are singular,
(ii) the odd-order \GGMs~
with different depth parameters
are singular, 
and 
(iii) the even-order \GGMs~
with different depth parameters
are equivalent. 
As a consequence,
the shallow-water \GGMs~
converge only weakly as $\dl \to 0$, while the deep-water
\GGMs~
converge in total variation as $\dl \to \infty$.
All these phenomena mentioned above
are completely new.

We also mention a recent work \cite{CHLO},
exhibiting the singular nature of the shallow-water regime
on the real line; see also Remark \ref{REM:new}.

\end{remark}

\subsection{Conservation laws of ILW}
\label{SUBSEC:1.2}

Before we start our statistical study of the (scaled) ILW,
we first need to carry out detailed analysis on
the derivation of suitable
(polynomial) conservation laws of the (scaled) ILW
along with their convergence properties
in the deep-water and shallow-water  limits.
It is worthwhile to note that
ILW \eqref{ILW} and the scaled ILW \eqref{sILW} differ only by the scaling  \eqref{scale1}.
Hence, once we obtain conservation laws of ILW \eqref{ILW}
in a form suitable to study their deep-water limits
(as $\dl \to \infty$),
the  scaling transform \eqref{scale1}
yields
the corresponding
conservation laws of the scaled ILW~\eqref{sILW}.
However,
if we simply use such conservation laws,
then we
encounter multiple divergent terms
(with negative powers of $\dl$)
in the shallow-water limit
($\dl \to 0$).
While it is possible to manually
group divergent terms to obtain a (conditionally) convergent limit
for the first few conservation laws,
the number of divergent terms grows
larger for higher order conservation laws
and it immediately becomes intractable.
In order to overcome this issue,
we will separately derive conservation laws of the scaled ILW
in a form suitable to
study their shallow-water limits.

We first consider the deep-water regime.
In this case, we follow
the work \cite{SAK}
by Satsuma,  Ablowitz, and Kodama
and
derive conservation laws of
 ILW
\eqref{ILW}
via the B\"acklund transform~\eqref{BT1};
see also \cite{ABFS}.
While the derivation of conservation laws for ILW is well known
in the literature,
there seems to be no proof on convergence of the conservation laws
(beyond a formal proof, claiming that
such convergence should follow from convergence of
 the B\"acklund transforms).
Our main task is to derive conservation laws
of ILW in the ``correct'' form such that
they converge to BO conservation laws in the deep-water limit
and, moreover,
the quadratic part is sign-definite,  which is crucial
in defining \GGMs~
associated with these deep-water conservation laws;
see
 Section \ref{SEC:cons1}
for details.
As a result, we obtain the following conservation laws
for ILW:\footnote{In the following,
we impose the mean-zero assumption for all the equations appearing in this paper,
and hence we will not refer to the conservation of mean in the remaining part of the paper.}
\begin{align*}
\text{mean:} \ \  \int_\T u dx
\end{align*}

\noi
and $\big\{E^\dl _{\frac{k}{2}}(u)\big\}_{k \in \Z_{\ge 0}}$, where $\Z_{\ge 0} := \N \cup\{0\}$
and $E^\dl_{\frac{k}{2}}(u)$ is given by
\begin{align}
\begin{split}
E^\dl _{0}(u) & = \frac12 \|u\|^2_{L^2}, \\
E^\dl_{\frac{k}{2}}(u)
&  = \frac12\sum_{\substack{\l=0\\ \text{even}}}^{k} a_{k, \l} \| \Gdl^{\frac{k-\l}{2}} u \|^2_{\dot{H}^{\frac{k}{2}}} + R^\dl_{\frac{k}{2}}(u), \quad k\in \N,
\end{split}
\label{E00}
\end{align}

\noi
for some {\it positive}\footnote{The positivity
of the coefficients $a_{k, \l}$ plays an important role  in
defining
 the base Gaussian measures $\mu^\dl_{\frac{k}{2}}$
in~\eqref{gauss1}.}
 constants $a_{k, \l}$ (see \eqref{DE2b})
such that
\begin{align}
\sum_{\substack{\l=0\\  \text{even}}}^{k} a_{k, \l} = 1, 
\label{E0}
\end{align}

\noi
where, for odd $k$ and even $\l$,  we used the convention:
\begin{align}
 \|\Gdl^{\frac{k-\l}{2}} u \|^2_{\dot{H}^{\frac k2}}
  = \int_\T \dx^\frac{\l}{2}u \cdot (\Gdl \dx)^{k-\l} \dx^\frac{\l}{2}u dx
\label{odd1}
\end{align}

\noi
by noting that 
$\Gdl \dx$ is a self-adjoint and positive operator with 
an even multiplier
(Lemma~\ref{LEM:K1}).
Here,
the interaction potential
 $R^\dl_{\frac{k}{2}}(u)$ consists of  cubic and higher order terms in $u$;
 see \eqref{R1}.
For readers' convenience, we write down the  case $k = 1, 2$:
\begin{align}
\begin{split}
E^\dl _{\frac12}(u) & =\frac12 \| \Gdl^\frac12 u\|^2_{\dot{H}^\frac12}+ \frac13  \int_\T  u^3 dx, \\
E^\dl _{1}(u) & = \frac18 \|u\|^2_{\dot{H}^1} + \frac38 \|\Gdl u\|^2_{\dot{H}^1}
+ \int_\T  \frac14 u^4 + \frac34 u^2 \Gdl \dx u + \frac{1}{4\dl} u^3  dx.
\end{split}
\label{E1}
\end{align}

\noi
Note that $E^\dl_{\frac12}(u)$ is the Hamiltonian for ILW \eqref{ILW}.
The positivity of the constants $a_{k, \l}$ ensures that the quadratic part
of the conservation law is sign-definite.
Moreover,
in view of
Lemma~\ref{LEM:K1}, we have
\begin{align}
\sum_{\substack{\l=0\\  \text{even}}}^{k} a_{k, \l} \| \Gdl^{\frac{k-\l}{2}} u \|^2_{\dot{H}^{\frac{k}{2}}}
\sim_{k, \dl} \| u \|_{\dot H^\frac{k}{2}}^2,
\label{E1a}
\end{align}

\noi
where the equivalence is uniform in $2\le \dl <  \infty$.
Namely,
$E^\dl_{\frac{k}{2}}(u)$ is indeed an $H^\frac k2$-level conservation law
of ILW.


As $\dl \to \infty$,
the operator $\Gdl$ converges to the Hilbert transform $\H$
(see Lemma \ref{LEM:K1}).
Then, by formally taking a limit as $\dl \to \infty$ in \eqref{E00}
with \eqref{E0},
we obtain the following conservation laws
for BO \eqref{BO}:
\begin{align}
\begin{split}
E^\BO _{0}(u) & =
   E^\infty_{0}(u)  =
\frac12 \|u\|^2_{L^2}, \\
E^\BO_{\frac{k}{2}}(u)
&  =
   E^\infty_{\frac{k}{2}}(u)
  =
\frac12 \| u \|^2_{\dot{H}^{\frac{k}{2}}} + R^\BO_{\frac{k}{2}}(u), \quad k\in \N,
\end{split}
\label{E2}
\end{align}

\noi
where once again
the interaction potential
 $R^\BO_{\frac{k}{2}}(u)
 = R^\infty_{\frac{k}{2}}(u)$
 consists of  cubic and higher order terms in $u$.
For example, when  $k = 1, 2$, we have
\begin{align*}
E^\BO _{\frac12}(u) & =\frac12 \|  u\|^2_{\dot{H}^\frac12}+ \frac13  \int_\T  u^3 dx, \\
E^\BO _{1}(u) & = \frac12 \|u\|^2_{\dot{H}^1}
+ \int_\T  \frac14 u^4 + \frac34 u^2 \H \dx u  dx.
\end{align*}

\noi
We point out that these quantities
$E^\BO_{\frac{k}{2}}(u)$
can also be directly constructed
from
the B\"acklund transform~\eqref{BTB1}
for  BO (see \eqref{chi2} and \eqref{DE3}),
and thus they are indeed conserved under the flow of BO.

\begin{theorem}[deep-water regime]\label{THM:1} 

Let $k \in \Z_{\ge 0}$.
Then,
given  $0 < \dl <  \infty$,
the quantity $ E^\dl_{\frac{k}{2}}(u) $
defined in \eqref{E00}
is conserved under the flow of  ILW  \eqref{ILW}.\footnote{In discussing
conservation of a  quantity under a given equation,
we always assume that
a solution $u$ is  sufficiently smooth
such that both the quantity and the flow make sense.}
Similarly,
the quantity $ E^\BO_{\frac{k}{2}}(u) $
defined in  \eqref{E2}
is  conserved under the flow of BO \eqref{BO}.
Moreover,
 for any $u\in H^{\frac k 2}(\T)$, we have
\begin{align}
\lim_{\dl\to\infty} E^\dl_{ \frac k 2}(u) = E_{\frac k 2 }^\BO(u).
\label{E4}
\end{align}

\end{theorem}

See 
Proposition \ref{PROP:cons1}
for a proof of Theorem \ref{THM:1}.
We prove Theorem \ref{THM:1}
by first deriving microscopic  conservation laws $\chi_n^\dl$ and $\chi_n^\BO$ for ILW and BO
via  the B\"acklund transforms~\eqref{BT1} and~\eqref{BTB1}, respectively.
Our goal is to derive conservation laws in the ``correct'' form as described above.
In fact,
a straightforward derivation of ILW conservation laws
via  
the microscopic  conservation laws $\chi_n^\dl$
yields
quadratic terms of different regularities
(not just the $\dot H^\frac k2$-norms appearing in \eqref{E00}; see Lemma~\ref{LEM:chi2})
which may not be suitable for constructing (weighted)
Gaussian measures.
Moreover,
in our recursive construction of conservation laws,
we ensure that the constants $a_{k, \l}$
are all positive,
guaranteeing that the quadratic part of $E^\dl_\frac k2(u)$ is sign-definite.
This plays a crucial role
in constructing  the Gaussian measures  $\mu^\dl_\frac k2$ in~\eqref{gauss1}
and studying their further properties (Proposition \ref{PROP:gauss1}).
Similar comments apply to the shallow-water regime.

As for BO,
 identities such as $\mathcal{H}^2 = -\Id$
 (on mean-zero functions)
make it relatively straightforward to derive conservation laws  from a  recursive formulation.
As for ILW, however,
there is no such identity for
 the operator $\Gdl$.
 Moreover, there are terms with negative powers of $\dl$
 (such as the last term of $E^\dl_1(u)$ in \eqref{E1}; see also \eqref{R1})
 which all vanish in the deep-water limit.
 As a result,
 the number of terms (including  quadratic ones;
 compare the quadratic parts in~\eqref{E00}
 and~\eqref{E2})
 becomes much larger than the BO case,
 requiring
 careful combinatorial computations.
Further details on the structure of these conservation laws can be found in Section~\ref{SEC:cons1}.

In order to prove the desired deep-water convergence \eqref{E4}, 
we introduce the notion of the {\it deep-water rank}
in  Definition \ref{DEF:ord1} (see also Definition \ref{DEF:mono1})
which allows us to organize the combinatorially complex structure
of the monomials appearing in the deep-water conservation laws.
Moreover, following
the recent works \cite{IS23, CLOP, CFLOP},
we employ a perturbative viewpoint
by introducing 
 the perturbation operator $\Qdl$:\footnote{The perturbation operator considered 
 in \cite{IS23, CLOP, CFLOP} is 
 slightly different from $\Qdl$ in~\eqref{Qdl1}
 and is given by 
  $\Tdl \dx - \H\dx$, where $\Tdl$ is the Tilbert transform defined in \eqref{Til0}.}
\begin{align}
 \Qdl = \Gdl\dx - \H\dx.
\label{Qdl1}
\end{align}

\noi
In \cite{IS23, CLOP, CFLOP},
by rewriting 
 ILW \eqref{ILW}
as
\begin{equation}
\dt u - \H\dx^2 u =\dx(u^2) + \Qdl \dx u
\label{ILWx}
\end{equation}

\noi
and exploiting a strong smoothing property of the perturbation operator
$\Qdl$ (see Lemma~2.2 in~\cite{IS23}
and Lemma~2.3 in~\cite{CLOP}), 
the authors of the aforementioned works established results on
low-regularity well-posedness of ILW, deep-water convergence, 
and 
long-time behavior of ILW solutions.
For our analysis, 
we use this perturbative viewpoint 
to replace the recursive relation~\eqref{chi1}
for the microscopic conservation laws $\chi_n^\dl$, 
involving $\Gdl$, 
by the recursive relation~\eqref{DE6}, 
involving $\H$ and $\Qdl$.
Together with the deep-water rank, 
the strong vanishing of $\Qdl$ (uniform in frequencies) as $\dl \to \infty$ (see \eqref{Q1})
will then allow us to prove~\eqref{E4}.

\medskip

Next,
we  consider the shallow-water regime.
As mentioned above,  the conservation laws
$E^\dl_{\frac{k}{2}}(u) $ in \eqref{E00}
under the scaling \eqref{scale1}
yield conservation laws of the scaled ILW \eqref{sILW}.
They are, however, not suitable for studying the shallow-water convergence,
since multiple terms diverge as $\dl \to 0$.
In order to overcome this difficulty,
we use an alternative B\"acklund transform~\eqref{BTX1}
derived in the works
\cite{GK80, Kupershmidt81}
by Gibbons and  Kupershmidt, 
from which we obtain 
a recursive relation 
for generating shallow-water microscopic conservation laws $h_n^\dl$;
see~\eqref{h1}.
There are two issues in directly using 
$\int_\T h_n^\dl dx $ to construct shallow-water conservation laws:

\smallskip
\begin{itemize}
\item[(i)]
the quadratic part comes with terms of different regularities, 
just as we saw in the deep-water regime,

\smallskip
\item[(ii)] the odd-order conservation laws
$\int_\T h_{2n+1}^\dl dx $ vanish
in the shallow-water limit; see Lemma~\ref{LEM:kdv1}\,(i).
 This in particular causes an issue
 in pursuing our statistical study
 since the corresponding \GGMs~
 would not have any meaningful limit.

\end{itemize}

\smallskip

\noi
In order to overcome these issues, 
we construct  new
microscopic conservation laws
$\wt h_{n}^\dl$ from $h_n^\dl$ (see~\eqref{ht0}),
which 
allows us to construct
a full set of  shallow-water conservation laws 
 with non-trivial shallow-water limits.
As a result, 
 we arrive at the following conservation laws $\big\{\wt E^\dl _{\frac k2}(v)\big\}_{k \in \Z_{\ge 0}}$  for the scaled ILW:
\begin{equation}
\begin{aligned}
\wt E^\dl _{0}(v) & =
\frac12 \|v\|^2_{L^2}, \\
\wt{E}^\dl_{\kk - \frac 12}(v)
&= \frac12 \sum_{\substack{\l=1\\\text{odd}}}^{2\kk-1}
\wt a_{2\kk-1, \l}\,
\dl^{\l-1} \|\Gd^{\frac{\l}{2}} v \|^2_{\dot{H}^{\kk - \frac 12}} + \wt{R}^\dl _{\kk - \frac 12}(v), \\
\wt{E}^\dl_{\kk}(v)
& = \frac12 \sum_{\substack{\l=0\\ \text{even}}}^{2\kk}
\wt a_{2\kk, \l}\,
 \dl^{\l} \| \Gd^{\frac \l 2} v \|^2_{\dot{H}^{\kk}} + \wt{R}^\dl_{ \kk }(v)
\end{aligned}
\label{E5}
\end{equation}

\noi
for $\kk \in \N$,
where, for odd $k$ and $\l$,  we used the convention:
\begin{align}
 \|\Gd^{\frac{\l}{2}} v \|^2_{\dot{H}^{\frac k2}}
  = \int_\T \dx^\frac{k-\l}{2}v \cdot (\Gd \dx)^\l \dx^\frac{k-\l}{2}v dx
\label{odd2}
\end{align}

\noi
by noting that 
$\Gd \dx$ is a self-adjoint and positive operator with 
an even multiplier
(Lemma~\ref{LEM:K1});
see also \eqref{odd1}
for an analogous convention.
Here, 
$\wt a_{k, \l}$ is 
a  positive constant (see \eqref{DEs3a})
with the following normalization on the leading coefficients:\footnote{The normalization $\wt a_{2\kk-1, 1}= 3$
is needed to cancel the factor $\frac 13$ which appears
in the shallow-water limit of $in \dl^{-1} \ft \Gdl(n)$; see Lemma \ref{LEM:L1}.
See also \eqref{T2} and \eqref{T3}.}
\begin{align}
\wt a_{2\kk-1, 1}= 3\qquad \text{and}\qquad \wt a_{2\kk, 0} = 1,
\label{E5a}
\end{align}

\noi
and the interaction potentials
$\wt{R}^\dl _{\kk - \frac 1 2}(v)$
and $\wt{R}^\dl _{\kk }(v)$
 consist of  cubic and higher order terms in $v$;
 see \eqref{Scons2}.
Here, we relabelled the indices
for
the conservation laws
(from $\frac k2$  to $\kk - \frac 12$
and $\kk$), 
since the structure of the conservation law
$\wt E^\dl _{\frac k2}(v)$
depends
sensitively
on the parity of $k$
(namely, $k = 2\kk - 1$ or $k = 2\kk$, $\kk \in \N$)
in the current shallow-water regime.
Moreover,
this relabelling allows us to
 emphasize that
$\wt{E}^\dl_{ \kk - \frac 12}(v) $
and $\wt{E}^\dl_{\kk}(v) $
appear in a pair; see \eqref{E7}.
See
 Section \ref{SEC:cons2}
for further  details on the derivation and the structure of the
shallow-water 
 conservation laws  $\wt{E}^\dl_{\kk - \frac 1 2}(v)$
 and  $\wt{E}^\dl_{\kk}(v)$
  in~\eqref{E5}.

By taking a formal limit of the B\"acklund transform
\eqref{BTX1} for the scaled ILW as $\dl \to 0$,
we obtain
the B\"acklund transform
\eqref{BTkdv} for KdV
(= the Gardner transform),
from which we derive  the following conservation laws for KdV:
\begin{align}
\begin{split}
\wt  E^\KDV _{0}(v) & =
\wt  E^0_{0}(v)
= \frac12 \|v\|^2_{L^2}, \\
\wt E^\KDV_{\kk}(v)
& = \wt  E^0 _{\kk}(v)
= \frac12 \| v\|_{\dot H^\kk}^2 + \wt R^\KDV_{\kk}(v), \quad \kk \in \N,
\end{split}
\label{E6}
\end{align}

\noi
where
the interaction potential 
 $\wt R^\KDV _{\kk}(v)$ consists of  cubic and higher order terms in $v$;
 see~\eqref{DEs4}.

We now state a result on convergence properties of
the shallow-water conservation laws, 
which is one of the main novelties of the current work.

\begin{theorem}[shallow-water regime]\label{THM:2}
Let $\kk \in \N$.
Then,
given  $0 < \dl <  \infty$,
the quantities 
$\wt  E^\dl_{0}(v) $, 
$\wt  E^\dl_{\kk - \frac 12}(v) $, 
and $\wt  E^\dl_{\kk}(v) $
defined in \eqref{E5}
are conserved under the flow of the scaled  ILW~\eqref{sILW}.
Similarly,
the quantities 
$ \wt  E^\KDV_{0}(v) $ and 
$ \wt  E^\KDV_{\kk}(v) $
defined in~\eqref{E6}
are  conserved under the flow of KdV~\eqref{kdv}.
Moreover,
 for any $v\in H^{\kk}(\T)$, we have
\begin{align}
\lim_{\dl\to0} \wt{E}^\dl_{\kk - \frac 12}(v) &=
\lim_{\dl\to0} \wt{E}^\dl_{ \kk }(v)
  = \wt  E^\KDV_{\kk}(v).
\label{E7}
\end{align}

\end{theorem}

Theorem \ref{THM:2} establishes a
novel 2-to-1 collapse
of the shallow-water conservation laws
 to the KdV conservation laws, 
 presenting
  the first construction of a complete family of 
 shallow-water conservation laws with non-trivial shallow-water limits;
 see Remark \ref{REM:cons2} below.
From~\eqref{E7} with \eqref{E5},
we see that while there is no regularity jump
in the convergence of
$\wt{E}^\dl_{ \kk }(v)$
to
$\wt  E^\KDV_{\kk}(v)$,
there is a regularity jump by $\frac 12$
in the convergence of
$\wt{E}^\dl_{\kk - \frac 12}(v)$
to
$\wt  E^\KDV_{\kk}(v)$,
exhibiting a singular nature of the shallow-water limit.
See Figure \ref{FIG:1}.

\begin{figure}[h]
\begin{tikzpicture}
%
\path
 node (A){$\wt E^\dl_{\kk}(v)$}
    (0:4cm)
    node (K) {$\wt  E^\KDV_{\kk}(v)$}
    (-90:1.5cm)
    node (B)
     {$\wt E^\dl_{\kk - \frac 12}(v)$} ;
 \path[thick, -stealth]
   (B)      edge[=>]
     node [sloped, below] {$\dl  \to 0$}
(K)
(A) edge              node [above] {$\dl \to 0$}     (K);

\end{tikzpicture}

\caption{While there is no regularity jump
in the convergence of
$\wt{E}^\dl_{ \kk }(v)$
to
$ \wt  E^\KDV_{\kk}(v)$,
there is a regularity jump by $\frac 12$
in the convergence of
$\wt{E}^\dl_{\kk - \frac 12}(v)$.}
\label{FIG:1}
\end{figure}


See 
Propositions \ref{PROP:Scons1}
and \ref{PROP:Scons2}
for a proof of Theorem \ref{THM:2}.
As mentioned above, 
we prove Theorem \ref{THM:2}
by starting with the B\"acklund transform \eqref{BTX1}, 
which generates  microscopic conservation laws $h_n^\dl$
(see~\eqref{h1}).
We  then 
construct  new
microscopic conservation laws
$\wt h^\dl_{n}$ from $h^\dl_n$ (see~\eqref{ht0})
which allows us to overcome the issues (i) and (ii) mentioned above;
see Lemma \ref{LEM:ht2} and
Propositions \ref{PROP:Scons1}
and \ref{PROP:Scons2}.
This process and the analysis on the structure
of $\wt h_n^\dl$
 involve lengthy algebraic, combinatorial computations.
In order to handle the combinatorial complexity, 
we introduce the notion of the {\it shallow-water rank}
in Definition \ref{DEF:ord2} (see also Definition \ref{DEF:mono2}).
Furthermore, as in the deep-water case, we also employ a perturbative viewpoint
by introducing the perturbation operator $\Qd$:
\begin{align}
\Qd = \Gd + \frac 13 \dx.
\label{Qdl2}
\end{align}

\noi
We point out that our argument for proving 
the shallow-water convergence \eqref{E7}
is more subtle than the deep-water case
due to 
 the slower convergence of the perturbation operator $\Qd$ 
 (in fact, of $\Qd \dx$ which is not uniform in frequencies; see Lemma \ref{LEM:L1}\,(iv))
and also insufficient powers of $\dl$ (tending to $0$).

\begin{remark}\label{REM:cons2}\rm
(i)
The numbers of
the microscopic conservation laws
for the scaled ILW and KdV coincide.
However,
upon integration over the torus,
 the odd-order microscopic conservation laws for KdV  vanish
 (see Lemma \ref{LEM:kdv1}~(i)),
 while those for the scaled ILW (whether $h_n^\dl$ defined in \eqref{h1}
 or $\wt h_n^\dl$ defined in \eqref{ht0}) do not vanish,
 which
 leaves KdV  with half as many (non-trivial) conservation laws as the scaled ILW.

\smallskip

\noi
(ii) The family $\big\{\wt E^\dl_{\frac k2}(v)\big\}_{k \in \Z_{\ge 0}}$
(and
$\big\{\wt E^\KDV_{\kk}(v)\big\}_{\kk \in \Z_{\ge 0}}$)
is the only possible polynomial conservation laws
of the scaled ILW (and KdV, respectively)
in the sense that
any other polynomial conservation laws are given
as their linear combination;
see \cite[Theorem 4]{M3}
and 
\cite[Theorems~2.16 and 2.28]{Kupershmidt81}.

\end{remark}

\subsection{Generalized Gibbs measures
associated with the conservation laws of ILW}
\label{SUBSEC:1.3}

%
We are now ready to state our results
on the  statistical study of ILW.
In this subsection, we discuss
the construction and convergence properties
of \GGMs~ associated with the conservation laws
 of the (scaled) ILW.
We discuss the dynamical problem in the next subsection.

We first consider the deep-water regime,
including the BO case.
Given $0 < \dl \le \infty$ and $k \in \Z_{\ge 0}$,
let $E^\dl_\frac{k}2(u)$
be as in \eqref{E00} and \eqref{E2}.
Our first goal is to construct a \GGM~ 
associated with $E^\dl_\frac{k}2(u)$
whose   formal density
 is given by\footnote{Hereafter, we use
$Z_{\dl, \frac{k}{2}}$, etc. to denote various normalization constants
which may vary line by line.}
\begin{align}
\rho^\dl_{\frac{k}{2}}(du) &= Z_{\dl, \frac{k}{2}}^{-1} \exp\Big(-E^\dl_{\frac{k}{2}}(u)\Big)  du
\label{rho1}
\end{align}

\noi
as a probability measure on periodic functions\,/\,distributions.
When $k = 0$, this corresponds to the white noise on $\T$ with formal density:
\begin{align}
\rho_0(du) = Z_0^{-1}e^{-\frac 12 \|u\|_{L^2}^2} du.
\label{white1}
\end{align}

\noi
Since it is independent of $0 < \dl \le \infty$,
there is nothing to study regarding its convergence as $\dl \to \infty$.
Moreover, the white noise is supported
on $H^{-\frac 12-\eps }(\T)$,
where well-posedness of ILW and BO is completely open.
Hence, in the remaining part of this paper,
we focus on $k \ge 1$; see Remark \ref{REM:white}.
See also \cite{OH1, OH6, BO, OQV, OQS2}
for a further discussion on the white noise.

Let $k \in \N$.
Then, from \eqref{rho1} with \eqref{E00} and \eqref{E2},
we can formally write $\rho^\dl_{\frac{k}{2}}$ as
the following weighted Gaussian measure:
\begin{align}
\rho^\dl_{\frac{k}{2}}(du)
&= Z_{\dl, \frac{k}{2}}^{-1} \exp\Big(-R^\dl_{\frac{k}{2}} (u) \Big) d\mu^\dl_{ \frac k 2}(u),
\label{rho2}
\end{align}

\noi
where
$\mu^\dl_{ \frac k 2}$ is the base Gaussian measure with formal density:
\begin{align}
\mu^\dl_{\frac{k}{2}} (du) = Z_{\dl, \frac{k}{2}}^{-1}
\exp\bigg(-\frac 12 \sum_{\substack{\l=0\\  \text{even}}}^{k} a_{k, \l} \| \Gdl^{\frac{k-\l}{2}} u \|^2_{\dot{H}^{\frac{k}{2}}}\bigg)  du.
\label{gauss1}
\end{align}

\noi
More precisely,
$\mu^\dl_{ \frac k 2}$
is defined as the induced probability measure under the map:
\begin{align}
\omega \in \Omega \longmapsto X^\dl_{\frac{k}{2}}(\o) =\frac{1}{\sqrt{2\pi}} \sum_{n\in \Z^*} \frac{g_n(\omega)}{\big(T_{\dl,\frac{k}{2}}(n)\big)^\frac12} e_n,
\label{Xdl1}
\end{align}
where
$\Z^* = \Z\setminus \{0\}$, $e_n(x) = e^{inx}$,
$\{g_n\}_{n\in\Z^*}$ is a sequence of independent standard\footnote{By convention,
we assume that  $g_n$
has mean 0 and variance $2$, $n \in \Z^*$.}  complex-valued Gaussian random variables
on a probability space $(\O, \F, \PP)$
conditioned that $g_{-n} = \cj{g_n}$, $n \in \Z^*$,
and
\begin{align}
T_{\dl, \frac{k}{2}}(n) = \sum_{\substack{\l=0\\ \text{even}}}^{k} a_{k, \l}
|n|^{k} | \ft{\Gdl}(n) |^{k-\l}
\label{T1}
\end{align}

\noi
with  $\ft{\Gdl}(n)$ as in \eqref{GG1}.
In view of \eqref{E1a}
(see also \eqref{DX1}),
one can easily show that $\mu^\dl_{\frac{k}{2}}$
is supported on $H^{\frac{k-1}2 - \eps}(\T)$
but that $\mu^\dl_{\frac{k}{2}}\big( H^{\frac{k-1}2}(\T) \big) = 0$.
When $\dl = \infty$, it follows from \eqref{E0}
with $\lim_{\dl \to \infty} \Gdl = \H$
that
$T_{\infty, \frac{k}{2}} (n)  = |n|^{k}$,
and thus
 the base Gaussian measure is given by
\begin{align}
\mu^\infty_{\frac{k}{2}} (du) = Z_{\infty, \frac{k}{2}}^{-1} \exp\bigg(- \frac 12 \| u \|^2_{\dot{H}^{\frac{k}{2}}}\bigg)  du.
\label{gauss2}
\end{align}

The  interaction potential
$R^\dl_{\frac{k}{2}} (u)$, consisting of cubic and higher order terms in $u$,
is not sign-definite (see \eqref{R1} below), exhibiting super-Gaussian growth.
As a result, the expression in \eqref{rho2}
can not be normalized to be a probability measure.
Following the work \cite{LRS, BO94, OST} on the construction of focusing Gibbs measures,
we instead consider a weighted Gaussian measure  with an $L^2$-cutoff.
More precisely,
let $\eta \in C^\infty([0, \infty); [0, 1])$
be a smooth cutoff function
supported on $[0,2]$
 with $\eta\equiv 1$ on $[0,1]$.
 Then, given $0 < \dl \le \infty$ and $K > 0$, we aim to construct the 
 deep-water 
 \GGM~ of the form:\footnote{A precise
 value of $K  > 0$ does not play any role in the remaining part of this paper
 and hence, for simplicity of notation,
 we suppress the $K$-dependence of
$\rho^\dl_{\frac{k}{2}}$
and $F^\dl_{\frac{k}{2}}$
in \eqref{rho3}.
We apply the same convention to  various objects throughout the paper.}
\begin{align}
\begin{split}
\rho^\dl_{\frac{k}{2}}(du)
& = Z^{-1}_{\dl, \frac k 2} \,
F^\dl_{\frac{k}{2}}(u)
 d\mu^\dl_{\frac k 2}(u)\\
:\! & = Z^{-1}_{\dl, \frac k 2} \, \eta_K\big(\| u\|_{L^2}\big)
\exp\Big(-R^\dl_{\frac k 2}(u)\Big)  d\mu^\dl_{\frac k 2}(u),
\end{split}
\label{rho3}
\end{align}

\noi
where $\eta_K(\,\cdot\,)=\eta(K^{-1}\, \cdot\,)$.
When $k = 1$, the support of
$\mu^\dl_{\frac{k}{2}}$ is outside $L^2(\T)$
and hence we need to replace the $L^2$-cutoff by a renormalized $L^2$-cutoff.
Since the $k = 1$ is already treated in~\cite{Tz10, OST2, LOZ},
we focus on $k \ge 2$ and suppress this issue; see \cite{LOZ}
for a further discussion
on the $k = 1$ case.

We construct the \GGM~ $\rho^\dl_\frac k2$ as a (unique) limit
of its frequency-truncated versions.
Given $N\in\N$, let $\P_N$ be the Dirichlet projector onto spatial frequencies $\{|n| \leq N\}$;
see~\eqref{Diri1}.
Then,
define the truncated \GGM~
$\rho^\dl_{\frac{k}{2}, N}$ by
\begin{align}
\begin{split}
\rho^\dl_{\frac{k}{2}, N}(du) & = Z^{-1}_{\dl, \frac{k}{2}, N}  \,
F^\dl_{\frac{k}{2}}(\P_N u)  d\mu^\dl_{\frac{k}{2}}(u)\\
 & = Z^{-1}_{\dl, \frac k 2,  N} \, \eta_K\big(\| \P_Nu\|_{L^2}\big)
\exp\Big(-R^\dl_{\frac k 2}(\P_Nu)\Big)  d\mu^\dl_{\frac k 2}(u).
\end{split}
\label{rho4}
\end{align}

\noi
The following theorem shows
that, as $N\to\infty$,
$\rho^\dl_{\frac k 2, N}$ converges in total variation
to
 $\rho^\dl_{\frac{k}{2}}$ in~\eqref{rho3}
 with a uniform rate of convergence 
 for  $2\le \dl \le \infty$.
 Moreover, we establish 
 deep-water convergence (as $\dl \to \infty$) of
  $\rho^\dl_{\frac{k}{2}}$.

\begin{theorem}[deep-water regime] \label{THM:3}

Let $k \ge 2$ be an integer and $K > 0$. Then, the following statements hold.

\smallskip

\noi{\rm(i)} Let $0<\dl\leq \infty$.
Then, given any finite $p\ge 1$, we have
\begin{equation}
\lim_{N\to \infty} F^\dl_{\frac{k}{2}}(\P_N u) = F^\dl_{\frac{k}{2}}(u)
 \quad \text{in } L^p(d\mu^\dl_{\frac{k}{2}}),
\label{conv1}
\end{equation}

\noi
where the convergence is uniform in $2 \le \dl \le \infty$.
As a consequence, as $N \to \infty$, 
the truncated \GGM~ $\rho^\dl_{\frac{k}{2}, N}$ in \eqref{rho4} converges, in the sense
of~\eqref{conv1},  to the limiting \GGM~ $\rho^\dl_{\frac{k}{2}}$ in~\eqref{rho3}.
In particular, $\rho^\dl_{\frac{k}{2}, N}$ converges to $\rho^\dl_{\frac{k}{2}}$ in total variation,
and
the limiting \GGM~ $\rho^\dl_{\frac{k}{2}}$ is
equivalent
 to the Gaussian measure with the $L^2$-cutoff
$\eta_K\big(\| u\|_{L^2}\big) d \mu^\dl_\frac k2(u)$.

\smallskip

\noi {\rm(ii)}
For $0<  \dl <\infty$, the \GGMs~ $\rho^\dl_{\frac{k}{2}}$  and
$\rho^\infty_{\frac{k}{2}}$  are equivalent.
Moreover,
 as $\dl \to \infty$,
 $\rho^\dl_{\frac{k}{2}}$ converges to $\rho^\infty_{\frac{k}{2}}$ in total variation.

\end{theorem}

In \cite{LOZ}, the second and third authors with G.\,Zheng
proved an analogue of Theorem~\ref{THM:3}
for $k = 1$, corresponding to the Gibbs measures.
Theorem \ref{THM:3} extends the result in \cite{LOZ}
to any $k \ge 2$.
In \cite{TV0},
Tzvetkov and Visciglia
constructed
the \GGMs~ associated with
the  higher order conservation laws of the BO equation \eqref{BO}.
Their measures were constructed
with smooth cutoffs on  the lower order conservation laws
 $E^\BO_{\frac j 2}(u)$, for $j=0, \ldots, k-1$
 (with a renormalization for $j = k -1$).
 In Theorem \ref{THM:3},
 we only impose an $L^2$-cutoff,
 thus providing 
 a significant simplification.
We also point out that
our  construction of the \GGM~  $\rho^\dl_\frac k2$
only with  an $L^2$-cutoff
also simplifies the invariance argument in the proof of Theorem~\ref{THM:5}
as compared to \cite{TV1, TV2};
see Remarks~\ref{REM:cutoff} and \ref{REM:TV} below.
A similar comment applies to the shallow-water regime.
Our simplified construction of the \GGM~
is based on 
the Bou\'e-Dupuis variational formula~\cite{BD98, Ust14}, 
recently popularized by Barashkov and Gubinelli \cite{BG}.
See also 
the (non-)construction of focusing Gibbs measures in \cite{OOT, OST2, OOT2, TW}
via the variational formula.

The construction of the \GGM~ $\rho^\dl_\frac k2$
for each fixed $0 < \dl \le \infty$
follows from
establishing a uniform (in $N$) estimate on  the truncated densities
 $F^\dl_{\frac{k}{2}}(\P_N u) $
  in $L^p(d\mu^\dl_\frac k2)$.
On the other hand,
in order to
prove the deep-water convergence in Theorem \ref{THM:3}\,(ii),
we need to   establish an $L^p(d\mu^\dl_\frac k2)$-estimate
on
 $F^\dl_{\frac{k}{2}}(\P_N u) $,
{\it uniformly in both $N \in \N$ and $\dl \gg1$}.
Our detailed analysis 
in Section \ref{SEC:cons1}
on the structure
of the deep-water conservation laws
and 
a very robust  variation formula
(Lemma \ref{LEM:var0})
 allow us to 
 achieve these goals, working
 only with the $L^2$-cutoff.
We also
note that
the base Gaussian measures $\mu^\dl_\frac k2$ are different
  for different values of $\dl$.
  In order to overcome this issue, we indeed
  work with
the random variable $X^\dl_{\frac{k}{2}}$ in~\eqref{Xdl1}
and
  estimate
$F^\dl_{\frac{k}{2}}(\P_N X^\dl_{\frac{k}{2}}) $ in $L^p(\O)$,
 uniformly in both  $N \in \N$ and $\dl \gg1$.

Once we obtain such a uniform $L^p$-bound,
the desired deep-water convergence
essentially follows from establishing
the corresponding deep-water convergence of
the Gaussian measure $\mu^\dl_\frac k2$;
see Proposition~\ref{PROP:gauss1},
where we in fact prove convergence of~$\mu^\dl_\frac k2$ in a stronger topology
(namely,  in the Kullback-Leibler divergence
defined in Definition \ref{DEF:KL}).
As for equivalence of the \GGMs~
with different depth parameters, 
we first prove
equivalence of the Gaussian measures $\mu^\dl_\frac k2$
and $\mu^\infty_\frac k2$
via
Kakutani's theorem (Lemma \ref{LEM:kak}),
using the Gaussian Fourier series representation~\eqref{Xdl1}.
Then, the claimed equivalence
of  $\rho^\dl_{\frac{k}{2}}$ and
$\rho^\infty_{\frac{k}{2}}$
follows
from Theorem \ref{THM:3}\,(i).

\medskip

We now turn  to the shallow-water regime.
Due to the more complicated structure of the shallow-water conservation laws in \eqref{E5}
and the singular nature of their shallow-water convergence,
the construction and convergence of the \GGMs~
associated with these conservation laws require more
careful analysis.
For $0 < \dl < \infty$, we first define the base Gaussian measure $\wt{\mu}^\dl_{\frac k 2}$,
$k \in \N$,
associated with the quadratic part
of the shallow-water conservation law
$\wt E^\dl _{\frac k2}(v)$ in \eqref{E5},
by setting, for $\kk \in \N$,
\begin{align}
\begin{split}
\wt{\mu}^\dl_{\kk - \frac 12} (dv)
& = \wt{Z}^{-1}_{\dl,\kk - \frac 12} \exp\bigg(
- \frac12 \sum_{\substack{\l=1\\\text{odd}}}^{2\kk-1}
\wt a_{2\kk-1, \l}\,
\dl^{\l-1} \|\Gd^{\frac{\l}{2}} v \|^2_{\dot{H}^{\kk - \frac 12}}
\bigg) dv,\\
\wt{\mu}^\dl_{\kk} (dv)
& = \wt{Z}^{-1}_{\dl,\kk} \exp\bigg(
-  \frac12 \sum_{\substack{\l=0\\ \text{even}}}^{2\kk}
\wt a_{2\kk, \l}\,
 \dl^{\l} \| \Gd^{\frac \l 2} v \|^2_{\dot{H}^{\kk}}
\bigg) dv.
\end{split}
\label{gauss3}
\end{align}

\noi
Namely,
$\wt{\mu}^\dl_{\kk - \frac 12}$
and $\wt{\mu}^\dl_\kk$ are defined as the induced probability measures under the following maps:
\begin{align}
\begin{split}
\omega \in \Omega & \longmapsto \wt X^\dl_{\kk - \frac 12}(\o)
 =\frac{1}{\sqrt{2\pi}} \sum_{n\in \Z^*} \frac{g_n(\omega)}{\big(\wt T_{\dl,\kk - \frac{1}{2}}(n)\big)^\frac12} e_n,\\
\omega \in \Omega & \longmapsto \wt X^\dl_{\kk}(\o) =\frac{1}{\sqrt{2\pi}} \sum_{n\in \Z^*} \frac{g_n(\omega)}{\big(\wt T_{\dl, \kk }(n)\big)^\frac12} e_n,
\end{split}
\label{Xdl2}
\end{align}

\noi
respectively,
where
$\{g_n\}_{n\in\Z^*}$ is as in \eqref{Xdl1}
and
\begin{align}
\begin{split}
\wt T_{\dl, \kk - \frac{1}{2}}(n) & =
\sum_{\substack{\l=1\\\text{odd}}}^{2\kk-1}
\wt a_{2\kk-1, \l}\,
\dl^{\l-1} |n|^{2\kk-1} |\ft \Gd(n)|^\l,  \\
\wt T_{\dl, \kk }(n) & =
 \sum_{\substack{\l=0\\ \text{even}}}^{2\kk}
\wt a_{2\kk, \l}\,
 \dl^{\l}
 |n|^{2\kk} |\ft \Gd(n)|^\l.
\end{split}
\label{T2}
\end{align}

\noi
Here,   $\ft{\Gd}(n)$ is the multiplier for $\Gd$ in \eqref{sILW}, 
given by
\begin{align}
\ft{\Gd}(n) =  \dl^{-1} \ft \Gdl(n),
\label{Gd1}
\end{align}

\noi
where $\ft \Gdl(n)$ is as in \eqref{GG1}.
In view of \eqref{T6}, 
one can show that $\wt \mu^\dl_{\frac{k}{2}}$, $k \in \N$,
is supported on $H^{\frac{k-1}2 - \eps}(\T)$
but that $\wt \mu^\dl_{\frac{k}{2}}\big( H^{\frac{k-1}2}(\T) \big) = 0$.

As in the deep-water regime, we restrict our attention to $k \ge 2$
in the following such that the Gaussian measure $\wt \mu^\dl_{\frac k2}$
is supported on $L^2(\T)$.
Then, given $K > 0$,
we define the shallow-water  \GGM~
$\wt{\rho}^\dl_{\frac k2 }$
associated with $\wt E^\dl_\frac k2(v)$ in \eqref{E5}
by setting
\begin{align}
\begin{split}
\wt{\rho}^\dl_{\frac k2}(d v)
& = \wt{Z}^{-1}_{\dl, \frac k2 }\, \wt F^\dl_\frac k2(v)
 d\wt{\mu}^\dl_{ \frac k 2}(v)\\
: \!& = \wt{Z}^{-1}_{\dl, \frac k2 } \, \eta_K\big(\|v\|_{L^2}\big)
\exp\Big( - \wt{R}^\dl_{ \frac k 2}(v) \Big) d\wt{\mu}^\dl_{ \frac k 2}(v).
\end{split}
\label{rho5}
\end{align}

\noi
Given $N \in \N$,
we also define their frequency-truncated version:
\begin{align}
\begin{split}
\wt \rho^\dl_{\frac{k}{2}, N}(dv)
& = \wt Z^{-1}_{\dl, \frac{k}{2}, N} \,  \wt F^\dl_{\frac{k}{2}}(\P_N v)  d\wt \mu^\dl_{\frac{k}{2}}(v).
\end{split}
\label{rho7}
\end{align}

Let us now discuss the $\dl = 0$ case.
We note that, in taking a limit of
$\wt T_{\dl, \kk - \frac{1}{2}}(n)$
(and $\wt T_{\dl, \kk}(n)$) in \eqref{T2}
as $\dl \to 0$,
all the terms except for $\l = 1$ (and $\l = 0$, respectively)
vanish.
In fact, we have 
\begin{align}
\lim_{\dl \to 0}
\wt T_{\dl, \kk - \frac{1}{2}}(n)  =
\lim_{\dl \to 0}
\wt T_{\dl, \kk }(n)  =
 n^{2\kk}.
\label{T3}
\end{align}

\noi
See 
 Lemma \ref{LEM:Tdl}.
Hence, for  $\kk \in \N$,
we define the \GGM~
$\wt \rho^0_\kk$
associated with
the KdV conservation laws $\wt E^\KDV_{\kk}(v)$ in~\eqref{E6}
by setting
\begin{align}
\begin{split}
\wt{\rho}^0_{\kk}(dv)
& = \wt{Z}^{-1}_{0, \kk }\, \wt F^0_\kk(v)
 d\wt{\mu}^0_{\kk}(v)\\
& = \wt{Z}^{-1}_{0,\kk}\, \eta_K\big(\|v\|_{L^2}\big) \exp\Big(-\wt{R}^0_{\kk}(v)\Big)
d\wt{\mu}^0_{\kk}(v)
\end{split}
\label{rho8}
\end{align}

\noi
along with their truncated version:
\begin{align*}
\wt{\rho}^0_{\kk, N}(dv)
 = \wt{Z}^{-1}_{0, \kk, N }\, \wt F^0_\kk(\P_N v)
 d\wt{\mu}^0_{\kk}(v),
\end{align*}

\noi
where
\begin{align}
\wt \mu^0_\kk := \mu^\infty_{\kk}
\label{gauss4}
\end{align}

\noi
is the Gaussian measure
defined in
\eqref{gauss2} with $k = 2\kk$.
Namely,
$\wt \mu^0_\kk$
is  the induced probability measure under the map:
\begin{align}
\o \in\O \longmapsto \wt X^0_\kk(\o) = 
\frac{1}{\sqrt{2\pi}} \sum_{n\in \Z^*} \frac{g_n(\omega)}{\big(\wt T_{0, \kk }(n)\big)^\frac12} e_n
: = 
\frac{1}{\sqrt{2\pi}} \sum_{n\in\Z^*} \frac{g_n(\o)}{|n|^\kk} e_n.
\label{Xdl3}
\end{align}

We now state our main result
on the construction and further properties
of the shallow-water \GGMs~ $\wt \rho^\dl_{\frac k2}$ in \eqref{rho5}.

\begin{theorem}[shallow-water regime]
\label{THM:4}
Let $k \ge 2$ be an integer and   $K > 0$. Then, the following statements hold. 

\smallskip

\noi{\rm(i)} Let
$0\le\dl <\infty$.
 Then, given any finite $p\ge 1$,  we have
\begin{equation}
\lim_{N\to \infty} \wt{F}^\dl_{\frac k 2}(\P_N v)
= \wt{F}^\dl_{\frac k 2}(v) \quad \text{in } L^p(d\wt{\mu}^\dl_{\frac k2}),
\label{conv2}
\end{equation}

\noi
where the convergence is uniform  in $0<  \dl\le 1$.
As a consequence, as $N \to \infty$,  the truncated \GGM~
$\wt{\rho}^\dl_{\frac k 2, N}$ in \eqref{rho7} converges, in the sense
of~\eqref{conv2},  to the limiting \GGM~
 $\wt{\rho}^\dl_{\frac k 2}$ in~\eqref{rho5}.
In particular, $\wt \rho^\dl_{\frac{k}{2}, N}$ converges to $\wt \rho^\dl_{\frac{k}{2}}$ in total variation,
and the limiting \GGM~
$\wt \rho^\dl_{\frac{k}{2}}$ is
equivalent
 to the Gaussian measure with the $L^2$-cutoff
$\eta_K\big(\| v\|_{L^2}\big) d \wt \mu^\dl_\frac k2(v)$.
When $k$ is even, the results extend to $\dl = 0$.

\smallskip

\noi {\rm(ii)}
Let $0<\dl<\infty$.
Then,  the \GGMs~ $\wt{\rho}^{\dl}_{ \frac k2}$
and
$\wt{\rho}^0_{\lceil \frac k2\rceil}$
are  singular.
Moreover, as $\dl \to 0$,
 the measure $\wt{\rho}^\dl_{\frac k2}$  converges weakly to $\wt{\rho}^0_{\lceil \frac k2\rceil}$.
Here,   $\lceil x \rceil$ denotes
the smallest integer greater than or equal to $x$.

\smallskip

\noi {\rm(iii)}
Let   $0<\dl_1, \dl_2 <\infty$.
When $k$ is odd,  the odd-order \GGMs~ $\wt{\rho}^{\dl_1}_{\frac k2}$
and
$\wt{\rho}^{\dl_2}_{\frac k2}$
are singular.
On the other hand, 
when $k$ is even, 
the even-order \GGMs~
 $\wt{\rho}^{\dl_1}_{\frac k2}$ and 
$\wt{\rho}^{\dl_2}_{\frac k2}$ are equivalent.

\end{theorem}

Theorem \ref{THM:4} manifests the singular nature
of the shallow-water convergence in several ways.
First of all,
Theorem \ref{THM:4} establishes a
 2-to-1 collapse
of the \GGMs~
$\wt{\rho}^\dl_{\kk - \frac 12}$ and $\wt{\rho}^\dl_{\kk}$
in the shallow-water limit.
In particular,
in the convergence
of $\wt{\rho}^\dl_{\kk - \frac 12}$
to $\wt{\rho}^0_{\kk}$, there is a regularity jump on the supports of the measures.
A similar regularity jump
of the supports was observed
in  \cite{LOZ} at the level of the scaled ILW Gibbs measure $\wt{\rho}^\dl_{\frac12}$
(constructed with a renormalized $L^2$-cutoff)
converging to the KdV Gibbs measure $\wt \rho^0_1$ (with a shifted $L^2$-cutoff);
see  \cite{LOZ} for further details.
Moreover, for $0 < \dl < \infty$, 
the \GGMs~
 $\wt{\rho}^\dl_{\kk - \frac 1 2}$
and   $\wt{\rho}^\dl_{\kk}$
  are 
 singular with respect to their shallow-water
 limit $\wt{\rho}^0_{\kk}$.
As a consequence,  while the shallow-water convergence takes place,
the convergence holds {\it only weakly},
exhibiting a sharp contrast to the  convergence in total variation observed in the deep-water limit.
We point out that, while
there is no regularity jump in the convergence of $\wt{\rho}^\dl_{\kk}$
to
$\wt{\rho}^0_{\kk}$,
stronger convergence (such as that in total variation)
is not possible due to their   singularity.
Lastly,
Theorem~\ref{THM:4}\,(iii)
exhibits an interesting dichotomy 
on singularity\,/\,equivalence 
of the odd-order and even-order \GGMs~
  with different (positive) depth parameters.

The proof of Theorem \ref{THM:4}\,(i)
is analogous to  that of
Theorem \ref{THM:3}\,(i) in the deep-water regime.
In particular, a key step is
establishing a uniform (in $N$ and also in $0 \le \dl \le 1$)
bound on
 the truncated densities
 $\wt F^\dl_{\frac{k}{2}}(\P_N u) $
  in $L^p(d \wt \mu^\dl_\frac k2)$.
As for Theorem~\ref{THM:4}\,(ii) and (iii),
we establish analogous results
for the base Gaussian measures
(Proposition \ref{PROP:gauss2}),
which are then used to deduce
the claimed results
for the \GGMs, 
following the argument in~\cite{LOZ}.

\begin{remark}\rm
In \cite[Theorem~1.5]{LOZ}, 
 singularity between the shallow-water (renormalized) Gibbs measure
$ \wt{\rho}^\dl_{ \frac 1 2}$, $0 < \dl < \infty$, 
 and its shallow-water limit 
  $\wt{\rho}^0_{ \frac 1 2}$
was established, but  
there was no discussion   
about singularity\,/\,equivalence 
of the Gibbs measures  with different (positive) depth parameters.
By noting that Proposition \ref{PROP:gauss2}
on singularity of the base Gaussian measures
with different depth parameters
also  holds for $k = 1$, 
we can extend Theorem~\ref{THM:4}\,(iii)
to the $k = 1$ case, establishing
  singularity
of the  scaled ILW  Gibbs measures  with different depth parameters.

\end{remark}

\subsection{Dynamical problem}
\label{SUBSEC:1.4}

We now state our main results
on the dynamical problems.
Given a   Banach space $B$,
we endow the space
$C(\R; B)$
with the
compact-open topology in time
(= the topology of uniform convergence on compact sets
in time).

\begin{theorem}
[deep-water regime]
\label{THM:5}

Let $k \geq 3$ be an integer. Then, the following statements hold.

\smallskip
\noi
\textup{(i)}
Let $0 < \dl \le \infty$.
Then,
the weighted Gaussian  measure  $\rho^\dl_{\frac k2}$
in~\eqref{rho3}
is invariant under the flow of ILW \eqref{ILW}
when $0 < \dl < \infty$
and under the flow of BO \eqref{BO} when $\dl = \infty$.

\smallskip
\noi
\textup{(ii)}
Given $0 < \dl \le \infty$,
there exists an $H^{\frac{k-1}{2}-\eps}$-valued random variable
 $u^\dl_0$
with $\Law (u^\dl_0) = \rho^\dl_{\frac k2}$
such that, as $\dl \to \infty$,   $u^\dl_0$ converges almost surely to $u^\infty_0$
in $H^{\frac{k-1}{2}-\eps}(\T)$.
Moreover,
as $\dl \to \infty$,
the unique global solution $u^\dl$ to ILW \eqref{ILW}
\textup{(}with the depth parameter $\dl$\textup{)}
with $u^\dl|_{t= 0} =  u^\dl_0$
converges almost surely
in $C(\R; H^{\frac{k-1}{2}-\eps}(\T))$
to the unique global solution $u$ to BO \eqref{BO}
with $u|_{t= 0} =  u^\infty_0$.

\end{theorem}

\begin{theorem}
[shallow-water regime]
\label{THM:6}

\smallskip
\noi
\textup{(i)}
Let $k\ge 3$ be an integer.
  Given  $0 < \dl <  \infty$,
the weighted Gaussian  measure  $\wt \rho^\dl_{\frac k2}$
in~\eqref{rho5}
is invariant under the flow of the scaled ILW \eqref{sILW}.

\smallskip
\noi
\textup{(ii)}
Let $\kk \ge 2$ be an integer.
The weighted Gaussian  measure  $\wt \rho^0_{\kk}$
in~\eqref{rho8}
is invariant under the flow of KdV \eqref{kdv}.

\smallskip
\noi
\textup{(iii)}
Let $\kk \ge 2$ be an integer.
Given $0 <  \dl < \infty $,
there exists an $H^{\kk - 1-\eps}$-valued random variable
 $v^\dl_0$
with $\Law (v^\dl_0) = \wt \rho^\dl_{\kk- \frac 12}$
such that, as $\dl \to 0$,   $v^\dl_0$ converges almost surely
in $H^{\kk - 1-\eps}(\T)$
to $v_0$ with $\Law (v_0) = \wt \rho^0_{\kk}$.
Moreover,
as $\dl \to 0$,
the unique global solution $ v^\dl$ to the scaled ILW \eqref{sILW}
\textup{(}with the depth parameter $\dl$\textup{)}
with $v^\dl|_{t= 0} =  v^\dl_0$
converges almost surely
in $C(\R; H^{\kk-1-\eps}(\T))$
to the unique global solution $v$ to KdV \eqref{kdv}
with $v|_{t= 0} =  v_0$.

\smallskip
\noi
\textup{(iv)}
Let $\kk \ge 2$ be an integer.
Given $0 <  \dl < \infty $,
there exists an $H^{\kk - \frac 12-\eps}$-valued random variable
 $v^\dl_0$
with $\Law (v^\dl_0) = \wt \rho^\dl_{\kk}$
such that, as $\dl \to 0$,   $v^\dl_0$ converges almost surely
in $H^{k - \frac 12 -\eps}(\T)$
to $v_0$ with $\Law (v_0) = \wt \rho^0_{\kk}$.
Moreover,
as $\dl \to 0$,
the unique global solution $ v^\dl$ to the scaled ILW \eqref{sILW}
\textup{(}with the depth parameter $\dl$\textup{)}
with $v^\dl|_{t= 0} =  v^\dl_0$
converges almost surely
in $C(\R; H^{\kk-\frac 12 -\eps}(\T))$
to the unique global solution $v$ to KdV \eqref{kdv}
with $v|_{t= 0} =  v_0$.

\end{theorem}

Theorems \ref{THM:5} and \ref{THM:6}
establish the first construction of invariant dynamics
 (with uniqueness)
 for the (scaled) ILW equation
 and its convergence  to invariant BO\,/\,KdV dynamics.
Compare these results with the dynamical results in \cite{LOZ},
where the authors treated the $k = 1$ case
but uniqueness was missing  due to the use of a compactness argument
(in the probabilistic setting).
Furthermore,
in Theorem~\ref{THM:6}\,(iii) and (iv),
we
establish a 2-to-1 collapse at the level of invariant dynamics,
which seems to be the first such result in the literature.
We, however,  emphasize  that Theorems \ref{THM:5}
and \ref{THM:6} are certainly not a main challenge in this paper
and that it is of much more significance
to construct the ``correct''
conservation laws in Theorems~\ref{THM:1}
and \ref{THM:2}, from which everything follows.

While a typical study of Gibbs measures
and the corresponding dynamics
leads to studying functions of low regularity,
we restrict our attention to $k \ge 3$
and thus the dynamical problems we consider
are not low-regularity well-posedness problems.
In fact,
Molinet and Vento~\cite{MV15}
proved
(unconditional) global well-posedness of
 ILW \eqref{ILW} (and hence of the scaled ILW~\eqref{sILW}) in $H^s(\T)$, $s\ge \frac 12$.
 In particular, 
 ILW (and the scaled ILW) is deterministically globally well-posedness
on the support of
the \GGM~ $\rho^\dl_{\frac k2}$
(and $\wt \rho^\dl_{\frac k2}$, respectively),
provided that $k \ge 3$.
See Remark~\ref{REM:white}.

Let us briefly discuss the strategy for proving 
Theorems \ref{THM:5} and \ref{THM:6}.
We focus on the deep-water regime.
Consider the following truncated ILW dynamics:
\begin{equation}
\partial_t u - \Gdl \partial_x^2 u = \P_N \dx ( \P_N u)^2.
\label{ILW2}
\end{equation}

\noi
As observed in the previous works \cite{Zhid1, NORS, TV1, TV2},
the main difficulty comes from the non-conservation of
 $E^\dl_\frac k2(\P_Nu)$ under the truncated  dynamics:
\begin{align*}
\frac{d}{dt} E^\dl_{\frac k2}(\P_N \Phi_N(t) (u_0)) \neq 0, \qquad k\geq 2,
\end{align*}

\noi
where $\Phi_N(t) = \Phi_N(t;\dl)$ denotes the solution map for the truncated ILW \eqref{ILW2},
sending $u_0$ to a solution $u(t) = \Phi_N(t) (u_0)$ at time $t$.
This in turn
implies
non-invariance of the truncated measure
$\rho^\dl_{\frac{k}{2}, N}$
defined in~\eqref{rho4}.

In order to overcome this difficulty,
we follow the strategy introduced in
\cite{TV1,TV2}
and prove almost invariance of the
 truncated measure
$\rho^\dl_{\frac{k}{2}, N}$
under the truncated dynamics \eqref{ILW2};
see Proposition \ref{PROP:main3},
where we need to take a time derivative of
the truncated density $F^\dl_{\frac{k}{2}} (\P_N \Phi_N(t)(u))$,
appearing in \eqref{rho4}.
At this point,
our construction of the truncated  \GGM~
$\rho^\dl_{\frac{k}{2}, N}$
 only with an $L^2$-cutoff
 leads to a simplification;
see Remark~\ref{REM:cutoff}.
Then, as in \cite{TV1,TV2},
the crucial ingredient is  the following
probabilistic  asymptotic conservation of
 $E^\dl_\frac k2(\P_N\Phi_N(t)(u))$ at time $t = 0$:
\begin{equation}
\lim_{N\rightarrow \infty} \bigg\| \frac{d}{dt} E^\dl_{\frac k2} (\P_N
\Phi_N(t)(u)) \Big|_{t = 0} \bigg\|_{L^p(d\mu^\dl_{\frac k2})} =0
\label{alc1}
\end{equation}

\noi
for any finite $p \ge 1$;
see Proposition \ref{PROP:AAS1}.
Thanks to our
detailed analysis 
on the structure
of the deep-water conservation laws
(see Subsection \ref{SUBSEC:A4})
with  
 a perturbative viewpoint, 
\eqref{alc1}
follows from a straightforward adaptation 
of the argument in \cite{TV0, TV1, TV2}.
We, however, point out that some care is needed
in the shallow-water regime, since 
the rate of convergence (as $N \to \infty$)
for the shallow-water probabilistic  asymptotic conservation 
for $0 < \dl < \infty$
(Proposition \ref{PROP:AAS2})
 diverges as $\dl \to 0$;
 see Remark \ref{REM:div} and Subsection \ref{SUBSEC:AC2}.

Once we establish~\eqref{alc1},
the desired invariance (Theorem~\ref{THM:5}\,(i)) follows
from \eqref{alc1} and a PDE approximation argument
(Lemma \ref{LEM:inv1}).
The  deep-water convergence
in Theorem~\ref{THM:5}\,(ii)
then
follows from the (weak) convergence of the \GGMs~
(Theorem~\ref{THM:3}), the
Skorokhod
representation theorem (Lemma~\ref{LEM:Sk}),
and the deterministic deep-water convergence result
on the ILW dynamics
 by the second author~\cite{Li2022}.
As a consequence  of the
unconditional well-posedness by Molinet and Vento \cite{MV15}
(including the case of the  limiting BO equation),
the deep-water convergence holds (pathwise) unconditionally,
namely, in the entire class  $C(\R; H^{\frac{k-1}{2}-\eps}(\T))$
without  any auxiliary function space.
A similar comment holds in the shallow-water limit.

Before we conclude this introduction by stating several remarks,
let us emphasize that the main novelties of this paper
lie in

\smallskip
\begin{itemize}
\item
the construction of appropriate conservation laws
of the (scaled) ILW
and their deep-water and shallow-water convergence
presented in Subsection \ref{SUBSEC:1.2}, and

\smallskip
\item
the construction of the \GGMs~
associated with these (scaled) ILW conservation laws
and their deep-water and shallow-water convergence
presented in Subsection  \ref{SUBSEC:1.3},

\end{itemize}

\noi
with a particular emphasis on the shallow-water convergence, 
where we establish a 2-to-1 collapse in each setting.
We note that 
 the dynamical results (Theorems \ref{THM:5} and \ref{THM:6})
follow from a straightforward
combination of the strategy in \cite{TV1, TV2}
with Theorems \ref{THM:3}
and \ref{THM:4}.

\begin{remark}\label{REM:cutoff}\rm

In proving almost invariance of the
 truncated \GGM~
$\rho^\dl_{\frac{k}{2}, N}$,
we just need to study the time derivative of
$E^\dl_\frac k2(\P_N \Phi_N(t)(u))$
thanks to the conservation of the $L^2$-norm under the
truncated dynamics~\eqref{ILW2}
and our simplified construction of the \GGMs~
only with an $L^2$-cutoff.
On the other hand,
  in \cite{TV1, TV2, DTV},
it was necessary to estimate the time derivatives
of (smooth) cutoff functions on the lower order conservation laws;
see, for example, the proof of Proposition 5.4 in \cite{TV1}.
We point out that, in \cite{TV1, TV2},
it was crucial to use {\it smooth} cutoff functions
on the lower order conservation laws precisely for this reason,
but that
our main results on the measure construction,
 invariance, and convergence
also hold even if we replace the smooth cutoff function $\eta_K\big(\|u\|_{L^2}\big)$
by a sharp $L^2$-cutoff $\ind_{\{\|u \|_{L^2} \le K\}}$
as in \cite{LRS,BO94, OST}.
See also Remarks \ref{REM:cutoff2}, 
\ref{REM:cutoff3}, 
and  \ref{REM:TV}.

\end{remark}

\begin{remark}\label{REM:KDV} \rm
Theorem \ref{THM:6}\,(ii) establishes
invariance of the \GGM~
associated with the $H^2$-conservation law
for KdV, which is missing in the literature.

\end{remark}

\begin{remark}\label{REM:white}\rm

(i)
In the current paper, we restrict our attention to  $k \ge 3$
in studying the dynamical problems.
This is due to the fact that we rely on the deterministic
global well-posedness of the (scaled) ILW
 in $H^\frac12(\T)$ by Molinet and Vento \cite{MV15}.
In a recent work~\cite{CLOP} with Pilod,
we in fact proved
 global well-posedness
 of the (scaled) ILW in $L^2(\T)$
 by viewing ILW as a perturbation of BO
 and applying the gauge transform for BO.
Thus, in order to handle the $k = 2$ case, 
using the deterministic global well-posedness in \cite{CLOP},
we need to adapt the argument in \cite{DTV}
to the ILW setting.
We plan to address this issue
together with the $k = 1$ case
in a forthcoming work
by combining the ideas from
\cite{DTV, Deng15} and \cite{CLOP, CFLOP}.

\smallskip

\noi
(ii) The $k = 0$ case corresponds to the white noise in \eqref{white1}.
It is known that the white noise is
invariant under the KdV dynamics;
see  \cite{QV, OH1, OH4, OH6, OQV, KMV}.
On the other hand,
despite
the
recent breakthrough results
by
G\'erard, Kappeler, and Topalov
\cite{GKT}
and by
Killip, Laurens, and Vi\c{s}an \cite{KLV}
on deterministic well-posedness of BO in
the entire subcritical regime
$H^s(\T)$, $s > -\frac 12$,
the question of invariance of the white noise for BO (and ILW)
remains a challenging open problem.

\end{remark}

\begin{remark}\label{REM:new}\rm

After the appearance of the first version of this paper, 
there have been intensive research activities on the ILW equation; 
see \cite{GL, HKV, CLOZ, CHLO}.

In  \cite{GL}, 
Gassot and Laurens
studied well-posedness of ILW~\eqref{ILW} 
in negative Sobolev spaces on $\T$.
By  combining
the perturbative viewpoint~\eqref{ILWx}
following  \cite{IS23, CLOP, CFLOP, FLZ}, 
the Birkhoff map for BO \cite{GKT}, 
and the quantity
$\be(\kk; u)$
(related to the perturbation determinant)
defined in 
\cite[Proposition 4.3]{KLV}, 
for given  $0 < \dl < \infty$, 
they proved that ILW~\eqref{ILW}
is globally well-posed in $H^s(\T)$, $s > - \frac12$, 
and that deep-water convergence 
holds in the same regularity range. 
Note that these results in~\cite{GL}
hold  only in the sense of 
sensible weak solutions as defined in 
\cite{OW1, OW, BOW}.
Since the approach in \cite{GL} relies on
the work~\cite{GKT} for the periodic setting, 
the corresponding results 
in negative Sobolev spaces on the real line are still open at this point.

In  \cite{HKV}, 
Harrop-Griffiths, Killip, and Vi\c{s}an, 
made a 
substantial progress in understanding  the complete integrability structure
of ILW \eqref{ILW}
in terms of the Lax pair structure and the perturbation determinant, 
and developed a 
{\it non-perturbative} approach
for {\it both} the deep-water and shallow-water regimes
on {\it both} geometries (namely, on the torus and the line).
In particular, 
they established
uniform-in-time a priori $H^s$ bounds for solutions to the (scaled) ILW
and their equicontinuity properties;
see \cite{HKV} for further details.

In \cite{CLOZ}, 
the authors with T.\,Zhao 
established shallow-water convergence 
of the scaled ILW~\eqref{sILW}
to KdV \eqref{kdv} in $L^2$ on both geometries.
Their approach was based on combining 
the complete integrability approach as in \cite{HKV}
and an infinite iteration of normal form reductions,
introduced and developed in \cite{GKO, KOY, OW}, 
which was completely new in the literature.

In \cite{CHLO}, 
the authors with B.\,Harrop-Griffiths
investigated 
regularity properties of the solution map
for the scaled ILW \eqref{sILW} 
 on the real line.
By decomposing the dynamics into 
the low frequency part and the residual part, 
they  showed that, when the depth parameter is sufficiently small,  
the solution map for the low frequency part
is analytic  in $L^2(\R)$, 
while the 
 solution map for the residual part fails to be~$C^2$.
Moreover, they established
shallow-water convergence  in $L^2(\R)$
of  the low frequency dynamics
 to  KdV.
This explained 
how this regularity gain 
of the solution map 
propagates in frequency,
namely, from the low frequency dynamics 
supported on $\{|\xi|\le \dl^{-1}\}$ 
to KdV, 
which 
 illustrates another aspect of the singular nature
of the 
 shallow-water limit.

\end{remark}

\begin{remark}\rm

In a recent preprint \cite{CGLLO}, 
the second and third authors with Chouk, Gubinelli, and J.\,Li
studied well-posedness issues of the following {\it modulated} ILW equation
($0 < \dl \le \infty$):
\begin{equation}
\label{MILW1}
\dt u-    \Gdl\dx^2 u \cdot \dt w =\dx (u^2) 
\end{equation}

\noi
and the modulated scaled ILW equation
($0 \le \dl <  \infty$):
\begin{equation}
\label{MILW2}
\dt v-    \wt \Gdl\dx^2 v \cdot \dt w =\dx (v^2),
\end{equation}

\noi
where $\Gdl$ and $\Gd$
are as in 
\eqref{GG1} and \eqref{sILW}.
Here,   $w:\R_+\to\R$ is an arbitrary continuous function of time, 
called a  {\it modulation}, 
satisfying a certain irregularity condition; 
see \cite[Definition 1.1]{CGLLO}
for the definition of irregularity in this context.
When $\dl = \infty$, 
\eqref{MILW1} corresponds to 
 the modulated BO:
\begin{equation}
\label{BO2}
\dt u-    \H\dx^2 u \cdot \dt w=\dx (u^2), 
\end{equation}

\noi
while, when $\dl = 0$, 
\eqref{MILW2} corresponds to 
the modulated KdV:
\begin{equation}
\label{kdv1}
\dt v+  \dx^3 v \cdot \dt w =\dx (v^2).
\end{equation}

\noi
It was shown in \cite{CGLLO} that,  
 given {\it any} $s \in \mathbb R$,  
the following statements hold,  provided that 
the modulation $w$ is sufficiently irregular:

\smallskip

\begin{itemize}
\item[(i)]
For each $0 < \dl \le  \infty$,  
 the modulated ILW \eqref{MILW1}
is  locally well-posed in $H^s(\mathbb T)$.
For each $0 \le \dl <   \infty$,  
the modulated scaled ILW~\eqref{MILW2}
is  locally well-posed in $H^s(\mathbb T)$.

\smallskip
\item[(ii)]
{\rm (deep-water convergence).}
Given $u_0 \in H^s (\T)$,
let $u^\dl$, $0<   \dl < \infty$,  and $u$ denote the local-in-time solutions
to the modulated ILW equation \eqref{MILW1}
and the modulated BO~\eqref{BO2}
with $u^\dl|_{t = 0} = u|_{t = 0} = u_0$, respectively.
Then, as $\dl \to \infty$, 
$u^\dl$ converges to $u$ in 
$C([0, \tau]; H^s(\T))$, 
where $\tau \in (0,  T]$ denotes the local existence time of the modulated BO 
\eqref{BO2}.

\smallskip
\item[(iii)]
{\rm (shallow-water convergence).}
Given $v_0 \in H^s (\T)$,
let $v^\dl$, $0<   \dl < \infty$,  and $v$ denote the local-in-time solutions
to the modulated scaled ILW equation \eqref{ILW2}
and the modulated KdV
\eqref{kdv1}
with $v^\dl|_{t = 0} = v|_{t = 0} = v_0$, respectively.
Then, as $\dl \to 0$, 
$v^\dl$ converges to $v$ in 
$C([0, \tau]; H^s(\T))$, 
where $\tau \in (0, T]$ denotes the local existence time of the modulated KdV 
\eqref{kdv1}.
%

\end{itemize}

\smallskip

\noi
Thanks to the modulation $w$, 
these results significantly improve
the corresponding results in the unmodulated case
(provided that 
the modulation $w$ is sufficiently irregular), 
thus exhibiting new examples of 
{\it regularization by noise};
see \cite{CGLLO} for a further discussion 
on regularization by noise.
See also \cite{CGLLO2, GLLO1}
for further results on the modulated ILW \eqref{MILW1}.

\end{remark}

\subsection{Organization of the paper}

In Section \ref{SEC:2}, we introduce notations and go over
deterministic and stochastic preliminary tools.
Sections \ref{SEC:cons1} and \ref{SEC:cons2}
are devoted to the construction of 
the deep-water and shallow-water conservation laws
in \eqref{E00} and \eqref{E5}, respectively, 
where we present proofs of Theorems \ref{THM:1} and \ref{THM:2}
and provide more detailed structures of
these conservation laws which play an important role in later sections.
This part occupies the major part of the paper.
In Sections \ref{SEC:DGM1} and 
\ref{SEC:SGM1}, we go over the construction 
of the \GGMs, thus providing proofs of Theorems~\ref{THM:3} and~\ref{THM:4}.
After establishing the probabilistic asymptotic conservation 
in Section~\ref{SEC:AC}, we establish invariance of the \GGMs~
and the claimed dynamical convergence in Theorems 
\ref{THM:5} and~\ref{THM:6}
in Section~\ref{SEC:INV}.

\section{Preliminaries}
\label{SEC:2}

\subsection{Notations and function spaces}
We start by introducing some  notations. Let $A\les B$ denote an estimate of the form $A\leq CB$ for some constant $C>0$. We write $A\sim B$ if $A\les B$ and $B\les A$, while $A\ll B$ will denote $A\leq c B$, for some small constant $c> 0$. 
We may write  $\les_{q}$ and $\sim_{q}$ to 
emphasize the dependence of the implicit constant on an external parameter $q$.
We will use the shorthand notation
\begin{align}
 n_{1\cdots k} := n_1 + \cdots +n_k.
 \label{sum1}
\end{align}

%
\noi
We set 
$\Z_{\ge 0} := \N \cup\{0\}$,  
$\Z^*:=\Z\setminus\{0\}$, 
and 
we let $\lceil x \rceil$ denote
the smallest integer greater than or equal to $x$.
We use $C>0$ to denote various constants, which may vary line by line,
and we may use subscripts to signify  dependence on external parameters.
We use $\eps > 0$ to denote an arbitrarily small constant.
Various constants depend on $\eps$ but,
for simplicity of notation,  we suppress such $\eps$-dependence.

Throughout this paper, we fix a probability space $(\Omega, \F, \PP)$. The realization $\omega \in\Omega$ is often omitted in writing. We will use $\Law(X)$ to denote the law of a random variable $X$.

Our convention for the Fourier transform is as follows:
\begin{align*}
\ft{f}(n) = \frac{1}{\sqrt{2\pi}} \int_\T f(x) e_{-n} dx
\quad \text{and} \quad 
f(x) = \F^{-1}(f)(x) = \frac{1}{\sqrt{2\pi}} \sum_{n\in\Z^*} \ft{f}(n) e_n(x)
\end{align*}

\noi
for $f$ on $\T$, where  $e_n(x) = e^{inx}$.
Then, we have
\begin{align*}
\|f\|_{L^2}^2 = \sum_{n \in \Z} |\ft f(n)|^2
\quad \text{and} \quad 
\ft{fg}(n) = \frac1{\sqrt{2\pi}} \sum_{\substack{n_1, n_2 \in \Z\\n = n_1 + n_2}}
\ft f(n_1) \ft g(n_2).
\end{align*}

\noi
Given $N\in\N$, we denote by $\P_N$ the Dirichlet projector onto spatial frequencies $\{|n| \leq N\}$ defined as follows:
\begin{align}
\P_{N} f (x) : = (D_N \ast f )(x) = \frac{1}{\sqrt{2\pi}} \sum_{0<|n| \leq N} \ft{f}(n) e^{inx},
\label{Diri1}
\end{align}
where $D_N(x) = \sum\limits_{|n| \leq N} e^{inx}$ is the Dirichlet kernel.
We set $\P_{>N} = \Id - \P_N$.

Given  $s\in\R$ and $1 \leq p \leq \infty$,
we define the $L^p$-based Sobolev space $W^{s,p}(\T)$ via the norm:
\begin{align*}
\| f\|_{W^{s,p}} := 
\| \jb{\dx}^s f \|_{L^p}
= 
\| \F^{-1} ( \jb{n}^s \ft{f}(n) ) \|_{L^p_x},
\end{align*}

\noi
where 
$\jb{\dx} = (1 - \dx^2)^\frac 12$
and $\jb{n}=(1+n^2)^\frac12$.
When $p=2$, we set $H^s(\T) = W^{s, 2}(\T)$.

Let $\P_{\ne 0}$ denote the projection onto non-zero frequencies.
Then, we set
$L^r_0(\T) = \P_{\ne 0}L^r(\T)$, 
$W^{s, r}_0(\T) = \P_{\ne 0}W^{s, r}(\T)$, 
and $H^s_0(\T) = \P_{\ne 0}H^s(\T)$.
Namely, they are the subspaces of $L^r(\T)$, $W^{s, r}(\T)$, 
and $H^s(\T)$, consisting of mean-zero functions, respectively.

 We will often use the short-hand notations 
 such as $L^q_T H^{s}_x$ and $L^p_\omega H^{s}_x$ for $L^q\big( [-T,T]; H^{s}(\T)\big)$ and $L^p(\Omega; H^{s}(\T))$, respectively.

\begin{lemma}
\noi{\rm(i) (interpolation).} Let  $s, s_0 \in \R$
such that $0 \le s_0 \le s$.
Then, we have 
\begin{align}
\|u\|_{H^{s_0}} &\le \|u\|^{\frac{s - s_0}{s}}_{L^2} \|u\|_{H^{s}}^{\frac{s_0}{s}}.
\label{interp1}
\end{align}

\noi{\rm(ii) (fractional Leibniz rule).} 
Let $s > 0$ and   $1<p_j,q_j, r\le\infty$ such that  $\frac{1}{p_j} + \frac{1}{q_j} =  \frac1r$, $j=1,2$.
Then, 
we have 
\begin{equation}\label{leib}
\| \jb{\dx}^s (fg) \|_{L^r} \les \|\jb{\dx}^s f\|_{L^{p_1}} \|g\|_{L^{q_1}} + \| f\|_{L^{p_2}} \|\jb{\dx}^s g \|_{L^{q_2}}.
\end{equation}

\end{lemma}

See \cite{BOZ} for the fractional Leibniz rule on the torus; see also 
\cite[Lemma 3.4\,(i)]{GKO}.

\subsection{On the multiplier operators}

In this paper, we work with various multiplier operators
in both the deep-water and shallow-water regimes.
In this subsection, 
we state several lemmas on key multipliers.

Given $0 < \dl < \infty$, 
let $\Kdl (n)$ be the multiplier of $\Gdl \dx$
given by 
\begin{align}
\Kdl (n) &= in \ft{\Gdl}(n) = n\coth(\dl n) - \dl^{-1},
 \quad n\in\Z^*, 
\label{sym1}
\end{align}

\noi
where $\Gdl$ and $\ft \Gdl(n)$ are as in \eqref{GG1}.
When $\dl = \infty$, we set
$\mathfrak{K}_\infty(n) = |n|$.
This multiplier plays an important role 
in the deep-water analysis.
We recall the following lemma on the basic properties
of $\Kdl(n)$; see Lemma 2.1 in \cite{LOZ} for a proof.

\begin{lemma}\label{LEM:K1}
Given $0 < \dl \le \infty$, 
let $\Kdl(n)$ be as in \eqref{sym1}.
Then, we have 
\begin{align*}
\max\Big( 0, |n| - \frac1\dl \Big) < \Kdl(n) = \Kdl(-n) < |n|
= \mathfrak{K}_\infty(n),
\end{align*}

\noi
for any $0 < \dl \le \infty$ and $n\in\Z^*$.
In particular, we have
\begin{align*}
\Kdl(n) \sim_\dl |n|
\end{align*}

\noi
for any $n \in \Z^*$, where the implicit constant is independent of 
 $2\leq \dl\leq\infty$.
Furthermore, for 
each fixed $n\in\Z^*$, $\Kdl(n)$ is strictly increasing in $\dl\geq1$ and converges to $|n|$ as $\dl\to\infty$.
\end{lemma}

Given $0 < \dl < \infty$, 
let $\Ldl (n)$ be the multiplier of $\Gd \dx$
given by 
\begin{align}
 \Ldl(n)  
 = in \ft{\Gd}(n)
 = \dl^{-1} \Kdl(n)
 = 
\dl^{-1} (n\coth(\dl n) - \dl^{-1}), 
\label{sym2}
\end{align}

\noi
where $\Gd$ 
and $\ft \Gd(n)$ are as in \eqref{sILW}
and \eqref{Gd1}, respectively.
When $\dl = 0 $, we set
$\mathfrak{L}_0(n) = \frac 13n^2$.
This multiplier plays an important role 
in the shallow-water analysis.
See (the proof of)   \cite[Lemma 2.3]{LOZ} for a proof
of the following lemma.
(Some of the bounds also follow
from 
Lemma~\ref{LEM:K1} 
with $  \Ldl(n)  
 = \dl^{-1} \Kdl(n)$.)

\begin{lemma}\label{LEM:L1}
The following statements hold.
\smallskip

\begin{itemize}
\item[\rm(i)] $0<\Ldl(n)< \min(\frac13 n^2, \frac1\dl|n|)$ for 
any $0 < \dl < \infty$ and $n\in\Z^*$.

\medskip
\item[\rm(ii)]
We have 
\begin{align}
\Ldl(n) 
  = 2 n^2 \sum_{k=1}^\infty \frac{1}{k^2\pi^2 + \dl^2n^2}.
\label{Ldl1}
\end{align}

\noi
In particular, 
for each fixed $n\in\Z^*$,
$\Ldl(n)$ is decreasing in $\dl$, 
and $\lim_{\dl \to 0} 
\Ldl(n) = \frac13 n^2
= \mathfrak{L}_0(n)
$.

\medskip
\item[\rm(iii)]
For any $0 < \dl < \infty$, we have 
\begin{align*}
\Ldl(n) \ges
\begin{cases}
\frac 1 \dl  |n|, & \text{if } \dl|n| \gg 1, \\
n^2, & \text{if } \dl|n| \les 1.
\end{cases}
\end{align*}

\medskip
\item[\rm(iv)]
Given $n \in \Z^*$, define $\hf(\dl, n)$ by  
setting $ n^2 \hf(\dl,n) =  n^2 - 3\Ldl(n)$.
Then, 
we have 
\begin{align}
& \hf(\dl,  n) =6\dl^2 n^2 \sum_{k=1}^\infty \frac{1}{k^2\pi^2(k^2\pi^2 + \dl^2 n^2)} 
\in (0, 1],
\label{hdef}\\
& \sum_{n\in\Z} \hf^2(\dl, n) = \infty, 
\quad \text{and} \quad \lim_{\dl\to 0} \hf(\dl, n) =0
\label{HX1}
\end{align}
for any $0<\dl<\infty$ and $n\in\Z^*$.

\end{itemize}

\end{lemma}

We also state basic bounds on the perturbation operator $\Qdl$
in \eqref{Qdl1} and $\Gd$ in \eqref{sILW}.

\begin{lemma}\label{LEM:T1}
The following bounds hold.

\smallskip

\begin{itemize}

\item[\rm(i)]
Let $\ft{\Qdl} (n)$ be the multiplier of the perturbation operator
$\Qdl = \Gdl \dx - \H\dx$
defined in~\eqref{Qdl1}.
Then, we have 
\begin{align}
|\ft{\Qdl} (n)| & \leq \frac1\dl
\label{Q1}
\end{align}

\noi
for any $0 < \dl < \infty$ and $n \in \Z^*$.

\medskip
\item[\rm(ii)]
Let $\ft \Gd(n)$ be as in \eqref{Gd1}.
Then, we have 
\begin{align}
|\ft{\Gd}(n)|  \le \min\Big(\frac1\dl, \frac13|n|\Big)
\label{dlGd}
\end{align}

\noi
for any $0 < \dl < \infty$ and $n \in \Z^*$.

\medskip
\item[\rm(iii)]
Let $ 1 < r < \infty$. 
Then, 
given $0 < \dl \le \infty$, 
the Hilbert transform $\H$
and the perturbation operator $\Qdl$ in \eqref{Qdl1} 
are  bounded on $L^r_0(\T) = \P_{\ne 0}L^r(\T)$, consisting
of mean-zero functions in $L^r(\T)$, 
satisfying 
 the bounds\textup{:}
\begin{align}
\|\H\|_{L^r_0\to L^r_0}, 
 & \le C_r < \infty, \label{Q2a}\\
 \|\Qdl\|_{L^r_0\to L^r_0} 
& \le C_{r}\dl^{-1}< \infty.
\label{Q2}
\end{align}

\noi
Moreover, given $0 <  \dl \le  \infty$, 
the operator $\Gdl = \dl \wt \Gdl $ in~\eqref{GG1}
and \eqref{sILW}
is  bounded on $L^r_0(\T)$
with 
 the bound\textup{:}
\begin{align}
 \| \Gdl\|_{L^r_0\to L^r_0}  = \|\dl \Gd\|_{L^r_0\to L^r_0} 
& \le C_{r}< \infty
\label{Q3}
\end{align}

\noi
and we also have 
\begin{align}
 \| \Gd \|_{W^{1, r}_0\to L^r_0}  
& \le C_{r}< \infty, 
\label{Q3a}
\end{align}

\noi
where 
$W^{1, r}_0(\T) = \P_{\ne 0} W^{1, r}(\T)$, 
consisting of  mean-zero functions in $W^{1, r}(\T)$.
In~\eqref{Q2a}, \eqref{Q2}, \eqref{Q3}, and \eqref{Q3a}, the constant
$C_r$ is independent of $0 < \dl \le \infty$.

\end{itemize}

\end{lemma}

\begin{remark}\label{REM:Gd1}\rm
From the series representation \eqref{Q4a} of 
$\ft \Gdl(n) = \dl \ft \Gd(n)$ below, 
we have 
\begin{align*}
\lim_{\dl \to 0} |\ft \Gdl(n)| = 
\lim_{\dl \to 0}|\dl \ft \Gd(n)| 
\les 
\lim_{\dl \to 0}\sum_{k=1}^\infty \frac{\dl |n|}{k^2\pi^2}
\sim \lim_{\dl \to 0}\dl |n|  = 0
\end{align*}

\noi
for any $n \in \Z^*$.

\end{remark}

\begin{proof}[Proof of Lemma \ref{LEM:T1}]
For~\eqref{Q1}, see the proof of Lemma 2.3 in \cite{CLOP}, 
while \eqref{dlGd} follows from Lemma \ref{LEM:L1}\,(i).

We now turn to (iii).
Let 
$\varphi \in C^\infty(\R; [0, 1])$ such that
$\varphi(\xi) = 0$ for $|\xi| \le \frac 12$ and 
$\varphi(\xi) = 1$ for $|\xi| \ge 1$.
Define a function $m_1(\xi)$ on $\R$ by setting
$m_1(\xi) = \varphi(\xi)( -i \sgn(\xi))$.
It is easy to check that 
 $|\dd_\xi^k m_1(\xi)| \les_k |\xi|^{-k}$
for any $k \in \Z_{\ge 0}$ and $\xi \in \R$.
Then, 
by the Mihlin multiplier theorem on $\R$ (see  \cite[Theorem 6.2.7 and (6.2.14)]{Gra1}), 
the Fourier  multiplier operator $T_{m_1}$
with multiplier $m_1$ is bounded on $L^r(\R)$, $1 < r < \infty$.
Noting that $m_1(n)$, $n \in \Z$,  
agrees with the multiplier for the Hilbert transform $\H$ on $\T$, 
boundedness of $\H$ on $L^r_0(\T)$
with the bound~\eqref{Q2a}
follows from 
the transference
(\cite[Theorem 4.3.7]{Gra1}),
since
$m_1(\xi)$ is continuous and thus is regulated everywhere (in the 
sense of  \cite[Definition 4.3.6]{Gra1}).

From \eqref{Qdl1}, we have 
\begin{align*}
\ft \Qdl(n) = 
 n \big(\coth (\dl n) - \sgn (n)\big) - \dl^{-1}
 = 
 \frac {2|n|}{e^{2\dl| n|} - 1} - \dl^{-1}
\end{align*}

\noi
for $n \in \Z^*$.
Let $m_2(\xi) 
= \varphi(\xi) \big(\frac {2|\xi|}{e^{2\dl |\xi|} - 1}- \dl^{-1}\big)$.
We easily see that $|\dd_\xi^k m_2(\xi)| \les_{k} \dl^{-1}|\xi|^{-k}$
for any $k \in \Z_{\ge 0}$ and $\xi \in \R$.
Hence, the bound \eqref{Q2}  follows from 
the Mihlin multiplier theorem on $\R$
and the transference as discussed above.

From \eqref{GG1}, we have 
\begin{align*}
\ft \Gdl(n) = 
-i\bigg(1+
 \frac 2{e^{2\dl n} - 1} - \frac 1{\dl n}\bigg) .
\end{align*}

\noi
By setting  $m_3(\xi) 
= -i \varphi(\xi) \big(1+ \frac 2{e^{2\dl \xi} - 1} - \frac 1{\dl \xi}\big)$, 
it follows from the Mihlin multiplier theorem on~$\R$ 
and the transference as discussed above that 
\begin{align}
 \| \Gdl\|_{L^r_0\to L^r_0}  
& \le C_{r, \dl}^{(1)}< \infty, 
\label{Q4}
\end{align}

\noi
where the constant 
$C_{r, \dl}^{(1)}$ is independent of $1 \les \dl \le \infty$.

In order to prove uniform boundedness
of $\Gdl$ for small $\dl$, 
we need to use a series representation.
 From \eqref{sym2} and \eqref{Ldl1} in Lemma \ref{LEM:L1}, 
we have 
\begin{align}
\ft \Gdl(n) = \dl \ft \Gd(n) = 
\frac{\dl \Ldl(n)}{in} 
&=  -2i \sum_{k=1}^\infty \frac{\dl n}{k^2\pi^2+ \dl^2 n^2}. 
\label{Q4a}
\end{align}

\noi
Let $m_4(\xi)
= -2i \varphi(\xi)  \sum_{k=1}^\infty \frac{\dl \xi}{k^2\pi^2+ \dl^2 \xi^2}$.  
Then, we have 
\begin{align*}
|m_4(\xi)| 
=  2 \varphi(\xi) \sum_{k=1}^\infty
\frac{1}{\pi^2 (\frac{k}{\dl | \xi|})^2+1} \frac{1}{\dl |\xi|}
 \le 2 \int_0^\infty \frac{dx}{\pi^2 x^2 + 1} \les 1
\end{align*}

\noi
for any $\xi \in \R$.
Similarly, a direct computation shows that 
 $|\dd_\xi^k m_4(\xi)| \les_{\dl, k} |\xi|^{-k}$
for any $k \in \Z_{\ge 0}$ and $\xi \in \R$, 
where the implicit constant  is independent of $0 < \dl \les 1$.
Hence, from 
the Mihlin multiplier theorem on $\R$
and the transference as discussed above, 
we obtain 
\begin{align}
 \| \Gdl\|_{L^r_0\to L^r_0}  
  = \|\dl \Gd\|_{L^r_0\to L^r_0} 
& \le C_{r, \dl}^{(2)}< \infty, 
\label{Q5}
\end{align}

\noi
where the constant 
$C_{r, \dl}^{(2)}$ is independent of $0 < \dl \les 1$.
Hence, the bound \eqref{Q3} follows from~\eqref{Q4}
and \eqref{Q5}.

Lastly, the bound \eqref{Q3a} follows from a similar discussion as above, 
using the series representation \eqref{Q4a}, written as 
\begin{align*}
 (in)^{-1} \ft \Gd(n) 
=  -2 \sum_{k=1}^\infty \frac{1}{k^2\pi^2+ \dl^2 n^2}. 
\end{align*}

This concludes the proof of Lemma \ref{LEM:T1}.
\end{proof}

Lastly,  we recall several identities.
First, recall  Cotlar's identity \cite{Cotlar}:
\begin{align}
2 \H\big( (\H u)  u  \big) = (\H u)^2   - u^2
\label{Cot1}
\end{align}

\noi
for a function $u$ on $\R$.
From \eqref{Cot1} and the ``polarization'' $2uv = (u+ v)^2 - u^2 - v^2$,
we obtain the following identity:
\begin{align}
\H\big( (\H u)  v + u \H v \big) = (\H u)( \H v ) - uv
\label{Cot2}
\end{align}

\noi
for functions $u$ and $v$ on $\R$.
In the current periodic setting,
we instead have
\begin{align}
\H\big( (\H u)  v + u \H v \big) = \P_{\ne 0} \big((\H u)( \H v ) - uv\big)
\label{Cot3}
\end{align}

\noi
for  mean-zero functions $u$ and $v$
on $\T$,
since we have $\H = \H \P_{\ne 0}$ by our convention; see
Footnote \ref{FT:1}.
However, by noting
that $(\H u)( \H v ) - uv$ has mean zero on $\T$
(recall that
$\H$ has a symbol
$-i \sgn(n)\cdot \ind_{n \ne 0}$),
we see that \eqref{Cot3} reduces to \eqref{Cot2} even on $\T$,
provided that $u$ and $v$ have mean zero on $\T$.

The following lemma establishes an analogous but slightly different identity
for the Tilbert transform $\Tdl$.

\begin{lemma}\label{LEM:T2}
Given $0 < \dl < \infty$,
let $\Tdl$
be the Tilbert transform defined as the Fourier multiplier operator
with
 multiplier
\begin{align}
\ft{\Tdl}(n) =
\begin{cases}
-i \coth(\dl n), &  n\in\Z^*,\\
0, & n = 0.
\end{cases}
\label{Til0}
\end{align}

\noi
Then, we have
\begin{align}
\Gdl = \Tdl - \dl^{-1} \dx^{-1}\P_{\ne 0}
\label{Til1}
\end{align}
and
\begin{align}
\Tdl\big((\Tdl u) v + u\Tdl v\big) =
\P_{\ne 0}\big(
(\Tdl u) (\Tdl v) - uv\big)
\label{Til2}
\end{align}

\noi
for  mean-zero functions $u$ and $v$
on $\T$,
where $\P_{\ne 0}$ denotes the projection onto non-zero frequencies.

\end{lemma}

\begin{proof}
The identity \eqref{Til1} follows from \eqref{GG1}
and \eqref{Til0}, while the identity \eqref{Til2}
follows from a direct computation.  We omit details.
\end{proof}

%
%

\subsection{On discrete convolutions}

Next, we recall the following basic lemma on a discrete convolution.

\begin{lemma}\label{LEM:SUM}
Let  $\al, \be \in \R$ satisfy
\[ \al \ge \be \ge 0 \qquad \text{and}\qquad  \quad \al+ \be > 1.\]
\noi
Then, we have
\[
 \sum_{n = n_1 + n_2} \frac{1}{\jb{n_1}^\al \jb{n_2}^\be}
\les \frac 1{\jb{n}^{ \be - \ld}}\]

\noi
for any $n \in \Z$, 
where $\ld = 
\max( 1- \al, 0)$ when $\al\ne 1$ and $\ld = \eps$ when $\al = 1$ for any $\eps > 0$.

\end{lemma}

Lemma \ref{LEM:SUM} follows
from elementary  computations.
See, for example,  
 \cite[Lemma 4.2]{GTV} and \cite[Lemma 4.1]{MWX}.

\subsection{Tools from stochastic analysis}

We first recall  the Wiener chaos estimate, 
which is a consequence of the  hypercontractivity 
of the Ornstein-Uhlenbeck semigroup due to Nelson~\cite{Ne65}.
See, for example,   \cite[Theorem I.22]{Simon}.
See also \cite[Proposition 2.4]{TTz}.

\begin{lemma}[Wiener chaos estimate]\label{LEM:hyp}
Let $\mathbf{g}=\{g_n\}_{n\in\Z}$ be an independent family of standard complex-valued Gaussian random variables satisfying $g_{-n} = \cj{g_n}$. Given $k\in\N$, let $\{Q_j\}_{j\in\N}$ be a sequence of polynomials in $\mathbf{g}$ of degree at most $k$. Then, for any finite $p\ge 1$,  we have
\begin{align*}
\bigg\| \sum_{j\in\N} Q_j(\mathbf{g}) \bigg\|_{L^p(\O)} \leq (p-1)^{\frac{k}{2}} \bigg\| \sum_{j\in\N} Q_j (\mathbf{g}) \bigg\|_{L^2(\O)}.
\end{align*}

\end{lemma}

Next, we recall 
Kakutani's theorem~\cite{Kakutani48} in the Gaussian
setting
(or  the Feldman-H\'ajek theorem \cite{Feldman, Hajek};
see also \cite[Theorem 2.9]{DP}).
For the following version of Kakutani's theorem, 
see Lemma 3.2 in \cite{LOZ}.

\begin{lemma}[Kakutani's theorem]\label{LEM:kak}
Let $\{A_n\}_{n\in\Z^*}$ and $\{B_n\}_{n\in \Z^*}$ be two sequences of independent, real-valued, mean-zero Gaussian random variables with $\E[A^2_n] = a_n >0$ and $\E[B_n^2] = b_n >0$ for all $n\in\Z^*$. Then, the laws of the sequences $\{A_n\}_{n\in\Z^*}$ and $\{B_n\}_{n\in\Z^*}$ are equivalent if and only if
\begin{equation*}
\sum_{n\in\Z^*} \bigg( \frac{a_n}{b_n} - 1 \bigg)^2 < \infty.
\end{equation*}
It they are not equivalent, then they are singular.
\end{lemma}

We now state   the  Skorokhod representation
theorem along a continuous parameter.
While \cite[Theorem 6.7]{Billingsley}
is stated for convergence along a discrete parameter, 
a straightforward modification (in particular, in 
 \cite[line 2 on p.\,71]{Billingsley})
 yields the following lemma.
See also \cite{OQS}
for an application of 
the  Skorokhod representation
theorem along a continuous parameter.

\begin{lemma}[Skorokhod representation theorem]\label{LEM:Sk}

Let $\M$ be a complete
separable metric space \textup{(}i.e.~a Polish space\textup{)}.
Fix
 $a, b \in \R$ with $a < b$.
Suppose that
 probability measures
 $\{\rho_\dl\}_{\dl \in [a, b)}$
 on $\M$ converge  weakly   to a probability measure $\rho$
as $\dl \to b$.
Then, there exist a probability space $(\wt \O, \wt \F, \wt\PP)$
and random variables $X_\dl, X:\wt \O \to \M$
such that
\begin{align*}
\Law( X_\dl) = \rho_\dl
\qquad \text{and}\qquad
\Law(X) = \rho ,
\end{align*}

\noi
and $X_\dl$ converges $\wt\PP$-almost surely to $X$ as $\dl\to b$.

\end{lemma}

%
%
%
%
%
%
%
%
%

Lastly, we recall the notion
of the total variation distance  and the 
Kullback-Leibler divergence (= relative entropy).
See Subsection 2.3 in \cite{LOZ} 
for a further  discussion on 
various modes of convergence
for  probability measures
and random variables;
see also 
\cite[Chapter~3]{Pollard}, 
 \cite[Chapter~2]{Tsy}, 
and 
 \cite{Dudley_book}.

\begin{definition}\label{DEF:KL}\rm

(i) 
Given   two
probability measures 
 $\mu$ and $\nu$
on a measurable space $(E, \mathcal{E})$,
the  total variation distance
$\dtv $ of
$\mu$ and $\nu$ is given by
  \begin{align}
\dtv(\mu, \nu) : = \sup\big\{  \vert \mu(A) - \nu(A) \vert : A\in\mathcal{E} \big\}.
\label{KL0}
 \end{align}
This metric induces a much stronger topology than the one induced by the weak convergence;
see \cite{LOZ} for such an example.

\smallskip

\noi
(ii)
Given two probability measures  $\mu$ and $\nu$, 
 the  Kullback-Leibler divergence $\dkl (\mu, \nu)$ between $\mu$ and $\nu$ 
 is defined by 
\begin{align}
\dkl(\mu, \nu) =
\begin{cases}
\displaystyle
\int_E \log \frac{d\mu}{d\nu} d\mu, & \text{if } \mu \ll \nu, \\
\infty, & \text{otherwise},
\end{cases}
\label{KL4}
\end{align}

\noi
which is nothing but the relative entropy of $\mu$ with respect to $\nu$.
\end{definition}

While the total variation distance
is a  metric, the  Kullback-Leibler divergence is not a metric.
For example, $\dkl(\,\cdot, \cdot\,)$ is not symmetric, and moreover,
the symmetrized version $\dkl(\mu, \nu) + \dkl(\mu, \nu)$ is not a metric,
either.
If $\mu$ and $\nu$ are product measures
of the form $\mu = \bigotimes_{n \in \N} \mu_n$
and $\nu = \bigotimes_{n \in \N} \nu_n$, then we have
\begin{align}
\dkl(\mu, \nu)
= \sum_{n \in \N} \dkl(\mu_n, \nu_n) .
\label{KL4a}
\end{align}

Lastly, we recall 
Pinsker's inequality;
see Lemma~2.5 in \cite{Tsy}
for a proof.

\begin{lemma}\label{LEM:KL2}
Let $\dtv$  and $\dkl$ be as in Definition \ref{DEF:KL}.
Then, we have
\begin{align*}
\dtv(\mu, \nu) \le
\frac {\sqrt{\dkl(\mu, \nu)}}{\sqrt2}  .
\end{align*}

\noi
In particular, convergence in 
 the  Kullback-Leibler divergence
 implies convergence in total variation.

\end{lemma}

\section{Deep-water conservation laws}
\label{SEC:cons1}

In this section, we present a proof of 
Theorem~\ref{THM:1}.
Namely, 
we  construct conservation laws in the deep-water regime
and establish their 
deep-water convergence; see Proposition~\ref{PROP:cons1}.
Moreover, we need to ensure that these deep-water conservation laws
are suitable for constructing the associated \GGMs~
along with their
deep-water convergence 
 (Theorem~\ref{THM:3})
 whose details are presented in 
Section \ref{SEC:DGM1}.

In Subsection~\ref{SUBSEC:A1},
 we
 start with the B\"acklund transform \eqref{BT1}
 for ILW
and derive a recursive formulation for microscopic conservation laws $\chi_n$;
see  \eqref{chi1}.
In Subsection~\ref{SUBSEC:A2},
we then study the structure
of the microscopic conservation laws $\chi_n$.
In particular, 
we introduce the notion of the {\it deep-water rank}  of a monomial 
(Definition \ref{DEF:ord1})
and also carry out detailed analysis on 
the linear and quadratic parts of the microscopic conservation laws
(Lemmas~\ref{LEM:chi1} and~\ref{LEM:chi2}).
In  Subsection~\ref{SUBSEC:A3},
we  define deep-water  conservation laws $E^\dl_\frac k2(u)$
and, by exploiting a perturbative viewpoint (see \eqref{DE6}), 
we establish
 convergence of the
deep-water  conservation laws $E^\dl_\frac k2(u)$
to the corresponding conservation laws for the BO equation~\eqref{BO}
(Proposition~\ref{PROP:cons1}).
We then take a closer look at
the cubic terms  in $u$
and determine their structures
(Lemmas \ref{LEM:cub1} and~\ref{LEM:cub2}).
This analysis on the cubic terms  plays an important  role in
the construction of  the associated \GGMs;
see Section \ref{SEC:DGM1} for further details. 
Finally, in Subsection~\ref{SUBSEC:A4},
by putting everything together,
we state the final structure
of the deep-water conservation laws~$E^\dl_\frac k2(u)$.

In this and the next sections,  we often suppress the dependence
on the depth parameter $0 < \dl < \infty$
for simplicity of notation
and, unless otherwise stated, we assume that all the functions under consideration
are smooth.

\subsection{B\"acklund transforms and microscopic conservation laws}
\label{SUBSEC:A1}

We first recall
 the B\"acklund transform for ILW \eqref{ILW}
 derived in \cite{SAK} (modulo constants);
 see also  \cite{KAS82, LR, Matsuno84}.
 Let $0 < \dl < \infty$ and $\mu >0$.
Given a (smooth) solution $u$ to ILW \eqref{ILW},
the B\"acklund transform yields a function $U = U(u;  \mu)$
such that the pair $(u, U)$
satisfies the following system:\footnote{The derivation in \cite{SAK}
 is on the real line but it equally applies to the current
 mean-zero periodic setting.}
\begin{align}
\begin{split}
2 u & = \mu(e^U-1) + (\Gdl  -i)\dx U + \dl^{-1}U ,
 \\
U_t & = \mu(e^U-1)U_x + \Gdl U_{xx} + U_x \Gdl U_x + \dl^{-1} U U_x.
\end{split}
\label{BT1}
\end{align}

\noi
Indeed,  given a smooth function $u$ on $\T\times \R$,
by assuming that $U$ is of the form
\begin{align}
U = \sum\limits_{n = 1}^\infty\mu^{-n}  \chi_n,
\label{BT1a}
\end{align}

\noi
where
 $\{\chi_n \}_{n \in \N} = \{\chi_n^\dl \}_{n \in \N}$
 is  a sequence
of smooth functions,  independent of $\mu >0 $,
we can use the first equation in \eqref{BT1}
to construct a smooth function $U = U(u; \mu)$
as follows.
Define the operator $\L_0 = \L_0(\dl)$ by setting
\begin{align}
\L_0 = (\Gdl  -i)\dx  + \dl^{-1}.
\label{BT1aa}
\end{align}

\noi
Then,
by rewriting the first equation in
 \eqref{BT1}
as
\begin{align*}
\mu U = 2u  - \mu \sum_{n = 2}^\infty\frac1{n!} U^n - \L_0U,
\end{align*}

\noi
substituting \eqref{BT1a}, and
comparing the powers of $\mu$,
we obtain the following recursive definition for $\chi_n= \chi^\dl_n$ in terms of $u$:
\begin{align}
\begin{split}
\chi_1 & = 2u, \\
\chi_n & = -\sum_{j=2}^n \frac{1}{j!}
\sum_{\substack{n_1, \dots, n_j \in \N\\n_{1\cdots j} = n}} \chi_{n_1} \cdots \chi_{n_j}
- \L_0 \chi_{n-1},
 \quad n\ge 2,
\end{split}
\label{chi1}
\end{align}

\noi
where  the notation  $n_{1\cdots j}  = n_1 + \cdots + n_j$ is as in \eqref{sum1}.
Hence, 
using \eqref{BT1a} and \eqref{chi1}, 
we can  construct $U$ from $u$.

The following lemma on  equivalence of ILW \eqref{ILW}
and its B\"acklund transform \eqref{BT1}
is well known but
seems to be
part of the folklore whose proof we could not find.
For readers' convenience,
we present its proof,
where, in the converse direction,
we impose the  additional assumption \eqref{BT1a}.

\begin{lemma}\label{LEM:BT1}
 Let $0 < \dl < \infty$.
Given $\mu \ge 0$, let $u:\R \times \T \to \R$ and $U:\R \times \T \to \C$
be smooth functions, satisfying the
first equation in \eqref{BT1}.
If
 $U$ satisfies
 the second equation in~\eqref{BT1},
then
$u = u(U)$
satisfies ILW \eqref{ILW}.
Conversely,
if $u$ satisfies  ILW \eqref{ILW}
and if, in addition, $U$ is of the form \eqref{BT1a},
then $U = U(u; \mu)$, $\mu > 0$,
satisfies
 the second equation in~\eqref{BT1}.
\end{lemma}

\begin{remark}\rm
 It is a priori not  clear to us, if, given a smooth function  $u$, 
we can find $U$ which satisfies the first equation in \eqref{BT1}.
The extra assumption \eqref{BT1a}
indeed guarantees existence of such $U$
as explained above.

\end{remark}

\begin{proof}[Proof of Lemma \ref{LEM:BT1}]

From
the first equation in  \eqref{BT1}
 with Lemma \ref{LEM:T2}
 and the anti self-adjointness of $\Gdl$
 (which in particular implies that $U_x \Gdl U_x$ has mean zero on $\T$),
 a direct (but rather lengthy) computation yields
 \begin{align}
\begin{split}
&2  (\dt u - \Gdl \dx^2 u - 2u \dx u )\\
& \quad = (\mu e^U  + (\Gdl -i) \dx + \dl^{-1})
\big[U_t - \mu(e^U-1) U_x - \Gdl U_{xx}- U_x \Gdl U_x - \dl^{-1} U U_x\big].
\end{split}
\label{BT2}
\end{align}

\noi
Hence, if $U$ satisfies the second equation
in  \eqref{BT1},  then we see that $u$ satisfies \eqref{ILW}.

Next, we prove the converse direction.
Given a smooth function $U$ on $\R \times \T$,
define $\L^U = \L^U(\mu)$ by
 \begin{align}
 \L^U = \mu e^U +  (\Gdl-i) \dx+ \dl^{-1}
 = \mu e^U + \L_0,
\label{BT3}
 \end{align}

\noi
where $\L_0$ is as in \eqref{BT1aa}.
Suppose that $u$ is a smooth solution to \eqref{ILW}
and that $U = U(u;\mu)$ of the form \eqref{BT1a} is defined by the first equation in \eqref{BT1}
(as explained above).
Then, from~\eqref{BT2} and \eqref{BT3}, we have
\begin{align}
\L^U F(U) = 0,
\label{BT4}
\end{align}

\noi
where
 $F(U)= F(U; \mu)$ denotes the expression in the brackets
on the right-hand side of \eqref{BT2}.
Our goal is to show that $F(U) = 0$.
In the following, we carry out analysis for fixed $t\in \R$
but we suppress the $t$-dependence
 for simplicity of notation.

From \eqref{BT4} with \eqref{BT3}, we have
\begin{align}
0 =\jb{\L^U F, \L_0 F}_{L^2} = \jb{\mu e^UF, \L_0 F}_{L^2}
+  \|\L_0 F\|_{L^2}^2.
\label{BT4a}
\end{align}

\noi
Then, from \eqref{BT4} with \eqref{BT3} and \eqref{BT4a}, 
 we obtain
\begin{align}
\begin{split}
0 & =\|\L^U F\|_{L^2}^2  
= \mu^2 \|e^UF\|_{L^2}^2 
+ 2\Re \jb{\mu e^UF, \L_0 F}_{L^2}+ \|\L_0 F\|_{L^2}^2\\
& = \mu^2 \|e^UF\|_{L^2}^2 - \|\L_0 F\|_{L^2}^2.
\end{split}
\label{BT5}
\end{align}

\noi
On the other hand,
from  the assumption \eqref{BT1a}
and \eqref{BT2}
(in particular, by expanding $e^U$ in the Taylor series),  
we have
\begin{align}
F(U)  = \sum_{n = 1}^\infty \mu^{-n} f_n
\label{BT5a}
\end{align}

\noi
 for some smooth functions $f_n$, independent of $\mu >0 $.
Expanding $U$ and $F$ through their series representations
 \eqref{BT1a} and~\eqref{BT5a}, we have
\begin{align}
\begin{split}
e^U F &  = \bigg( 1 + \sum_{j =  1}^\infty \frac{1}{j!} U^j \bigg)
\bigg( \sum_{n = 1}^\infty \mu^{-n} f_n \bigg) \\
& = \bigg( 1 + \sum_{n= 1}^\infty \mu^{-n} \sum_{j=1}^n \frac{1}{j!}
\sum_{\substack{n_1, \dots, n_j = 1\\n_{1\cdots j}=n}}^\infty \chi_{n_1} \cdots \chi_{n_j}\bigg)
\bigg( \sum_{n= 1}^\infty \mu^{-n} f_n \bigg) \\
& = \mu^{-1} f_1 + \sum_{n= 2}^\infty \mu^{-n} \bigg( f_n +
\sum_{\substack{n_{12}=n\\1\le n_1 \le n-1}}  f_{n_1} \wt{U}_{n_2}  \bigg) \\
& =: \sum_{n = 1}^\infty \mu^{-n} A_n,
\end{split}
\label{BT6}
\end{align}

\noi
where
\begin{align}
\begin{split}
A_1  & = f_1, \qquad 
A_n  = f_n + \sum_{\substack{n_{12}=n\\1\le n_1 \le n-1}} f_{n_1} \wt{U}_{n_2}, \quad n \ge 2, \\
\wt{U}_{n} & =  \sum_{j=1}^n  \frac{1}{j!} \sum_{n_{1\cdots j}=n} \chi_{n_1} \cdots \chi_{n_j},
\quad n \in \N.
\end{split}
\label{BT6a}
\end{align}

\noi
Then,
by substituting \eqref{BT5a} and \eqref{BT6} in \eqref{BT5},
we see that the coefficient
of each power of $\mu$ must be $0$.
From the coefficient of  $\mu^0$ with \eqref{BT6a}, we obtain
\begin{align}
 f_1 =A_1 =  0.
\label{BT7}
\end{align}

\noi
From the coefficient of $\mu^{-1}$, we have
\begin{align*}
 \int_\T A_1 \cj {A_2} + A_2 \cj {A_1} dx = 0
\end{align*}

\noi
which is satisfied in view of \eqref{BT7}.
From the coefficient of $\mu^{-2}$, we have
\begin{align}
0 =  \int_\T A_1 \cj{A_3} + |A_2|^2 + A_3 \cj{A_1} dx - \int_\T |\L_0 f_1|^2 \, dx = \int_\T |A_2|^2  dx,
\label{BT7a}
\end{align}

\noi
where we used \eqref{BT7} at the second equality.
Hence, from the definition of $A_2$ in
\eqref{BT6a} with~\eqref{BT7} and \eqref{BT7a},
we conclude that
\begin{align}
f_2 = A_2 = 0.
\label{BT7b}
\end{align}

\noi
In view of \eqref{BT7} and \eqref{BT7b},
we see that the coefficient of $\mu^{-3}$ in \eqref{BT5} is $0$.

We now proceed by induction and  show
 $f_n = 0$ for any $n\in\N$.
Given $n \ge 3$, assume that $f_k = 0$ for any $1\le k \le n-1$.
Then,
from \eqref{BT6a},
we see that
\begin{align}
A_k = 0
\label{BT7c}
\end{align}

\noi
 for any $1 \le k \le n-1$.
Then, from  the coefficient of
 $\mu^{-2n+2}$
in \eqref{BT5}, we obtain
\begin{align*}
0 & =  \sum_{n_{12}=2n} \int_\T A_{n_1}\cj{A_{n_2}}  dx
- \sum_{n_{12}= 2n-2} \int_\T \L_0 f_{n_1} \cj{ \L_0 f_{n_2}}  dx  =  \int_\T |A_n|^2  dx,
\end{align*}

\noi
where
we used the inductive hypothesis and \eqref{BT7c} at the second equality.
Namely, we have $A_n = 0$
and thus, we obtain from \eqref{BT6a}
and the inductive hypothesis that $f_n = 0$
for any $n \in \N$.
Therefore,
from \eqref{BT5a},
we conclude that $F(U) = 0$.
Namely, $U$
satisfies
 the second equation in \eqref{BT1}.
\end{proof}

Suppose that $U$ satisfies  \eqref{BT1}.
Then,
from the second equation in \eqref{BT1},
we see that the mean of $U$ is conserved.\footnote{In this paper, we impose
the mean-zero assumption on a solution $u$
to  ILW \eqref{ILW} (and BO \eqref{BO} when $\dl = \infty$), but we do {\it not}
impose such a condition on $U$ which is constructed from $u$ as in \eqref{BT1a} and \eqref{chi1}.
A similar comment applies to the shallow-water case.}
Then, by assuming that $U$
is of the form \eqref{BT1a}, we formally obtain
\begin{align*}
0 = \frac d{dt}
 \int_\T U dx
=
\sum\limits_{n = 1}^\infty\mu^{-n} \cdot  \frac d{dt}
 \int_\T\chi_ndx
\end{align*}

\noi
for any $\mu > 0$.
By comparing the coefficients of $\mu^{-n}$
on both sides,
we then conclude that
\begin{align*}
\frac d{dt}
\int_\T \chi_n dx
= 0.
\end{align*}

\noi
Namely,
for each $n \in \N$, $\chi_n$ is a microscopic conservation law for ILW \eqref{ILW}.
In Subsection \ref{SUBSEC:A3}, 
we will construct deep-water conservation laws $E^\dl_{\frac k2}(u)$
by taking suitable linear combinations of 
(the real part of) $\int_\T \chi_n dx $;
see Proposition \ref{PROP:cons1}.

Recall from
Lemma \ref{LEM:K1}
that   $\Gdl$ converges to $\H$ as $\dl\to\infty$.
Thus, in the deep-water limit ($\dl \to \infty$),
the B\"acklund transform \eqref{BT1} for ILW
reduces to that for  BO \eqref{BO}:
\begin{align}
\begin{split}
2 u & = \mu (e^U - 1) + \H U_x - i U_x, \\
U_t & = \mu(e^U - 1) U_x + \H U_{xx} + U_x \H U_x.
\end{split}
\label{BTB1}
\end{align}

\noi
Note that
Lemma \ref{LEM:BT1} also holds for $\dl = \infty$
with the same proof.
By taking $\dl \to \infty$ in \eqref{chi1} and~\eqref{BT1aa},
we obtain the following  recursive definition
for the BO microscopic conservation laws $\chi_n^\BO$:
\begin{equation}
\begin{aligned}
\chi_1^\BO& = 2u, \\
\chi_n^\BO& = - \sum_{j=2}^n \frac{1}{j!}
\sum_{\substack{n_1, \dots, n_j \in \N\\n_{1\cdots j} = n}}
 \chi_{n_1}^\BO \cdots \chi_{n_j}^\BO - (\H-i)\dx \chi_{n-1}^\BO, \quad n\ge 2.
\end{aligned}
\label{chi2}
\end{equation}

\subsection{On the deep-water microscopic conservation laws}
\label{SUBSEC:A2}

\noi
In this subsection,
 we study the structure of the deep-water microscopic conservation laws $\chi_n$.
 From the recursive formulation~\eqref{chi1},
we see that $\chi_n$ consists of monomial terms
in $u$ with $\dx$ and $\dl^{-1}$.\footnote{Also with $\Gdl$.
However, we do not need to keep track of the number of $\Gdl$ in the following.}
We first introduce the notion of the {\it deep-water rank}  of a monomial $p(u)$, 
which plays an important role
in studying the structure of the deep-water conservation laws
constructed
in the next subsection.
See \cite[p.\,118]{Matsuno84}
for a similar concept 
in a 
 simpler context of the BO equation.

\begin{definition}\label{DEF:ord1}
\rm
Given a monomial $p(u)$,
define $\#u$, $\#\dx$, and  $\#\dl^{-1}$ 
by
\begin{align*}
\#u &  = \#u(p(u)) = \text{homogeneity of $p(u)$ in $u$},\\
\#\dx &  = \#\dx(p(u))= \text{number of (explicit) appearance of $\dx$ in $p(u)$},\\
\#\dl^{-1}&  = \#\dl^{-1}(p(u))= \text{number of (explicit) appearance of $\dl^{-1}$ in $p(u)$}.
\end{align*}

\noi
We then define the {\it deep-water rank}  of a monomial $p(u)$, denoted by $\ord(p)$, by setting
\begin{align}
\ord (p)  = \# u + \# \dx + \# \dl^{-1}.
\label{ord1a}
\end{align}

\noi
For simplicity, we refer to the deep-water rank as the rank in the following.\footnote{In Definition \ref{DEF:ord2}, 
we also define the shallow-water rank, which we also refer to as the rank.
Since we never consider the deep-water and shallow-water regimes
at the same time, there is no confusion in using this convention.
The same comment applies to the notation $\ord(p)$ introduced in \eqref{ord1a}
and \eqref{ord1b}.}
If all monomials in a given polynomial $p(u)$ have the same rank, say $n$,
then we say that the polynomial $p(u)$ (and its integral over $\T$) has rank $n$.

\end{definition}

The following lemma characterizes  the rank of  $\chi_n$.

\begin{lemma}\label{LEM:chi0}
Let $0 < \dl < \infty$. Given any  $n \in \N$,
the deep-water microscopic conservation law 
$\chi_n = \chi_n^\dl$
defined in~\eqref{chi1}
has rank $n$.

\end{lemma}

\begin{proof}
We prove the claim by induction.
Since $\chi_1 = 2u$,
the claim  holds when $n = 1$.
Now, given  $n\ge 2$,  assume that the claim holds for $\chi_k$
for any  $1\leq k \leq n-1$.
Hence,  $\chi_k$ has rank $k$
for any  $1\leq k \leq n-1$.
Then,  in view of Definition \ref{DEF:ord1}
with \eqref{BT1aa}, we
see that
the monomials $\chi_{n_1} \cdots \chi_{n_j}$ (with $n_1 + \cdots + n_j = n$)
and $\L_0 \chi_{n-1}$,  appearing in \eqref{chi1},
all have rank $n$,
and thus the claim holds.
Therefore, by induction,
we conclude the proof of Lemma \ref{LEM:chi0}.
\end{proof}

In the next subsection,
 by taking a suitable linear combination of  $\chi_n$, $1\le n \le k+2$,
we will construct deep-water conservation laws
$E^\dl_{\frac{k}{2}}(u)$
 in the form \eqref{E00},
where
 all quadratic (in~$u$) terms
have  exactly $k$ derivatives
with  {\it positive} coefficients.
We point out that  the positivity
of the coefficients is crucial in
constructing
 the base Gaussian measures $\mu^\dl_{\frac{k}{2}}$
in~\eqref{gauss1}.
Thus, our main goal in this subsection
is to obtain
 a precise
 description
of  the quadratic part of the conservation law $ \int_\T \chi_n  dx$;
see Lemma \ref{LEM:chi2}.
For this purpose, we first need to understand the structure
of the linear part of $\chi_n$.

\begin{lemma}\label{LEM:chi1}
Let $0 < \dl < \infty$. Given $n \in \N$, let $L_n = L_n^\dl$ be the linear \textup{(}in $u$\textup{)}
part
of
$\chi_n = \chi_n^\dl$
defined in~\eqref{chi1}.
Then, we have
\begin{align}
L_n = 2(-\L_0)^{n-1}u
\label{chi4}
\end{align}

\noi
for any $n \in \N$.
In particular, under the current mean-zero assumption on $u$, we have
\begin{align}
\int_\T L_n dx = 0
\label{chi5}
\end{align}

\noi
for any $n \in \N$.

\end{lemma}

\begin{proof}
When $n = 1$, it follows from \eqref{chi1} that $L_1 = 2u$.
First, note that each monomial in
 $\chi_n$ is at least linear in $u$.
Thus,  the first term on the right-hand side of the second equation in~\eqref{chi1}
is at least quadratic in $u$.
Hence, it follows from \eqref{chi1}
that
\begin{align*}
L_n = -\L_0  L_{n-1}
\end{align*}

\noi
for any $n \ge 2$,
which yields \eqref{chi4}, since $L_1 = 2u$.
The second claim \eqref{chi5} follows
from~\eqref{chi4} with \eqref{BT1aa}
and  the mean-zero assumption on $u$.
\end{proof}

The following lemma
provides  a description of
the quadratic part of the macroscopic conservation law $  \int_\T \chi_n  dx$
(whose quadratic part turns out to be real-valued).

\begin{lemma}\label{LEM:chi2}
Let $0 < \dl < \infty$. Given $n \in \N$, let $Q_n = Q_n^\dl$ be the quadratic
 \textup{(}in $u$\textup{)}
part
of $\chi_n = \chi_n^\dl$
defined in~\eqref{chi1}.
Then, we have
 $Q_1=0$ and
\begin{align}
 \int_\T Q_n  dx
&  = 2(-1)^{n+1} \sum_{m=0}^{n-2} \frac{1}{\dl^{n-2-m}}
\binom{n}{m+2}
 \sum_{\substack{\l=0\\ \textup{even}}}^{m}  \binom{m+1}{\l+1} \| \Gdl^{\frac{m-\l}{2}} u \|^2_{\dot{H}^{\frac{m}{2}}}
\label{quad1}
\end{align}

\noi
for any $n \ge 2$, 
where we used the convention \eqref{odd1}.

\end{lemma}

Before proceeding to a proof of Lemma \ref{LEM:chi2},
we recall the following combinatorial formulas.
For  non-negative integers $q, \l, m, n$ satisfying $n \geq q\geq 0$, we have
\begin{align}
\sum_{k=0}^\l \binom{\l-k}{m} \binom{q+k}{n} &= \binom{\l+q+1}{m+n+1}, \label{aux1} \\
\sum_{k=0}^{\l} \binom{k}{m}& = \binom{\l+1}{m+1} , \label{aux2}
\end{align}

\noi
where it is understood that
\begin{align}
\binom k m = 0 \quad \text{ for $k < m$}.
\label{aux3}
\end{align}

\noi
For \eqref{aux1}, see
\cite[(5.26) on p.\,169]{GKP}.
The second identity \eqref{aux2}
follows
from setting $n=q=0$ in \eqref{aux1}.

\begin{proof}[Proof of Lemma \ref{LEM:chi2}]
From \eqref{chi1}, we have $Q_1 = 0$ and
\begin{align}
Q_n = -\frac12 \sum_{n_{12}=n} L_{n_1} L_{n_2} - \L_0 Q_{n-1}
\label{quad2}
\end{align}

\noi
for $n \ge 2$.

We first study the first term on the right-hand side of \eqref{quad2}.
By Lemma \ref{LEM:chi1}, \eqref{BT1aa},
and integration by parts with the anti self-adjointness of $\Gdl$, we have
\begin{align}
-& \frac12\int_\T  \sum_{n_{12}=n} L_{n_1} L_{n_2} dx\notag \\
&= -2 \sum_{n_{12}=n} (-1)^{n-2}
\int_\T (\L_0^{n_1-1}  u ) (\L_0^{n_2-1}  u )  dx\notag \\
& = 2 (-1)^{n+1} \sum_{n_{12}=n}
\int_\T u (i\dx + \Gdl\dx +\dl^{-1})^{n_1-1} (-i\dx + \Gdl \dx + \dl^{-1})^{n_2-1} u  dx \notag \\
& = 2 (-1)^{n+1} \sum_{n_{12}=n} \sum_{m_1=0}^{n_1-1} \sum_{m_2=0}^{n_2-1} \frac{1}{\dl^{{n_{12}-2 -m_{12}}}} \binom{n_1-1}{m_1} \binom{n_2-1}{m_2} \notag \\
&\qquad \times \int_\T u (i \dx + \Gdl \dx)^{m_1} (-i\dx + \Gdl\dx)^{m_2} u  dx \notag \\
& = 2 (-1)^{n+1}  \sum_{m=0}^{n-2} \frac{1}{\dl^{n-2-m}}  \sum_{n_{12}=n} \sum_{\substack{0\le m_1 \le n_1-1\\ 0\le m_2 \le n_2-1\\ m=m_{12}}} \binom{n_1-1}{m_1} \binom{n_2-1}{m_2}\notag \\
& \qquad \times \int_\T u(i \dx +\Gdl\dx)^{m_1} (-i\dx + \Gdl \dx )^{m_2} u  dx.
\label{quad3}
\end{align}

\noi
For fixed $0\le m\le n-2$, we can write the inner sums as
\begin{align}
\begin{split}
& \sum_{n_{12}=n}  \sum_{\substack{0\le m_1 \le n_1-1\\ 0\le m_2 \le n_2-1\\ m=m_{12}}}  \binom{n_1-1}{m_1} \binom{n_2-1}{m_2} \\
& \quad \quad \times
\int_\T u(i \dx +\Gdl\dx)^{m_1} (-i\dx + \Gdl \dx )^{m_2} u  dx\\
&\quad  = \sum_{n_{12}=n} \sum_{\substack{0\le m_1 \le n_1-1\\ 0\le m_2 \le n_2-1\\ m=m_{12}}} \sum_{\substack{0\le \l_1\le m_1\\ 0\le \l_2 \le m_2}}
 \binom{n_1-1}{m_1} \binom{n_2-1}{m_2} \\
 &\quad  \quad \times
 \binom{m_1}{\l_1}  \binom{m_2}{\l_2} (-1)^{\l_2}  i^{\l_{12}}  \int_\T u \Gdl^{m-\l_{12}} \dx^m u  dx.
\end{split}
\label{quad4}
\end{align}

\noi
In view of the anti self-adjointness of 
$\Gdl$ and $\dx$, we see that the integral on the right-hand side vanishes unless
 $\l_{12} = \l_1 + \l_2$ is even.
Then, we have
\begin{align}
 \eqref{quad4}
 = \sum_{\substack{\l= 0\\ \text{even}}}^m (-1)^{\frac \l2}
\al_{n, m, \l}
\int_\T u \Gdl^{m-\l} \dx^m u  dx ,
\label{quad5}
\end{align}

\noi
where the coefficient $\al_{n, m, \l}$ is given by
\begin{align*}
\al_{n,  m, \l}=
\bigg[ \sum_{n_{12}=n}  \sum_{\substack{0\le m_1 \le n_1-1\\ 0\le m_2 \le n_2-1\\ m=m_{12}}} \sum_{\substack{0\le \l_1\le m_1\\ 0\le \l_2\le m_2\\ \l_{12}= \l}}\binom{n_1-1}{m_1} \binom{n_2-1}{m_2} \binom{m_1}{\l_1} \binom{m_2}{\l_2}  (-1)^{\l_2}  \bigg] .
\end{align*}

\noi
By \eqref{aux3}
and applying  \eqref{aux1} for the sum in $n_1$ (with
$k = n_1$, $q=0$, $\l=n-2$, $m=m-m_1$, and $n=m_1$) and
then  in $m_1$ (with $k = m_1$, $q=0$, $\l=m$,  $m=\l-\l_1$, and $n=\l_1$), we have
\begin{align}
\al_{n,  m, \l}
 & = \sum_{\substack{0 \le m_1 \le m}} \sum_{\substack{0\le \l_1\le m_1\\ 0 \le \l-\l_1 \le m-m_1}} \sum_{n_1=m_1+1}^{n-1+m_1-m}\binom{n_1-1}{m_1} \binom{n-n_1-1}{m-m_1}\notag\\
 & \hphantom{XXXXXXXX}\times  \binom{m_1}{\l_1} \binom{m-m_1}{\l-\l_1} (-1)^{\l-\l_1}\notag\\
 & = \sum_{0\le m_1\le m} \sum_{\substack{0\le \l_1\le m_1\\ 0\le \l - \l_1 \le m-m_1}} \binom{m_1}{\l_1} \binom{m-m_1}{\l-\l_1}\notag\\
  & \hphantom{XXXXXXXX}\times   (-1)^{\l-\l_1} \sum_{n_1=0}^{n-2} \binom{n_1}{m_1} \binom{n-2-n_1}{m-m_1}
  \notag\\
 & = (-1)^\l \binom{n-1}{m+1} \sum_{\l_1=0}^\l \sum_{m_1=\l_1}^{m-(\l-\l_1)} \binom{m_1}{\l_1} \binom{m-m_1}{\l-\l_1}(-1)^{\l_1} \notag\\
 & =  (-1)^\l \binom{n-1}{m+1} \binom{m+1}{\l+1}\sum_{\l_1=0}^\l (-1)^{\l_1} \notag\\
 & = \binom{n-1}{m+1} \binom{m+1}{\l+1},
\label{quad7}
\end{align}

\noi
where we used the fact that $\l$ is even in the last step.
Hence, it follows from \eqref{quad5},  \eqref{quad7},
and
integration by parts,  using the facts that $\Gdl \dx$ is self-adjoint and
that $\l$ is even
with the convention \eqref{odd1}, 
 that
\begin{align}
 \eqref{quad4}
 = \binom{n-1}{m+1} \sum_{\substack{\l=0\\\text{even}}}^m \binom{m+1}{\l+1} \|\Gdl^{\frac{m-\l}{2}} u \|^2_{\dot{H}^{\frac{m}{2}}}.
\label{quad8}
\end{align}

Therefore, putting \eqref{quad2}, \eqref{quad3}, \eqref{quad4}, and \eqref{quad8} together
with \eqref{BT1aa},
we obtain the following recursive relation:
\begin{align}
\begin{split}
 \int_\T Q_n  dx & = 2 (-1)^{n+1} \sum_{m=0}^{n-2} \frac{1}{\dl^{n-2-m}} \binom{n-1}{m+1} \sum_{\substack{\l=0\\ \text{even}}}^m \binom{m+1}{\l+1} \| \Gdl^{\frac{m-\l}{2}} u \|^2_{\dot{H}^{\frac{m}{2}}} \\
& \quad - \frac1\dl  \int_\T Q_{n-1} dx \\
& =: A_n - \frac1\dl  \int_\T Q_{n-1}dx,
\end{split}
\label{quad8a}
\end{align}

\noi
for $n \ge 2$ and  $Q_1 = 0$.
 Let $\be_{n,m,\l} = 2 \binom{n-1}{m+1} \binom{m+1}{\l+1}$.
 Then, from \eqref{quad8a}, we have
 \begin{align*}
\int_\T Q_n d x& =  \sum_{j=0}^{n-2} \frac{(-1)^j}{\dl^j} A_{n-j} \\
& = \sum_{j=0}^{n-2} (-1)^{n-j+1 +j} \sum_{m=0}^{n-j-2} \frac{1}{\dl^{n-j-2-m+j}} \sum_{\substack{\l=0\\\text{even}}}^m \be_{n-j,m,\l} \|\Gdl^{\frac{m-\l}{2}} u \|^2_{\dot{H}^{\frac{m}{2}}} \\
& = (-1)^{n+1} \sum_{m=0}^{n-2} \frac{1}{\dl^{n-2-m}} \sum_{\substack{\l=0\\\text{even}}}^{m} \bigg(\sum_{j=0}^{n-2-m} \be_{n-j, m, \l}\bigg) \| \Gdl^{\frac{m-\l}{2}} u \|^2_{\dot{H}^{\frac{m}{2}}} \\
& = 2 (-1)^{n+1} \sum_{m=0}^{n-2} \frac{1}{\dl^{n-2-m}} \sum_{\substack{\l=0\\\text{even}}}^{m} \binom{m+1}{\l+1} \binom{n}{m+2} \| \Gdl^{\frac{m-\l}{2}} u\|^2_{\dot{H}^{\frac{m}{2}}},
\end{align*}

\noi
where, in the last step, we used
\eqref{aux3} and  \eqref{aux2}:
\begin{align*}
\sum_{j=0}^{n-2-m} \be_{n-j,m,\l} &  =2 \binom{m+1}{\l+1} \sum_{j=0}^{n-2-m} \binom{n-1-j}{m+1}
 = 2 \binom{m+1}{\l+1}\sum_{j=m+1}^{n-1} \binom{j}{m+1}\\
& = 2 \binom{m+1}{\l+1} \binom{n}{m+2}.
\end{align*}

\noi
This proves \eqref{quad1}.
\end{proof}

\subsection{Construction and convergence of  the deep-water conservation laws}
\label{SUBSEC:A3}

In this subsection,
by taking a linear combination of $ \Re \int_\T \chi_n dx$,
we construct deep-water conservation laws
and
prove
their convergence to the corresponding
BO conservation laws,
where,  in the latter,
a perturbative viewpoint plays a crucial role;
see Proposition~\ref{PROP:cons1}.

Given $0 < \dl < \infty$,
let  $\Qdl$ be the perturbation operator defined in~\eqref{Qdl1}.
Then,
from  \eqref{BT1aa} and \eqref{Qdl1},  we can rewrite \eqref{chi1} as
\begin{align}
\begin{split}
\chi_1 & = 2u, \\
\chi_n & =-\sum_{j=2}^n \frac{1}{j!} 
\sum_{\substack{n_1, \dots, n_j \in \N\\n_{1\cdots j} = n}}
\chi_{n_1} \cdots \chi_{n_j}\\
& \quad - (\H -i )\dx \chi_{n-1} - (\dl^{-1} + \Qdl) \chi_{n-1}, \quad n\ge 2.
\end{split}
\label{DE6}
\end{align}

\noi
We introduce the classes $\Pc_n(u)$
of (monic)\footnote{Here, ``monic" means that the fundamental form in \eqref{mono1}
is monic.} monomials in $u$,
adapted to our perturbative viewpoint~\eqref{DE6}.
Compare these with the classes of monomials
introduced in  \cite[Section~2]{TV0}
in the context of the BO equation,
where there is no perturbation operator $\Qdl$.

\begin{definition}
\label{DEF:mono1}
\rm
Let $u \in C^\infty(\T)$ with $\int_\T u dx = 0$.
Given $n \in \N$,
define
the classes $\Pc_n(u)$ of monomials by setting
\begin{align*}
\Pc_1 (u)&  = \Big\{  \H^{\al_1} \Qdl^{\be_1} \dx^{\g_1} u :   \al_1\in \{0,1\}, \,
  \be_1, \g_1 \in \Z_{\ge 0}\Big\}, \\
\Pc_2 (u)&  = \Big\{  \big[\H^{\al_1} \Qdl^{\be_1} \dx^{\g_1} u\big] \big[\H^{\al_2} \Qdl^{\be_2} \dx^{\g_2} u\big]:
 \al_1,\al_2 \in\{0,1\}, \,   \be_1 , \be_2, \g_1,\g_2 \in \Z_{\ge 0} \Big\}, \\
\Pc_n (u)&  = \bigg\{  \prod_{\l=1}^k \H^{\al_\l} \Qdl^{\be_\l} p_{j_\l}(u) :
k \in \{2, \ldots, n\}, \, j_\l \in\N, \, j_{1\cdots k} = n,
    \\
& \hphantom{lXXXXXXXXXXX}
p_{j_\l} (u) \in \Pc_{j_\l}(u),\,
\al_\l \in \{0,1\},  \,  \be_\l \in \Z_{\ge 0}
 \bigg\}.\end{align*}

 \noi
Here, we used the fact that,
in \eqref{DE6}, the Hilbert transform comes with $\dx$,
namely, $\H$ always acts on mean-zero functions
and that
 $\H^2 = -\Id$
on mean-zero functions; see also Footnote~\ref{FT:1}.
Given a monomial $p(u) \in \Pc_n(u)$,
we define
its fundamental form  $\wt p(u)  \in \Pc_n(u)$
which is obtained
 by dropping  all the $\H$ and $\Qdl$ operators
 appearing in $p(u)$.

 Given a monomial $p(u) \in \Pc_n(u)$,
 suppose that its fundamental form $\wt p(u)$ is given by
\begin{align}
\wt p(u) = \prod_{j=1}^n \dx^{\g_j} u.
\label{mono1}
\end{align}
 Then, we have
 \begin{align*}
\#u (p(u))  = n\qquad \text{and}
\qquad
\#\dx  (p(u)) = \gamma_1 + \dots + \gamma_n
\end{align*}

\noi
in the sense of  Definition \ref{DEF:ord1}.
With our perturbative viewpoint \eqref{DE6}, we also let
\begin{align*}
\#\Qdl &  = \#\Qdl(p(u))= \text{number of appearance of $\Qdl$ in $p(u)$}.
\end{align*}

\noi
Then, from Lemma \ref{LEM:chi0} and \eqref{ord1a} with \eqref{Qdl1}, we have\footnote{Recall from Definition \ref{DEF:ord1}
that $\#\dx$ counts the number of {\it explicit} appearance of $\dx$.
Namely, it does not count $\dx$ appearing in the definition \eqref{Qdl1} of $\Qdl$.}
\begin{align}
n = \ord(p) = \# u + \# \dx +  \# \dl^{-1} + \# \Qdl 
\label{mono2}
\end{align}

\noi
for any monomial $p(u)$ in
the $n$th microscopic conservation law $\chi_n$.
Moreover, given a monomial $p(u) \in \Pc_n(u)$
with its fundamental form $\wt p(u)$ as in \eqref{mono1},
we let $|p(u)|$  denote the maximum number of (explicit) derivatives
(applied to a single factor $u$)
by setting
\begin{align*}
|p(u)| &=\max_{j=1, \dots, n} \g_j,
\end{align*}

\noi
where $\g_j$ is as in \eqref{mono1}.

\end{definition}

We are now ready to define the deep-water conservation laws $E^\dl_\frac k2(u)$
and present a proof of
Theorem \ref{THM:1}.
Recall that we are working under the mean-zero assumption on $u$
throughout the paper.

\begin{proposition}\label{PROP:cons1}
\textup{(i)}
Let $0 < \dl < \infty$.
Given $k\in \Z_{\ge 0}$,
define $E^\dl_{\frac{k}{2}}(u)$ recursively by
\begin{align}
E^\dl_{\frac{k}{2}}(u)
&  = \frac{(-1)^{k+1}}{4 b_k} 
\bigg[ \Re \int_\T  \chi_{k+2}^\dl  dx
- 4(-1)^{k+1}\sum_{j=0}^{k-1} \frac{1}{\dl^{k-j}}
 \binom{k+2}{j+2}
b_j
  E^\dl_{\frac j2}(u) \bigg]
\label{DE1}
\end{align}

\noi
with the understanding that the second sum vanishes when $k = 0$,
where $\chi_n^\dl$ is the microscopic conservation law defined  in \eqref{chi1}
and $b_k$ is given by 
\begin{align}
b_k = \sum_{\substack{\l=0\\ \textup{even}}}^{k}
 \binom{k+1}{\l+1} .
\label{DE1a}
\end{align}

\noi
Then,
$E^\dl_{\frac{k}{2}}(u) $ is conserved under the flow of ILW~\eqref{ILW}.
Moreover,
\begin{itemize}
\item[\rm (i.a)]
 $E^\dl_{\frac{k}{2}}(u) $ has
  rank $k+2$
in the sense of Definition \ref{DEF:ord1},

\smallskip
\item[\rm (i.b)]
 the  quadratic part 
 of  $E^\dl_{\frac{k}{2}}(u) $
 does not depend on $\dl$ in an explicit manner and
 is given by
\begin{align}
\frac 1{2 b_k}
 \sum_{\substack{\l=0\\ \textup{even}}}^{k}
 \binom{k+1}{\l+1} \| \Gdl^{\frac{k-\l}{2}} u \|^2_{\dot{H}^{\frac{k}{2}}} , 
\label{DE2}
\end{align}

\noi
where we used the convention \eqref{odd1}.

\end{itemize}

\noi
In particular, we have
\begin{align}
E^\dl_0(u) &= \frac12 \|u\|^2_{L^2},
\label{DE2a}
\end{align}

\noi
and, defining $a_{k, \l}$ by
\begin{align}
a_{k, \l} = 
\frac 1{ b_k}
  \binom{k+1}{\l+1},
\label{DE2b}
\end{align}

\noi
the conservation law  $E^\dl_{\frac{k}{2}}(u)$ in \eqref{DE1}
reduces to the form
\eqref{E00} with \eqref{E0}.

\medskip

\noi
\textup{(ii)}
By extending the definition \eqref{DE1} to the  case $\dl = \infty$,
we obtain the following conservation laws for BO \eqref{BO}\textup{:}
\begin{align}
E^\BO_{\frac k 2}(u) &  =  \frac{(-1)^{k+1} }{4b_k}
 \Re \int_\T \chi_{k+2}^\BO  dx ,
\label{DE3}
\end{align}

\noi
where $\chi_n^\BO$
is the microscopic conservation law for BO defined  in \eqref{chi2}.
Moreover,

\begin{itemize}
\item[\rm (ii.a)]
 $E^\BO_{\frac{k}{2}}(u) $ has
  rank $k+2$ with $\#\dl^{-1} = 0$,

\smallskip
\item[\rm (ii.b)]
 $E^\BO_{\frac{k}{2}}(u) $
is of the form \eqref{E2}.

\end{itemize}

\noi
Given $0 < \dl < \infty$,
 we have
\begin{align}
E^\dl_{\frac k 2}(u) = E_{\frac k 2}^\BO (u)  + \EE^\dl_{\frac k 2}(u), 
\label{DE3b}
\end{align}

\noi
where $\EE^\dl_{\frac k 2}(u)$
consists of  the terms in $E^\dl_{\frac k 2}(u)$
which depend on  $\dl$ in an explicit manner
and the terms with $\Gdl \dx$, where $\Gdl \dx$ is replaced by
the perturbation operator $\Qdl$ defined in~\eqref{Qdl1}.
Moreover, we have
\begin{align}
\lim_{\dl\to\infty}
\EE^\dl_{\frac k 2}(u) = 0,
\qquad \text{namely,}\qquad
\lim_{\dl\to\infty} E^\dl_{\frac k 2}(u) = E_{\frac k 2 }^\BO(u)
\label{DE4}
\end{align}

\noi
 for any $u\in H^{\frac k 2}(\T)$.

\end{proposition}

\begin{proof}
Proceeding inductively,
we see that
 $E^\dl_{\frac k 2}(u)$ in \eqref{DE1}
 is defined as a  linear combination of the conserved quantities $\Re\int \chi_n dx$.
 Hence,
 $E^\dl_{\frac k 2}(u)$ is also  conserved under the flow of ILW (and of BO when $\dl = \infty$).

From \eqref{DE1} and \eqref{chi1} with \eqref{BT1aa}, 
we see that \eqref{DE2a} holds
and thus
the claim (i.a)  holds true when $k = 0$.
We proceed by induction and assume that the claim (i.a)
holds for any $0 \le j \le k-1$.
On the one hand,
it follows from Lemma~\ref{LEM:chi0} that
$\Re \int_\T \chi_{k+2} \, dx$ has rank $k+2$.
On the other hand, by the inductive hypothesis,
we see that
$\dl^{-(k-j)} E^\dl_{\frac j2}(u)$ has rank $(k-j) + (j+2) = k+2$.
Hence,  we conclude from \eqref{DE1} that  $E^\dl_{\frac k 2}(u)$ has rank $k+2$.
The claim (ii.a) follows from Lemma \ref{LEM:chi0}
(which also holds for $\dl = \infty$).

Next, we prove the claim (i.b).
Once again, we proceed by induction.
In view of \eqref{DE2a},
the claim (i.b) holds when $k = 0$.
Now,
assume that the claim (i.b)
holds for $E^\dl_{\frac{j}{2}}(u) $,  $0 \le j \le k-1$.
Then,
from \eqref{DE1}
and the inductive hypothesis,
the quadratic part in $u$ of $E^\dl_\frac k2(u)$
is given by
\begin{align*}
& \frac{(-1)^{k+1}}{4b_k} 
\bigg[ \Re \int_\T  Q_{k+2}  dx 
- 2(-1)^{k+1}\sum_{j=0}^{k-1} \frac{1}{\dl^{k-j}}
 \binom{k+2}{j+2}
 \sum_{\substack{\l=0\\ \textup{even}}}^{j}
 \binom{j+1}{\l+1} \| \Gdl^{\frac{j-\l}{2}} u \|^2_{\dot{H}^{\frac{j}{2}}} \bigg],
\end{align*}

\noi
which agrees with \eqref{DE2}
in view of
 Lemma \ref{LEM:chi2},
 thus proving the claim (i.b).
The claim (ii.b) directly follows from
\eqref{DE3}  and Lemma \ref{LEM:chi2} (which
 also holds for $\dl = \infty$)
 with 
$\| \Gdl^{\frac{m-\l}{2}} u \|^2_{\dot{H}^{\frac{m}{2}}}\Big|_{\dl = \infty}
= \|  u \|^2_{\dot{H}^{\frac{m}{2}}}$
and \eqref{DE1a}.

Let $\chi_{n, 0}$ be the $\dl$-free
 part of $\chi_n$.
Namely, 
$\chi_{n, 0}$
is the collection of the monomials in $\chi_n$
 which do not depend on $\dl$
in an explicit manner or on $\Qdl$.
 Then, from \eqref{DE6},
we see that
 $\chi_{n,0}$ satisfies the following recursive relation:
\begin{align}
\begin{split}
\chi_{1,0}& = 2u, \\
\chi_{n, 0}& = - \sum_{j=2}^n \frac{1}{j!}
\sum_{\substack{n_1, \dots, n_j \in \N\\n_{1\cdots j} = n}}
 \chi_{n_1,0} \cdots \chi_{n_j, 0}
- (\H-i)\dx \chi_{n-1,0}, \quad n \ge 2.
\end{split}
\label{DE7}
\end{align}

\noi
By comparing \eqref{DE7} and \eqref{chi2}, we see
that
 \begin{align}
\chi_{n,0} = \chi_n^\BO
\label{DE8}
 \end{align}

\noi
 for any $n \in \N$.
Now, let $E^\dl_{\frac k 2, 0}(u)$  be
 the $\dl$-free part of
  $E^\dl_{\frac k 2}(u)$
defined in \eqref{DE1}.
Then, from~\eqref{DE1} and  \eqref{DE3} with \eqref{DE8}, we obtain
\begin{align}
 E^\dl_{\frac k 2, 0}(u) =\frac{(-1)^{k+1}}{4b_k}
\bigg[ \Re \int_\T  \chi_{k+2,0} \, dx  \bigg] = E_{\frac k 2}^\BO(u).
\label{DE9}
\end{align}

\noi
Hence, from \eqref{DE3b}, \eqref{DE9},
and the definition of
$ E^\dl_{\frac k 2, 0}(u)$, we have
\begin{align}
 \EE^\dl_{\frac k 2}(u) =
E^\dl_{\frac k 2}(u) -
E^\dl_{\frac k 2, 0}(u),
\label{DE10}
\end{align}

\noi
where  all the terms in
$ \EE^\dl_{\frac k 2}(u)$ must involve
 at least one power of $\dl^{-1}$ or  $\Qdl$.

 It remains to prove \eqref{DE4}.
When $k = 0$,
we have $E^\dl_0(u) = E_{0 }^\BO(u)
= \frac12 \|u\|^2_{L^2} $
and thus there is nothing to prove.
In the following, we assume $k \ge 1$.
From 
\eqref{DE3b} with 
Parts (i.a) and (ii.a)
(see also \eqref{mono2}), 
\eqref{DE10}, 
 and \eqref{chi5} in Lemma~\ref{LEM:chi1},
we see that  any monomial in  $\EE^\dl_{\frac k 2}(u)$ satisfies
\begin{align}
\# u + \# \dx  + \# \dl^{-1} + \# \Qdl = k+2,
\quad \# \dl^{-1} + \# \Qdl \ge 1,  \quad \text{and} \quad \# u \ge 2,
\label{DE11}
\end{align}

\noi
where it is understood that
$\# \dx$ counts the number of {\it explicit} appearance of $\dx$
(namely, it does not count $\dx$ appearing in $\Qdl$).
In particular, we have
\begin{align}
\#\dx \le k - 1.
\label{DE11a}
\end{align}
Thus, each monomial in
 $\EE^\dl_{\frac k 2}(u)$ is of the form
 $\dl^{-m}\int_\T p(u)  dx$ for some
 $0\le m \le k$
 and
   $p(u)\in \Pc_j(u)$
   with
$j = \#u \in\{2, \dots, k+2-m\}$.
Then, by integration by parts
with Definitions \ref{DEF:ord1} and \ref{DEF:mono1}
and \eqref{DE11a}, we have
\begin{align}
\begin{split}
\EE^\dl_{\frac k 2}(u)
& =\sum_{m=1}^{k}    \sum_{j =2}^{k+2-m}
\sum_{\l =0}^{k+2-m-j} \sum_{\substack{p(u) \in \Pc_j(u)\\
\eqref{DE11} \text{ holds}\\ \#\dx(p) = \l \\
|p(u)| \le \lceil \frac{k-1}{2} \rceil}  } \dl^{-m}
c(p) \int_\T p(u)  dx \\
& \quad + \sum_{j=2}^{k+1} \sum_{\l=0}^{k+1-j}
\sum_{\substack{p(u) \in \Pc_j(u) \\
\eqref{DE11} \text{ holds}\\ \#\dx (p)= \l\\
\#\Qdl(p) \ge  1  \\ |p(u)| \le \lceil \frac{k-1}{2} \rceil}}
c(p) \int_\T p(u)  dx 
\end{split}
\label{DE12}
\end{align}

\noi
for some constants $c(p) \in \R$, independent of $\dl$, 
where $\lceil x \rceil$ denotes
the smallest integer greater than or equal to $x$.
As for the second term on the right-hand side of \eqref{DE12},
we have $\#\Qdl\ge 1$ since it does not depend on $\dl$ in an explicit manner
but is part of $\EE^\dl_{\frac k 2}(u) $
(which must involve at least one power of $\dl^{-1}$ or  $\Qdl$).

Fix $p(u) \in \Pc_j(u)$
with $|p(u)| \le \lceil \frac{k-1}{2} \rceil$,  appearing  on the right-hand side of \eqref{DE12}.
(i)~When $k = 1$, we have
$\lceil \frac{k-1}{2}\rceil = 0$ and thus there is no explicit $\dx$
on any factor of $u$ in $p(u)$.
Then, 
 by
Cauchy-Schwarz's inequality, 
we have 
\begin{align}
\bigg|\int_\T p(u)  dx \bigg|
\les \|u \|_{L^2}^2
\les \big(1 + \|u \|_{H^{\frac k2}}\big)^{k+1}.
\label{DE13}
\end{align}

\noi
(ii)~When $k\ge 3$ is odd, it follows from  \eqref{DE11a}
that there are at most two factors of $u$ in $p(u)$ with
$\lceil \frac{k-1}{2} \rceil = \frac{k-1}2$ derivatives,
while all the other factors have at most
$\frac{k-3}2  < \frac k2  - \frac 12$ derivatives.
In this case,  
 by viewing the product on the physical side as
an iterated convolution on the Fourier side
and 
applying 
Young's inequality 
and 
Cauchy-Schwarz's inequality,\footnote{Here, we work on the Fourier side
such that we can drop $\H$ and apply \eqref{Q1} when $\#\Qdl(p) \ge 1$
by 
simply  putting  
all the terms in an absolute value on the Fourier side.
This is due to the fact  that these operators  may act on a product.
Namely, it is not enough to  apply 
the $L^2$-boundedness of $\H$ or \eqref{Q1}
to  the right-hand sides of 
 \eqref{DE13a} and~\eqref{DE13b}.
A similar comment applies to the proof of Proposition \ref{PROP:Scons2}.} 
we have
\begin{align}
\bigg|\int_\T p(u)  dx \bigg|
\les
\big(1 + \|u \|_{H^{\frac{k-1}{2}}}\big)^{2}
 \big(1 + \|u \|_{\F L^{\frac{k-3}{2}, 1}}\big)^{k-1}
\les \big(1 + \|u \|_{H^{\frac k2}}\big)^{k+1}, 
\label{DE13a}
\end{align}

\noi
where $\F L^{s, r}(\T)$ denotes the Fourier-Lebesgue space defined by the norm:
\begin{align}
\| u \|_{\F L^{s, r}} = \| \jb{n}^s \ft u(n) \|_{\l^r_n}.
\label{FL1}
\end{align}

\noi
(iii)~When $k\ge 2$ is even, it follows from  \eqref{DE11a}
that the maximum number of derivatives on any factor is
$\lceil \frac{k-1}{2} \rceil = \frac{k}2$.
Moreover, there can be at most one term with
$\frac{k}2$-many derivatives, while
all the other terms have at most
$\frac{k}2 - 1$ derivatives.
Hence, proceeding as in \eqref{DE13a}, we have
\begin{align}
\bigg|\int_\T p(u)  dx \bigg|
\les \big(1+ \| u \|_{H^{\frac k2}}\big)\big(1 + \|u \|_{\F L ^{\frac{k}{2} - 1, 1}}\big)^{k}
\les \big(1 + \|u \|_{H^{\frac k2}}\big)^{k+1}.
\label{DE13b}
\end{align}

\noi
In all the cases (i), (ii), and (iii),
if, in addition, we have  $\#\Qdl(p(u)) \ge  1$,
in view of \eqref{Q1},
we obtain a power of $\dl^{-1}$ in \eqref{DE13}, 
\eqref{DE13a}, and \eqref{DE13b}.
Hence, from \eqref{DE12}, \eqref{DE13}, \eqref{DE13a}, and \eqref{DE13b}
with $ \# \dl^{-1} + \# \Qdl \ge 1$, we obtain
\begin{align}
|\EE^\dl_{\frac k 2}(u) |
 \les \frac1\dl \big( 1 + \| u\|_{H^{\frac k 2}}\big)^{k+1} \too 0,
\label{DE13c}
\end{align}

\noi
as $\dl \to \infty$.
This proves \eqref{DE4} for $k \ge 1$.
This concludes the proof of Proposition \ref{PROP:cons1}.
\end{proof}

In the remaining part of this subsection,
we take a closer look at the structure of the cubic terms in the 
deep-water conservation law
$E^\dl_{\frac k 2}(u)$.

\begin{lemma}\label{LEM:cub1}
Let  $k\in\N$.
Then,  the cubic \textup{(}in $u$\textup{)} part in $E^\dl_{\frac k 2}(u)$
consists of monomials  of the form
\begin{align}
\frac{1}{\dl^\s}\int \prod_{j=1}^3 \big(\Qdl^{\al_j} (\H\dx)^{\be_j} \dx^{\g_j} u\big)  dx
\label{DF1}
\end{align}

\noi
for some
 $\s, \al_j, \b_j,\g_j \in \Z_{\ge 0}$ such that $\s + \sum_{j=1}^3 (\al_j+\be_j+\g_j)=k-1$.

\end{lemma}

\begin{proof}
From \eqref{DE1},  we know that $E^\dl_{\frac k 2}(u)$
is given by a linear combination of
\begin{align*}
\frac{1}{\dl^{k-j}} \int \Re \chi_{j+2}  dx
\end{align*}

\noi
for $j=0, \ldots, k$.
Then, \eqref{DF1} follows 
from \eqref{DE6} (which in particular implies that $\H$ must come with $\dx$), 
\eqref{mono2}, 
and noting that if $\Qdl$, $\H\dx$, or $\dx$ acts on the product
of two factors, then by integration by parts, we can move it onto the
remaining factor
(thanks to the fact that $\#u = 3$).
\end{proof}

Fix $k \in \N$.
Integrating by parts, we can rewrite the cubic terms in $E^\dl_{\frac{k}{2}}(u)$
to guarantee that they  have at most $m = \lceil \frac {k-1}2\rceil$ derivatives on each factor.
While
most contributions  have at most $m-1$ derivatives on each factor,
we need to pay particular attention to the terms
where there is a factor with $m$ derivatives
(even after integration by parts).
Let
$B^\dl_\frac k2(u)$ denote the collection
of the worst contributions from the cubic terms in $E^\dl_{\frac{k}{2}}(u)$,
defined by

\begin{itemize}
\item[(i)]
when $k = 2m$ for some $m \in \N$, 
\begin{align}
B^\dl_{\frac{2m}{2}}(u) &= \sum_{\substack{p(u) \in \mathcal{P}_3(u)\\ \wt{p}(u) = u \dx^{m-1}u\dx^m u
\\ \#\Qdl(p) = 0}} c(p) \int_\T p(u)  dx ,
\label{DF3a}
\end{align}

\item[(ii)]
when $k = 2m+1$ for some $m \in \Z_{\ge 0}$, 
\begin{align*}
B^\dl_{\frac{2m+1}{2}}(u) &= \sum_{\substack{p(u) \in \mathcal{P}_3(u)\\ \wt{p}(u) = u \dx^{m}u\dx^m u \\
\#\Qdl(p) = 0}} c(p) \int_\T p(u)  dx 
\end{align*}

\end{itemize}

\noi
for some constants $c(p) \in \R$, independent of $\dl$, 
where $\wt p(u)$ denotes the fundamental form
of a monomial $p(u)$
in the sense of Definition \ref{DEF:mono1}.

We have the following characterization of $B^\dl_\frac k2(u)$.

\begin{lemma}\label{LEM:cub2}
\textup{(i)}
Let $m \in \N$
and $0 < \dl < \infty$.
Then, there exist  $c_0,c(p) \in \R$, independent of $\dl$, such that
\begin{align}
B^\dl_{\frac{2m}{2}}(u) & = c_0 \int_\T u (\H\dx^{m-1}u) (\dx^{m} u )  dx + \sum_{\substack{p(u) \in \mathcal{P}_3(u)\\ |p(u)| \leq m-1\\ \#\Qdl(p) = 0}} {c}(p) \int_\T p(u)  dx , 
\label{DF4}
\end{align}

\noi
where 
$ \mathcal{P}_3(u)$, $|p(u)|$, and $\#\Qdl(p)$ are as in Definition \ref{DEF:mono1}.

\smallskip

\noi
\textup{(ii)}
Let $m\in\Z_{\ge 0}$
and $0 < \dl < \infty$.
 Then, $B^\dl_{\frac{2m+1}{2}}(u)$ is given by
\begin{align}
B^\dl_{\frac{2m+1}{2}}(u)
& =
\sum_{\vec \al = (\al_1, \al_2, \al_3) \in \mathcal{A}}
 c_{\vec \al}
\int_\T (\H^{\al_1} u ) (\H^{\al_2} \dx^{m} u) (\H^{\al_3} \dx^m u ) dx
\label{DF5}
\end{align}
for some constants $c_{\vec \al}  \in \R$,
 independent of $\dl$,
where the index set $\mathcal A$ is given by 
\begin{align}
\mathcal A = \big\{(0, 0, 0), (0, 0, 1),  (0, 1, 1)\big\}.
\label{DF6}
\end{align}

\end{lemma}

\begin{proof}
(i) 
From \eqref{DF3a}
and \eqref{mono2}, 
we see that 
any monomial in $B^\dl_{\frac{2m}{2}}(u)$ satisfies
 $\# \dl^{-1} = \# \Qdl =0$.
Namely, it  must come from  the $\dl$-free
part  $E^\dl_{\frac {2m} 2, 0}(u)$, 
and hence from $E_{\frac {2m} 2}^\BO(u)$
in view of \eqref{DE9}.
Then, \eqref{DF4} follows
from  \cite[Proposition 2.2]{TV1}.

\medskip

\noi
(ii) Similarly, 
any monomial in $B^\dl_{\frac{2m+1}{2}}(u)$ satisfies
 $\# \dl^{-1} = \# \Qdl =0$
 and thus 
  must come from  the $\dl$-free
part  $E^\dl_{\frac {2m+12} 2, 0}(u)$, 
and hence from $E_{\frac  {2m+12} 2}^\BO(u)$
in view of \eqref{DE9}.
Then, \eqref{DF5} follows
from  
\cite[Proposition~8.1]{TV2}.
\end{proof}

\subsection{Structure of the deep-water conservation laws}
\label{SUBSEC:A4}

We conclude this section by stating the final structure
of the deep-water conservation law $E^\dl_{\frac k2}(u)$
defined in \eqref{DE1}.

Let $k \in \Z_{\ge 0}$.
Then, from (the proof of) Proposition \ref{PROP:cons1},
 we
see that
the deep-water conservation law $E^\dl_{\frac k2}(u)$
defined in \eqref{DE1}
is indeed given by \eqref{E00}:
\begin{align}
\begin{split}
E^\dl _{0}(u) & = \frac12 \|u\|^2_{L^2}, \\
E^\dl_{\frac{k}{2}}(u)
&  = \frac12\sum_{\substack{\l=0\\ \text{even}}}^{k} a_{k, \l} \| \Gdl^{\frac{k-\l}{2}} u \|^2_{\dot{H}^{\frac{k}{2}}} + R^\dl_{\frac{k}{2}}(u), \quad k\in \N,
\end{split}
\label{cons1}
\end{align}

\noi
with  $a_{k, \l}$ as in \eqref{DE2b}.
Here, the interaction potential  $R^\dl_{\frac{k}{2}}(u)$,
satisfying $\# u \ge 3$,  is given by
\begin{align}
R^\dl_{\frac{k}{2}}(u) = 
A^\dl_{\frac k 2,\frac{k}{2}}(u) + \sum_{\l=1}^{k-1} \frac{1}{\dl^{k-\l}} A^\dl_{\frac k 2,\frac{\l}{2}}(u),
\label{R1}
\end{align}

\noi
where $A^\dl_{\frac k 2,\frac{\l}{2}}(u)$ is defined as follows,  depending on the parity of $\l$;
when $\l = 2m$ is even, we have
\begin{align}
A^\dl_{\frac k 2,\frac{2m}{2}}(u)
&=
\sum_{\substack{p(u) \in \mathcal{P}_3(u) \\ \wt{p}(u) = u \dx^{m-1}u \dx^m u \\ \#\Qdl(p) = 0}}
c(p) \int_\T p(u)  dx \notag \\
& \quad +  \sum_{\substack{p(u) \in \mathcal{P}_j(u), \\ j=3, \ldots, 2m+2 \\
\#\dx(p)+ \#\Qdl(p)   = 2m +2 - j\\ |p(u)| \leq m-1 }} \hspace{-5mm}
c(p) \int_\T p(u)  dx
\label{Aeven}
\end{align}

\noi
with suitable constants  $c(p) \in \R$,
 independent of $\dl$, 
while,
when $\l = 2m+1$ is odd, we have
\begin{align}
A^\dl_{\frac k2,\frac{2m+1}{2}}(u)
&=
\sum_{\substack{p(u) \in \mathcal{P}_3(u) \\ \wt{p}(u) = u \dx^{m}u \dx^m u \\
\#\Qdl(p) = 0}} c(p) \int_\T p(u)  dx \notag \\
& \quad + \sum_{\substack{p(u) \in \mathcal{P}_3(u) \\ 
\wt{p}(u) = \dx^\kk u \dx^{m-1}u \dx^m u, \, \kk = 0, 1 \\
\#\Qdl(p) = 1-\kk}} c(p) \int_\T p(u)  dx\notag \\
& \quad +  \sum_{\substack{p(u) \in \mathcal{P}_4(u) \\ \wt{p}(u) = u^2 \dx^{m-1}u \dx^m u \\
\#\Qdl(p) = 0}} c(p) \int_\T p(u)  dx \notag \\
& \quad +  \sum_{\substack{p(u) \in \mathcal{P}_j(u), \\ j=3, \ldots, 2m+3 \\
\#\dx(p) + \#\Qdl(p) = 2m+3 - j\\ |p(u)| \leq m-1}}
 \hspace{-1cm} c(p) \int_\T p(u)  dx
  \label{Aodd}
\end{align}

\noi
with the understanding that, when $m = 0$, 
there is no contribution from the second, third, and fourth
terms on the right-hand side of \eqref{Aodd}.
In particular, we have\footnote{Note that there is no $\H$ in \eqref{Aodd1}
since $\H$ must come with  $\dx$ (before integration by parts)
but $\#\dx = 0$.}
\begin{align}
A^\dl_{\frac k2,\frac{1}{2}}(u) = c_0 \int_\T u^3 dx
\label{Aodd1}
\end{align}

\noi
for some constant $c_0 \in \R$, 
 independent of $\dl$.
Here, we used the following facts:
\smallskip
\begin{itemize}
\item
When $\l = 2m$ and $|p(u)| = m$, 
it follows from Proposition \ref{PROP:cons1}\,(i.a)
(see also \eqref{DE11}) with \eqref{R1} that 
\begin{align*}
 \# \dx    + \# \Qdl = k+2 - \#u - \# \dl^{-1}\le  2m -1, 
\end{align*}

\noi
which implies that $\wt{p}(u) = u \dx^{m-1}u \dx^m u$
(after integration by parts)
since, if $\wt{p}(u) = \dx^{\kk-1} u \dx^{m-\kk}u \dx^m u$
for some $2\le \kk \le \frac{m+1}{2}$,
then integration by parts allows us to reduce
to the case $|p(u)| \le m-1$.

\smallskip
\item
Similarly, when $\l = 2m + 1$ and $|p(u)| = m$, 
we have 
\begin{align*}
 \# \dx    + \# \Qdl = k+2 - \#u - \# \dl^{-1}\le  2m +3 - \#u. 
\end{align*}

\noi
Thus, when $\#u = 4$, we must have $\wt{p}(u) = u^2 \dx^{m-1}u \dx^m u$. 
When $\#u = 3$, we have
either $\wt{p}(u) = u \dx^{m}u \dx^m u$
or $\dx^\kk u \dx^{m-1}u \dx^m u$
with $\#\Qdl = 1-\kk$,  $\kk = 0, 1$
(after integration by parts).

\end{itemize}

\smallskip

\noi
Moreover, it follows from
Lemma~\ref{LEM:cub2} that
 the leading  contribution
 of the cubic terms in $A^\dl_{\frac k 2, \frac{2m}{2}}(u)$
 is given by
\begin{align*}
 c_0 \int_\T u (\H \dx^{m-1} u ) (\dx^m u) dx
\end{align*}
\noi

\noi
for some $c_0 \in \R$,
while
 the leading  contribution
 of the cubic terms in $A^\dl_{\frac{k}{2}, \frac{2m+1}{2}}(u)$
 is given by
\begin{align*}
\sum_{\vec \al = (\al_1, \al_2, \al_3)\in \A}c_{\vec \al}
\int_\T (\H^{\al_1}u)(\H^{\al_2} \dx^m u )( \H^{\al_3} \dx^m u ) dx
\end{align*}

\noi
for some constants $c_{\vec \al} \in \R$,
where
 $\A=\{(0,0,0), (0,0,1),
(0,1,1)\}$
is as in \eqref{DF6}.

\section{Shallow-water conservation laws}
\label{SEC:cons2}

In this section, we present a proof of 
Theorem~\ref{THM:2}.
Namely, 
we construct  conservation laws in the shallow-water regime
and establish their  shallow-water convergence;
see Propositions~\ref{PROP:Scons1} and~\ref{PROP:Scons2}.
In particular, we establish a 2-to-1 collapse of the shallow-water conservation laws
to the KdV conservation laws, 
as stated in \eqref{E7}.
The construction
of a full set of 
 shallow-water conservation laws with non-trivial shallow-water limits
 is particularly important 
 in order to study the shallow-water convergence problem
 for the associated 
 \GGMs;
 see
Section \ref{SEC:SGM1}
for further details.

As explained in Section \ref{SEC:1}, 
while the scaling transform in \eqref{scale1} applied to the deep-water conservation laws $E^\dl_{\frac k 2}(u)$ constructed in Section~\ref{SEC:cons1} 
yields conservation laws for the scaled ILW~\eqref{sILW}, 
they are not suitable for studying  
  the shallow-water convergence problems.
In order to overcome this issue,   
  we start with an alternative  B\"acklund transform~\eqref{BTX1} for the scaled ILW and derive a recursive formulation for the microscopic conservation laws $h_n$
 in Subsection~\ref{SUBSEC:B1}.  
In Subsection~\ref{SUBSEC:B2}, we then study the structure of the microscopic conservation laws $h_n$ and,
by taking suitable linear combinations of $h_n$, 
we construct alternative
microscopic conservation laws 
 $\wt{h}_n$ (see  \eqref{ht0}), 
 whose quadratic parts, upon integration over $\T$, 
have fixed numbers of derivatives;  
see Lemma~\ref{LEM:ht2}. 
Moreover, these new
microscopic conservation laws 
 $\wt{h}_n$ 
 guarantee non-trivial shallow-water limits
 for the associated shallow-water conservation laws.
  In Subsection~\ref{SUBSEC:B3}, we then define the shallow-water conservation laws $\wt{E}^\dl_{\frac k2 }(v)$ and establish  a 2-to-1 collapse of the shallow-water conservation laws
(Propositions~\ref{PROP:Scons1}
  and~\ref{PROP:Scons2}). Lastly, by combining the earlier results, we state the final structure of the shallow-water conservation laws $\wt{E}^\dl_{\frac k2 }(v)$
  in Subsection~\ref{SUBSEC:B4}.
Analogously to the deep-water case studied in the previous section, 
we introduce the notion of the {\it shallow-water rank} (Definition \ref{DEF:ord2}),
which is more complicated than the deep-water rank, reflecting
the more challenging nature of the shallow-water problem.
This shallow-water rank
plays a crucial role in understanding the structure
of the shallow-water conservation laws, 
thus allowing us to prove the desired 2-to-1 collapse.


\subsection{B\"acklund transforms and microscopic conservation laws}
\label{SUBSEC:B1}

We first recall the B\"acklund transform for the 
scaled ILW \eqref{sILW} derived in \cite{GK80, Kupershmidt81} (modulo constants).
Let $0<\dl<\infty$ and $\mu>0$. Given a (smooth) solution $v$ to the scaled ILW \eqref{sILW}, the B\"acklund transform yields a function $V=V(v;\mu)$ such that the pair $(v,V)$ satisfies the following system:
\begin{align}
\begin{split}
v & =  \frac{1}{2\dl^2} \big\{ 2i\mu \dl V - (1-i \mu^{-1}\dl) (e^{2i\mu \dl V} - 1)  \big\}
+\mu V_x + i\mu \dl \Gd V_x,
\\
 V_t & =  \frac{1}{\dl^2}   \big\{ 2i\mu \dl V - (1-i \mu^{-1}\dl ) (e^{2i\mu \dl V} - 1) \big\} V_x
+ \Gd V_{xx} +2i \mu \dl V_x \Gd V_x, 
\end{split}
\label{BTX1}
\end{align}

\noi
where $\Gd$ is as in \eqref{sILW}.
For now,  assume that $V$ is of the form
\begin{equation}
\label{BTX1a}
V = \sum_{n=0}^\infty \mu^n h_n,
\end{equation}
where $\{h_n\}_{n\in\Z_{\ge0}} = \{h^\dl_n\}_{n\in\Z_{\ge0}}$ is a sequence of smooth functions, independent of $\mu>0$.
Then, by first rewriting the first equation in \eqref{BTX1} as
\begin{equation}
V =- v - \frac{1}{2\dl^2} \sum_{j=2}^\infty \frac{1}{j!} (2i\mu \dl V)^j
+ 
\frac{i\mu^{-1}}{2 \dl}
 \sum_{j=2}^\infty \frac{1}{j!} (2i\mu \dl V)^j
 + \mu \wt{\L}_0 V,
\label{BTX1b}
\end{equation}

\noi
where $\wt{\L}_0 = \wt{\L}_0(\dl)$ is given by 
\begin{equation}
\label{BTX1aa}
\wt{\L}_0 = (1 + i \dl \Gd)\dx  , 
\end{equation}

\noi
substituting \eqref{BTX1a}, and comparing powers of $\mu$, we obtain the following recursive definition for $h_n = h^\dl_n$ in terms of $v$:
\begin{equation}\label{h1}
\begin{aligned}
h_0 & = -v , \\
h_1 & = 
-i \dl h_0^2 + \wt{\L}_0 h_0
= -i \dl v^2 - \wt{\L}_0 v, \\
h_n & = -\frac{1}{2\dl^2} \sum_{j=2}^n \frac{(2i\dl)^j}{j!}
 \sum_{\substack{n_1, \ldots, n_j \in \Z_{\ge 0} \\ n_{1\cdots j} = n-j}} 
 h_{n_1} \cdots h_{n_j} \\
& \quad + \frac{i}{2\dl} \sum_{j=2}^{n+1} \frac{(2i\dl)^j}{j!} 
\sum_{\substack{n_1, \ldots, n_j \in \Z_{\ge 0} \\ n_{1\cdots j} = n+1-j}}
 h_{n_1} \cdots h_{n_j} + \wt{\L}_0 h_{n-1}, \quad   n\ge 2,
\end{aligned}
\end{equation}

\noi
where  $n_{1\cdots j} = n_1 + \cdots +n_j$ is as in \eqref{sum1}.
Hence, as in the deep-water case, 
 given a smooth function $v$ on $\T\times \R$, 
by recursively constructing $h_n$ via \eqref{h1}
and using \eqref{BTX1a}, we
can 
 construct a smooth function $V=V(v;\mu)$.

In the following lemma, 
we establish 
 equivalence between 
 the scaled ILW \eqref{sILW} and its B\"acklund transform \eqref{BTX1}. As its proof does not seem to be readily available in the literature, we include it for readers' convenience. 
As in the deep-water case,  we impose the additional assumption \eqref{BTX1a}
 in the converse direction.

\begin{lemma}
\label{LEM:BTX1}
Let $0<\dl <\infty$. Given $\mu>0$, let $v:\R\times\T \to \R$ and $V:\R\times\T \to \C$ be smooth functions, satisfying the first equation in \eqref{BTX1}. If $V$ satisfies the second equation in \eqref{BTX1}, then $v=v(V)$ satisfies the scaled ILW \eqref{sILW}. Conversely, if $v$ satisfies the scaled ILW \eqref{sILW} and if, in addition, $V$ is of the form \eqref{BTX1a}, then $V=V(v;\mu)$ satisfies the second equation in~\eqref{BTX1}.
\end{lemma}

\begin{remark}\rm
(i)
As in  deep-water case, 
given   a smooth function $v$ on $\R \times \T$.
 it is in general not clear to us 
how to find a function $V$ satisfying the first equation in \eqref{BTX1}
 without the extra assumption \eqref{BTX1a}.

\smallskip

\noi(ii)  Fix $0<\dl<\infty$ and $\mu > 0$.
With  $\wt{\L}_0$  as in \eqref{BTX1aa}, 
define  $\wt{\L}_1 = \wt{\L}_1(\dl,\mu)$ by 
\begin{equation*}
\wt{\L}_1  = \Id  - \mu \wt{\L}_0 = \Id  - \mu(1 + i\dl\Gd)\dx.
\end{equation*}

\noi
 Then, we can rewrite \eqref{BTX1b} as
\begin{equation}
\label{BTX1bb}
V = \wt{\L}_1^{-1} \bigg(- v - \frac{1}{2\dl^2} \sum_{j=2}^\infty \frac{1}{j!} (2i\mu \dl V)^j
+ 
\frac{i\mu^{-1}}{2 \dl}
 \sum_{j=2}^\infty \frac{1}{j!} (2i\mu \dl V)^j \bigg), 
\end{equation}

\noi
where $\wt{\L}_1^{-1}$ is the Fourier multiplier operator with multiplier:
\begin{equation*}
\ft{\wt{\L}_1^{-1}}(n) =
\begin{cases}
\big(1- i \dl^{-1}\mu (\dl n + \dl n \coth(\dl n) -1)\big)^{-1}, & n\in\Z^*, \\
1, & n=0.
\end{cases}
\end{equation*}

\noi
Note  that $|\ft{\wt{\L}_1^{-1}}(n)|\le 1$ for any $n \in \Z$.
Fix $T>0$ and $v\in C([-T,T]; H^s(\T))$ for some $s\gg1$. Then, 
for sufficiently small $\mu = \mu(\|v\|_{C_T H^s_x}, \dl)>0$, a standard contraction argument shows that there exists a unique solution $V=V(v;\mu)$ to \eqref{BTX1bb} in the ball 
$B\subset C([-T,T];H^s(\T))$
of radius $\sim \| v\|_{C_T H^s_x}$ centered at the origin. We should, however, note that, 
even if $v$ satisfies the scaled ILW \eqref{sILW},  it is not clear if $V$ defined in \eqref{BTX1bb}  satisfies the second equation in \eqref{BTX1}.

\end{remark}

We now present a proof of Lemma~\ref{LEM:BTX1}.
\begin{proof}[Proof of Lemma~\ref{LEM:BTX1}]

We proceed as in the proof of Lemma~\ref{LEM:BT1}.
From the first equation in~\eqref{BTX1} and  Lemma~\ref{LEM:T2}
with $\Gd = \dl^{-1} \Gdl$, 
a direct computation yields
\begin{equation}
\begin{aligned}
& \dt v - \Gd\dx^2 v - 2 v \dx v  \\
& \quad 
=\wt \L^V \Big[ V_t - \frac{1}{\dl^2} 
\big\{ 2i\mu \dl V - (1- i \mu^{-1}\dl ) (e^{2i\mu \dl V} - 1) \big\} V_x  - \Gd V_{xx} -2i \mu \dl V_x  \Gd V_x\Big]
\end{aligned}
\label{BTX2}
\end{equation}

\noi
where the operator $\wt \L^V = \wt \L^V(\mu)$ is given by
\begin{equation}
\label{BTX3}
\begin{aligned}
\wt \L^V &= -(1+i\mu \dl^{-1}) e^{2i\mu \dl V} 
+ i \mu \big( (\dl \Gd - i)\dx + \dl^{-1}   \big)\\
&  = -(1+i\mu \dl^{-1}) e^{2i\mu \dl V} + i \mu  \L_0
\end{aligned}
\end{equation}

\noi
with $ \L_0$ as in \eqref{BT1aa}.
Hence, if $V$ satisfies  the second equation in \eqref{BTX1}, 
then it follows from~\eqref{BTX2}  that $v$ satisfies the  scaled ILW \eqref{sILW}.

Next, we prove the converse direction. 
Suppose that $v$ is a smooth solution to \eqref{sILW}
and that $V = V(v;\mu)$ of the form \eqref{BTX1a} is defined by the first equation in \eqref{BTX1}
(as explained above).
Then, from \eqref{BTX2} and \eqref{BTX3}, we have
\begin{equation}
\label{BTX4}
\wt \L^V \wt F(V) =0,
\end{equation}

\noi
where $\wt F(V)=\wt F(V;\mu)$ denotes the expression in the brackets on the right-hand side of \eqref{BTX2}.
In the following, we show that $\wt F(V)=0$. 
For simplicity, we will suppress the $t$-dependence. From \eqref{BTX4} and \eqref{BTX3}, we have
\begin{equation}
0 = \jb{\wt \L^V \wt F, i \mu \L_0 \wt F}_{L^2} = - (1+ i \mu\dl^{-1})
  \jb{e^{2i\mu \dl V} \wt F, 
i \mu \L_0 \wt F}_{L^2} 
+ \mu^2 \| \L_0 \wt F \|_{L^2}^2.
\label{BTX4a}
\end{equation}

\noi
Then, arguing as in \eqref{BT5}, 
it follows from \eqref{BTX4} with \eqref{BTX3} and \eqref{BTX4a} that 
\begin{equation}
\label{BTX5}
0 = \| \wt \L^V \wt F \|_{L^2}^2  = (1+\mu^2 \dl^{-2}) \| e^{2i\mu \dl V} \wt F \|^2_{L^2} - \mu^2 \| \L_0 \wt F \|^2_{L^2}.
\end{equation}

From the assumption \eqref{BTX1a}, we have
\begin{equation}
\label{BTX5a}
\wt F(V) = \sum_{n=0}^\infty \mu^n \wt f_n
\end{equation}

\noi
for some smooth functions $\wt f_n$ independent of $\mu>0$. 
Then, from \eqref{BTX1a} and \eqref{BTX5a}, we have
\begin{equation}
\label{BTX6}
\begin{aligned}
e^{2i\mu\dl V} \wt F 
& = \bigg( 1 + \sum_{j=1}^\infty \frac{1}{j!} (2i \mu \dl V)^j \bigg)  \bigg(\sum_{n=0}^\infty \mu^n \wt f_n \bigg)  \\
& = \bigg( 1 + \sum_{n=1}^\infty \mu^n  \sum_{j=1}^n \frac{(2i\dl)^j}{j!} \sum_{\substack{n_1, \ldots, n_j =0 \\ n_{1\cdots j}=n-j}}^\infty h_{n_1} \cdots h_{n_j} \bigg) \bigg(\sum_{n=0}^\infty \mu^n \wt f_n \bigg) \\
& = : \sum_{n=0}^\infty \mu^n B_n,
\end{aligned}
\end{equation}

\noi
where
\begin{equation}
\label{BTX6a}
\begin{aligned}
B_0 & = \wt f_0, \qquad 
 B_n =\wt  f_n + \sum_{\substack{n_{12}=n\\0\le n_1 \le n-1}} \wt f_{n_1} \wt{V}_{n_2}, \quad n\in \N,\\
\wt{V}_n & = \sum_{j=1}^{n} \frac{(2i\dl)^j}{j!} \sum_{\substack{n_1,\ldots, n_j =0 \\ n_{1\cdots j} = n-j}}^\infty h_{n_1} \cdots h_{n_j} , \quad n \in \N.
\end{aligned}
\end{equation}

\noi
By substituting \eqref{BTX5a} and \eqref{BTX6} in \eqref{BTX5}, we see that the coefficient of each power of $\mu$ must be 0. From the coefficient of $\mu^0$
with \eqref{BTX6a}, we obtain
\begin{equation}
\label{BTX7}
B_0 = \wt f_0 = 0,
\end{equation}

\noi
while from the coefficient of $\mu^1$, we have
\begin{equation*}
2 \Re \int_\T B_0 \cj{B}_1  dx = 0,
\end{equation*}

\noi
which holds in view of \eqref{BTX7}. 
From the coefficient of $\mu^2$, we obtain
\begin{equation}
\label{BTX7a}
0 = \int_\T \dl^{-2} |B_0|^2 + |B_1|^2 + 2 \Re B_0\cj B_2 - |\L_0 \wt f_0|^2 dx  = \int_\T |B_1|^2 dx  ,
\end{equation}

\noi
where we used \eqref{BTX7} at the second equality. 
Hence,  from the definition of $B_1$ in \eqref{BTX6a}
with~\eqref{BTX7} and \eqref{BTX7a}, 
we obtain
\begin{equation*}
\wt f_1 = B_1 = 0.
\end{equation*}

We now proceed by induction and show that $\wt f_n=0$ for $n\in\Z_{\ge0}$. 
Given $n\ge2$, assume that $\wt f_k=0$ for $0\le k \le n-1$. Then, from \eqref{BTX6a}, we see that
\begin{equation}
\label{BTX7c}
B_k =0
\end{equation}
for $0\le k \le n-1$. From the coefficient of $\mu^{2n}$ in \eqref{BTX5}, we get
\begin{align*}
0 & =  \sum_{ n_{12}=2n}
 \int_\T B_{n_1} \cj B_{n_2} dx 
 + \dl^{-2} \sum_{n_{12}=2n-2} \int_\T B_{n_1} \cj B_{n_2}
  - \sum_{n_{12}=2n-2} \int_\T \L_0 \wt f_{n_1} \cj{\L_0\wt  f_{n_2}} dx \\
& = \int_\T |B_n|^2  dx, 
\end{align*}

\noi
where we used 
the inductive hypothesis and
\eqref{BTX7c} at the second equality. 
Hence, we obtain  $B_n = 0$, 
and, 
from \eqref{BTX6a} and the inductive hypothesis, we conclude that $\wt f_n=0$
for any $n \in \Z_{\ge 0}$.
Therefore,
from \eqref{BTX5a},
we conclude that $\wt F(V) = 0$.
Namely, $V$
satisfies
 the second equation in \eqref{BTX1}.
\end{proof}

Suppose that $V$ satisfies \eqref{BTX1}. Then, from the second equation in \eqref{BTX1}
with the fact that 
 $V_x \Gd V_x$ has mean zero on $\T$
 (which follows from the anti self-adjointness of $\Gd$), 
we see that the mean of $V$ is conserved. If $V$ is also of the form \eqref{BTX1a}, we formally obtain
\begin{align*}
0 = \frac{d}{dt} \int_\T  V dx = \sum_{n=0}^\infty \mu^n \cdot \frac{d}{dt} \int_\T h_n d x
\end{align*}

\noi
for any $\mu>0$. 
Since the coefficients of all powers of $\mu$ must be 0, we conclude that
\begin{equation}
\label{BTX8}
\frac{d}{dt} \int_\T h_n dx =0
\end{equation}

\noi
for any $n\in\Z_{\ge0 }$.
Namely, for any $n\in\Z_{\ge0 }$, $h_n$ is a microscopic conservation law for the scaled ILW \eqref{sILW}. 
As in the deep-water case, 
we will define the shallow-water conservation laws $\wt{E}^\dl_{\frac k 2}(u)$ as 
(the real part of) 
suitable linear combinations  
of 
$\int_\T h_n dx$; see \eqref{ht0} and Proposition~\ref{PROP:Scons1}.

We conclude this subsection by recalling the B\"acklund transform 
for KdV \eqref{kdv}
and its microscopic conservation laws.
Recall from Lemma~\ref{LEM:L1} that $\Gd\dx$ converges to $-\frac13\dx^2$ as $\dl\to0$. Thus, in the shallow-water limit ($\dl\to0$), the B\"acklund transform \eqref{BTX1} for the scaled ILW~\eqref{sILW} converges to that for KdV~\eqref{kdv}, also known as the  Gardner transform \cite{M1}:
\begin{equation}
\label{BTkdv}
\begin{aligned}
v&= -V + \mu^2 V^2 + \mu V_x , \\
V_t & = -2 ( V- \mu^2 V^2) V_x - \frac13 V_{xxx}.
\end{aligned}
\end{equation}

\noi
Note that Lemma~\ref{LEM:BTX1} also holds for the case $\dl=0$. 
Moreover, by setting $\dl = 0$ in \eqref{BTX1aa} and~\eqref{h1}
we obtain the following recursive definition for the KdV microscopic conservation laws $h^\KDV_n = h^0_n$:
\begin{equation}
\begin{aligned}
h^{\KDV}_0 & = -v, \\
h^{\KDV}_1 & = -\dx v, \\
h^{\KDV}_n & =   \sum_{\substack{n_1, n_2 \in  \Z_{\ge 0} \\n_{12}=n-2}} h^{\text{KdV}}_{n_1} h^{\text{KdV}}_{n_2} + \dx h^{\text{KdV}}_{n-1} , \quad n\ge 2.
\end{aligned}
\label{hkdv}
\end{equation}

\noi
Arguing as above, we have
\begin{equation}
\label{BTX9}
\frac{d}{dt} \int_\T h_n^\KDV dx =0
\end{equation}

\noi
for $n\in\Z_{\ge0 }$.

The following lemma collects well-known properties
of 
 the KdV microscopic conservation laws $h^\KDV_n$.

\begin{lemma}
\label{LEM:kdv1}

{\rm(i)} Given any $n\in\Z_{\ge0}$, we have 
\begin{equation}
\label{hkdv1}
\int_\T h^\KDV_{2n+1} dx = 0.
\end{equation}

\noi
{\rm(ii)} Given $n\in\Z_{\ge0}$, let $L^\KDV_n$ and $Q^\KDV_n$ denote 
the linear and quadratic {\rm(}in $v${\rm)} parts of $h^\KDV_n$
defined in \eqref{hkdv}, respectively. Then, we have
\begin{equation}
\label{hkdv1b}
L^\KDV_n = - \dx^{n} v.
\end{equation}
Moreover, under the mean-zero assumption on $v$, we have
\begin{align}
\label{hkdv1a}
\int_\T L^\KDV_{n} dx& = 0, \\
\label{hkdv1aa}
\int_\T Q^\KDV_{2n} dx& = (-1)^{n-1} \|v\|^2_{\dot{H}^{n-1}} ,  \quad n \in \N.
\end{align}


\end{lemma}

\begin{proof}
(i) 
The identity \eqref{hkdv1}
is well known in the literature; see \cite[Proposition~2.7]{Kupershmidt81}
for a formal proof.  See also \cite[(25) on p.\,1207]{M2}.
For readers' convenience, we first present the argument in \cite{Kupershmidt81, M2}.
Let $V$ be as in  \eqref{BTX1a} (with $\dl = 0$) and write it as
\begin{align}
V = V_+ + V_- := \sum_{n=0}^\infty \mu^{2n} h_{2n}^\KDV+ 
\sum_{n=0}^\infty \mu^{2n+1} h_{2n+1}^\KDV.
\label{hkdv1x}
\end{align}

\noi
By substituting \eqref{hkdv1x} into 
the first equation in \eqref{BTkdv}
and collecting the contributions coming from the odd powers of $\mu$, 
we obtain
\begin{align*}
0 = -V_- + 2 \mu^2 V_+ V_- + \mu \dx V_+.
\end{align*}

\noi
Then, by solving for $V_-$, we obtain
\begin{align}
V_- = \frac{\mu \dx V_+}
{1 -  2 \mu^2 V_+} = - \frac{1}{2\mu} \dx \log (1 -  2 \mu^2 V_+), 
\label{hkdv1y}
\end{align}

\noi
from which we conclude that 
\begin{align*}
0 = \int_\T V_- dx = \sum_{n=0}^\infty \mu^{2n+1}\cdot \int_\T  h_{2n+1}^\KDV dx.
\end{align*}

\noi
Since the coefficient of $\mu^{2n+1}$ must be 0 for each $n \in \Z_{\ge 0}$, we conclude \eqref{hkdv1}.
We point out that the step in~\eqref{hkdv1y} is rather formal without any justification.

In the following, we present an alternative proof of \eqref{hkdv1}, using the recurrence relation \eqref{hkdv}.
We claim that, for $n\in\Z_{\ge0}$, the odd-order microscopic conservation laws $h^\KDV_{2n+1}$ satisfy the following identity:
\begin{equation}
h^\KDV_{2n+1} = \dx \bigg( \sum_{j=1}^{n+1} \frac{2^{j-1}}{j}
 \sum_{\substack{ n_1, \ldots, n_{j} \in  \Z_{\ge 0}  \\ n_{1\cdots j} = n+1 - j  }} h^\KDV_{2n_1} \cdots h^\KDV_{2n_{j}} \bigg),
\label{hkdv2}
\end{equation}

\noi
from which \eqref{hkdv1} follows. 

We proceed by induction on $n \in \Z_{\ge0}$.
When  $n=0$, it follows from \eqref{hkdv} that 
\begin{equation}
h^\KDV_1 = \dx h^\KDV_0 = \text{RHS of \eqref{hkdv2}} |_{n=0}.
\label{hkdv2a}
\end{equation}

\noi
When  $n=1$, from \eqref{hkdv} and \eqref{hkdv2a},  we have
\begin{align*}
h^\KDV_3 & = 2 h_0^\KDV h_1^\KDV + \dx h^\KDV_2 \\
 &= 2 h_0^\KDV \dx h^\KDV_0 + \dx h^{\KDV}_2 \\
& = \dx \big( (h^\KDV_0)^2 + h^\KDV_2 \big) \\
& = \text{RHS of \eqref{hkdv2}} |_{n=1}.
\end{align*}

Fix an integer $n\ge 2$ and assume that \eqref{hkdv2} holds for $0 \le k \le n-1$. 
Then, from \eqref{hkdv} and the inductive hypothesis, we have
\begin{align}
\begin{split}
h^\KDV_{2n+1} 
& = 2 \sum_{\substack{k_1, k_1 \in  \Z_{\ge 0} \\k_{12}=n-1 }} h^\KDV_{2k_1} h^\KDV_{2k_2 + 1} + \dx h^\KDV_{2n} \\
& = 2 \sum_{\substack{k_1, k_2 \in  \Z_{\ge 0}  \\ k_{12} = n-1}} h^\KDV_{2k_1} \dx \bigg( \sum_{j=1}^{k_2+1} \frac{2^{j-1}}{j}  \sum_{\substack{m_1, \ldots, m_j  \in \Z_{\ge 0}  \\ m_{1\cdots j} = k_2 +1 - j}}
h^\KDV_{2m_1} \cdots h^\KDV_{2m_j} \bigg) + \dx h_{2n}^\KDV.
\end{split}
\label{hkdv3}
\end{align}

\noi
By rearranging the right-hand side of \eqref{hkdv3}, we obtain

\noi
\begin{equation}
\begin{aligned}
h_{2n+1}^\KDV & =  \sum_{j=1}^{n} \frac{2^j}{j} \sum_{k_1=0}^{n-j} 
\sum_{\substack{m_1, \ldots, m_j \in \Z_{\ge 0} \\ m_{1\cdots j} = n-k_1 -j}} \sum_{\substack{0 \le \al_1,\ldots, \al_j \le 1 \\ \al_{1\cdots j} = 1}} h^\KDV_{2k_1} \prod_{\l=1}^j \dx^{\al_\l} h^\KDV_{2 m_\l}  + \dx h^\KDV_{2n} \\
& = \sum_{j=1}^n \frac{2^j}{j} \sum_{\substack{m_1, \ldots, m_{j+1} \in \Z_{\ge 0} \\ m_{1\cdots (j+1)} = n-j}} \sum_{\substack{ 0\le \al_1, \ldots, \al_{j+1} \le 1 \\ \al_{1\cdots j} = 1 \\ \al_{j+1} =0   }} \prod_{\l=1}^{j+1} \dx^{\al_\l} h_{2 m_\l}^{\KDV} + \dx h^\KDV_{2n}.
\end{aligned}
\label{hkdv3aa}
\end{equation}

\noi
Note that we have 
\begin{equation}
\begin{aligned}
&  \sum_{\substack{m_1, \ldots, m_{j+1} \in \Z_{\ge 0} \\ m_{1\cdots (j+1)} = n-j}}
 \sum_{\substack{ 0\le \al_1, \ldots, \al_{j+1} \le 1 \\ \al_{1\cdots j} = 1 \\ \al_{j+1} =0   }} \prod_{\l=1}^{j+1} \dx^{\al_\l} h_{2 m_\l}^{\KDV} \\
& \hphantom{XXXXXX} = j   \sum_{\substack{m_1, \ldots, m_{j+1} \in \Z_{\ge 0} \\ m_{1\cdots (j+1)} = n-j}}  (\dx h^\KDV_{2m_1}) h^\KDV_{2m_2} \cdots h^\KDV_{2m_{j+1}} 
\end{aligned}
\label{hkdv3a}
\end{equation}

\noi
for $1\le j \le n$.
Similarly, we have
\begin{equation}
\begin{aligned}
&  \sum_{\substack{m_1, \ldots, m_{j+1} \in \Z_{\ge 0} \\ m_{1\cdots (j+1)} = n-j}} \sum_{\substack{ 0\le \al_1, \ldots, \al_{j+1} \le 1 \\ \al_{1\cdots (j+1)} = 1  }} \prod_{\l=1}^{j+1} \dx^{\al_\l} h_{2 m_\l}^{\KDV} \\
& \hphantom{XXXXXX}
 = (j+1)  \sum_{\substack{m_1, \ldots, m_{j+1} \in \Z_{\ge 0} \\ m_{1\cdots (j+1)} = n-j}}  (\dx h^\KDV_{2m_1}) h^\KDV_{2m_2} \cdots h^\KDV_{2m_{j+1}} .
\end{aligned}
\label{hkdv3b}
\end{equation}

\noi
Hence, from 
\eqref{hkdv3aa},  \eqref{hkdv3a},  \eqref{hkdv3b}, 
and the product rule,
 we obtain
\begin{align*}
h^\KDV_{2n+1} 
& =  \sum_{j=1}^n \frac{2^j}{j+1}  
\sum_{\substack{m_1, \ldots, m_{j+1} \in \Z_{\ge 0} \\ m_{1\cdots (j+1)} = n-j}}
 \sum_{\substack{ 0\le \al_1, \ldots, \al_{j+1} \le 1 \\ \al_{1\cdots (j+1)} = 1  }} \prod_{\l=1}^{j+1} \dx^{\al_\l} h_{2 m_\l}^{\KDV} + \dx h^\KDV_{2n} \\
& = \dx \bigg( \sum_{j=2}^{n+1} \frac{2^{j-1}}{j} 
\sum_{\substack{m_1, \ldots, m_{j} \in \Z_{\ge 0} \\ m_{1\cdots j} = n+1-j}} h^\KDV_{2m_1} \cdots h^\KDV_{2m_j} \bigg) + \dx h^\KDV_{2n} \\
& = \dx \bigg( \sum_{j=1}^{n+1} \frac{2^{j-1}}{j} 
\sum_{\substack{m_1, \ldots, m_{j} \in \Z_{\ge 0} \\ m_{1\cdots j} = n+1-j}}
h^\KDV_{2m_1} \cdots h^\KDV_{2m_j} \bigg)
\end{align*}

\noi
which agrees with the right-hand side of \eqref{hkdv2}. 
Therefore, by induction, we conclude that  the identity \eqref{hkdv2} 
holds for any $n \in \Z_{\ge 0}$.

\medskip

\noi
(ii) 
From \eqref{hkdv},  we see that \eqref{hkdv1b} holds for $n=0,1$. 
Fix  $n \ge 2$. Since $h^\KDV_n$ is at least linear in $v$,  
it follows from \eqref{hkdv} that 
\begin{align*}
L^\KDV_n = \dx L^\KDV_{n-1},
\end{align*}

\noi
which yields \eqref{hkdv1b} by induction. 
The identity \eqref{hkdv1a} follows from \eqref{hkdv1b} and the mean-zero assumption on $v$
(which is needed for $n = 0$).

Let $n\in\N$. From  \eqref{hkdv} and \eqref{hkdv1b}, we have 
\begin{align*}
\int_\T Q^\KDV_{2n} dx & = \sum_{n_{12}=2n-2} \int_\T L_{n_1}^\KDV L_{n_2}^\KDV dx 
 = \sum_{n_1=0}^{2n-2} (-1)^{n_1} \int_\T v \dx^{2n-2} v dx \\
& = (-1)^{n-1} \|v\|^2_{\dot{H}^{n-1}},
\end{align*}

\noi
which yields \eqref{hkdv1aa}.
\end{proof}

\subsection{On the shallow-water microscopic conservation laws}
\label{SUBSEC:B2}

In this subsection, 
by taking linear combinations of $h_n$ defined in~\eqref{h1}, 
we construct  new
microscopic conservation laws
$\wt h_{n}$ (see~\eqref{ht0}).
In the next subsection, we 
use these new microscopic conservation laws
$\wt h_{n}$ 
to
construct the shallow-water conservation laws $\wt{E}^\dl_{\frac k 2}(v)$ in~\eqref{DEs}. 
The introduction of 
the new microscopic conservation laws
$\wt h_{n}$ 
 allows us to guarantee that 
\begin{itemize}
\item[(i)] all the quadratic terms in  $\wt{E}^\dl_{\frac k 2}(v)$
have exactly $k$ derivatives with \textit{positive} coefficients 
(Proposition \ref{PROP:Scons1}; see also Lemma \ref{LEM:ht2}), 

\smallskip

\item[(ii)]
the shallow-water conservation laws $\wt{E}^\dl_{\frac k 2}(v)$ have non-trivial limits as $\dl\to0$
(Proposition~\ref{PROP:Scons2}).

\end{itemize}

\smallskip

\noi
While the point (i) already appeared in the deep-water regime, 
the point (ii)
 is a new difficulty specific to  the shallow-water regime.
Indeed, in view of Lemma \ref{LEM:kdv1}\,(i), 
if we simply use (the real part of) $\int_\T h_n dx$ to define a shallow-water conservation law, 
 then we would get a trivial shallow-water limit
 when $n$ is odd in the sense that 
 $\int_\T h_{2n+1} dx$ converges to $0$ as $\dl \to 0$.
 This in particular causes an issue
 in pursuing our statistical study
 since the corresponding \GGM~
 would not have any meaningful limit.
As we see in the next subsection, 
the new microscopic conservation laws $\wt h_n$
allow us to construct
a full set of  shallow-water conservation laws 
 with non-trivial shallow-water limits
 as claimed in Theorem~\ref{THM:2}.

We first state a lemma on 
the structure of the linear part of 
the microscopic conservation laws
$h_n$.

\begin{lemma}
\label{LEM:h1}
Let $0<\dl<\infty$. Given $n\in\Z_{\ge0}$, let ${L}_n = {L}^\dl_n$ be the linear \textup{(}in $v$\textup{)} part of $h_n = h^\dl_n$ defined in \eqref{h1}.
Then, we have
\begin{equation}
\label{h4}
{L}_n =  -(\wt{\L}_0)^{n} v
\end{equation}
for any $n\in\Z_{\ge0}$, 
where $\wt{\L}_0$ is as in \eqref{BTX1aa}. In particular, under the current mean-zero assumption on $v$, we have
\begin{equation}
\label{h5}
\int_\T {L}_n dx = 0
\end{equation}
for any $n\in\Z_{\ge0}$.
\end{lemma}

\begin{proof}

In view of \eqref{h1}, 
the claims \eqref{h4} and \eqref{h5}
obviously hold when $n=0,1$.
Noting that  each monomial in $h_n$ is at least linear in $v$, 
we see that 
the linear contribution 
on the right-hand side of the third equation in 
\eqref{h1} comes from the last term
$\wt{\L}_0 h_{n-1}$, 
and thus we have 
\begin{align*}
{L}_n = \wt{\L}_0 {L}_{n-1}
\end{align*}
for any $n\ge 2$, which yields \eqref{h4} by induction. The second claim \eqref{h5} follows from \eqref{h4} with 
the mean-zero assumption on $v$
(which is needed for $n = 0$)
 and 
\eqref{BTX1aa}.
\end{proof}

We now introduce
 new microscopic conservation laws $\wt{h}_n = \wt{h}^\dl_n$,  $n\in\N$, 
 by setting
\begin{equation}
\label{ht0}
\wt h_{n}  = \frac{(-i\dl)^{\pf(n)}}{\dl^2} \sum_{j=1}^{n} \frac{1}{(i\dl)^{n-j}} h_j , 
\end{equation}

\noi
where $\pf(n)$ denotes the parity of an integer $n$ defined by 
\begin{align}
\pf(n) = 
\begin{cases}
0, & \text{if $n$ is even},\\
1, & \text{if $n$ is odd}.
\end{cases}
\label{par1}
\end{align}

\noi
As we see in Lemma \ref{LEM:ht2} below, 
the quadratic parts of (the integrals of) these new 
microscopic conservation laws $\wt{h}_n$
have simple structures.
In Proposition \ref{PROP:Scons1}, 
we use  $\wt{h}_n$
to construct the shallow-water conservation laws
$\wt E_\frac k2^\dl(v)$.

The next lemma provides a more convenient form  for $\wt{h}_n$ 
which we will use in the subsequent analysis.

\begin{lemma}\label{LEM:ht2a}
Let $0<\dl<\infty$. Given $n\in\N$, the  new microscopic conservation law $\wt{h}_{n}$
 defined in \eqref{ht0} can be written as follows\textup{:}
\begin{equation}
\label{ht1}
\begin{aligned}
\wt{h}_{n} & = - \frac{(-i\dl)^{\pf(n)}}{2i\dl^3} \sum_{j=2}^{n+1} \frac{(2i\dl)^j}{j!} 
\sum_{\substack{n_1, \ldots, n_j \in \Z_{\ge 0} \\ n_{1\cdots j} = n+1 - j}} h_{n_1} \ldots h_{n_j}\\ 
& \quad + 
\frac{(-i\dl)^{\pf(n)}}{\dl^2}
\sum_{j=1}^{n} \frac{1}{(i\dl)^{n-j}} \wt{\L}_0 h_{j-1}, 
\end{aligned}
\end{equation}

\noi
where  $h_j$,  $\wt{\L}_0$, and $\pf(n)$ are as in 
 \eqref{h1}, 
\eqref{BTX1aa}, and \eqref{par1}, respectively.
\end{lemma}

\begin{proof}
 From \eqref{ht0} and \eqref{h1}, we have
\begin{align}
\label{ht1a}
\wt{h}_{1} = \frac{1}{i\dl} h_1  = - h_0^2 + \frac{1}{i\dl} \wt{\L}_0 h_0,
\end{align}

\noi
 which agrees with \eqref{ht1}.

Let $n \ge 2$.
From \eqref{h1}, we have
\begin{equation}
\label{ht1aa}
\begin{aligned}
& \sum_{j=1}^n \frac{1}{(i\dl)^{n-j}} h_j \\
& \ \  =   \frac{1}{(i\dl)^{n-1}} h_1  + \frac{1}{2} \sum_{j=2}^{n} \frac{1}{(i\dl)^{n+2-j}} \sum_{k=2}^j \frac{(2i\dl)^k}{k!} \sum_{\substack{n_1, \ldots, n_k\in \Z_{\ge 0} \\ n_{1\cdots k} = j-k}} h_{n_1} \cdots h_{n_k} \\
& \ \ \quad - \frac{1}{2}  \sum_{j=3}^{n+1} \frac{1}{(i\dl)^{n+2-j}} \sum_{k=2}^{j} \frac{(2i\dl)^k}{k!} \sum_{\substack{n_1, \ldots, n_k=0 \\ n_{1\cdots k} = j-k}}^\infty h_{n_1} \cdots h_{n_k} + \sum_{j=2}^n \frac{1}{(i\dl)^{n-j}} \wt{\L}_0 h_{j-1} \\
&\ \  = - \frac{1}{2i\dl} \sum_{j=2}^{n+1} \frac{(2i\dl)^j}{j!} 
 \sum_{\substack{n_1, \ldots, n_j\in \Z_{\ge 0} \\ n_{1\cdots j} = n+1-j}} h_{n_1} \cdots h_{n_j} + \sum_{j=1}^n \frac{1}{(i\dl)^{n-j}} \wt{\L}_0 h_{j-1},
\end{aligned}
\end{equation}

\noi
where we used \eqref{ht1a} at the last equality. 
Then, \eqref{ht1} follows from \eqref{ht0} and \eqref{ht1aa}.
\end{proof}

The following lemma 
shows that  the quadratic part of the macroscopic conservation law $\int_\T \wt{h}_n dx$
has a particularly simple structure, involving only the $\dot H^{\frac{n-1} 2}$-norm.

\begin{lemma}
\label{LEM:ht2}

Let $0<\dl<\infty$. Given $n\in\N$, let $\wt{Q}_n = \wt{Q}_n^\dl$ be the quadratic \textup{(}in $v$\textup{)} part of $\wt{h}_n = \wt{h}_n^\dl$ defined in \eqref{ht0}.
Then, we have
\begin{equation}
\label{hquad1}
 \int \wt{Q}_{n} dx  
 = (-1)^{\frac{n-2-\pf(n)}{2}} \sum_{\substack{\l=0\\ \pf(\l) = \pf(n-1)}}^{n-1} \binom{n}{\l} 
 \dl^{\l - 1 + \pf(n)}    \| \Gd^{\frac \l 2 } v \|^2_{\dot{H}^{\frac{n-1}2}}, \\
\end{equation}

\noi
where 
$\pf(n)$ is as in \eqref{par1}
and  we used the convention \eqref{odd2}
when $n$ is even and $\l$ is odd.

\end{lemma}

Before proceeding to a proof of Lemma~\ref{LEM:ht2}, 
we first recall the so-called  Rothe-Hagen identity; see \cite[(5.63) on p.\,202]{GKP}. For integers $m,n,p$ with $p\ge0$, we have
\begin{equation}
\label{comb3}
\sum_{\l=0}^p \binom{m -1 + \l}{\l} \binom{ n + p  - 1-\l}{p-\l} = \binom{m + n + p - 1}{p},
\end{equation}

\noi
where we use 
 the convention \eqref{aux3}.

\begin{proof}[Proof of Lemma~\ref{LEM:ht2}]

Let $n\in\N$.
From  Lemma \ref{LEM:ht2a} with \eqref{BTX1aa}, we have
\begin{equation}
\label{ht2a}
\int_\T \wt{Q}_{n} dx  = \frac{(-i\dl)^{\pf(n)}}{i \dl} 
\sum_{n_{12}=n -1} \int_\T L_{n_1} L_{n_2} dx, 
\end{equation}

\noi
where $L_n=L^\dl_n$ denotes the linear (in $v$) part of
the original microscopic conservation law $h_n$.
Proceeding as in \eqref{quad3}, \eqref{quad4}, and \eqref{quad5}
with  \eqref{ht2a}, Lemma \ref{LEM:h1},  \eqref{BTX1aa}, and integration by parts (along with the 
anti self-adjointness of $\Gd$), we have
\begin{align}
\begin{split}
\int_\T \wt{Q}_n dx 
& =\frac{(-i\dl)^{\pf(n)}}{i \dl} 
 \sum_{n_{12}=n-1} (-1)^{n_1} \int_\T v (1 - i \dl \Gd )^{n_1} (1+i\dl\Gd)^{n_2} \dx^{n-1}v dx \\
& =\frac{(-i\dl)^{\pf(n)}}{i \dl}  \sum_{n_{12}= n-1}  \sum_{m_1=0}^{n_1} \sum_{m_2=0}^{n_2}  \binom{n_1}{m_1} \binom{n_2}{m_2} (-1)^{n_1+m_1} \\
& \hphantom{XXXXXXXXXXXXXX}
 \times (i\dl)^{m_{12}} \int_\T v \,\Gd^{m_{12}}\dx^{n-1} v dx   \\
& = 
(-1)^{\frac{n-2-\pf(n)}{2}} \sum_{\substack{\l=0\\ \pf(\l) = \pf(n-1)}}^{n-1}\al_{n, \l}\, \dl^{\l - 1 + \pf(n)}  
\| \Gd^{\frac \l 2} v \|^2_{\dot{H}^{\frac{n-1}{2}}} ,
\end{split}
\label{ht2aa}
\end{align}

\noi
%
where, in the last step,  we used the convention \eqref{odd2}
when $n$ is even and $\l$ is odd.
Here, 
 the coefficient $\al_{n, \l}$ is given by 
\begin{equation*}
\al_{n, \l} = \sum_{n_{12}=n-1} 
\sum_{\substack{0\le m_1\le n_1\\ 0 \le m_2 \le n_2\\ \l=m_{12}}} \binom{n_1}{m_1}\binom{n_2}{m_2} (-1)^{n_1+m_1} .
\end{equation*}

\noi
By a change of variables ($\wt n_1 = n_1 - m_1$)
and \eqref{comb3} (with $m=\wt{n}_1+1$, $\l = m_1$, $n =n-\wt{n}_1 -\l $, and $p=\l$), we have
\begin{align}
\begin{split}
\al_{n, \l} & = \sum_{\substack{0\le n_1 \le n-1\\ 0 \le m_1 \le n_1 \\ 0 \le \l - m_1 \le n-1- n_1}} \binom{n_1}{m_1} \binom{n-1-n_1}{\l-m_1}  (-1)^{n_1+m_1} \\
& = \sum_{0\le \wt{n}_1 \le n-1-\l} (-1)^{\wt{n}_1} 
\sum_{m_1=0}^\l \binom{\wt{n}_1+m_1}{m_1} \binom{n-1-\wt{n}_1 - m_1}{\l-m_1}  \\
& =\binom{n}{\l} \sum_{0\le \wt{n}_1 \le n-1-\l} (-1)^{\wt{n}_1} \\
& =  \binom{n}{\l},
\end{split}
\label{ht2aaa}
\end{align}

\noi
where we used the fact that, in view of \eqref{aux3},  the contribution 
after the second equality vanishes
when 
$\wt n_1 > n - 1 - \l$
and that $n - 1 - \l$ is even under the condition $\pf(\l) = \pf(n-1)$.
Therefore,
the identity
\eqref{hquad1} follows from \eqref{ht2aa}
and \eqref{ht2aaa}.
\end{proof}

We conclude this subsection by introducing
the {\it shallow-water rank}, which plays an important role 
in understanding the structure of the shallow-water conservation laws.
See \cite{M2, M3} for a similar concept 
in a simpler context of the KdV equation.

\begin{definition}\label{DEF:ord2}
\rm
Given a monomial $p(v)$,
define $\# v$, $\#\dx$, $\#\Gd$, and $\#\dl$
by
\begin{align*}
\# v &  = \# v(p(v)) = \text{homogeneity of $p(v)$ in $v$},\\
\#\dx &  = \#\dx(p(v))= \text{number of (explicit) appearance of $\dx$ in $p(v)$},\\
\#\Gd&  = \#\Gd(p(v))= \text{number of (explicit) appearance of $\Gd$ in $p(v)$},\\
\#\dl&  = \#\dl (p(v))= \text{number of (explicit) appearance of $\dl$ in $p(v)$}.
\end{align*}

\noi
We then define the {\it shallow-water rank} of a monomial $p(v)$, denoted by $\ord(p)$, by setting
\begin{align}
\ord (p)  = \# v + \frac12 \big(\# \dx + \#\Gd -  \# \dl\big).
\label{ord1b}
\end{align}

\noi
For simplicity, we refer to the shallow-water rank as the rank in the following.
If all monomials in a given polynomial $p(v)$ have the same  rank, say $n$,
then we say that the polynomial $p(v)$ (and its integral over $\T$) has rank $n$.

\end{definition}

From 
\eqref{h1} and  \eqref{ht0}, 
we see that 
 both $h_n$ and $\wt{h}_n$
 consist of monomials in $v$ with $\dx$, $\dl$, and $\Gd$.
The following lemma characterizes the rank of the microscopic conservation laws $h_n$ and $\wt{h}_n$ for the scaled ILW \eqref{sILW} and of the KdV microscopic conservation laws $h^\KDV_n$ defined in \eqref{hkdv}

\begin{lemma}
\label{LEM:h0}

\noi
{\rm(i)} Given any $n\in\Z_{\ge0}$, the shallow-water microscopic conservation law $h_n$ defined in \eqref{h1} has rank $1 + \frac12 n $.
Moreover, each monomial in $h_n$ satisfies 
\begin{equation}
\label{saux1a}
0\le \#\Gd \le \min(\#\dx, \#\dl),   \quad \# v + \#\dx \le n +1, 
\quad \text{and} \quad 
0 \le  \#\dl\le n.
\end{equation}

\noi{\rm(ii)} Given any $n\in\Z_{\ge0}$, the KdV microscopic conservation law $h^\KDV_n$ defined in \eqref{hkdv} has rank $1 + \frac12 n $ with $\#\dl=\#\Gd=0$.

\noi{\rm(iii)} Given any $n\in\N$, 
the new shallow-water microscopic conservation law $\wt{h}_{n}$ 
 defined in~\eqref{ht0} has  rank $2 + \frac 12 n - \frac 12 \pf(n)$, 
 where $\pf(n)$ is as  in \eqref{par1}.
Moreover, each monomial in $\int_\T \wt h_ndx$ satisfies 
\begin{equation}
\begin{split}
& 0 \le \#\Gd\le \min(\#\dx, \#\dl+ 1 -  \pf(n)),  \quad   \# v + \#\dx \le n+1, 
\quad \text{and}\\ \quad
& 0 \le  \#\dl\le n + \pf(n) - 2.
\end{split}
\label{DEX1}
\end{equation}

\end{lemma}

\begin{proof}
(i) 
From \eqref{h1} and \eqref{BTX1aa}
with \eqref{ord1b}, we see that 
 the first claim holds for $n=0,1$. 
 Fix an integer  $n\ge 2$ and  assume that the first claim holds for $h_k$ with $0\le k \le n-1$.
 Then,  we see that 
 the summand (with fixed $j = 2, \dots, n$) of 
 the first sum (in $j$) on the right-hand side  of \eqref{h1} has rank $j + \frac12 (n_1 + \cdots + n_j) - \frac12 (j-2) = 1 + \frac12 n$, where we used  $n_1+\cdots + n_j = n-j$.
Similarly,  the summand (with fixed $j = 2, \dots, n+1$) of 
 the second sum (in $j$) on the right-hand side  of~\eqref{h1} has rank 
  $j + \frac12 (n_1 + \cdots + n_j) - \frac12 (j-1) = 1 + \frac12 n$, 
   where we used
$n_1+ \cdots + n_j = n+1-j$.
In view of \eqref{BTX1aa}
with the inductive hypothesis, 
the last term on the right-hand side of  \eqref{h1} also has rank $1 + \frac12 n$. 
Therefore, by induction, we conclude that $h_n$ has rank $1 + \frac 12 n$.

The first bound in \eqref{saux1a} follows from the fact  that 
$\Gd$  comes with $\dl$ and $\dx$ in \eqref{h1} through $\wt{\L}_0$ defined in \eqref{BTX1aa}.
From \eqref{h1}, we see that the second and third bounds in 
\eqref{saux1a} hold when $n = 0, 1$.
For $n \ge 2$, the claim follows easily 
from  induction, using \eqref{h1} with \eqref{BTX1aa}.

\medskip

\noi
(ii)  While this result was shown  in \cite{M2}, 
we present a  proof for readers' convenience.
From~\eqref{hkdv}, 
we see that the claim holds for $n = 0, 1$.
 Fix an integer  $n\ge 2$ and  assume that the claim holds for $h_k^\KDV$ with $0\le k \le n-1$.
 Then,
 by the inductive hypothesis, we see that 
all monomials $h_{n_1}^\KDV h_{n_2}^\KDV$ (with $n_1 + n_2= n-2)$ and $\dx h_{n-1}^\KDV$ 
in \eqref{hkdv}
have rank $1 + \frac12 n$.
Therefore, the claim follows from  induction.

\medskip

\noi
(iii)
 From from \eqref{ht0}, \eqref{ord1b},  and Part (i), 
 we easily see that 
 $\wt{h}_{n}$  has  rank $2 + \frac 12 n - \frac 12 \pf(n)$.
From Lemma \ref{LEM:ht2a} with \eqref{BTX1aa}, we have 
\begin{equation}
\int_\T\wt{h}_{n}dx   = - \frac{(-i\dl)^{\pf(n)}}{2i\dl^3} \sum_{j=2}^{n+1} \frac{(2i\dl)^j}{j!} 
\sum_{\substack{n_1, \ldots, n_j \in \Z_{\ge 0} \\ n_{1\cdots j} = n+1 - j}}
\int_\T h_{n_1} \ldots h_{n_j} dx.
\label{DEX2}
\end{equation}

\noi
The  first two bounds in \eqref{DEX1} follow
from \eqref{DEX2} and \eqref{saux1a}.
The upper bound 
$  \#\dl\le n + \pf(n) - 2$
follows from \eqref{ht0}
and \eqref{saux1a}.

It remains to show $\#\dl \ge 0$.
When $n$ is odd, 
 $\#\dl \ge 0$
follows 
from \eqref{DEX2} and \eqref{saux1a}.
When $n$ is even,  
the contribution from $j \ge 3$ in~\eqref{DEX2} satisfies $\#\dl \ge 0$.
It remains to consider the contribution from $j = 2$ in \eqref{DEX2}.
In view of the third bound in \eqref{saux1a}, 
 it remains to consider 
\begin{equation}
\label{saux3}
\dl^{-1}\sum_{\substack{n_1, n_2 \in \Z_{\ge 0}\\n_{12}=n-1}} \int_\T h_{n_1, 0} h_{n_2, 0}  dx, 
\end{equation}

\noi
where $h_{n,0}=h_{n,0}^\dl$
denotes  the coefficient of $\dl^0$ in $h_n$.\footnote{Note that $h_{n,0}$ can still depend on $\dl$ through $\Gd$.} In fact, we show that \eqref{saux3} vanishes. 
From \eqref{h1} with \eqref{BTX1aa}, we obtain the following recurrence relation for $h_{n,0}$:
\begin{equation}
\label{hn0}
\begin{aligned}
h_{0,0} & = -v, \\
h_{1,0} & = - \dx v, \\
h_{n,0} & = \sum_{\substack{n_1, n_2 \in \Z_{\ge 0}\\n_{12}=n-2}} h_{n_1,0} h_{n_2,0} + \dx h_{n-1,0}.
\end{aligned}
\end{equation}
By comparing \eqref{hn0} and \eqref{hkdv}, we see that
\begin{equation}
\label{hn00}
h_{n,0} = h^\KDV_n
\end{equation}

\noi
for any $n\in\Z_{\ge 0}$. 
In particular, from \eqref{saux3}, \eqref{hn0}, and \eqref{hn00}, we have 
\begin{align*}
\eqref{saux3} = \dl^{-1} \int_\T h^\KDV_{n+1} dx
\end{align*}

\noi
which vanishes in view of 
Lemma \ref{LEM:kdv1}\,(i) since $n+1$ is odd.
\end{proof}

\subsection{2-to-1 collapse of the shallow-water conservation laws}
\label{SUBSEC:B3}

In this subsection,
we construct  shallow-water conservation laws and prove their convergence to the corresponding KdV conservation laws, 
thus establishing Theorem \ref{THM:2}.

As in Definition \ref{DEF:mono1}, 
we introduce the  classes $\wt{\Pc}_n(v)$ of (monic) monomials in $v$, adapted to the 
current shallow-water setting.

\begin{definition}\rm
\label{DEF:mono2}

Let $v\in C^\infty(\T)$ with $\int_\T  v dx =0$. Given $n\in\N$, we defined the classes $\wt{\Pc}_n(v)$ of monomials by setting
\begin{align*}
\wt{\Pc}_1 (v)&  = \Big\{ \Gd^\al \dx^{\b} v : \,   \al, \b \in \Z_{\ge0 } \Big\}, \\
\wt{\Pc}_2 (v)&  = \Big\{ \big[\Gd^{\al_1} \dx^{\be_1} v\big] \big[\Gd^{\al_2} \dx^{\be_2} v\big]: \,  \al_1,\al_2, \be_1 , \be_2 \in \Z_{\ge0} \Big\}, \\
\wt{\Pc}_n (v)&  = \bigg\{ \prod_{\l=1}^k \Gd^{\al_\l} \dx^{\be_\l} p_{j_\l}(v) : \, k\in\{2, \ldots, n\}, \ j_\l \in \N, \ j_{1\cdots k} =n, \\
& \hspace{4cm}   p_{j_\l}(v) \in \wt{\Pc}_{j_\l}(v), \  \al_\l, \be_\l \in\Z_{\ge0}\bigg\}.
\end{align*}

\noi
Given a monomial $p(v) \in \wt{\Pc}_n(v)$, we define its fundamental form $\wt{p}(v) \in \wt{\Pc}_n(v)$ which is obtained by dropping  all the $\Gd$ operators appearing in $p(v)$.
Moreover, given a   monomial $p(v)$ 
with the  fundamental form of the form:
\begin{equation*}
\wt{p}(v) = \prod_{j=1}^n \dx^{\be_j} v, 
\end{equation*}

\noi
we let $|p(u)|$ denote the maximum number of derivatives on each factor by setting
\begin{align*}
|p(v)| = \max_{j=1, \ldots, n}\be_j.
\end{align*}

\noi
Given a monomial $p(v)\in \wt \Pc_n(v)$, 
we also define   its  semi-fundamental form
$\cj p(v)$, 
obtained
 by dropping  all the  $\dl \Gd$  operators
 appearing in $p(v)$
 (but keeping $\Gd$ unpaired  with $\dl$).
We use  semi-fundamental forms
only  in Subsection \ref{SUBSEC:AC2}.

\end{definition}

We are now ready to define the shallow-water conservation laws $\wt{E}^\dl_{\frac k 2}(v)$ and 
prove their shallow-water convergence. 
Theorem~\ref{THM:2} follows from Propositions~\ref{PROP:Scons1} and~\ref{PROP:Scons2}.
Recall that we are working under the mean-zero assumption on $v$
throughout the paper.

\begin{proposition}
\label{PROP:Scons1}

{\rm(i)} Let $0<\dl<\infty$. 
Define $\wt{E}^\dl_{\frac k 2}(v)$, $k \in \Z_{\ge0}$,  by setting
\begin{equation}
\label{DEs}
\begin{aligned}
\wt{E}^\dl_0(v) & = - \frac12 \int_\T \wt{h}_1^\dl dx, \\
 \wt{E}^\dl_{\kk-\frac12}(v) & = (-1)^{\kk+1} \frac{3}{4\kk} \Re \int_\T \wt{h}^\dl_{2\kk} dx , \\
 \wt{E}^\dl_{\kk}(v) & = (-1)^{\kk+1} \frac12 \Re \int_\T \wt{h}^\dl_{2\kk+1} dx 
\end{aligned}
\end{equation}

\noi
for  $\kk\in\N$, where $ \wt{h}_n^\dl$ is the microscopic conservation law defined in \eqref{ht0}.
Then, $\wt{E}^\dl_{\frac k 2}(v)$ is conserved under the flow of the scaled ILW \eqref{sILW}. 
Moreover, for any $k \in \Z_{\ge 0}$,
$\wt{E}^\dl_{\frac k 2}(v)$ satisfies 
the following statements.

\smallskip

\noi
\textup{(i.a)} When $k = 2\kk-1$, $\kk \in \N$, is odd, 

\smallskip

\begin{enumerate}

\item[\rm(i.a.1)]
the odd-order shallow-water conservation law
 $\wt{E}^\dl_{\kk-\frac12}(v)$ has rank $\kk+2$ in the sense of Definition~\ref{DEF:ord2}, 
 where each  monomial in  $\wt{E}^\dl_{\kk-\frac12}(v)$ satisfies 
\begin{align}
\begin{split}
& 0 \le \#\Gd\le \min(\#\dx, \#\dl+1), \quad \# v + \#\dx \le 2\kk+1 ,   \quad \text{and} \\
& 0 \le  \#\dl\le 2\kk  - 2, 
\end{split}
\label{DEs1a}
\end{align}

\smallskip

\item[\rm(i.a.2)]
 the quadratic {\rm(}in $v${\rm)} part of  $\wt{E}^\dl_{\kk-\frac12}(v)$ is given by
\begin{equation}
\label{DEs2a}
\frac{3}{4\kk} \sum_{\substack{\l=1\\ \textup{odd}}}^{2\kk-1} \binom{2\kk}{\l} \dl^{\l-1} \|\Gd^{\frac\l2} v\|^2_{\dot{H}^{\kk-\frac12}} , 
\end{equation}

\noi
where  we used the convention \eqref{odd2}.

\end{enumerate}

\smallskip

\noi
\textup{(i.b)} When $k = 2\kk$, $\kk \in \Z_{\ge 0}$, is even, 

\smallskip
\begin{enumerate}

\item[\rm(i.b.1)]
the even-order shallow-water conservation law
 $\wt{E}^\dl_{\kk}(v)$ has rank $\kk+2$ in the sense of Definition~\ref{DEF:ord2}, 
 where each monomial in  $\wt{E}^\dl_{\kk}(v)$ satisfies
\begin{equation}
\label{DEs1b}
\begin{split}
& 0\le \#\Gd \le \min(\#\dx, \#\dl),  \quad \# v + \#\dx \le 2\kk + 2,
 \quad \text{and}\\
& 0 \le  \#\dl\le 2\kk, 
\end{split}
\end{equation}

\smallskip

\item[\rm(i.b.2)]
the quadratic {\rm(}in $v${\rm)} part of  $\wt{E}^\dl_{\kk}(v)$ is given by
\begin{equation}
\label{DEs2b}
\frac12 \sum_{\substack{\l=0\\\textup{even}}}^{2\kk} \binom{2\kk+1}{\l} \dl^\l \| \Gd^{\frac \l2} v \|^2_{\dot{H}^\kk}.
\end{equation}

\end{enumerate}

\noi
In particular, we have
\begin{equation}
\label{DEs3}
\wt{E}^\dl_0(v) = \frac12 \|v\|^2_{L^2},
\end{equation}

\noi
and, defining $\wt{a}_{k, \l}$, $k \in \N$,  by
\begin{equation}
\label{DEs3a}
\wt{a}_{2\kk-1, \l} =  \frac{3}{2\kk} \binom{2\kk}{\l} \qquad \text{and} \qquad
\wt{a}_{2\kk, \l} = \binom{2\kk+1}{\l},
\end{equation}

\noi
the conservation law $\wt{E}^\dl_{\frac k 2}(v)$ in \eqref{DEs} reduces to the form \eqref{E5}
with \eqref{E5a}.

\medskip

\noi
\textup{(ii)} Given $\kk \in\Z_{\ge0}$, we obtain the following conservation laws for KdV \eqref{kdv}:
\begin{equation}
\label{DEs4}
\wt E^\KDV_\kk(v) = (-1)^\kk \frac12 \int_\T h^\KDV_{2\kk+2}(v) dx,
\end{equation}
where $h^\KDV_n$ is the microscopic conservation law for KdV defined in \eqref{hkdv}. Moreover,

\smallskip
\begin{enumerate}
\item[\rm(ii.a)] 
$\wt E^\KDV_\kk(v)$ has rank $\kk+2$ with $\# \dl = \#\Gd=0$, 

\smallskip

\item[\rm(ii.b)] 
$\wt E^\KDV_\kk(v)$ is of the form \eqref{E6}.
\end{enumerate}

\end{proposition}

\begin{proof}

The conservation of $\wt{E}^\dl_{\frac k 2}(v)$ 
under the flow of the scaled ILW~\eqref{sILW}
follows from  \eqref{ht0} and  \eqref{BTX8}.
Similarly, 
the conservation of $\wt E^\KDV_\kk(v)$
under the flow of KdV \eqref{kdv} 
follows from~\eqref{BTX9}.
The claims (i.a.1), (i.b.1), and (ii.a)
follow from 
Lemma~\ref{LEM:h0}.

The claim \eqref{DEs2b} for $\kk = 0$, namely \eqref{DEs3}, 
follows from 
\eqref{ht1a} with \eqref{BTX1aa}.
More generally, the claims (i.a.2) and (i.b.2) follow
from  Lemma \ref{LEM:ht2}.
Lastly, the claim in (ii.b)  follows from 
 Lemma~\ref{LEM:kdv1}~(ii).
\end{proof}

\begin{proposition}
\label{PROP:Scons2}

Given $k \in \N$ and $0<\dl<\infty$, we have 
\begin{equation*}
 \wt{E}^\dl_{\frac k2}(v) = \wt E^\KDV_{\lceil \frac k2 \rceil}(v)
+ \wt{\EE}^\dl_{\frac k2}(v), 
\end{equation*}

\noi
where
$\lceil x \rceil$ denotes
the smallest integer greater than or equal to $x$, 
and
 $\wt{E}^\dl_{\frac k 2}(v)$ and $\wt  E^\KDV_{\lceil \frac k2 \rceil}(v)$ are as in \eqref{DEs} and \eqref{DEs4}, respectively. 
Here,  $\wt{\EE}^\dl_{\frac k2}(v)$ consists of the terms in $\wt{E}^\dl_{\frac k 2}(v)$
which depend on $\dl$ in an explicit manner and those with $\Gd$, where $\Gd$ is replaced by 
the perturbation operator $\Qd = \Gd + \frac13 \dx$
defined in \eqref{Qdl2}.
Moreover, \noi
for any $\kk \in \N$ and $v\in H^\kk(\T)$, 
we have
\begin{equation}
\label{EEc1}
\lim_{\dl\to0} \wt{\EE}^\dl_{\kk-\frac12}(v) = \lim_{\dl\to0} \wt{\EE}^\dl_{\kk}(v) =0,
\end{equation}
namely
\begin{equation*}
\lim_{\dl\to0} \wt{E}^\dl_{\kk-\frac12}(v) = \lim_{\dl\to0} \wt{E}^\dl_{\kk}(v)  = \wt  E^\KDV_{\kk}(v).
\end{equation*}

\end{proposition}

In the deep-water case (Proposition \ref{PROP:cons1}), 
the deep-water convergence followed from an explicit power of $\dl^{-1}$;
see \eqref{DE13c}.
In the current shallow-water setting, 
the situation is more subtle
due 
to the slower convergence of the perturbation operator $\Qd$ in \eqref{Qdl2}
and also insufficient powers of $\dl$ (tending to $0$), 
and thus we need to proceed with more care.

\begin{proof}[Proof of Proposition \ref{PROP:Scons2}]

Let $\wt{E}^\dl_{\frac k 2, 0}(v)$ be the $\dl$-free part of $\wt{E}^\dl_{\frac k2 }(v)$ defined in \eqref{DEs}, namely the collection of all terms that do not depend on $\dl$ explicitly and 
with $\Gd$ replaced by $-\frac13\dx$;\footnote{\label{FT:2}This latter process of replacing $\Gd$ 
by $-\frac 13\dx$ is relevant only when $k$ is odd.
When $k = 2\kk$ is even, it follows from  \eqref{DEs1b}
that 
 $\#\Gd \le \#\dl$ holds for any monomial in $\wt E^\dl_{\kk}(v)$
 and thus any monomial with $ \#\Gd\ge 1$ 
 (which implies $\#\dl \ge 1$) is removed
 in constructing the $\dl$-free part.}
 see also Remark \ref{REM:Gd1}.
Then,  we define $\wt{\EE}^\dl_{\frac k 2}(v)$ by
\begin{equation}
\label{EEb}
\wt{\EE}^\dl_{\frac k 2}(v) = \wt{E}^\dl_{\frac k 2}(v) - \wt{E}^\dl_{\frac k2, 0}(v),
\end{equation}

\noi
where all the terms in $\wt{\EE}^\dl_{\frac k2 }(v)$ involve at least one power of $\dl$ or
the perturbation  operator $\Qd = \Gd + \frac13\dx$.

\medskip

\noi
\underline{\bf Part 1:}
In this part, we prove \eqref{EEc1}.
In the following, we assume  $0<\dl \le 1$.

\smallskip

\noi
$\bullet$  {\bf Case A:}
 $k=2\kk-1$ is odd.
\\
\indent
From Proposition~\ref{PROP:Scons1}~(i.a), we see that 
any  monomial in $\wt{\EE}^\dl_{\kk-\frac12}(v)$ satisfies
\begin{equation}
\begin{aligned}
\label{EEodd1}
& \# v + \frac12 (\# \dx + \#\Gd - \#\dl) = \kk +2, \quad \# v \ge 2, \\
&    0 \le \#\Gd \le \min(\#\dx, \#\dl+1), \quad  \# v + \#\dx \le 2\kk+1,
\quad \text{and} \quad 0 \le \#\dl \le 2\kk - 2, 
\end{aligned}
\end{equation}

\noi
where we ignore the extra $\dx$ operator in $\Qd = \Gd+ \frac13\dx$ in determining $\#\dx$
but count occurrences of $\Qd$  in $\#\Gd$.
With our perturbative viewpoint \eqref{Qdl2}, we set
\begin{align*}
\#\Qd &  = \#\Qd(p(v))= \text{number of appearance of $\Qd$ in $p(v)$}.
\end{align*}

\noi
Then, 
 from~\eqref{EEb}
and the definition of the $\dl$-free part 
 $\wt E^\dl_{\kk-\frac12, 0}(v)$, 
we also have 
\begin{equation}
\label{EEodd1a}
\text{(a) } \#\dl \ge 1, \quad\text{or} \quad 
\text{(b) } \#\dl=0, \quad \#\Qd = 1, 
\end{equation}
where the terms satisfying (b) arise from those in $\wt{E}^\dl_{\kk-\frac12}(v)$ with no powers of $\dl$ 
(which in particular implies $\#\Gd \le 1$ in view of \eqref{EEodd1})
and one instance of $\Gd$, which we rewrite as $\Gd = - \frac13\dx + \Qd$ with the first contribution going into $\wt{E}^\dl_{\kk-\frac12, 0}(v)$ 
 and the latter going into $\wt{\EE}^\dl_{\kk-\frac12}(v)$.
Thus, each monomial in $\wt{\EE}^\dl_{\kk-\frac12}(v)$ is of the form $\dl^m \int_\T p(v) dx$ for some 
$0\le m \le 2\kk-2$ and $p(v) \in \wt{\Pc}_j(v)$ with $j= \# v \in \{2, \ldots, 2\kk+1\}$. 
Then, after integration by parts, 
we have 
\begin{equation}
\wt{\EE}^\dl_{\kk-\frac12}(v) = \sum_{m=0}^{2\kk-2} \sum_{j=2}^{\kk+2 + \frac12m} \sum_{\l=0}^{2\kk+1-j} \sum_{\substack{ p(v) \in \wt{\Pc}_j(v) \\ \text{\eqref{EEodd1}-\eqref{EEodd1a} hold}  \\ \#\dx (p) = \l \\ \# \Gd(p) \le \#\dl+1 = m+1 \\ |p(v)| \le \kk  }}  c(p)  \dl^m \int_\T p(v) d x.
\label{EEodd2}
\end{equation}

 Fix $p(v) \in \wt{\Pc}_j(v)$ with $|p(v)| \le \kk$ appearing on the right-hand side of \eqref{EEodd2}. Let $\wt{p}(v) = \prod_{k=1}^j \dx^{\be_j} v$ denote the fundamental form of the monomial  $p(v)$, 
 where  $\be_1 + \cdots + \be_j \le 2\kk+1-j$ and $0\le \be_k \le \kk$.
We note from \eqref{EEodd1} with $\#v \ge 2$ that there exists at most one $k$ such that $\be_k = \kk$.

\smallskip

\noi
$\bullet$ {\bf Subcase A.1:}
 $m=\#\dl \ge 1$ and $\wt m =  \#\Gd \le \#\dl$.\\
 \indent
 From \eqref{dlGd}
 in Lemma \ref{LEM:T1} with  $0<\dl\le1$, we have 
\begin{equation}
\label{EEodd2a}
|\dl^m \ft{\Gd^{\wt{m}}}(n) | \les \dl^m \min\Big( \frac1\dl, |n| \Big)^{m} \le
\min( 1,  \dl |n|)
\le  \dl |n|
\end{equation}

\noi
for any $n\in\Z^*$. 
Note from \eqref{EEodd1} that $\# \dx \le 2\kk-1$. Consequently, accounting for the loss of 
one derivative 
from estimating $\dl^m \Gd^{\wt{m}}$ in \eqref{EEodd2a}, 
$p(v)$ has at most $2\kk$ derivatives in total. 
Hence, by integration by parts, we can reduce it to the form where there are at most two factors of $v$ with $\kk$ derivatives, while the remaining factors have at most $\kk-1 < \kk-\frac12$ derivatives. 
Then, 
 by
viewing the product on the physical side as
an iterated convolution on the Fourier side
and applying 
Young's inequality 
and 
Cauchy-Schwarz's inequality, 
we have 
\begin{equation}
\label{EEodd2b}
\begin{split}
\bigg| \dl^m \int_\T p(v) dx \bigg| 
& \les \dl \big(1+\|v\|_{H^\kk}\big)^2 
\big(1+ \| v\|_{\F L ^{\kk-1, 1}}\big)^{2\kk-1}
 \les \dl \big(1+ \|v\|_{H^\kk}\big)^{2\kk+1}\\
& \too0, 
\end{split}
\end{equation}

\noi
as $\dl \to 0$, 
where the $\F L^{s, r}$-norm is as in \eqref{FL1}.

\smallskip

\noi
$\bullet$ {\bf Subcase A.2:}
 $m = \#\dl \ge 1$, $\#\Gd = \#\dl+1$,  and $j = \#v\ge3$.\\ 
 \indent
 From \eqref{dlGd}
 in Lemma \ref{LEM:T1} with  $0<\dl\le1$, we have 
\begin{equation}
\label{EEodd2c}
|\dl^m \ft{\Gd^{m+1}}(n) | \le \dl^m \min\Big( \frac1\dl, |n|\Big)^{m+1} \le \dl |n|^2
\end{equation}

\noi
for any $n\in\Z^*$.
Note from \eqref{EEodd1} that $\#\dx \le 2\kk+1- j \le 2\kk-2$. Thus, accounting for the loss of 2 derivatives from  \eqref{EEodd2c}, we have at most $2\kk$ derivatives in total. 
Hence, by proceeding as in Subcase A.1, 
we obtain \eqref{EEodd2b}.

\smallskip

\noi
$\bullet$ {\bf Subcase A.3:}
 $m = \#\dl \ge 1$, $\#\Gd = \#\dl+1$,  and $j = \#v = 2$.\\
\indent
Since $\dl^mp(v)$ is quadratic in $v$ with $\#\dl \ge 1$, 
it follows from 
\eqref{DEs2a} in Proposition~\ref{PROP:Scons1} that the collection of such  $\dl^mp(v)$ is given by 
\begin{equation}
\label{EEodd3a}
P(v) = \frac{3}{4\kk}\sum_{\substack{\l=3\\ \text{odd}}}^{2\kk-1} \binom{2\kk}{\l} \dl^{\l-1} \| \Gd^{\frac \l 2} v \|^2_{\dot{H}^{\kk-\frac12}}.
\end{equation}

\noi
From \eqref{dlGd} in Lemma \ref{LEM:T1}, we have
\begin{equation}
\label{EEodd3b}
| \dl^{\l-1} \ft {\Gd^\l}(n) n^{2\kk-1} | 
\les
\begin{cases}
\eta^{\l-1} |n|^{2\kk}, & \text{for }\dl|n| \le \eta, \\
|n|^{2\kk}, & \text{for }\dl|n| > \eta
\end{cases}
\end{equation}

\noi
for any $n\in\Z^*$ and $\eta>0$.
Therefore, from \eqref{EEodd3a} and \eqref{EEodd3b}
and choosing $\eta = \sqrt{\dl}$, 
it follows from the dominated convergence theorem that 
\begin{equation}
\bigg| \int_\T P(v) dx \bigg|  \les \dl \| v\|^2_{H^\kk} + \sum_{|n| > \frac{1}{\sqrt{\dl}}} |n|^{2\kk} |\ft{v}(n)|^2
\too0, 
\label{EEodd3c}
\end{equation}

\noi
as $\dl \to 0$.

\smallskip

\noi
$\bullet$ {\bf Subcase A.4:}
 $m = \#\dl =0$ and $\#\Qd = 1$.\\
\indent
In this case,   we can write $p(v)$ as 
\begin{equation*}
p(v) = \Qd
\bigg(\prod_{k=1}^{j_0} \dx^{\be_k} v \bigg)
 \cdot \prod_{k=j_0+1}^j \dx^{\be_k} v  
\end{equation*}

\noi
for some $2 \le j \le \kk +2$ with $1\le j_0 \le j$, 
where $\be_1 + \cdots + \be_j \le 2\kk+1-j$
and $0\le \be_k \le \kk$ for any $k = 1, \dots, j$
with 
 at most one $k$ such that $\be_k = \kk$.
 From Lemma~\ref{LEM:L1} with \eqref{Qdl2}, we have 
 $|\ft{\Qd}(n)| = \frac13 |n| |\hf(\dl, n) |$, 
 where $\hf$ is as  in \eqref{hdef}.
Under $n_{1\dots j} = 0$, 
it follows from the triangle inequality that
 $\max_{k = 1, \dots, j}|n_k|\sim n^*_2$, 
 where  $n^*_2$ denotes the second largest in $|n_k|$, $k = 1, \dots, j$.
Then,   by  \eqref{hdef}
 and 
Young's 
and Cauchy-Schwarz's inequalities, 
we have
\begin{equation}
\label{EEodd5a}
\begin{aligned}
\bigg| \int_\T p(v) dx \bigg| 
& \les \sum_{\substack{n_1, \ldots, n_j \in \Z^*\\ n_{1\cdots j} =0}}
\max_{k = 1, \dots, j}|n_k| \cdot  |\hf(\dl, n_{1\cdots j_0})|  \prod_{k=1}^j |n_k|^{\be_k} |\ft{v}(n_k)| \\
& \sim  \sum_{\substack{n_1, \ldots, n_j \in \Z^*\\ n_{1\cdots j} =0}}
n^*_2 \cdot  |\hf(\dl, n_{1\cdots j_0})|  \prod_{k=1}^j |n_k|^{\be_k} |\ft{v}(n_k)| \\
& \les (1+ \| v \|_{H^\kk})^{\kk+2},
\end{aligned}
\end{equation}

\noi
uniformly in $0<\dl<\infty$, 
where,  in the last step, 
we  used the fact that there exists at most one $k$ such that $\be_k = \kk$.
 Moreover, from the pointwise (in $n$) bound
 and convergence of $\hf$  in~\eqref{HX1}  and the dominated convergence theorem, 
 we conclude that 
$\lim_{\dl \to 0} \eqref{EEodd5a} =0$.
 
 \smallskip
 
Therefore, putting all the subcases together, we obtain
\begin{equation}
\label{EEc1x}
\lim_{\dl\to0} \wt{\EE}^\dl_{\kk-\frac12}(v)  =0.
\end{equation}

\medskip

\noi
$\bullet$ 
{\bf Case B:}  $k=2\kk$ is even.\\
\indent
From Proposition~\ref{PROP:Scons1}~(i.b), we see that 
any  monomial in $\wt{\EE}^\dl_{\kk}(v)$ satisfies
\begin{equation}
\begin{aligned}
\label{EEaux1}
& \# v + \frac12 (\# \dx + \#\Gd - \#\dl) = \kk +2, \quad  \# v \ge 2,\\
&     \quad 0 \le \#\Gd \le \min(\#\dx, \#\dl), \quad  \# v + \#\dx \le 2\kk+2, 
\quad\text{and} \quad
1\le  \# \dl \le 2\kk, 
\end{aligned}
\end{equation}

\noi
where we used the fact that the terms with $\#\dl = 0$ 
(which implies $\#\Gd = 0$ in view of~\eqref{DEs1b})
are included in the $\dl$-free part $\wt{E}^\dl_{\kk, 0}(v)$.
Thus, each monomial in $\wt{\EE}^\dl_{\kk}(v)$ is of the form $\dl^{m} \int_\T p(v) dx$ for some $1\le m \le 2\kk$ and $p(v) \in \wt{\Pc}_j(v)$ with $j=\# v \in\{2, \ldots, 2\kk+2 \}$.
 Then, after integration by parts, 
we have
\begin{equation}
\wt{\EE}^\dl_{\kk} (v)  =
\sum_{m=1}^{2\kk} \sum_{j=2}^{\kk+2+  \frac 1 2m } \sum_{\l=0}^{2\kk+2-j} \sum_{\substack{ p(v) \in\wt{\Pc}_j(v) \\ \eqref{EEaux1} \text{ holds} \\ \#\dx(p) = \l \\ \#\Gd(p) \le m \\ |p(v) | \le \kk      }} c(p) 
\dl^m \int_\T p(v) dx .
\label{EEaux2}
\end{equation}

Fix $p(v) \in \wt{\Pc}_j(v)$ with $|p(v)| \le \kk$ appearing on the right-hand side of \eqref{EEaux2}. Let $\wt{p}(v) = \prod_{k=1}^j \dx^{\be_j} v$ denote the fundamental form of the monomial $p(v)$, where $0\le \be_k \le \kk$ and $\be_1 + \cdots + \be_j \le 2\kk +2 -j$.
Note that, unlike Case A, 
there can be two  $k$'s such that $\be_k = \kk$.

\smallskip

\noi
$\bullet$ {\bf Subcase B.1:}
 $m = \#\dl \ge \#\Gd+1$.\\
 \indent 
Proceeding as in  \eqref{EEodd2b}
with  \eqref{EEodd2a}, 
(but using the upper bound by $1$ in \eqref{EEodd2a}), we have 
\begin{equation}
\label{EEaux3i}
\bigg|  \dl^m \int_\T p(v)  dx \bigg| 
 \les \dl \big(1+ \|v\|_{H^\kk}\big)^{2\kk+2}
 \too0, 
\end{equation}

\noi
as $\dl \to 0$.

\smallskip

\noi
$\bullet$ {\bf Subcase B.2:}
 $m=\#\dl = \#\Gd \ge 1$ and $j = \#v \ge 3$.\\
 \indent
By noting that  at most one factor of $v$ in $p(v)$ has $\kk$ derivatives, 
we can proceed as in Subcase A.1
and obtain \eqref{EEaux3i}.

\noi
\smallskip

\noi
$\bullet$ {\bf Subcase B.3:}
$m=\#\dl = \#\Gd \ge 1$ and $j = \#v =2$.\\
\indent
Since $\dl^m p(v)$ is quadratic in $v$ with $\#\dl \ge 1$, 
it follows from 
\eqref{DEs2b} in Proposition~\ref{PROP:Scons1} that the collection of such  $\dl^m p(v)$ is given by 
\begin{equation}
\label{EEaux3iii}
\frac12 \sum_{\substack{m=2\\\text{even}}}^{2\kk} \binom{2\kk+1}{m} \dl^m \| \Gd^{\frac m 2} v \|^2_{\dot{H}^\kk}.
\end{equation}

\noi
From \eqref{dlGd} in Lemma \ref{LEM:T1}, we have 
\begin{equation}
\label{EEaux3iiia}
| \dl^{m} \ft {\Gd^m}(n) n^{2\kk} | 
\les \begin{cases}
\eta^m |n|^{2\kk} , & \text{for }\dl|n| \le \eta, \\
|n|^{2\kk}, & \text{for }\dl|n| > \eta
\end{cases}
\end{equation}

\noi
for any $n\in\Z^*$ and $\eta>0$.
Therefore, from \eqref{EEaux3iii} and \eqref{EEaux3iiia}, and choosing $\eta = \sqrt{\dl}$, 
it follows from the dominated convergence theorem that 
\begin{equation*}
\dl^m \| \Gd^{\frac m2} v\|^2_{\dot{H}^\kk}  \le \dl \| v\|^2_{\dot{H}^\kk} + \sum_{|n| > \frac{1}{\sqrt{\dl}}} |n|^{2\kk} |\ft{v}(n)|^2\too0, 
\end{equation*}

\noi
as $\dl \to 0$.

\smallskip
 
Therefore, putting all the subcases together, we obtain
\begin{equation}
\lim_{\dl\to0} \wt{\EE}^\dl_{\kk}(v)  =0.
\label{EEc1y}
\end{equation}

\noi
From \eqref{EEc1x}
and \eqref{EEc1y}, we obtain
\eqref{EEc1}.

\medskip

\noi
\underline{\bf Part 2:}
It remains to show
\begin{equation}
\label{EE0}
\wt{E}^\dl_{\kk-\frac12 ,0}(v) = \wt{E}^\dl_{\kk, 0}(v) = \wt E^\KDV_{\kk}(v).
\end{equation}

We first consider $\wt{E}^\dl_{\kk, 0}(v)$.
In this case, from 
\eqref{DEs} and \eqref{ht1}
in Lemma \ref{LEM:ht2a}, we have
\begin{equation}
\label{EEa}
\wt{E}^\dl_{\kk, 0}(v)  = (-1)^{\kk} \frac12 \Re \sum_{n_{12}=2\kk} \int_\T h_{n_1, 0} h_{n_2,0}dx  ,
\end{equation}

\noi
where 
$h_{n, 0}$ denotes  the coefficient of $\dl^0$ in $h_n$.
Then, from 
 \eqref{EEa},   \eqref{hn00},  \eqref{hkdv}, and \eqref{DEs4}, we have 
\begin{equation*}
\begin{aligned}
\wt{E}^\dl_{\kk,0}(v) &= (-1)^\kk \frac12 \sum_{n_{12}=2\kk } \int_\T h^\KDV_{n_1} h^\KDV_{n_2} dx 
 = (-1)^\kk \frac12 \int_\T h_{2\kk+2}^\KDV dx \\
& = \wt E^\KDV_\kk(v),
\end{aligned}
\end{equation*}

\noi
which establishes \eqref{EE0} for $\wt{E}^\dl_{\kk,0}(v)$.

Next, we  show \eqref{EE0} for $\wt{E}^\dl_{\kk-\frac12, 0}(v)$. 
Since $\wt{E}^\dl_{\kk-\frac12,0}(v)$ 
is the  $\dl$-free part of  $\wt{E}^\dl_{\kk-\frac12}(v)$, 
it follows from \eqref{DEs2a} that 
 the  quadratic part in $\wt{E}^\dl_{\kk-\frac12,0}(v)$ is given by 
\begin{equation*}
\frac{3}{4\kk} \binom{2\kk}{1} \frac13 \| v\|^2_{\dot H^\kk} = \frac12 \|v\|_{\dot H^\kk}^2,
\end{equation*}

\noi
which agrees with the quadratic part of $E^\KDV_{\kk}(v)$ in \eqref{E6}. 
Suppose for now that  $\wt{E}^\dl_{\kk-\frac12, 0}(v)$ is conserved for KdV \eqref{kdv}.
Then, we can conclude 
\eqref{EE0}  for $\wt{E}^\dl_{\kk-\frac12,0}(v)$
from the fact that 
any polynomial conservation law under KdV 
with $\|v\|^2_{\dot{H}^\kk}$  as its  leading order quadratic part can be written as a linear combination of the conservation laws $\wt E^\KDV_{k}(v)$ for $k=0, \ldots, \kk$
defined in~\eqref{DEs4}; 
see, for example,  \cite[Theorem 2.16]{Kupershmidt81}.

In the following, we show that 
 $\wt{E}^\dl_{\kk-\frac12, 0}(v)$ is indeed conserved under KdV \eqref{kdv}. 
Let $v_0\in H^\kk(\T)$ and $0<\dl \le 1$. From Lemma~\ref{LEM:GWP2}, 
there exists $T>0$,  independent of $0 < \dl \le 1$,  and unique solutions $v^\dl, v^\KDV \in C([-T,T];H^\kk(\T))$ 
to the scaled ILW  and KdV, respectively, with $v^\dl\vert_{t=0} = v^\KDV \vert_{t=0} = v_0$ such that 
\begin{align}
\lim_{\dl\to0} \| v^\dl - v^\KDV \|_{L^\infty_T H^\kk_x} =0
\label{EE1}
\end{align}

\noi
and
\begin{align}
\| v^\dl \|_{L^\infty_T H^\kk_x} \le \| v^\KDV \|_{L^\infty_T H^\kk_x} + 1
\label{EE1a}
\end{align}

\noi
for any $0 < \dl \le 1$.

From \eqref{EEb} and the conservation of  $\wt{E}^\dl_{\kk-\frac12}$  under the scaled ILW, 
we can write 
\begin{equation}
\label{EE1aa}
\begin{aligned}
&\wt{E}^\dl_{\kk-\frac12,0}(v^\KDV(t)) - \wt{E}^\dl_{\kk-\frac12, 0}(v_0) \\
& =  
\wt{E}^\dl_{\kk-\frac12,0}(v^\KDV(t)) - \wt{E}^\dl_{\kk-\frac12}(v_0) 
+ \wt{\EE}^\dl_{\kk-\frac12}(v_0) \\
& =  
\wt{E}^\dl_{\kk-\frac12,0}(v^\KDV(t)) - \wt{E}^\dl_{\kk-\frac12}(v^\dl(t)) 
+ \wt{\EE}^\dl_{\kk-\frac12}(v_0) \\
& = \Big( \wt{E}^\dl_{\kk-\frac12, 0} (v^\KDV(t)) - \wt{E}^\dl_{\kk - \frac12, 0}(v^\dl(t)) \Big) - \wt{\EE}^\dl_{\kk-\frac12} (v^\dl(t)) + \wt{\EE}^\dl_{\kk-\frac12}(v_0) .
\end{aligned}
\end{equation}

\noi
From Proposition \ref{PROP:Scons1}\,(i.a)
with Definitions~\ref{DEF:ord2} and \ref{DEF:mono2},
we see that any monomial in 
  $\wt{E}^\dl_{\kk - \frac12, 0}(v)$
  is of the form  $\int_\T p(v) dx$ for some  $p(v) \in \wt{\Pc}_j(v)$ for $j=2, \ldots, \kk+2$, 
  satisfying\footnote{Here, before considering the $\dl$-free part, 
it is possible to have   $\# \dl(p) = 0 $  but $\#\Gd(p) =1$ under \eqref{DEs1a}.
Then, in constructing the $\dl$-free part, we replaced $\Gd$ by $- \frac13 \dx$
which increased $\#\dx$ by 1.
This explains the upper bound $2\kk + 2$ in \eqref{EE1x}.} 
\begin{equation}
\# \dl(p) = \#\Gd(p) =0, \quad \# v + \# \dx \le 2\kk + 2, \quad \text{and} 
\quad |p(v)| \le \kk,
\label{EE1x}
\end{equation}

\noi
where the last bound follows from the fact that 
$\#v \ge 2$ (which implies $\# \dx \le 2\kk $)
and  integration by parts. 
Thus,  there exist at most two factors of $v$ with $\kk$ derivatives, while the remaining factors have at most $\kk-1 < \kk-\frac12$.
Moreover, by using   the multilinearity of the polynomials $p(v)$
and applying 
Young's inequality 
and 
Cauchy-Schwarz's inequality
on the Fourier side as in \eqref{EEodd2b},  
\eqref{EE1a},  and~\eqref{EE1}, we have
\begin{align}
\label{EE2a}
\begin{split}
| \wt{E}^\dl_{\kk-\frac12, 0}(v^\KDV(t)) - \wt{E}^\dl_{\kk-\frac12, 0}(v^\dl(t)) |
&   \les \big(1+\|v^\KDV\|_{L^\infty_T H^\kk_x}\big)^{\kk+1} \| v^\dl - v^\KDV \|_{L^\infty_T H^\kk_x }\\
&  \too 0,
\end{split}
\end{align}

\noi
as $\dl\to0$.

%
From \eqref{EEc1}, we have 
\begin{equation}
\label{EE2b}
 |\wt{\EE}^\dl_{\kk-\frac12}(v_0)| \too 0, 
\end{equation}

\noi
as $\dl\to0$.
We point out that, by a slight modification of the argument in Part 1 leading to~\eqref{EEc1}, 
we also have 
\begin{equation}
\label{EE2bx}
 | \wt{\EE}^\dl_{\kk-\frac12} (v^\dl(t)) | \too 0, 
\end{equation}

\noi
as $\dl\to0$.
Indeed, for the contributions from monomials, 
satisfying the assumptions
of Subcase~A.1 and  A.2, 
it follows
from 
\eqref{EEodd2b} 
(with $v$ replaced by $v^\dl(t)$)
with 
the uniform bound~\eqref{EE1a}
that 
\begin{align}
\bigg| \dl^m \int_\T p(v^\dl(t)) dx \bigg| 
\too 0, 
\label{RRX1}
\end{align}

\noi
as $\dl \to 0$.
In 
Subcase A.3, 
 we used the dominated convergence theorem, 
and thus  more care is needed.
Let $P(v)$ be as in Subcase A.3 above.
Then, by writing 
\begin{align}
\bigg| \int_\T P(v^\dl(t)) dx \bigg|
\le \bigg| \int_\T P(v^\dl(t)) - P(v(t)) dx \bigg|
+ \bigg| \int_\T P(v(t)) dx \bigg|, 
\label{RRX2}
\end{align}

\noi
it follows from 
\eqref{EEodd3c}
that the second term on the right-hand side of \eqref{RRX2} tends to $0$
as $\dl \to 0$.
As for the first term, 
it follows from \eqref{EE1} and \eqref{EE1a} that 
\begin{align}
\bigg| \int_\T P(v^\dl(t)) - P(v(t)) dx \bigg|
\les \big(\| v^\dl \|_{L^\infty_T H^\kk_x} + \| v \|_{L^\infty_T H^\kk_x}\big) 
\| v^\dl - v \|_{L^\infty_T H^\kk_x}
 \too 0 ,
 \label{RRX3}
\end{align}

\noi
as $\dl\to0$. 
For the remaining  cases, namely
Subcase A.4, 
the required modification is analogous to that for Subcase A.3 discussed above, 
and thus we omit details.
Hence, \eqref{EE2bx} follows from \eqref{RRX1}, 
\eqref{RRX2}, and \eqref{RRX3}.

Lastly, 
recalling that, by definition,   $\wt{E}^\dl_{\kk-\frac12, 0}$ is independent of $\dl$,
it follows from~\eqref{EE1aa}, \eqref{EE2a}, \eqref{EE2b}, and \eqref{EE2bx} that 
\begin{equation*}
\wt{E}^\dl_{\kk-\frac12, 0}(v^\KDV(t)) = \wt{E}^{\dl}_{\kk-\frac12, 0}(v_0)
\end{equation*}

\noi
for any $ |t| \le T$.
This  shows that $\wt{E}^\dl_{\kk-\frac12, 0}(v)$ is  conserved under KdV 
and thus must be equal to $E^\KDV_{\kk}(v)$
since they have the same quadratic part.
This concludes 
 the proof of Proposition~\ref{PROP:Scons2}.
\end{proof}

\subsection{Structure of the shallow-water conservation laws}
\label{SUBSEC:B4}

We conclude this section by stating the final structure of the shallow-water conservation laws $\wt{E}^\dl_{\frac k2 }(v)$ defined in \eqref{DEs}.

Let $k\in\Z_{\ge0}$. Then, from 
(the proofs of) Propositions~\ref{PROP:Scons1}
and~\ref{PROP:Scons2}, we see that the shallow-water conservation law $\wt{E}^\dl_{\frac k2 }(v)$ defined in \eqref{DEs} is indeed given by \eqref{E5}:

\noi
\begin{align*}
\wt{E}^\dl_0(v) & = \frac12 \|v\|^2_{L^2}, \\
\wt{E}^\dl_{\kk-\frac12}(v)& = \frac12 \sum_{\substack{\l=1\\\text{odd}}}^{2\kk-1}
 \wt{a}_{2\kk-1, \l} \, \dl^{\l-1} \| \Gd^{\frac \l 2} v\|^2_{\dot{H}^{\kk-\frac12}} + \wt{R}^\dl_{\kk-\frac12}(v), \\
\wt{E}^\dl_\kk (v) & = \frac12 \sum_{\substack{\l=0\\ \text{even}}}^{2\kk} 
\wt{a}_{2\kk, \l}\, \dl^\l \| \Gd^{\frac \l2 } v \|^2_{\dot{H}^\kk} + \wt{R}^\dl_{\kk}(v),
\end{align*}

\noi
for $\kk\in\N$  with $\wt{a}_{k, \l}$ as in \eqref{DEs3a}. 
Here, the interaction potentials
 $\wt{R}^\dl_{\kk-\frac12}(v)$ and $\wt{R}^\dl_{\kk}(v)$, satisfying $\#v\ge 3$,  are 
given by 
\begin{equation}
\label{Scons2}
\begin{aligned}
\wt{R}^\dl_{\kk-\frac12}(v) 
& = \sum_{m=0}^{2\kk-2}  \sum_{j=3}^{\kk+1} \sum_{\substack{p(v) \in \wt{\Pc}_j(v) \\ \#\Gd(p) = m + 1
\le \#\dx(p)\\ 
\#\dx(p) = 2\kk +3 - 2j \\ |p(v)| \le \kk -1}} c(p) \dl^m   \int_\T p(v) dx\\
& \quad + \sum_{m=0}^{2\kk-2} \sum_{j=3}^{\kk+2 + \frac 12 m} \sum_{\substack{ p(v) \in \wt{\Pc}_j(v) \\ 
\#\Gd(p) \le \min(m,  \#\dx (p)) \\ \# \dx(p) \le 2\kk+1-j \\ |p(v)| \le \kk -1 }} c(p) \dl^m  \int_\T p(v) dx ,\\
\wt{R}^\dl_{\kk}(v) 
&  = \sum_{m=0}^{2\kk} \sum_{j=3}^{\kk+2 + \frac 12 m} \sum_{\substack{ p(v) \in \wt{\Pc}_j(v) \\ 
\#\Gd(p) \le\min(m, \#\dx(p)) \\ \#\dx(p) \le 2\kk+2 -j \\ |p(v)| \le \kk  } } c(p) \dl^m \int_\T p(v) dx
\end{aligned}
\end{equation}

\noi
with suitable constants $c(p) \in \R$, independent of $\dl$, 
where the condition $\#\dx(p) = 2\kk +3 - 2j$
in the first sum comes from 
the fact that 
$\wt{R}^\dl_{\kk-\frac12}(v) $ has rank $\kk + 2$ with
$ \#\Gd = \#\dl + 1$.

\section{Deep-water \GGMs}
\label{SEC:DGM1}

In this section,  we 
study the construction and convergence properties
of the \GGMs~ in the deep-water regime, 
claimed in 
Theorem~\ref{THM:3}.
As pointed out in~\cite{LOZ}, 
the main observation is the following:
 
 \smallskip
 
 \begin{itemize}
 \item[(i)]
 In order to construct the limiting \GGM~ $\rho^\dl_\frac k2$ 
 for each fixed $0 < \dl \le \infty$
 (Theorem \ref{THM:3}\,(i)), 
it suffices to prove an
$L^p$-integrability of 
the truncated density $F^\dl_\frac k2 (\P_N u)$, 
uniformly in $N \in \N$ but for each fixed $0 < \dl \le \infty$.

\smallskip

\item[(ii)] In order to prove convergence of the \GGMs~
in the deep-water limit  (Theorem~\ref{THM:3}\,(ii)), 
we need to establish
 an
$L^p$-integrability of 
the truncated density $F^\dl_\frac k2 (\P_N u)$, 
uniformly in 
 {\it both} $N \in \N$ and $\dl \gg 1$. 
 
 \end{itemize}

\smallskip

\noi
Here, we need to study an $L^p$-integrability of  
the truncated density 
  with respect to the Gaussian measure~$\mu^\dl_\frac k2$ in \eqref{gauss1}
which is different for different values of $\dl$.
In order to establish a uniform (in $\dl$) bound, 
it is therefore more convenient to work with the associated Gaussian process $X^\dl_\frac k2$
in \eqref{Xdl1}
and the underlying probability measure $\PP$ on $\O$.

The main tool for proving such  a uniform (in $N \in \N$ and $\dl \gg1$)
 $L^p$-integrability bound
 is a  variational approach; see Subsection \ref{SUBSEC:DGM2}, 
 where the structural property
 of the deep-water conservation laws $E^\dl_\frac k2(u)$
 established in 
  Section \ref{SEC:cons1} plays a crucial role.
 Another key ingredient in establishing 
  convergence
 of the deep-water \GGMs~ $\rho^\dl_\frac k2$
is the corresponding convergence property of 
the associated Gaussian measures $\mu^\dl_\frac k2$.
In Subsection \ref{SUBSEC:DGM1}, we carry out detailed analysis
on 
the  Gaussian measures $\mu^\dl_\frac k2$; see Proposition \ref{PROP:gauss1}.

In Subsection \ref{SUBSEC:DGM3}, 
we present a proof of Theorem \ref{THM:3}\,(i)
on the construction of the deep-water \GGM~ for fixed $0 < \dl \le \infty$.
We then present 
 a proof of Theorem~\ref{THM:3}\,(ii)
on deep-water convergence  of the \GGM~ 
in Subsection \ref{SUBSEC:DGM4}.

In the following, 
we
often suppress the dependence on  $k$ and $K$
for  various constants.
We apply this convention in the remaining part of this paper.
Recall also our convention for using 
 $\eps > 0$ to denote an arbitrarily small constant, 
 where we suppress $\eps$-dependence of various constants.

\subsection{Equivalence  of the base Gaussian measures}
\label{SUBSEC:DGM1}

Given 
$0< \dl \le \infty$ and 
$k\in\N$, 
let $\mu^\dl_{\frac k 2}$
be  the  Gaussian measure  defined in \eqref{gauss1} (and \eqref{gauss2} when $\dl = \infty$).
Recall that 
 $\mu^\dl_{\frac{k}{2}}$ is the induced probability measure 
 under the map $\o \in \O \mapsto X^\dl_{\frac{k}{2}}(\o)$, 
 where $X^\dl_{\frac{k}{2}}$ is as in \eqref{Xdl1}.
 From Lemma~\ref{LEM:K1} with \eqref{T1} and \eqref{E0}, we have 
\begin{align}\label{DX1}
0\le T_{\dl, \frac k 2} (n) \sim_\dl |n|^k 
\qquad \text{and} \qquad \lim_{\dl\to\infty}T_{\dl,\frac k 2}(n) = 
T_{\infty, \frac{k}{2}} (n)  = |n|^k 
\end{align}

\noi
for any $0<\dl \le \infty$ and $n\in\Z^*$, 
where the implicit constant in the first bound is independent of   $2 \le \dl \le \infty$. 
With \eqref{DX1}, one can easily show that
 $X^\dl_{\frac{k}{2}}$ belongs to  
 $W^{\frac{k-1}{2}-\eps, r}(\T) \setminus W^{\frac{k-1}{2}, r}(\T)$ for any $1 \le r \le \infty$,   almost surely. 
The following proposition
establishes regularity, equivalence, and convergence properties
of the Gaussian measure $\mu^\dl_\frac k2$
and the associated random variable 
$X^\dl_\frac k2$, 
which generalizes 
 Proposition 3.1 in \cite{LOZ}
on the $k = 1$ case; see also 
 Proposition 3.4 in~\cite{LOZ}.

\begin{proposition}\label{PROP:gauss1}
Let $k \in \N$.  Then, the following statements hold.

\smallskip

\noi{\rm(i)}
Given any $0 < \dl \le \infty$, 
  finite $p \ge 1$,  and $1 \le r \le \infty$, 
the sequence $\{\P_N X^\dl_{\frac k 2}  \}_{N\in\N}$ 
converges
to   $X^\dl_{\frac{k}{2}}$ in $L^p(\Omega; W^{\frac{k-1}{2}-\eps, r}(\T))$
as $N \to \infty$.
 Moreover, given any finite $p \ge 1$  
 and $1 \le r \le \infty$, 
 there exist  $C_\dl = C_\dl( p, r) > 0$ and $\ta > 0$ such that 
\begin{align}
\sup_{N\in\N} \sup_{2\leq \dl \leq \infty} 
\Big\|  \| \P_N X^\dl_{\frac{k}{2}} \|_{W^{\frac{k-1}{2}-\eps, r}_x} \Big\|_{L^p(\Omega)} &< C_\dl  <\infty, 
 \label{DX2}\\
\sup_{2\leq \dl \leq \infty} \Big\| \| \P_M X^\dl_{\frac{k}{2}} - \P_N X^\dl_{\frac{k}{2}} 
\|_{W^{\frac{k-1}{2}-\eps, r}_x} \Big\|_{L^p(\Omega) } & \leq \frac{C_\dl}{N^\ta}
 \label{DX3}
\end{align}

\noi
for any $M\geq N\ge 1$,
where the constant $C_{\dl}$ is  independent of $2 \leq \dl \leq \infty$.
 In particular, the rate of convergence is uniform in $2\leq \dl \leq \infty$.

\smallskip

\noi
{\rm(ii)}  
Given any finite $p \ge 1$ and $ 1\le r \le \infty$,  
 $X^\dl_{\frac{k}{2}}$ converges to 
$X^\infty_{\frac k 2}$ in $L^p(\Omega;W^{\frac{k-1}{2}-\eps, r}(\T))$
 as $\dl\to\infty$. In particular, $\mu^\dl_{\frac{k}{2}}$ converges weakly to $\mu^\infty_{\frac{k}{2}}$ as $\dl\to \infty$.

\smallskip
\noi
{\rm(iii)} 
For any $0<\dl<\infty$, the Gaussian measures $\mu^\dl_{\frac{k}{2}}$ and $\mu^\infty_{\frac{k}{2}}$ are equivalent.

\smallskip
\noi{\rm(iv)} 
As $\dl \to \infty$, the measure $\mu^\dl_{\frac{k}{2}}$ converges to $\mu^\infty_{\frac{k}{2}}$ in the Kullback-Leibler divergence
defined in Definition \ref{DEF:KL}. In particular, $\mu^\dl_{\frac{k}{2}}$ converges to $\mu^\infty_{\frac{k}{2}}$ in total variation
as $\dl \to \infty$.

\end{proposition}

\begin{remark}\label{REM:gauss1}
\rm

The regularity threshold $s =\frac{k-1}{2}$ is sharp in the sense
that $X^\dl_\frac k2$ does not belong to $W^{\frac{k-1}{2}, r}(\T)$
for any $r\ge 1$, almost surely.
The same comment applies to the shallow-water regime
(Proposition \ref{PROP:gauss2}).

\end{remark}

Proposition \ref{PROP:gauss1}
follows from a straightforward modification of the proof of Proposition 3.1 in \cite{LOZ}
(and also of Proposition 3.4 in \cite{LOZ} with $k = 1$), 
but for readers' convenience
and due to its importance, we present  its proof.

\begin{proof}[Proof of Proposition \ref{PROP:gauss1}]
(i) We only prove the difference estimate \eqref{DX3} since
\eqref{DX2} follows from a similar computation.
We first consider the case $r < \infty$.
When $p\ge r\ge 1$, 
it follows from 
 Minkowski's inequality, the Wiener chaos estimate (Lemma~\ref{LEM:hyp}), 
 \eqref{Xdl1}, 
 and~\eqref{DX1} that 
\begin{align}
\begin{split}
\| &  \P_M X^\dl_{\frac{k}{2}} - \P_N X^\dl_{\frac{k}{2}}\|_{L^p_\omega W^{\frac{k-1}{2}-\eps, r}_x} 
  \leq p \big\| \jb{\dx}^{\frac{k-1}{2}-\eps}(\P_M X^\dl_{\frac{k}{2}} - \P_N X^\dl_{\frac{k}{2}})\big\|_{L^r_x L^2_\o}
  \\
& = p  \bigg(\sum_{N<|n|\leq M} \frac{\jb{n}^{k-1-2\eps}}{T_{\dl, \frac k 2}(n)} \bigg)^\frac12 
 \sim_\dl p \bigg( \sum_{N < |n| \leq M} \frac{1}{|n|^{1+2\eps}}\bigg)^\frac12 \les  \frac{p}{N^{\theta }}
\end{split}
\label{DX3a}
\end{align}

\noi
for some $\ta > 0$.
Note that the implicit constant is independent  of $2 \leq \dl \leq \infty$.
When $1 \le p < r$, it follows from H\"older's inequality (in $\o$)
and \eqref{DX3a} that 
\begin{align}
\begin{split}
\|  \P_M X^\dl_{\frac{k}{2}} - \P_N X^\dl_{\frac{k}{2}}\|_{L^p_\omega W^{\frac{k-1}{2}-\eps, r}_x} 
& \leq \|  \P_M X^\dl_{\frac{k}{2}} - \P_N X^\dl_{\frac{k}{2}}\|_{L^r_\omega W^{\frac{k-1}{2}-\eps, r}_x} \\
& \les_\dl  \frac{r}{N^{\theta }}, 
\end{split}
\label{DX3aa}
\end{align}

\noi
where 
the implicit constant is independent  of $2 \leq \dl \leq \infty$.

When $r = \infty$, from Sobolev's inequality, we have 
\begin{align}
\|  \P_M X^\dl_{\frac{k}{2}} - \P_N X^\dl_{\frac{k}{2}}\|_{L^p_\omega W^{\frac{k-1}{2}-\eps, \infty}_x} 
\les \|  \P_M X^\dl_{\frac{k}{2}} - \P_N X^\dl_{\frac{k}{2}}\|_{L^p_\omega W^{\frac{k-1}{2}-\frac 12 \eps, r_0}_x} 
\label{DX3ab}
\end{align}

\noi
for any  finite $r_0 > \frac 2\eps \gg1$.
Then, the  bound \eqref{DX3} follows from 
\eqref{DX3ab} with \eqref{DX3a} and \eqref{DX3aa}
(with $r$ and $\eps$ replaced by $r_0$ and $\frac \eps 2$, respectively).

\medskip

\noi
(ii) 
From \eqref{T1} (with \eqref{E0} and \eqref{sym1}) and 
 Lemma \ref{LEM:K1}, we have 
\begin{align}
0\le |n|^k - T_{\dl, \frac{k}{2}}(n) 
& = \sum_{\substack{\l=0 \\ \text{even}}}^k a_{k, \l} |n|^\l \big[ |n|^{k-\l} - (\Kdl(n))^{k-\l} \big] \nonumber\\
& <  \sum_{\substack{\l=0 \\  \text{even}}}^k a_{k, \l} |n|^\l \big[|n|^{k-\l} - (|n| - \dl^{-1})^{k-\l} \big]  \nonumber\\
& = \dl^{-1} \sum_{\substack{\l=0 \\  \text{even}}}^{k-1} a_{k, \l}   \sum_{j=0}^{k-\l-1}|n|^{k-1-j} (|n| - \dl^{-1})^{j} \nonumber,
\end{align}

\noi
from which we obtain
\begin{align}\label{DX3b}
0\le |n|^k - T_{\dl, \frac{k}{2}}(n) 
\les
\begin{cases}
\dl^{-1} |n|^{k-1}, & \text{for } |n| > \dl^{-1},\\
\dl^{-k} , & \text{for } |n|\le  \dl^{-1}
\end{cases}
\end{align}

\noi
for any $n \in \Z^*$.
Let $2\le \dl < \infty$.
Then, 
proceeding as in \eqref{DX3a} with \eqref{DX1} and \eqref{DX3b}, we have 
\begin{equation}\label{DX4}
\begin{aligned}
 \| X^\dl_{\frac{k}{2}} - X^\infty_{\frac k 2} \|_{L^p_\o W^{\frac{k-1}{2} - \eps, r }_x}
& \le p  \bigg( \sum_{n\in\Z^*}  \jb{n}^{k-1-2\eps} \bigg[ \frac{1}{(T_{\dl,\frac{k}{2}}(n))^\frac12} - \frac{1}{|n|^{\frac{k}{2}}} \bigg]^2 \bigg)^\frac12 \\
& \les p \bigg( \sum_{n\in\Z^*} \frac{1}{\jb{n}^{1+2\eps}} \frac{|n|^{k} - T_{\dl,\frac{k}{2}}(n)}{T_{\dl, \frac{k}{2}}(n)}  \bigg)^\frac12\\
& \les \frac p{\dl^\frac 12} \bigg(\sum_{n \in \Z^*} \frac{1}{\jb{n}^{2+2\eps}} \bigg)^\frac12 \les \frac{p}{\dl^\frac12} \too 0, 
\end{aligned}
\end{equation}

\noi
as $\dl \to \infty$, 
where, in the second inequality, we used 
the fact that $\sqrt{a}-\sqrt{b}\leq \sqrt{a-b}$ for $a\geq b \geq 0$.
Recalling that  $\mu^\dl_{\frac{k}{2}} = \Law(X^\dl_{\frac{k}{2}})$
for $0 < \dl \le \infty$, 
we then conclude weak convergence of $\mu^\dl_{\frac{k}{2}}$ to $\mu^\infty_{\frac{k}{2}}$.

\medskip

\noi
(iii) Let $0 < \dl \le \infty$. Recalling $g_{-n} = \cj {g_n}$, we have 
\begin{align*}
X^\dl_{\frac k 2}(x;\omega)
& = \frac{1}{\sqrt{2\pi}} \sum_{n\in\N} \frac{2}{(T_{\dl,  \frac{k}{2}}(n))^\frac12} 
\Big( \Re g_n \cos(nx) - \Im g_n \sin(nx) \Big).
\end{align*}

\noi
For $n\in\N$, set
\begin{align*}
A_n &= \frac{\Re g_n}{(T_{\dl,  \frac{k}{2}}(n))^\frac12},  \quad  A_{-n} = - \frac{\Im g_n}{(T_{\dl,  \frac{k}{2}}(n))^\frac12}, \quad
B_n = \frac{\Re g_n}{|n|^{ \frac{k}{2}}}, \quad  B_{-n}= - \frac{\Im g_n}{|n|^{ \frac{k}{2}}}
\end{align*}
with $a_{\pm n} = \E[A^2_{\pm n}] = (T_{\dl,  \frac{k}{2}}(n))^{-1}$ and $b_{\pm n} = \E[B^2_{\pm n}] = |n|^{-k}$. Then, using \eqref{DX1} and \eqref{DX3b}, we have
\begin{align}
\begin{split}
\sum_{n\in \Z^*} \bigg( \frac{a_n}{b_n} - 1 \bigg)^2 
&  = \sum_{n\in \Z^*} \bigg( \frac{|n|^k}{T_{\dl, \frac{k}{2}}(n)} -1  \bigg)^2\\
&  \les \frac{1}{\dl^{2k}} \sum_{0<|n| \le \frac1\dl} \frac{1}{|n|^{2k}} + \frac{1}{\dl^2} \sum_{|n| > \frac1\dl} \frac{1}{|n|^2} < \infty.
\end{split}
\label{DX6}
\end{align}

\noi
Hence, from  Kakutani's theorem (Lemma~\ref{LEM:kak}), we conclude 
equivalence of  $\mu^\dl_{\frac{k}{2}}$ and $\mu^\infty_{\frac{k}{2}}$.

\medskip

\noi
(iv) In the following, we focus on proving 
convergence in the Kullback-Leibler divergence.
Once we prove convergence in the Kullback-Leibler divergence, 
convergence in total variation follows from 
Pinsker's inequality (Lemma \ref{LEM:KL2}).

Given $0< \dl \le \infty$, we first write $\mu^\dl_\frac k2$   as the product
of Gaussian measures on $\R$:
\begin{align*}
d\mu^\dl_\frac k2
& = \bigg(\bigotimes_{n \in \N} \frac{T_{\dl, \frac k2}^\frac 12(n)}{\sqrt {2\pi}}
e^{-\frac 1{2} T_{\dl, \frac k2}(n) (\Re \ft u(n))^2 } d\Re \ft u(n)\bigg)\\
& \quad \times  \bigg(\bigotimes_{n \in \N} \frac{T_{\dl, \frac k2}^\frac 12(n)}{\sqrt {2\pi}}
e^{-\frac 1{2} T_{\dl, \frac k2}(n) (\Im \ft u(n))^2 } d\Im \ft u(n)\bigg).
\end{align*}

\noi
With $x = (x_1, x_2) \in \R^2$, we then have
\begin{align}
d\mu^\dl_\frac k2
& = \bigotimes_{n \in \N} \frac{T_{\dl, \frac k2}(n)}{ 2\pi}
e^{-\frac 1{2} T_{\dl, \frac k2}(n) |x|^2} dx
=: \bigotimes_{n \in \N} d\mu^{\dl, (n)}_{\frac k2}
\label{KD1}
\end{align}

\noi
 The Radon-Nikodym derivative $d \mu^{\dl, (n)}_\frac k2/ d \mu^{\infty, (n)}_\frac k2$
is given by
\begin{align}
\frac{d \mu^{\dl, (n)}_\frac k2}{d \mu^{\infty, (n)}_\frac k2}
= \frac{T_{\dl, \frac k2}(n)}{n^k} e^{\frac 1{2} (n^k - T_{\dl, \frac k2}(n)) |x|^2}.
\label{KD2}
\end{align}

\noi
Then, arguing as in \cite[(3.9)]{LOZ} with 
Part (iii), \eqref{KL4a},  \eqref{KL4}, 
 \eqref{KD1},  and \eqref{KD2}, we have
\begin{align}
 \dkl\big(\mu^\dl_\frac k2, \mu^\infty_\frac k2\big)
= \sum_{n \in \N} \dkl (\mu^{\dl, (n)}_\frac k2,  \mu^{\infty, (n)}_\frac k2)
= \sum_{n \in \N}  \phi\bigg(\frac{n^k}{T_{\dl, \frac k2}(n)}\bigg),
\label{KD3}
\end{align}

\noi
where $\phi(t) = t - 1 - \log t$.
Note that $\phi(1) = 0$ and $\phi'(t) > 0$ for $t > 1$.
Recalling from Lemma \ref{LEM:K1} with \eqref{T1}
that  $\frac{n^k}{T_{\dl, \frac k2}(n)}$ decreases to 1 as $\dl \to \infty$, 
we see that the summand on the right-hand side of \eqref{KD3}
decreases to $\phi(1) = 0$
as $\dl \to \infty$.
Moreover, 
by noting  $\phi(t) \le (t - 1)^2$ for $t \ge 1$, it follows from \eqref{DX6} that 
\begin{align*}
\sum_{n \in \N}  \phi\bigg(\frac{n^k}{T_{\dl, \frac k2}(n)}\bigg)
\le
\sum_{n \in \N}  \frac{(n^k - T_{\dl, \frac k2}(n))^2}{T_{\dl, \frac k2}^2(n)}<\infty
\end{align*}

\noi
for any $\dl > 0$.
Hence, by 
the dominated convergence theorem, we 
conclude that 
\begin{align*}
\lim_{\dl \to \infty}
 \dkl\big(\mu^\dl_\frac k2, \mu^\infty_\frac k2\big)
= \lim_{\dl \to \infty}\sum_{n \in \N}  \phi\bigg(\frac{n^k}{T_{\dl, \frac k2}(n)}\bigg)
= \sum_{n \in \N} \lim_{\dl \to \infty} \phi\bigg(\frac{n^k}{T_{\dl, \frac k2}(n)}\bigg)
= 0.
\end{align*}

\noi
This concludes the proof
of Proposition \ref{PROP:gauss1}.
\end{proof}

\subsection{Uniform integrability of the deep-water densities}
\label{SUBSEC:DGM2}
Our main goal in this subsection is to 
establish the following uniform (in $N$ and also in $2\le \dl \le\infty$)
$L^p$-integrability of  the truncated density $
F^\dl_{\frac k2}(\P_N u)$ defined in \eqref{rho4}, where $F^\dl_{\frac{k}{2}}(u)$ is given by 
\begin{align}
\begin{split}
F^\dl_{\frac{k}{2}}(u)  =  \eta_K\big(\| u\|_{L^2}\big)
\exp\Big(-R^\dl_{\frac k 2}(u)\Big) 
\end{split}
\label{rhox1}
\end{align}

\noi
for a fixed constant  $K>0$.

\begin{proposition}\label{PROP:V1}
Let $k \ge 2$ be an integer and $K > 0$.
Then, given any  $0<\dl\leq \infty$ and finite  $p\ge 1$,
there exists $ C_{\dl, p}> 0$ such that 
\begin{align*}
\sup_{N\in\N}  \| F^\dl_{\frac{k}{2}}( \P_N X^\dl_{\frac{k}{2}}) \|_{L^p(\Omega)} 
& = \sup_{N\in\N} \| F^\dl_{\frac{k}{2}}( \P_N u) \|_{L^p(d\mu^\dl_{\frac{k}{2}})}
\leq C_{\dl, p} < \infty, 
\end{align*}

\noi
where the constant $C_{\dl, p}$ is  independent of 
$2\le \dl \le \infty$.
Here, 
$X^\dl_\frac k2$ and 
$\mu^\dl_\frac k2$ are as in \eqref{Xdl1} and~\eqref{gauss1}, respectively.
\end{proposition}

We prove  Proposition~\ref{PROP:V1}
via a variational approach, popularized by  Barashkov and Gubinelli~\cite{BG}.
In particular, 
we follow the approach
developed in \cite{OOT, OST2, OOT2, TW} 
to study 
focusing Gibbs measures. 
Recall that
\begin{align*}
\eta_K (x) 
\le \ind_{\{|x|\le 2K\}} \le
 \exp\big( -  A |x|^\al\big) \exp\big(A (2K)^\al\big)
\end{align*}

\noi
for any $K, A , \al > 0$.
Then, 
by setting
\begin{align*}
\F^\dl_{ \frac{k}{2}}(u)
=\exp\Big(-R^\dl_{\frac{k}{2}}(u) - A \| u\|_{L^2}^{\al} \Big)
\end{align*}

\noi
for some $A, \al> 0$ (to be chosen later), 
we have 
\begin{align*}
F^\dl_{\frac{k}{2}}(\P_N u) \le C_{A, K} \cdot 
\F^\dl_{ \frac{k}{2}}(\P_N u).
\end{align*}

\noi
Hence,
Proposition \ref{PROP:V1} follows once we prove the following
uniform bound on $\F^\dl_{ \frac{k}{2}}(\P_N u)$.

\begin{proposition}\label{PROP:V2}
Let $k \ge 2$ be an integer.
Then, given  any $0<\dl\leq \infty$ and finite  $p\ge 1$,
there exist $A = A(p) > 0$,  $\al > 0$,  and $C_{\dl, p} > 0$ such that 
\begin{align}
&\sup_{N\in \N} \|\F^\dl_\frac k2 (\P_NX^\dl_{\frac{k}{2}})  \|_{L^p(\O)}
=
\sup_{N \in \N}\|\F^\dl_\frac k2 (\P_N u) \|_{L^p(d\mu^\dl_{\frac{k}{2}})}
\le C_{\dl, p} < \infty, 
\label{V10}
\end{align}

\noi
where the constant $C_{\dl, p}$ is  independent of 
$2\le \dl \le \infty$.
Here, $A = A(p)$ is independent of $2\le \dl \le \infty$, while
 $\al$ is independent of $p$ and $0 < \dl \le \infty$.

\end{proposition}

Let us introduce some notations before stating the variational formula.
Let $W(t)$ be a cylindrical Brownian motion in 
\begin{align*}
L^2_0(\T) = \P_{\ne 0} L^2(\T)
\end{align*}

\noi
of mean-zero functions on $\T$.
%
Namely, we have
\begin{align}
W(t) = \frac 1{\sqrt{2\pi}}\sum_{n \in \Z^*} B_n(t) e_n,
\label{P1}
\end{align}

\noi
where  
$\{B_n\}_{n \in \Z^*}$ is a sequence of mutually independent complex-valued\footnote{By convention, we normalize $B_n$ such that $\text{Var}(B_n(t)) = 2 t$.}
Brownian motions such that 
$\cj{B_n}= B_{-n}$, $n \in \Z^*$. 
Let  $\TT_{\dl, \frac{k}{2}}$
 be the 
 Fourier multiplier operator
with  multiplier $T_{\dl, \frac{k}{2}}(n)$ defined in~\eqref{T1}.
Then, 
we define a centered Gaussian process $\cX^\dl_{\frac k2}$
by setting
\begin{align}
\cX^\dl_\frac k2  (t)
=  \TT_{\dl, \frac{k}{2}}^{-\frac 12}    W(t)
=\frac{1}{\sqrt{2\pi}} \sum_{n\in \Z^*} \frac{B_n(t)}{\big(T_{\dl,\frac{k}{2}}(n)\big)^\frac12} e_n.
\label{P2}
\end{align}

\noi
By comparing \eqref{P2} with \eqref{Xdl1}, 
we have 
\begin{align}
\Law \big( \cX^\dl_\frac k2  (1)\big)
=\Law \big(X^\dl_\frac k2 \big) .
\label{P2a}
\end{align}

We use  $\Ha$ to denote the collection of drifts,
belonging to
$L^2([0,1]; L^2_0(\T))$, $\PP$-almost surely, 
which are progressively measurable
with respect to the filtration
induced by 
$W$;
see  
\cite[p.\,477]{Zhang} for a  discussion of filtration
in an abstract Wiener space.

We now state the Bou\'e-Dupuis variational formula \cite{BD98, Ust14}.
See also \cite[Theorem 3.2]{Zhang}
and \cite[Appendix A]{TW}.

\begin{lemma}\label{LEM:var0}
Let $k \ge 2$ be an integer,  $0 < \dl \le \infty$, 
and $N \in \N$.
Suppose that  $G:C^\infty(\T) \to \R$
is measurable such that $\E\big[|G(\P_N X^\dl_{\frac k2})|^p\big] < \infty$
and $\E\big[\big|\exp  \big( G(\P_NX^\dl_{\frac k2} )\big)\big|^q \big] < \infty$ 
for some $1 < p, q < \infty$ with $\frac 1p + \frac 1q = 1$.
Then, we have
\begin{align}
\begin{split}
&  \log \E\Big[\exp\big(G(\P_N X^\dl_{\frac k2})\big)\Big]\\
& \quad = \sup_{\dr \in  \Ha}
\E\bigg[ G\big( \P_N\cX^\dl_{\frac k2}(1) + \P_N I_{\dl, \frac k 2}(\dr)(1)\big) - \frac{1}{2} \int_0^1 \| \dr(t) \|_{L^2_x}^2 dt \bigg],
\end{split}
\label{var1}
\end{align}

\noi
where  
$I_{\dl, \frac k 2}(\dr)$ is  defined by
\begin{align*}
I_{\dl, \frac k 2}  (\dr)(t) = \int_0^t \TT_{\dl, \frac{k}{2}}^{-\frac 12}     \dr(t') dt'
\end{align*}

\noi
and the expectation $\E = \E_\PP$
is taken with respect to the underlying probability measure~$\PP$.

\end{lemma}

In view of \eqref{P2a}, 
with a slight abuse of notation, 
we use
$X^\dl_\frac k2$ to denote 
$\cX^\dl_\frac k2  (1)$
in applying the variational formula \eqref{var1} in the remaining part
of this subsection.

Given finite $p \ge 1$,  we prove Proposition \ref{PROP:V2}
by applying the variational formula (Lemma~\ref{LEM:var0})
for 
 \begin{align*}
 G(u)=- p \big(R^\dl_{\frac{k}{2}}( u) + A  \| u \|_{L^2}^{\al} \big).
 \end{align*}
 
 \noi
Given $N \in \N$, 
it is easy to see that $G(u)$ satisfies the assumption of Lemma \ref{LEM:var0}
in view of 
the bound $R^\dl_{\frac{k}{2}}( \P_N u) \le C_N < \infty$, 
uniformly in $u \in L^2(\T)$ (see, for example, the proof of Lemma~\ref{LEM:FN2}), 
and 
Proposition \ref{PROP:gauss1}.
Proceeding as in the proof of  \cite[Lemma 4.7]{GOTW}
with~\eqref{T1}, 
we have 
\begin{align}
\| I_{\dl, \frac k 2}(\theta)(1) \|_{H^\frac{k}{2}_x}^2 \leq a_\dl \int_0^t \| \theta (t) \|^2_{L^2_x}  dt
\label{var2}
\end{align}

\noi
 for any $\theta \in \Ha$, 
where the constant $a_\dl>0$ is independent of  $2 \leq \dl \leq \infty$.
Then, from Lemma~\ref{LEM:var0}  and \eqref{var2}, 
we see that 
Proposition~\ref{PROP:V2} follows once
we establish an 
upper bound on 
\begin{align}
\begin{split}
\mathcal{M}_{\dl,\frac k2 , p, N}(\Dr) 
& =\E \bigg[ -p R^\dl_{\frac{k}{2}}( 
\P_N X^\dl_{\frac k 2} + \P_N\Dr) \\
& \hphantom{XXX}- p A\| \P_N X^\dl_{\frac k 2} + \P_N\Dr \|_{L^2}^{\al} 
- \frac12 a_\dl^{-1}  \|\Dr\|^2_{H^{\frac{k}{2}}} \bigg]
\end{split}
\label{var3}
\end{align}

\noi
for each $0 < \dl \le \infty$, 
uniformly in $N\in\N$  and $\Dr\in H^{\frac{k}{2}}_0(\T) = \P_{\ne 0}H^{\frac{k}{2}}(\T)$
and also 
uniformly in $2\le \dl \le \infty$.

For simplicity of notation, 
we set 
\begin{align}
X^\dl_N =  \P_N X^\dl_{\frac k 2}\qquad \text{and}\qquad 
\Dr_N = \P_N \Dr, 
\label{var4}
\end{align}

\noi
where both $X^\dl_{\frac k 2}$ and $\Dr$ have mean zero on $\T$.
Then, from \eqref{R1}, we have 
\begin{align}
R^\dl_{\frac{k}{2}} (X^\dl_N + \Dr_N) 
& = A^\dl_{\frac k2 , \frac{k}{2}}(X^\dl_N + \Dr_N)
+ \sum_{\l=1}^{k-1} \frac{1}{\dl^{k-\l}} A^\dl_{\frac k 2, \frac{\l}{2}}( X^\dl_N + \Dr_N), 
\label{R2}
\end{align}

\noi
where
 $A^\dl_{\frac k 2,\frac{\l}{2}}(u)$  is as in 
\eqref{Aeven}, \eqref{Aodd}, and \eqref{Aodd1}.
In the following, 
we separately estimate
the contribution
from the lower order terms ($1\le \l \le k-1$) in \eqref{R2}, 
which is handled deterministically
(Lemma \ref{LEM:var1}), 
and from the leading order term ($\l = k$) in \eqref{R2}, 
for which we need to use probabilistic estimates; see  Lemma \ref{LEM:var2}.

The first lemma controls the lower order terms 
in \eqref{R2}.

\begin{lemma}\label{LEM:var1}
Let $k\ge 2$ be an integer.
Given $0<\dl\leq \infty$ and small $\eps_0 > 0$,  
 there exist $\al \ge 1$,  independent of $\dl$ and $\eps_0$,  and 
  $C_{\dl, \eps_0}>0$ such that
\begin{align}
|A^\dl_{\frac k 2, \frac{\l}{2}} (u_1 + u_2)| 
\leq C_{\dl, \eps_0} \Big( 1 + \|u_1\|^\al_{W^{\frac{k-1}{2}-\eps, \infty}} + \|u_2\|_{L^2}^{\al} \Big) 
+ \eps_0 \|u_2\|_{H^\frac{k}{2}}^2
\label{DV0}
\end{align}

\noi
for any $1 \le \l \le k-1$, 
where the constant   $C_{\dl, \eps_0}>0$  is independent of $2\le \dl \le \infty$
and  $\al$ is independent of  $0 < \dl \le \infty$.

\end{lemma}

The next lemma controls the leading contribution in \eqref{R2}.

\begin{lemma}\label{LEM:var2}
Let $k\ge 2$ be an integer.
Given $0<\dl\leq \infty$ and small $\eps_0 > 0$,  
 there exist $\al \ge 1$,  independent of $\dl$ and $\eps_0$,  and 
  $C_{\dl, \eps_0}>0$ such that
\begin{align}
\big|\E \big[ A^\dl_{\frac k2 , \frac{k}{2}}(X^\dl_N+\Dr_N) \big] \big|
 \leq 
 C_{\dl, \eps_0}+ 
 \E\Big[ C_{\dl, \eps_0} \|\Dr_N\|_{L^2}^{\al}   + \eps_0 \|\Dr_N\|_{H^\frac{k}{2}}^2 \Big], 
\label{DDV1}
\end{align}

\noi
uniformly in $N \in \N$, 
where the constant   $C_{\dl, \eps_0}>0$  is independent of $2\le \dl \le \infty$
and  $\al$ is independent of  $0 < \dl \le \infty$.

\end{lemma}

\begin{remark}\rm
From \eqref{Aeven} and 
\eqref{Aodd}, we see that the leading order term $A^\dl_{\frac k2 , \frac{k}{2}}(u)$
involves two factors with $\dx^\frac{k-1}{2} u$.\footnote{When $k$ is even, it is 
$|\dx|^\frac 12 \dx^\frac{k-2}{2} u$ to be precise.}
Since $X^\dl_{\frac k2}$ does not belong to $W^{\frac{k-1}{2}, r}(\T)$
for any $1\le r \le \infty$, almost surely, 
we need to rely on a probabilistic argument in handling
the leading order term 
$A^\dl_{\frac k2 , \frac{k}{2}}(X^\dl_N+\Dr_N)$
in Lemma \ref{LEM:var2}.

\end{remark}

We now present a proof of 
Proposition~\ref{PROP:V2}
(and hence of Proposition~\ref{PROP:V1})
by assuming Lemmas \ref{LEM:var1} and \ref{LEM:var2}.

\begin{proof}[Proof of  Proposition~\ref{PROP:V2}]
It follows from  \eqref{var3}, Lemmas \ref{LEM:var1} and \ref{LEM:var2}, 
and Proposition \ref{PROP:gauss1}\,(i) that given any small $\eps_0 > 0$,  we have 
\begin{align}
\begin{split}
\mathcal{M}_{\dl,\frac k 2,  p, N}(\Dr) 
& \leq pC_{\dl, \eps_0}+
\E\Big[\big(Cp \eps_0 - \tfrac12 a_\dl^{-1}\big) \|\Dr_N\|_{H^\frac{k}{2}}^{2}\\
& \quad +p C_{\dl, \eps_0}  \|\Dr_N\|_{L^2}^\al
- pA \|X^\dl_N+\Dr_N\|_{L^2}^\al  \Big], 
\end{split}
\label{var5}
\end{align}

\noi
uniformly in $N\in \N$ 
and $\Dr\in H^{\frac{k}{2}}(\T)$.
Moreover, the constant
$C_{\dl, \eps_0}$ is independent of 
 $2\le \dl \le \infty$.
From  \cite[(5.35)]{OOT}, given $\al > 0$, 
there exists $C_\al > 0$ such that 
\begin{align}
|a+b|^\al \ge \frac 12 |a|^\al - C_\al |b|^\al
\label{var6a}
\end{align}

\noi
for any $a, b \in \R$.
Thus, we have 
\begin{align}
- \| X^\dl_N+\Dr_N\|_{L^2}^\al
\le - \frac 12
\|  \Dr_N\|_{L^2}^\al
+ C_\al \| X^\dl_N\|_{L^2}^\al
\label{var6}
\end{align}

\noi
for some $C_\al>0$.
Hence, by choosing $\eps_0 = \eps_0(p, a_\dl)
= \eps_0(p, \dl)
>0$ sufficiently small
and then $A = A(p, \eps_0) = A(p, \dl) \gg1$
sufficiently large, 
it follows from \eqref{var5} and \eqref{var6} with 
 Proposition~\ref{PROP:gauss1}\,(i)
 that 
 \begin{align}
\begin{split}
\mathcal{M}_{\dl,\frac k 2,  p, N}(\Dr) 
& \les_{\dl, p} 1, 
\end{split}
\label{var7}
\end{align}

\noi
uniformly in $N\in \N$ 
and $\Dr\in H^{\frac{k}{2}}(\T)$, 
where the implicit constant is  independent of 
 $2\le \dl \le \infty$.
Since $a_\dl$ is independent of  $2\le \dl \le \infty$, 
the constants $\eps_0$ and $A$ 
are also  independent of  $2\le \dl \le \infty$.
 Therefore, 
\eqref{V10} with uniformity in $2\le \dl \le \infty$ follows from 
\eqref{var7} and Lemma~\ref{LEM:var0}.
This proves
Proposition~\ref{PROP:V2}
 (and hence Proposition~\ref{PROP:V1}).
\end{proof}

The remaining part of this subsection is devoted to proofs of 
Lemmas \ref{LEM:var1} and \ref{LEM:var2}.
We first state  a preliminary lemma.

\begin{lemma}\label{LEM:var0a}

Let $k \ge 2$ be an integer.
Let $0<\dl\leq \infty$
and
$s, \s_1, \s_2 \in \R$ with $\s_1 \le \s_2 \le \frac k2$
such that
$2k - 2\s_1 - 2\s_2 > 1$ and 
\begin{align}
s < 
\begin{cases}
k - \s_1- \s_2 - 1, & \text{if }k - 2\s_1 < 1,\\
  \frac 12k - \s_2 - \frac 12, & \text{if }k - 2\s_1 \ge 1.
\end{cases}
\label{VL3}
\end{align}

\noi
Then, given any finite $p \ge 1$, $1\le r \le \infty$, 
and $\al_1, \al_2 \in \{0, 1\}$, we have 
\begin{align}
\E\bigg[ \big\|\P_{\ne 0} \big((\jb{\dx}^{\s_1}\H^{\al_1}\P_N X^\dl_{\frac k 2})(\jb{\dx}^{\s_2}
\H^{\al_2}
\P_N  X^\dl_{\frac k 2})\big)\big\|_{W^{s, r}}^p  \bigg]
< C_\dl<\infty,
\label{VL2}
\end{align}

\noi
uniformly in $N \in \N$, 
where the constant $C_\dl>0$ is independent of  $2\leq \dl \leq \infty$.

\end{lemma}

\begin{remark}\label{REM:var0a}\rm
The bound \eqref{VL2} also holds
even if we replace
$\jb{\dx}^{\s_j}$ by $\jb{\dx}^{\s_j - \kk_j}\dx^{\kk_j} $
for any $\kk_j \in \N$ with $\kk_j \le \s_j$, $j = 1, 2$.

\end{remark}

\begin{proof}
We only present the case $\al_1 = \al_2 = 0$
since the other cases can be treated in a similar manner.
Proceeding as in the proof of Proposition \ref{PROP:gauss1}\,(i)
with  Lemma \ref{LEM:hyp}, it suffices to consider the case $p = r = 2$. 
We first consider the case $k - 2\s_1 \ne 1$.
From \eqref{Xdl1}, 
\eqref{DX1}, and  
Lemma \ref{LEM:SUM}, we have 
\begin{align*}
\text{LHS of \eqref{VL2}}
 & \les \E \Bigg[  \sum_{n\in \Z^*} \jb{n}^{2s} \bigg| \sum_{\substack{n=n_1+n_2\ne 0\\ 0<|n_j| \leq N}}  \frac{g_{n_1} g_{n_2}}
{(T_{\dl, \frac k2}(n_1) T_{\dl ,  \frac k2} (n_2))^\frac12} \jb{n_1}^{\s_1} \jb{n_2}^{\s_2} \bigg|^2 \Bigg] \\
& \sim_\dl  \sum_{n \in \Z} \jb{n}^{2s}
\sum_{\substack{n=n_1+n_2\\0<|n_j| \leq N}} 
\frac1{\jb{n_1}^{k - 2\s_1} \jb{n_2}^{k - 2\s_2}}\\
& \les  \sum_{n \in \Z} \frac{1}{\jb{n}^{-2s + k - 2\s_2 + \min (k - 2\s_1 - 1, 0)}}< \infty, 
\end{align*}

\noi
provided that \eqref{VL3} holds, 
where the implicit constant can be chosen to be independent
of  $2 \leq \dl \leq \infty$.
When $k - 2\s_1 = 1$, in view of Lemma \ref{LEM:SUM}, 
essentially the same computation
as above holds, yielding \eqref{VL2} under~\eqref{VL3}.
\end{proof}

We now present a proof of 
Lemma \ref{LEM:var1}.

\begin{proof}[Proof of Lemma \ref{LEM:var1}]
Fix  an integer $k \ge 2$.
We first consider the case $\l = 1$.
From 
\eqref{Aodd1}, 
 Sobolev's inequality, \eqref{interp1}, and Young's inequality, we have 
\begin{align*}
|A_{\frac k2, \frac 12}(u_1+u_2) |
& \le  C \Big(\|u_1\|^3_{L^3} +   \|u_2\|^{\frac{3k-1}{k}}_{L^2}\|u_2\|^{\frac{1}{k}}_{H^{\frac{k}{2}}}\Big)\\
& \le  C_{\eps_0}\Big( \|u_1\|_{W^{\frac{k-1}{2}-\eps, \infty}}^{3} +\|u_2\|_{L^2}^{\frac{2(3k-1)}{2k-1}} \Big)+ \eps_0 \|u_2\|^2_{H^{\frac{k}{2}}} 
\end{align*}

\noi
for any $\eps_0> 0$.

In the following, we consider the case $2 \le \l \le k - 1$, 
which in particular implies $k \ge 3$.

\smallskip

\noi
$\bullet$ {\bf Case 1:} 
$\l= 2m$ for some $1 \le m \le \frac{k-1}{2}$.\\
\indent
Let us consider the first term on the right-hand side of \eqref{Aeven}.
Noting that $\#\Qdl = 0$, it
suffices to consider a monomial of the form 
 $p(u) = u \dx^{m-1} u \dx^m u$
 thanks to 
 Lemma \ref{LEM:cub1} and 
 the $L^r$-boundedness of the Hilbert transform $\H$
 for $ 1 < r < \infty$ (Lemma \ref{LEM:T1}\,(iii)).\footnote{For the terms with the $W^{s, \infty}$-norm, 
 we can instead work with  the $W^{s, r}$-norm with $r\gg 1$
 to exploit the boundedness of the Hilbert transform.
 This may cause an $\eps$-loss of regularity (via Sobolev's inequality; 
 see, for example, \eqref{DV5})
 but does not affect the end result.
Thus,  for simplicity of the presentation,  we  work with 
 the $W^{s, \infty}$-norm in the following.
 The same comment applies to the proofs of Lemmas \ref{LEM:var2}
 and \ref{LEM:Svar2}.}
By crudely estimating with 
(fractional) integration by parts, 
Cauchy-Schwarz's inequality, 
and the fractional Leibniz rule~\eqref{leib}, 
we have 
\begin{align}
\bigg| \int_\T u \dx^{m-1}u \dx^m u dx \bigg| 
  \le \| u \dx^{m-1}u\|_{H^\frac 14}
\|u\|_{H^{m- \frac 14}}
 \les \|u\|_{W^{\frac 14, \infty}} \|u\|_{H^{\frac {2k-3}4}}^2.
\label{DV2}
\end{align}

\noi
On the one hand, we have 
\begin{align}
\|u\|_{W^{\frac 14, \infty}}
 + \|u\|_{H^{\frac {2k-3}4}}
\les \| u \|_{W^{\frac{k -1}{2}-\eps, \infty}}.
\label{DV3}
\end{align}

\noi
On the other hand, by 
Sobolev inequality and 
\eqref{interp1}, we have 
\begin{align}
\|u\|_{W^{\frac 14, \infty}}
\|u\|_{H^{\frac {2k-3}4}}^2
\les 
\|u\|_{H^{\frac {3+\eps}4 }} \|u\|_{H^{\frac {2k-3}4}}^2
\les \| u \|_{L^2}^{3- \g} \|u\|_{H^\frac k2}^\g, 
\label{DV4}
\end{align}

\noi
where
$\g =\frac{3+\eps}{2k} +  \frac{2k- 3}{k}    < 2$.
Hence, the bound \eqref{DV0} in this case follows from substituting $u = u_1 + u_2$ into
the left-hand side of  \eqref{DV2}
and applying 
multilinearity of $p(u)$, 
\eqref{DV3},  \eqref{DV4}, and Young's inequality.

Next, we consider
the second term on the right-hand side of \eqref{Aeven}.
Recall from Definition~\ref{DEF:mono1}
that a monomial $p(u) \in \mathcal{P}_j(u)$
comes with the Hilbert transform $\H$
and the perturbation operator $\Qdl$ in \eqref{Qdl1}
over various factors in a {\it nested} manner.
In the following, we first apply H\"older's inequality iteratively 
to estimate $\int_\T p(u) dx$
and, in doing so, 
we can simply drop 
the Hilbert transform $\H$
and the perturbation operator $\Qdl$, 
 thanks to their $L^r$-boundedness
 for $ 1 < r < \infty$ (Lemma \ref{LEM:T1}\,(iii)).\footnote{As  pointed out above,  the $W^{s, \infty}$-norm in \eqref{DV5} 
 should really be  the $W^{s, r}$-norm with $r\gg 1$.
 For simplicity of the presentation,  however, we work with 
 the $W^{s, \infty}$-norm in the following.}
Thus, the contribution from 
the second term on the right-hand side of \eqref{Aeven}
is bounded by 
\begin{align}
\begin{split}
& 
C_\dl
 \sum_{j=3}^{2m+2}
 \sum_{\#\dx =0}^{2m+2-j} 
  \sum_{\substack{  \al_1, \dots ,  \al_j=  0 \\ \al_{1 \cdots j} = \#\dx}} ^{m-1}
 \|u\|_{H^{\al_1}} \|u\|_{H^{\al_2}} 
 \prod_{i=3}^j \|u\|_{W^{\al_i, \infty}}\\
  & \quad \le
  C_\dl 
\sum_{j=3}^{2m+2}
 \sum_{\substack{  \al_1, \dots ,  \al_j=  0 \\ \al_{1 \cdots j} = 2m+2- j}} ^{m-1}
 \|u\|_{H^{\al_1}} \|u\|_{H^{\al_2}} 
 \prod_{i=3}^j \|u\|_{W^{\al_i, \infty}}, 
\end{split}
\label{DV5}
\end{align}

\noi
where $C_\dl $ is independent of $ 2\le \dl \le \infty$
in view of \eqref{Q1}.
Since $\al_i \le m-1 \le \frac{k-1}{2} - 1$, we have 
\begin{align}
\|u\|_{W^{\al_i, \infty}}
\les \|u\|_{W^{\frac{k-1}{2} - \eps, \infty}}. 
\label{DV6}
\end{align}

\noi
On the other hand, by setting 
\begin{align}
\g  = \frac{2}{k}\bigg(
\al_1 +  \al_2 + \sum_{i = 3}^j \Big(\al_i + \frac 12 + \eps\Big)\bigg)
\le \frac{2}{k}\bigg(2m+1 - \frac 12 j + (j-2)\eps\bigg) < 2, 
\label{DV6a}
\end{align}

\noi
it follows from Sobolev's inequality and  \eqref{interp1} that 
\begin{align}
\begin{split}
 \|u\|_{H^{\al_1}} \|u\|_{H^{\al_2}} 
 \prod_{i=3}^j \|u\|_{W^{\al_i, \infty}}
& \les  \|u\|_{H^{\al_1}} \|u\|_{H^{\al_2}} 
 \prod_{i=3}^j \|u\|_{H^{\al_i+\frac 12 + \eps}}\\
& \les \|u \|_{L^2}^{j - \g}   \|u\|_{H^\frac k2}^{\g}.
\end{split}
\label{DV7}
\end{align}

\noi
Hence, the bound \eqref{DV0} in this case follows from substituting 
 $u = u_1 + u_2$ into
the left-hand side of  \eqref{DV5}
and applying \eqref{DV6},  \eqref{DV7}, and Young's inequality
(since $\g < 2$).

\medskip

\noi
$\bullet$ {\bf Case 2:} 
$\l=2m+1$ for $1\le m \le \frac{k-2}{2}$.\\
\indent
Proceeding as in \eqref{DV5}, 
we can bound the contribution 
from  \eqref{Aodd} by 
 \begin{align}
\begin{split}
 &   \sum_{j=3}^{2m+3}  \sum_{\#\dx =0}^{2m+3 -j} 
 \sum_{\substack{p(u) \in \mathcal{P}_j(u) \\ \wt{p}(u) = \prod_{i=1}^j \dx^{\al_i} u \\
 \#\Qdl(p) = 2m+3-j-\#\dx \\
0 \le  \al_1, \dots, \al_j \le m}}
 \bigg| \int_\T p(u)  dx \bigg| \\
 & \quad \le
 C_\dl \sum_{j=3}^{2m+3} 
  \sum_{\substack{ \al_1, \dots, \al_j = 0\\ \al_{1\cdots j}= 2m+3-j}}^m
  \| u\|_{H^{\al_1}} \|u\|_{H^{\al_2}} 
 \prod_{i =3}^j \|u\|_{W^{\al_i, \infty}}, 
\end{split}
\label{DV7a}
 \end{align}

\noi
where $C_\dl $ is independent of $ 2\le \dl \le \infty$.
Since $\al_i \le m \le \frac{k-2}{2} $, \eqref{DV6} still holds in this case.
By noting that 
\begin{align}
\g  = \frac{2}{k}\bigg(
\al_1 +   \al_2 + \sum_{i = 3}^j \Big(\al_i + \frac 12 + \eps\Big)\bigg)
\le \frac{2}{k}\bigg(2m+2 - \frac 12 j +  (j-2)\eps\bigg) < 2
\label{DV8}
\end{align}

\noi
in this case, 
the bound \eqref{DV7}  holds with $\g$  in \eqref{DV8}.
Hence, the bound \eqref{DV0} in this case also follows from substituting 
 $u = u_1 + u_2$ into
the left-hand side of  \eqref{DV7a}
and applying~\eqref{DV6},  \eqref{DV7}, and Young's inequality.
\end{proof}

We now present a proof of
Lemma \ref{LEM:var2} on the case $\l = k$, 

\begin{proof}[Proof of Lemma \ref{LEM:var2}]

We separately estimate the contribution
in  the cases $k = 2m $ and $k = 2m+1$.
In the following, all the estimates hold uniformly in $N \in \N$
and
 all the constants are understood to be uniform in $2\le \dl \le \infty$.
As explained in the proof of Lemma \ref{LEM:var1}, 
for simplicity of the presentation, 
we drop 
the Hilbert transform $\H$
and the perturbation operator $\Qdl$, 
appearing in a monomial $p(u) \in \Pc_j(u)$.
We recall that 
$X^\dl_N$ and $\Dr_N$ in \eqref{var4} have mean zero on $\T$.

\smallskip

\noi
$\bullet$ {\bf Case 1:} 
$k = 2m$ for some $m \in \N$.\\
\indent
Any monomial $p(u)$ in the second term on the right-hand side of \eqref{Aeven}
satisfies $|p(u)| \le m - 1$.
Thus, 
 we can estimate
the contribution from the second  term on the right-hand side of \eqref{Aeven}
deterministically as in Case~1 of the proof of Lemma \ref{LEM:var1};
see \eqref{DV5}-\eqref{DV7}.
In particular, note that \eqref{DV6} and \eqref{DV6a} hold even when $m = \frac k2$
(with $j \ge 3$).
Then, by taking an expectation and applying 
Proposition \ref{PROP:gauss1}\,(i), 
we obtain \eqref{DDV1}.

We now consider the first term on the right-hand side of 
 \eqref{Aeven}.
By dropping $\H$ and $\Qdl$
(as explained in Case 1 of the proof of Lemma \ref{LEM:var1}), 
it suffices to estimate  $p(u) = u \dx^{m-1} u \dx^m u$.
 Then, we have 
\begin{align}
\begin{split}
&  \E\bigg[\int_\T p(X^\dl_N + \Dr_N) dx \bigg]\\
& \quad  = 
\E\bigg[ \int_\T  X^\dl_N \dx^{m-1} X^\dl_N \dx^m X^\dl_N + (X^\dl_N+\Dr_N )\dx^{m-1} (X^\dl_N+\Dr_N) \dx^m \Dr_N \\
& \hphantom{XXXX}
+ (X^\dl_N+\Dr_N)\dx^{m-1}\Dr_N \dx^m X^\dl_N + \Dr_N \dx^{m-1} X^\dl_N \dx^m X^\dl_N dx\bigg].
\end{split}
\label{DDV11a}
\end{align}

Noting that $\E[g_{n_1}g_{n_2}g_{n_3}] = 0$
for any $n_1, n_2, n_3 \in \Z^*$, 
it follows from \eqref{Xdl1} that 
\begin{align*}
\E \bigg[ \int_\T X^\dl_N \dx^{m-1} X^\dl_N \dx^m X^\dl_N  dx \bigg] =0.
\end{align*}

\noi
As for the second term on the right-hand side of \eqref{DDV11a}, 
it follows from H\"older's inequality, 
Sobolev's inequality, 
and \eqref{interp1} that 
\begin{align}
\bigg|&  \int_\T (X^\dl_N+\Dr_N )\dx^{m-1} (X^\dl_N+\Dr_N) \dx^m \Dr_N dx \bigg|\notag \\
& \le  \|X^\dl_N+\Dr_N\|_{L^\infty} \|X^\dl_N+\Dr_N\|_{H^\frac{k - 2}{2}} \|\Dr_N\|_{H^\frac k2} \notag \\
& \les
\Big( \|X^\dl_N\|_{L^\infty} + \|\Dr_N\|_{H^\frac{1+\eps}{2}}\Big)
\Big(\|X^\dl_N\|_{H^\frac{k - 2}{2}} + \|\Dr_N\|_{H^\frac{k - 2}{2}}\Big)
\|\Dr_N\|_{H^\frac k2} \label{DDV13}\\
& \les
\Big( \|X^\dl_N\|_{L^\infty} + \|\Dr_N\|_{L^2}^{\frac{k-1-\eps}{k}}
\|\Dr_N\|_{H^\frac k 2}^\frac{1+\eps}{k}\Big)\notag \\
& \quad \times \Big(\|X^\dl_N\|_{H^\frac{k - 2}{2}} 
+ \|\Dr_N\|_{L^2}^{\frac{2}{k}}
\|\Dr_N\|_{H^\frac k2}^\frac{k-2}{k}\Big)
\|\Dr_N\|_{H^\frac k2}.
\notag 
\end{align}

\noi
By noting that $\frac{1+ \eps}{k} + \frac{k-2}k + 1 < 2$
and  applying Young's inequality and 
Proposition \ref{PROP:gauss1}\,(i), we obtain
\begin{align*}
  \E\big[\eqref{DDV13}\big]
\le C + \E\Big[C_{\eps_0} \| \Dr_N \|_{L^2}^\al +   \eps_0 \| \Dr_N\|_{H^\frac{k}2}^2\Big]
\end{align*}

\noi
for some $\al \gg 1$.
\noi
As for the third term on the right-hand side of \eqref{DDV11a}, 
integrating by parts, we have 
\begin{align}
\begin{split}
\int_\T
(X^\dl_N+\Dr_N)\dx^{m-1}\Dr_N \dx^m X^\dl_N dx 
& = - \int_\T \dx(X^\dl_N+\Dr_N) \dx^{m-1} \Dr_N \dx^{m-1} X^\dl_N dx\\
& \quad - \int_\T (X^\dl_N+\Dr_N) \dx^{m} \Dr_N  \dx^{m-1} X^\dl_N  dx,
\end{split}
\label{DDV15}
\end{align}

\noi
where the first term on the right-hand side 
satisfies $|p(u)| \le m - 1$
and thus can be 
handled 
in the same manner as  the second term on the right-hand side of \eqref{Aeven}
discussed above, 
while the second term on the right-hand side of \eqref{DDV15} was already treated in \eqref{DDV13}. 
Finally, as for the 
last term on the right-hand side of \eqref{DDV11a}, 
by Cauchy-Schwarz's inequality 
(recall that $\Dr_N$ has mean zero on $\T$)
and Cauchy's inequality, we have\footnote{Here, 
we crucially use the fact that the Hilbert transform (if any) does not act on multiple factors but 
acts only on each factor; see Lemma \ref{LEM:cub1}.}
\begin{align}
\begin{split}
\bigg| \int_\T \Dr_N \dx^{m-1} X^\dl_N \dx^m X^\dl_N  dx \bigg| 
& \le \| \Dr_N\|_{H^m}
\big\| \P_{\ne 0} (\dx^{m-1} X^\dl_N \dx^m X^\dl_N)\big\|_{H^{-m}}\\
& \le
\eps_0 \| \Dr_N\|_{H^\frac k2}^2
+ C_{\eps_0} \big\| \P_{\ne 0} ( \dx^{\frac k2 -1} X^\dl_N \dx^\frac k2 X^\dl_N)\big\|_{H^{-\frac k2}}^2.
\end{split}
\label{DDV16}
\end{align}

\noi
Then, 
from  Lemma \ref{LEM:var0a}
 (see also Remark \ref{REM:var0a}), 
we obtain
\begin{align}
  \E\big[\eqref{DDV16}\big]
& \leq \eps_0 \E\Big[\|\Dr_N\|^2_{H^{\frac{k}{2}}} \Big] + 
 C_{\eps_0} .
\label{DDV17}
\end{align}

Putting everything together, we conclude that 
\eqref{DDV1} holds in this case.

\medskip

\noi
$\bullet$ {\bf Case 2:} 
$k = 2m + 1$ for some $m \in \N$.\\
\indent
Any monomial $p(u)$ in the last term on the right-hand side of \eqref{Aodd}
satisfies $|p(u)| \le m - 1 =  \frac{k-3}{2}$, 
and thus we can estimate
the last term on the right-hand side of \eqref{Aodd}
deterministically as in Case~2 of the proof of Lemma \ref{LEM:var1}.
Then, by taking an expectation and applying Proposition \ref{PROP:gauss1}\,(i), 
we obtain \eqref{DDV1}.

We now consider the 
 first term on the right-hand side of  \eqref{Aodd}.
As before, it suffices to consider $p(u) = u  \dx^m u  \dx^m u$.
Then, we have   
\begin{align}
\begin{split}
 \E\bigg[\int_\T p(X^\dl_N + \Dr_N) dx \bigg]
& = \E \bigg[ \int_\T X^\dl_N (\dx^m X^\dl_N)^2 + \Dr_N (\dx^m X^\dl_N)^2 \\
& \hphantom{XXX}+ (X^\dl_N+\Dr_N) \dx^m \Dr_N  \dx^m (2X^\dl_N+\Dr_N) dx \bigg].
\end{split}
\label{DDV2}
\end{align}

\noi
Noting that $\E[g_{n_1}g_{n_2}g_{n_3}] = 0$
for any $n_1, n_2, n_3 \in \Z^*$, 
it follows from \eqref{Xdl1} that 
\begin{align*}
\E \bigg[ \int_\T X^\dl_N (\dx^m X^\dl_N)^2 dx \bigg] =0.
\end{align*}

\noi
From 
Cauchy-Schwarz's inequality 
(recall that $\Dr_N$ has mean zero on $\T$)
and Cauchy's inequality, we have
\begin{align*}
\bigg| \int_\T \Dr_N (\dx^m X^\dl_N)^2  dx \bigg|
& \le \|\Dr_N\|_{H^{m+\frac12}} 
\big\|\P_{\ne 0}((\dx^m X^\dl_N)^2)\big\|_{H^{-m-\frac12}} \\
& \leq \eps_0 \|\Dr_N\|^2_{H^{\frac{k}{2}}} 
+ C_{\eps_0} \big\|\P_{\ne 0}((\dx^{\frac{k-1}{2}} X^\dl_N)^2)\big\|_{H^{-\frac k2}}^2.
\end{align*}

\noi
Then, 
from Lemma \ref{LEM:var0a}
 (see also Remark \ref{REM:var0a}), 
we obtain
\begin{align*}
\bigg| \E\bigg[\int_\T \Dr_N (\dx^m X^\dl_N)^2  dx\bigg] \bigg|
& \leq \eps_0 \E\Big[\|\Dr_N\|^2_{H^{\frac{k}{2}}}\Big]  + 
 C_{\eps_0} .
\end{align*}

\noi
As for the last term on the right-hand side of \eqref{DDV2}, 
from  (fractional) integration by parts, Cauchy-Schwarz's inequality, the fractional Leibniz rule~\eqref{leib}, 
 \eqref{interp1}, and Sobolev's inequality, we have
\begin{align}
\begin{split}
\bigg| & \int_\T (X^\dl_N+\Dr_N) \dx^m \Dr_N \dx^m (2X^\dl_N+\Dr_N)  dx \bigg| \\
& \les \| (X^\dl_N+\Dr_N) \dx^\frac{k-1}{2}\Dr_N \|_{H^\frac 14} 
\Big(\| X^\dl_N \|_{H^{\frac{k-1}{2} - \frac 14}} + \|\Dr_N\|_{H^{\frac{k-1}{2} - \frac 14}}\Big) \\
& \les \Big(\| X^\dl_N\|_{W^{\frac 14, \infty}}
+ \| \Dr_N\|_{H^{\frac{3+\eps}4 }}\Big)
\|\Dr_N \|_{H^\frac{2k-1}{4}} \\
& \quad \times
\Big(\| X^\dl_N \|_{H^{\frac{k-1}{2} - \eps}} + 
\|\Dr_N\|_{L^2}^\frac{3}{2k} \|\Dr_N\|_{H^{\frac{k}{2}}}^\frac{2k-3}{2k} \Big) \\
& \les \Big(\| X^\dl_N\|_{W^{\frac 14, \infty}}
+ \|\Dr_N\|_{L^2}^\frac{2k - 3- \eps}{2k} \|\Dr_N\|_{H^{\frac{k}{2}}}^\frac{3+ \eps}{2k} \Big)\\
& \quad \times \|\Dr_N\|_{L^2}^\frac{1}{2k} \|\Dr_N\|_{H^{\frac{k}{2}}}^\frac{2k-1}{2k} 
\Big(\| X^\dl_N \|_{H^{\frac{k-1}{2} - \eps}} + 
\|\Dr_N\|_{L^2}^\frac{3}{2k} \|\Dr_N\|_{H^{\frac{k}{2}}}^\frac{2k-3}{2k} \Big).
\end{split}
\label{DDV5}
\end{align}

\noi
By noting that 
$\frac{3+ \eps}{2k} + 
\frac{2k-1}{2k}  + \frac{2k-3}{2k} < 2$
and  applying Young's inequality and 
Proposition \ref{PROP:gauss1}\,(i), 
 we obtain
\begin{align}
  \E\big[\eqref{DDV5}\big]
   \le C_{\eps_0} + \E\Big[C_{\eps_0} \| \Dr_N \|_{L^2}^\al +   \eps_0 \| \Dr_N\|_{H^\frac{k}2}^2\Big]
\label{DDV6}
\end{align}

\noi
for some $\al \gg 1$.

Next, we consider the 
second term on the right-hand side of  \eqref{Aodd}.
In this case, it suffices to consider 
 $\wt p(u) = \dx^\kk u  \dx^{m-1}u \dx^m u$
 with $\#\Qdl = 1-\kk$, $k = 0, 1$.
 Moreover, we may assume $m = \frac{k-1}{2}\ge 2$ since
 the case $m = 1$ reduces to 
 $p(u) = u  \dx^m u  \dx^m u$ (or even easier  $p(u) = u^2  \dx u$)
 which we handled above. 
 Proceeding as in~\eqref{DDV5} with 
 Lemma \ref{LEM:T1}\,(iii), we can bound the contribution in this case by 
 \begin{align}
\bigg| & \int_\T \dx^\kk (X^\dl_N+\Dr_N) \dx^{m-1} (X^\dl_N+\Dr_N) \dx^m 
(X^\dl_N+\Dr_N)  dx \bigg| \notag \\
& \les_\dl  \| \dx^\kk (X^\dl_N+\Dr_N) \dx^\frac{k-3}{2}(X^\dl_N+\Dr_N) \|_{H^\frac 14} \notag \\
& \quad \times
\Big(\| X^\dl_N \|_{H^{\frac{k-1}{2} - \frac 14}} + \|\Dr_N\|_{H^{\frac{k-1}{2} - \frac 14}}\Big) \notag \\
& \les_\dl  \Big(\| X^\dl_N\|_{W^{\frac 54, \infty}}
+ \| \Dr_N\|_{H^{\frac{7+\eps}4 }}\Big)
\|(X^\dl_N+\Dr_N) \|_{H^\frac{2k-5}{4}} 
\label{DDV7}
\\
& \quad \times
\Big(\| X^\dl_N \|_{H^{\frac{k-1}{2} - \eps}} + 
\|\Dr_N\|_{L^2}^\frac{3}{2k} \|\Dr_N\|_{H^{\frac{k}{2}}}^\frac{2k-3}{2k} \Big) \notag \\
& \les  \Big(\| X^\dl_N\|_{W^{\frac 54, \infty}}
+ \|\Dr_N\|_{L^2}^\frac{2k - 7- \eps}{2k} \|\Dr_N\|_{H^{\frac{k}{2}}}^\frac{7+ \eps}{2k} \Big)\notag \\
& \quad \times 
\Big(\| X^\dl_N \|_{H^{\frac{k-1}{2} - \eps}} + 
\|\Dr_N\|_{L^2}^\frac{5}{2k} \|\Dr_N\|_{H^{\frac{k}{2}}}^\frac{2k-5}{2k} \Big)\notag \\
& \quad \times
\Big(\| X^\dl_N \|_{H^{\frac{k-1}{2} - \eps}} + 
\|\Dr_N\|_{L^2}^\frac{3}{2k} \|\Dr_N\|_{H^{\frac{k}{2}}}^\frac{2k-3}{2k} \Big).\notag 
\end{align}

\noi
By noting that 
$\frac{7+ \eps}{2k} + 
\frac{2k-5}{2k}  + \frac{2k-3}{2k} < 2$
and  applying Young's inequality and 
Proposition \ref{PROP:gauss1}\,(i),  we obtain
\begin{align*}
  \E\big[\eqref{DDV7}\big]
\le C_{\dl, \eps_0} + 
\E\Big[C_{\dl, \eps_0}\| \Dr_N \|_{L^2}^\al +   \eps_0 \| \Dr_N\|_{H^\frac{k}2}^2\Big]
\end{align*}

\noi
for some $\al \gg 1$, 
where the constant $C_{\dl, \eps_0} > 0$ is  independent of $2\le \dl \le \infty$.

Lastly, we consider the 
third term on the right-hand side of  \eqref{Aodd}.
In this case, it suffices to consider 
 $p(u) =  u^2  \dx^{m-1}u \dx^m u$.
Then,  proceeding as in~\eqref{DDV5}, we have\footnote{Here, 
if there is the Hilbert transform $\H$ acting on 
multiple factors including the last factor
$\dx^m (2X^\dl_N+\Dr_N)$, 
then by the anti self-adjointness of $\H$, 
we  move the action of $\H$ onto the remaining factors
(which in particular does not include the last factor
$\dx^m (2X^\dl_N+\Dr_N)$ such that we can apply Cauchy-Schwarz's inequality
to separate the last factor as in the first step of \eqref{DDV9}.}
 \begin{align}
\bigg| & \int_\T  (X^\dl_N+\Dr_N)^2 \dx^{m-1} (X^\dl_N+\Dr_N) \dx^m (X^\dl_N+\Dr_N)  dx \bigg| 
\notag \\
& \les \|  (X^\dl_N+\Dr_N)^2 \dx^\frac{k-3}{2}(X^\dl_N+\Dr_N) \|_{H^\frac 14} 
\Big(\| X^\dl_N \|_{H^{\frac{k-1}{2} - \frac 14}} + \|\Dr_N\|_{H^{\frac{k-1}{2} - \frac 14}}\Big) \notag \\
& \les \Big(\| X^\dl_N\|_{W^{\frac 14, \infty}}^2
+ \| \Dr_N\|_{H^{\frac{3+\eps}4 }}^2\Big)
\|(X^\dl_N+\Dr_N) \|_{H^\frac{2k-5}{4}} \notag \\
& \quad \times
\Big(\| X^\dl_N \|_{H^{\frac{k-1}{2} - \eps}} + 
\|\Dr_N\|_{L^2}^\frac{3}{2k} \|\Dr_N\|_{H^{\frac{k}{2}}}^\frac{2k-3}{2k} \Big) 
\label{DDV9} \\
& \les \Big(\| X^\dl_N\|_{W^{\frac 14, \infty}}^2
+ \|\Dr_N\|_{L^2}^\frac{2k - 3- \eps}{k} \|\Dr_N\|_{H^{\frac{k}{2}}}^\frac{3+ \eps}{k} \Big)\notag \\
& \quad \times 
\Big(\| X^\dl_N \|_{H^{\frac{k-1}{2} - \eps}} + 
\|\Dr_N\|_{L^2}^\frac{5}{2k} \|\Dr_N\|_{H^{\frac{k}{2}}}^\frac{2k-5}{2k} \Big)\notag \\
& \quad \times 
\Big(\| X^\dl_N \|_{H^{\frac{k-1}{2} - \eps}} + 
\|\Dr_N\|_{L^2}^\frac{3}{2k} \|\Dr_N\|_{H^{\frac{k}{2}}}^\frac{2k-3}{2k} \Big).
\notag 
\end{align}

\noi
By noting that 
$\frac{3+ \eps}{k} + 
\frac{2k-5}{2k}  + \frac{2k-3}{2k} < 2$
and  applying Young's inequality and 
Proposition \ref{PROP:gauss1}\,(i),  we obtain
\begin{align*}
  \E\big[\eqref{DDV9}\big]
\le C_{\eps_0} + \E\Big[C_{\eps_0} \| \Dr_N \|_{L^2}^\al +   \eps_0 \| \Dr_N\|_{H^\frac{k}2}^2\Big]
\end{align*}

\noi
for some $\al \gg 1$.

This concludes the proof of Lemma \ref{LEM:var2}.
\end{proof}

\subsection{Construction of the deep-water \GGMs}
\label{SUBSEC:DGM3}

In this subsection, we present a proof of
Theorem~\ref{THM:3}\,(i).
We first state a lemma
on 
uniform (in $N$ and $\dl$)
convergence of the truncated interaction potential 
$R^\dl_{\frac{k}{2}} (\P_N X^\dl_{\frac{k}{2}})$, 
where $R^\dl_{\frac{k}{2}}(u)$ 
and $X^\dl_{\frac{k}{2}}$ are as in 
\eqref{R1}
 and \eqref{Xdl1}, respectively.

\begin{lemma}\label{LEM:R1}
Let $k \ge 2$ be an integer.
Given any $0 < \dl \le \infty$
and  finite $p \ge 1$,  
the sequence $\{R^\dl_{\frac{k}{2}} (\P_N X^\dl_{\frac{k}{2}})\}_{N\in\N}$
converges
to   $R^\dl_{\frac{k}{2}} (X^\dl_{\frac{k}{2}})$
 in $L^p(\Omega)$ as $N \to \infty$.
 Moreover, given any finite $p \ge 1$, 
 there exist    $C_{\dl} = C_{\dl}(p) > 0$  and $\ta > 0$ such that 
\begin{align*}
\sup_{N\in\N} \sup_{2\leq \dl \leq \infty} \|R^\dl_{\frac{k}{2}} (\P_N X^\dl_{\frac{k}{2}}) \|_{L^p(\Omega) } &
\le C_\dl < \infty, \\
\| R^\dl_{\frac{k}{2}} (\P_M X^\dl_{\frac{k}{2}}) - R^\dl_{\frac{k}{2}} (\P_N X^\dl_{\frac{k}{2}}) \|_{L^p(\Omega) } & \leq \frac{C_{\dl}}{N^\theta}
\end{align*}

\noi
for any $M\geq N \geq 1$, 
where the constant $C_{\dl}$ is independent of $2 \leq \dl \leq \infty$.
In particular, the rate of convergence is uniform in $2\leq \dl \leq \infty$.

\end{lemma}

Before proceeding to a proof of Lemma~\ref{LEM:R1}, 
we present a sketch of a proof of  Theorem~\ref{THM:3}\,(i).

\begin{proof}[Proof of Theorem~\ref{THM:3}\,(i)]
Given finite $p \ge 1$, 
the claimed $L^p(\O)$-convergence \eqref{conv1} of the truncated density
$F^\dl_{\frac{k}{2}}(\P_N X^\dl_{\frac{k}{2}})$, which is uniform in $ 2\le \dl \le \infty$, follows
from 

\smallskip

\begin{itemize}
\item[(i)] 
 the uniform (in $N\in \N$) $L^q(\O)$-bounds, for each $0 < \dl \le \infty$,  on 
the truncated density 
$F^\dl_{\frac{k}{2}}(\P_N X^\dl_{\frac{k}{2}})$ for some $q > p$
(Proposition \ref{PROP:V1}), 
which is also uniform 
in $2\leq \dl\le \infty$, 

\smallskip

\item[(ii)] 
the  convergence in probability, for each $0 < \dl \le \infty$,  of
$R^\dl_{\frac{k}{2}}(\P_N X^\dl_{\frac{k}{2}})$ (Lemma \ref{LEM:R1})
and 
$\| \P_N X^\dl_{\frac{k}{2}}\|_{L^2_x}$
(Proposition \ref{PROP:gauss1}\,(i)), 
which is  uniform 
in $2\leq \dl\le \infty$,

\smallskip

\item[(iii)] 
the uniform continuous mapping theorem
(Lemma 2.9 in \cite{LOZ})
which implies uniform (in $2\le \dl\le \infty$)
 convergence in probability of
the truncated density 
$F^\dl_{\frac{k}{2}}(\P_N u)$ for some $q > p$.

\end{itemize}

\noi
See Step 2 in the proof of  \cite[Proposition 3.6]{LOZ}
for further details.
See also 
\cite[Remark~3.8]{Tz08}
and the end of the proof of \cite[Proposition 1.2]{OTh}.

\end{proof}

\begin{remark}\label{REM:cutoff2}
\rm 

As mentioned in Remark \ref{REM:cutoff}, 
it is also possible to work with 
a sharp $L^2$-cutoff $\ind_{\{\|u \|_{L^2} \le K\}}$
in place of the smooth $L^2$-cutoff
$ \eta_K\big(\| u\|_{L^2}\big)$.
In this case, from the almost sure convergence of 
$\| \P_N X^\dl_{\frac{k}{2}}\|_{L^2}$
which is uniform in $2\le \dl\le \infty$
(Proposition \ref{PROP:gauss1}\,(i))
and the fact that 
\begin{align}
\PP\Big(\| X^\dl_{\frac{k}{2}}\|_{L^2} = K\Big) = 0
\label{RN3}
\end{align}

\noi
for any $0 < \dl \le \infty$ and $K > 0$,
we obtain  almost sure convergence of 
$\ind_{\{\|\P_N X^\dl_{\frac{k}{2}} \|_{L^2} \le K\}}$, 
(and hence convergence in $L^q(\O)$, $1 \le q < \infty$,  via the bounded convergence theorem) 
which is uniform in $2\le \dl\le \infty$,  
and the rest follows as above.
As for \eqref{RN3}, see Lemma 2.4 in~\cite{OST2}
(in a more complicated situation of the renormalized $L^2$-norm).

\end{remark}

Before presenting  a proof of 
Lemma \ref{LEM:R1}, let us introduce a definition.

\begin{definition}\label{DEF:pair}\rm
Given $\{ n_j\}_{j = 1}^\l\subset \Z^*$, 
we say that we have a 
 {\it pair}, if
there exist distinct
$j_1, j_2 \in \{1, \dots, \l\}$ such that 
$n_{j_1}  + n_{j_2} = 0$.
For example, 

\smallskip
\begin{itemize}
\item[(i)]
when $\l = 3$, there is no pair 
under $n_{123} = 0$, since $n_j \ne 0$, $j = 1, 2, 3$, 

\smallskip
\item[(ii)]
when $\l = 4$, we have either no pair or two pairs.

\end{itemize}

\end{definition}

We now present a proof of 
 Lemma~\ref{LEM:R1}.

\begin{proof}[Proof of Lemma~\ref{LEM:R1}]

Fix an integer $k \ge 2$ and finite $p \ge 1$.
Let $A_{\frac k 2, \frac\l2}$ be as in \eqref{Aeven}, \eqref{Aodd}, and \eqref{Aodd1}.
Then, in view of \eqref{R1}, 
it suffices to prove that there exist 
 $C_{\dl} = C_\dl(p)> 0$
and $\ta > 0$ 
 such that 
\begin{align*}
\|A_{\frac k 2, \frac\l2}(  X^\dl_M) - A_{\frac k 2, \frac\l2}( X^\dl_N) \|_{L^p(\Omega)} & \leq  \frac{C_{\dl}}{N^{\ta}}, 
\end{align*}

\noi
for any $M \ge N\ge 1$ and $1 \le \l \le k$, 
where the constant $C_{\dl}$ is independent of $2\leq \dl\leq\infty$
and $1 \le \l \le k$.
Here, we used the notation \eqref{var4}.
In the following,  all the implicit constants are  independent of $N,  M \in \N$ and $ 2 \le \dl \le \infty$.
In view of the Wiener chaos estimate (Lemma \ref{LEM:hyp}), 
we only consider the case $p = 2$.

\medskip

We first consider the case  $\l=1$. 
By 
Proposition \ref{PROP:gauss1}\,(i), we have 
\begin{align*}
\|A_{\frac k 2, \frac12}(X^\dl_M) - A_{\frac k 2 , \frac12}(X^\dl_N)\|_{L^2_\o} 
& \les  \|X^\dl_M-X^\dl_N\|_{L^{6}_\o L^3_x} 
\Big(\|X^\dl_M\|^2_{L^{6}_\o L^3_x}  + \|X^\dl_N\|^2_{L^{6}_\o L^3_x} \Big) \\
&  \les_\dl \frac{1}{N^{\frac{k-1}{2}-\eps}}.
\end{align*}

\smallskip

\noi
$\bullet$ {\bf Case 1:} 
$\l = 2m $ for some $m \in \N$ with $m \le \frac k2$.\\
\indent
Denote the $j$th term on the right-hand side of \eqref{Aeven} by $A^{(j)}_{\frac{k}{2}, \frac\l 2}$, 
$j = 1, 2$.
Then, we have 
\begin{align}
\begin{split}
A_{\frac{k}{2}, \frac\l 2}(X^\dl_M) - A_{\frac{k}{2}, \frac\l2}( X^\dl_N)  
& = 
\Big(A^{(1)}_{\frac{k}{2}, \frac\l 2}(X^\dl_M) - A^{(1)}_{\frac{k}{2}, \frac\l2}( X^\dl_N)\Big)\\
& \quad + \Big(A^{(2)}_{\frac{k}{2}, \frac\l 2}(X^\dl_M) - A^{(2)}_{\frac{k}{2}, \frac\l2}( X^\dl_N)\Big)\\
& =: \I + \II.
\end{split}
\label{RN5}
\end{align}

We first estimate $\I$.
Without loss of generality, 
we assume that 
 $ p(u) = u \dx^{m-1} u \dx^m u$. 
 Then, from \eqref{Xdl1}, we have 
 \begin{align}
\begin{split}
\|\I\|^2_{L^2(\Omega)}
&  = \frac 1{(2\pi)^2} \E \bigg[ \sum_{n_{123}=0} 
\Big(\ind_{\substack{0<|n_j| \leq M\\ j=1,2,3}} 
- \ind_{\substack{0<|n_j| \leq N\\ j=1,2,3}} \Big)
  \frac{g_{n_1} g_{n_2} g_{n_3}  (in_2)^{m-1} (in_3)^{m}}{ (T_{\dl, \frac k 2}(n_1) T_{\dl, \frac k2}(n_2) 
  T_{\dl, \frac k2}(n_3))^\frac12}  \\
& \hphantom{XXXX}\times \sum_{\wt n_{123}=0} 
\Big(\ind_{\substack{0<|\wt n_j| \leq M\\ j=1,2,3}} 
- \ind_{\substack{0<|\wt n_j| \leq N\\ j=1,2,3}} \Big)
  \frac{\cj{g_{\wt n_1} g_{\wt n_2} g_{\wt n_3}} (i \wt n_2)^{m-1}(i\wt  n_3)^{m}}
  { (T_{\dl, \frac k 2}(\wt n_1) T_{\dl, \frac k 2}(\wt n_2) T_{\dl, \frac k 2}(\wt n_3))^\frac12}\bigg].
\end{split}
\label{RN5a}
\end{align}

\noi
Note that, in the sums above,   there is no pair in $(n_1, n_2, n_3)$
and  in $(\wt n_1, \wt n_2, \wt n_3)$
in the sense of Definition~\ref{DEF:pair}.\footnote{In particular, 
$\I$ belongs to the homogeneous Wiener chaoses of order $3$.}
Thus, 
in order to compute the expectation above, 
we need to take all possible pairings between $(n_1, n_2, n_3)$
and $(\wt n_1, \wt n_2, \wt n_3)$.
By Jensen's inequality, however, 
we see that it suffices to consider the case
$n_j = \wt n_j$, $j = 1, 2, 3$.
See the discussion on $\<31p>$ in Section 4 of \cite{MWX}.
See also Section~10 in~\cite{Hairer}
and the proof of Lemma B.1 in \cite{OOT}.

As a result, 
by applying  \eqref{DX1}
and noting that $\max(|n_1|, |n_2|) \ges N$ in \eqref{RN5a} (which follows from 
$\max_{j = 1, 2, 3} |n_j| >  N$ and the triangle inequality), 
we have 
\begin{align}
\begin{split}
\| \I \|_{L^2(\O)}^2 
&\les_\dl \sum_{n_{123} = 0} 
\Big(\ind_{\substack{0<| n_j| \leq M\\ j=1,2,3}} 
- \ind_{\substack{0<| n_j| \leq N\\ j=1,2,3}} \Big)^2
 \frac{1}{\jb{n_1}^k\jb{n_2}^{k - 2m + 2}\jb{n_3}^{k - 2m}} \\
& \les \sum_{\substack{n_1, n_2 \in \Z\\\max(|n_1|, |n_2|) \ges N}}
 \frac{1}{\jb{n_1}^k\jb{n_2}^{2}}
   \les\frac{1}{N},
\end{split}
\label{RN6}
\end{align}

\noi
where, in the last step, we used the fact that 
 $k \ge 2$.

Next, we consider
 $\II$ in \eqref{RN5}.
 From the definition of 
 $A^{(2)}_{\frac{k}{2}, \frac\l 2}$, 
 we have $|p(u)|\le m - 1\le \frac k2 - 1$.
 Then, from 
 multilinearity of the monomials, 
 \eqref{DV5},  \eqref{DV6}, 
\begin{align}
\|u\|_{H^{\al_i}}
\le \|u\|_{H^{m-1}}
\les \|u\|_{H^{\frac{k-1}{2} - \eps - (\frac 12-\eps)}}
\label{RN6a}
\end{align}

\noi
for $i =1, 2$, 
 and Proposition \ref{PROP:gauss1}\,(i), 
 we have 
  \begin{align*}
 \| \II \|_{L^2(\O)}
&  \les  N^{-\frac 12+\eps}
\Big(1 + \|X^\dl_M \|_{L^{4m+4}_\o  W^{\frac{k-1}{2} - \eps, \infty}}^{2m+1}
+ \|X^\dl_N \|_{L^{4m+4}_\o W^{\frac{k-1}{2} - \eps, \infty}}^{2m+1} \Big)\\
& \hphantom{Xl}\times
\|X^\dl_M - X^\dl_N\|_{L^{4m+4}_\o W^{\frac{k-1}{2} - \eps, \infty}}\\
& \les N^{-\frac 12 + \eps}.
\end{align*}

\smallskip

\noi
$\bullet$ {\bf Case 2:} 
$\l = 2m + 1 \ge 3$ for some $m \in \N$ with $m \le \frac{k-1}2$.\\
\indent
In this case, we have $k \ge 3$.
In view of \eqref{Aodd}, we write 
\begin{align*}
A^\dl_{\frac{k}{2}, \frac{2m+1}{2}}(X^\dl_M) - A^\dl_{\frac{k}{2}, \frac{2m+1}{2}}( X^\dl_N)  
& =: \III + \IV + \V + \VI, 
\end{align*}

\noi
where
$\III$, $\IV$,  $\V$, and $\VI$ correspond to the contributions
from the first, second,  third, and fourth terms on the right-hand side of \eqref{Aodd}.

We first estimate $\III$ and $\IV$ when $k \ge 3$.
Proceeding as in \eqref{RN5a} and \eqref{RN6}
with Lemma~\ref{LEM:SUM}, we have 
\begin{align*}
\| \III \|_{L^2(\O)}^2 
&\les_\dl \sum_{n_{123} = 0} 
\Big(\ind_{\substack{0<| n_j| \leq M\\ j=1,2,3}} 
- \ind_{\substack{0<| n_j| \leq N\\ j=1,2,3}} \Big)^2
 \frac{1}{\jb{n_1}^k\jb{n_2}^{k - 2m}\jb{n_3}^{k - 2m}} \\
& \les \sum_{\substack{n_{123} = 0\\\max(|n_1|, |n_2|) \ges N}}
 \frac{1}{\jb{n_1}^k\jb{n_2}\jb{n_3}}
   \les\frac{1}{N^{1-\eps}}.
\end{align*}

\noi
Similarly, with~\eqref{Q1}
and $k \ge 3$, we have
\begin{align*}
\| \IV \|_{L^2(\O)}^2 
&\les_\dl 
\sum_{\kk = 0}^1\frac 1{\dl^{1-\kk}}
\sum_{\substack{n_{123} = 0\\\max(|n_1|, |n_2|) \ges N}}
 \frac{1}{\jb{n_1}^{k-2\kk} \jb{n_2}^{k-2m + 2}\jb{n_3}^{k-2m}}\\
& \les_\dl \frac{1}{N^{1-\eps}}, 
\end{align*}

\noi
where we used the fact that $k - 2\kk \ge 1$.
In handling $\III$ and $\IV$, 
we used the fact that there is no pair in $(n_1, n_2, n_3)$.

Let us now consider  $\V$ coming form the third term on the right-hand side of \eqref{Aodd}.
Without loss of generality, 
we assume that 
 $ p(u) = u^2 \dx^{m-1} u \dx^m u$. 
 Then, from \eqref{Xdl1}, we have 
 \begin{align}
\V
  = \frac 1{(2\pi)^\frac{3}{2}}  \sum_{n_{1234 }=0} 
\Big(\ind_{\substack{0<|n_j| \leq M\\ j=1,\dots, 4}} 
- \ind_{\substack{0<|n_j| \leq N\\ j=1,\dots, 4}} \Big)
  \frac{\prod_{j = 1}^4 g_{n_j}  (in_3)^{m-1} (in_4)^{m}}{ \prod_{j = 1}^4 (T_{\dl, \frac k 2}(n_j))^\frac 12}.
\label{RN9a}
\end{align}

\noi
Since it is quartic, 
we have either no pair or two pairs
in the sense of Definition \ref{DEF:pair}.
Write $\V = \V_1 + \V_2$, 
where $\V_1$ denotes the contribution from frequencies with no pair
in 
\eqref{RN9a}, 
while 
$\V_2$ denotes the contribution from frequencies with two pairs in 
\eqref{RN9a}.
Then, proceeding with Jensen's inequality  as in \eqref{RN5}-\eqref{RN5a}
with \eqref{DX1}, 
we have 
\begin{align*}
\| \V_1 \|_{L^2(\O)}^2 
&\les_\dl 
\sum_{\substack{n_{1234} = 0\\\max(|n_1|, |n_2|, |n_3|) \ges N}}
 \frac{1}{\jb{n_1}^{k} \jb{n_2}^k\jb{n_3}^3\jb{n_4}}
\les \frac{1}{N^{2}}.
\end{align*}

\noi
On the other hand, when there are  two pairs in \eqref{RN9a}, 
by Minkowski's inequality, we have 
\begin{align*}
\| \V_2 \|_{L^2(\O)}
&\les_\dl 
\sum_{\substack{n_1, n_2 \in \Z^*\\ \max(|n_1|, |n_2|) \ges N}}
 \frac{1}{\jb{n_1}^{2} \jb{n_2}^2}
\les \frac{1}{N}.
\end{align*}

Lastly, we bound $\VI$
corresponding to the last term on the right-hand side of \eqref{Aodd}, 
where
 we have $|p(u)|\le m - 1\le \frac {k-1}2 - 1$.
 Then, 
 proceeding as in Case 2 of the proof of Lemma \ref{LEM:var1}
 with 
 \eqref{RN6a} 
 and Proposition \ref{PROP:gauss1}\,(i), 
 we have 
  \begin{align*}
& \| \VI \|_{L^2(\O)}\\
& \quad \le  N^{-1 + \eps}
\Big(\|X^\dl_M \|_{L^{4m+6}_\o W^{\frac{k-1}{2} - \eps, \infty}}^{2m+2} 
+ \|X^\dl_N \|_{L^{4m+6}_\o W^{\frac{k-1}{2} - \eps, \infty}}^{2m+2} \Big)\\
& \hphantom{XXl}\times
\|X^\dl_M - X^\dl_N\|_{L^{4m+6}_\o W^{\frac{k-1}{2} - \eps, \infty}}\\ 
& \quad \les N^{-1 + \eps}. 
\end{align*}

This concludes
the proof of Lemma \ref{LEM:R1}.
\end{proof}

\subsection{Deep-water convergence of the \GGMs}
\label{SUBSEC:DGM4}

In this subsection, we present a   proof of Theorem~\ref{THM:3}\,(ii).
We first state  a preliminary result 
on  the $L^p(\O)$-convergence properties of the truncated density
$F^\dl_{\frac{k}{2}} (\P_N u)$ defined  in \eqref{rhox1}.

\begin{lemma}\label{LEM:FN1}
Let $k \ge 2$ be an integer and $K > 0$.
Then, given any finite $p \ge 1$ and $N\in\N$, we have 
\begin{align}
\lim_{\dl\to\infty} \|F^\dl_{\frac{k}{2}} (\P_N X^\dl_{\frac{k}{2}}) - F^\infty_{\frac{k}{2}} (\P_N X^\infty_{\frac k 2}) \|_{L^p(\O)} & = 0, \label{FN1} \\
\lim_{\dl\to\infty} \|F^\dl_{\frac{k}{2}} (\P_N X^\infty_{\frac k 2}) - 
F^\infty_{\frac{k}{2}} (\P_N X^\infty_{\frac k 2}) \|_{L^p(\O)} & = 0. \label{FN2}
\end{align}

\end{lemma}

In order to prove  Lemma \ref{LEM:FN1}, 
we need the following auxiliary lemma
on 
an $N$-dependent 
(but uniform in $2\le \dl\le \infty$)
deterministic bound
on the truncated density $F^\dl_{\frac{k}{2}}(\P_N u) $.

\begin{lemma}\label{LEM:FN2}
Let $k \ge 2$ be an integer and $K > 0$.
Then, given any $N\in\N$, 
there exists a finite constant $C_{k, K, N}>0$ such that 
\begin{align}\label{unif1}
\sup_{2\leq \dl \leq \infty}
\big|F^\dl_{\frac{k}{2}}(\P_N u) \big| \leq C_{k, K,  N} < \infty
\end{align}

\noi
for any $u \in L^2(\T)$.
\end{lemma}

\begin{proof}
From  \eqref{rhox1}, 
we have 
\begin{align*}
F^\dl_{\frac{k}{2}}(\P_N u)
=  \eta_K\big(\| \P_N u\|_{L^2}\big)
\exp\Big(-R^\dl_{\frac k 2}(\P_N u)\Big)  , 
\end{align*}

\noi
where $R^\dl_{\frac k 2}(u)$ is as in \eqref{R1}.
Recall from \eqref{R1}, \eqref{Aeven}, \eqref{Aodd},  and \eqref{Aodd1}
that $R^\dl_{\frac k 2}(u)$
is a polynomial in $u$ with derivatives, non-positive powers of $\dl$, 
$\Qdl$, and the Hilbert transform.
Then, 
the bound \eqref{unif1} follows from 
Bernstein's inequality  ($\|\P_N u\|_{W^{s, \infty}}\les N^{s + \frac 12} \|\P_N u\|_{L^2}$
for any $s \ge 0$), 
Lemma \ref{LEM:T1}, 
and the fact that $\|\P_N u\|_{L^2}\le 2K$.
\end{proof}

We now present a proof of Lemma \ref{LEM:FN1}.

\begin{proof}[Proof of Lemma \ref{LEM:FN1}]

We first prove \eqref{FN2}.
From (the proof of) \eqref{DE4} in Proposition~\ref{PROP:cons1}, 
we have 
\begin{align}
\lim_{\dl \to \infty} R^\dl_\frac k2(u) = 
R^\infty_\frac k2(u)
\label{FN3}
\end{align}

\noi
for any 
$u \in H^\frac k2 (\T)$, 
where
$R^\infty_{\frac{k}{2}}(u)
= 
R^\BO_{\frac{k}{2}}(u)$ is as in \eqref{E2}. 
Then, 
from \eqref{rhox1} and \eqref{FN3}, 
we obtain 
\begin{align}
\lim_{\dl \to \infty} F^\dl_{\frac{k}{2}} (\P_N X^\infty_{\frac k 2}(\o)) 
= 
F^\infty_{\frac{k}{2}} (\P_N X^\infty_{\frac k 2}(\o)) 
\label{FN4}
\end{align}

\noi
for each $\o \in \O$.
Therefore, \eqref{FN2} follows from \eqref{FN4}, 
Lemma \ref{LEM:FN2}, and the bounded convergence theorem.


Next, we prove \eqref{FN1}.
For each fixed  $N \in \N$, 
it follows from \eqref{Xdl1} and \eqref{DX1}
that 
$\P_N X^\dl_{\frac{k}{2}}$
converges to $\P_N X^\infty_{\frac k 2}$
in $H^\frac k2(\T)$ for any $\o \in \O$.
Thus, we have 
\begin{align}
\lim_{\dl \to \infty}  \eta_K\big(\| \P_N X^\dl_{\frac{k}{2}}(\o)\|_{L^2}\big)
=  \eta_K\big(\| \P_N X^\infty_{\frac{k}{2}}(\o)\|_{L^2}\big)
\label{FN5}
\end{align}

\noi
for any $\o \in \O$.
We claim that 
\begin{align}
\begin{split}
& \lim_{\dl \to \infty} 
\ind_{\{\|\P_N X^\dl_{\frac{k}{2}}(\o)\|_{L^2}\le 2K\}}
 R^\dl_\frac k2( \P_N X^\dl_{\frac{k}{2}}(\o))\\
& \quad =  
\ind_{\{\|\P_N X^\infty_{\frac{k}{2}}(\o)\|_{L^2}\le 2K\}}
R^\infty_\frac k2( \P_N X^\infty_{\frac{k}{2}}(\o)).
\end{split}
\label{FN6}
\end{align}

\noi
For now, let us assume \eqref{FN6} and prove \eqref{FN1}.
From \eqref{rhox1}, \eqref{FN5},  and \eqref{FN6}, 
we have
\begin{align}
\lim_{\dl \to \infty}  F^\dl_\frac k2( \P_N X^\dl_{\frac{k}{2}}(\o))
=  F^\infty_\frac k2( \P_N X^\infty_{\frac{k}{2}}(\o))
\label{FN6a}
\end{align}

\noi
for any  $\o \in \O$.
Then, \eqref{FN1} follows from \eqref{FN6a}, 
Lemma \ref{LEM:FN2}, and the bounded convergence theorem.

It remains to prove \eqref{FN6}.
Recall the definition \eqref{R1} of $R^\dl_\frac k2(u)$.
As for the second term on the right-hand side of \eqref{R1}, 
arguing as in the proof of Lemma \ref{LEM:FN2}, we have 
\begin{align}
\sup_{u \in L^2}
\bigg|\ind_{\{\|\P_N u\|_{L^2}\le 2K\}}\sum_{\l=1}^{k-1} \frac{1}{\dl^\l} A_{\frac k 2,\frac{k-\l}{2}}(\P_N u)\bigg|
\les \frac {C_N}\dl \too 0, 
\label{FN7}
\end{align}

\noi
as $\dl \to \infty$.
Let 
$D^\dl_\frac k2(u)$
be the collection of monomials $p(u)$ in $A^\dl_{\frac k2 , \frac{k}{2}}(u)$
such that $\#\Qdl(p) \ge 1$.
In view of 
\eqref{Q1}, 
we have 
\begin{align}
\sup_{u \in L^2}
\bigg|\ind_{\{\||\P_N  u\|_{L^2}\le 2K\}}
D^\dl_\frac k2(\P_Nu)\bigg|
\les \frac {C_N}\dl \too 0, 
\label{FN8}
\end{align}

\noi
as $\dl \to \infty$.

Let  $G^\dl_\frac k2 (u) = A^\dl_{\frac k2, \frac k2}(u) - D^\dl_\frac k2(u)$.
Then, by noting that $G^\dl_\frac k2$ is $\dl$-free
in the sense of the proof of Proposition \ref{PROP:cons1}, 
we see that $G^\dl_\frac k2(u)$ is independent of $\dl$
and that  $G^\dl_\frac k2(u) = R^\infty_\frac k2(u)$.
See also~\eqref{DE7} and 
\eqref{DE9}.
Hence, from the $H^\frac k2$-convergence of 
$\P_N X^\dl_{\frac{k}{2}}(\o)$
to $\P_N X^\infty_{\frac k 2}(\o)$, 
we obtain
\begin{align}
\lim_{\dl \to \infty}
G^\dl_\frac k2(\P_N X^\dl_{\frac{k}{2}}(\o))
 = \lim_{\dl \to \infty}
R^\infty_\frac k2(\P_N X^\dl_{\frac{k}{2}}(\o))
= R^\infty_\frac k2(\P_N X^\infty_{\frac k 2}(\o))
\label{FN9}
\end{align}

\noi
for any $\o \in \O$.
Hence, the claim \eqref{FN6} follows from  \eqref{FN7},  \eqref{FN8},  and  \eqref{FN9} with \eqref{R1}.
\end{proof}

We conclude this section by presenting  a  proof of Theorem~\ref{THM:3}\,(ii).
\begin{proof}[Proof of Theorem~\ref{THM:3}\,(ii)]

The claimed equivalence
of  $\rho^\dl_{\frac{k}{2}}$ and
$\rho^\infty_{\frac{k}{2}}$  
follows from 
(a) 
the equivalence
of  $\rho^\dl_{\frac{k}{2}}$ and 
 the Gaussian measure with the $L^2$-cutoff
$\eta_K\big(\| u\|_{L^2}\big) d \mu^\dl_\frac k2(u)$
for any $0 < \dl \le \infty$, 
and 
(b) Proposition \ref{PROP:gauss1}\,(iii).

In the following,
by closely following the argument in 
Subsection 3.3 of  \cite{LOZ}, 
we prove
convergence in total variation of
 $\rho^\dl_{\frac{k}{2}}$  to $\rho^\infty_{\frac{k}{2}}$. 
Given any  $N\in\N$, from the triangle inequality, we have 
\begin{align*}
\dtv \big(\rho^\dl_\frac k2 ,\rho^\infty_{\frac k2}\big) 
& \leq \dtv\big(\rho^\dl_{\frac k2}, \rho^\dl_{\frac k2, N}\big) 
+ \dtv\big(\rho^\dl_{\frac k2, N}, \rho^\infty_{\frac k2, N}\big) 
+ \dtv \big(\rho^\infty_{\frac k2, N}, \rho^\infty_{\frac k2}\big)\\
& \leq 2 \sup_{2\leq \dl \leq \infty} 
\dtv\big(\rho^\dl_{\frac k2}, \rho^\dl_{\frac k2, N}\big) 
+ \dtv \big(\rho^\dl_{\frac k2, N}, \rho^\infty_{\frac k2, N}\big)
\end{align*}

\noi
for any $2 \le \dl < \infty$.
In view of the uniform (in $2\le \dl \le \infty$)
convergence in total variation of $\rho^\dl_{\frac k2, N}$ to $\rho^\dl_{\frac k2}$
(Theorem~\ref{THM:3}(i)), 
by taking 
 a limit as $\dl\to\infty$ and then  a limit as $N\to\infty$, we obtain
\begin{align*}
\lim_{\dl\to\infty} 
\dtv \big(\rho^\dl_\frac k2 ,\rho^\infty_{\frac k2}\big) 
 \leq \lim_{N\to\infty} \lim_{\dl\to\infty} 
 \dtv\big(\rho^\dl_{\frac k2, N}, \rho^\infty_{\frac k2, N}\big).
\end{align*}

\noi
Hence, it suffices to prove 
\begin{align}
\lim_{\dl\to\infty} 
\dtv \big(\rho^\dl_{\frac k2, N}, \rho^\infty_{\frac k2, N}\big)
= 0
\label{DT2}
\end{align}

\noi
for some $N \in \N$ (but we will prove \eqref{DT2} for any $N \in \N$).

Recall from \eqref{FN1} in Lemma \ref{LEM:FN1} that 
the partition function 
$Z_{\dl,\frac k2,  N} = \| F^\dl_\frac k2(\P_Nu)\|_{L^1(d\mu^\dl_\frac k2)}$ 
of the truncated ILW \GGM~ $\rho^\dl_{\frac k2, N}$
converges
to the partition function 
$Z_{\infty,\frac k2,  N} = \| F^\infty_\frac k2(\P_Nu)\|_{L^1(d\mu^\infty_\frac k2)}$ 
of the truncated BO \GGM~ $\rho^\infty_{\frac k2, N}$.
Then, 
proceeding as in \cite[(3.57)]{LOZ}, 
with 
 \eqref{KL0}, we have
\begin{align}
\begin{split}
&  \lim_{\dl \to \infty} \dtv \big(\rho^\dl_{\frac k2, N}, \rho^\infty_{\frac k2, N}\big)
 =  \lim_{\dl \to \infty}\sup_{A\in\mathcal{B}} \big|\rho^\dl_{\frac k2, N}(A) - \rho^\infty_{\frac k2, N}(A)\big| \\
& \quad \le Z_{\infty,\frac k2,  N} ^{-1}
 \lim_{\dl \to \infty}
 \int_{H^{\frac{k-1}{2}-\e} } 
| F^\dl_\frac k2(\P_Nu)  -  F^\infty_\frac k2 (\P_N u) |\, d \mu^\infty_\frac k2(u) \\
 & \quad \quad  +  Z_{\infty,\frac k2,  N} ^{-1}
 \lim_{\dl \to \infty}
  \int_{H^{\frac{k-1}{2}-\e}} 
F^\dl_\frac k2(\P_N u)\bigg|  \frac {d\mu^\dl_\frac k2}{d \mu^\infty_\frac k2} (u) - 1\bigg|\, d \mu^\infty_\frac k2(u)\\
& \quad =:  \lim_{\dl \to \infty} \I_{\dl, N} +  \lim_{\dl \to \infty}\II_{\dl,N},   
\end{split}
\label{DT3}
\end{align}

\noi
where  $\mathcal{B}$ denote the Borel sets in $H^{\frac{k-1}{2}-\eps}(\T)$.
From \eqref{FN2} in Lemma \ref{LEM:FN1}, 
we obtain 
\begin{align}
\lim_{\dl \to \infty} \I_{\dl, N} = 0
\label{DT4}
\end{align}

\noi
for any $N \in \N$.
On the other hand, from Lemma \ref{LEM:FN2}, 
 Scheff\'e's theorem
(Lemma 2.1 in \cite{KS}; see also Proposition 1.2.7 in \cite{Tsy}), 
and Proposition~\ref{PROP:gauss1}\,(iv), we obtain
\begin{align}
\begin{split}
\lim_{\dl\to\infty}\II_{\dl, N}
& \leq  C_{k,K,N} 
 \lim_{\dl \to \infty}  \int_{H^{\frac{k-1}{2}-\e}} 
\bigg|  \frac {d\mu^\dl_\frac k2}{d \mu^\infty_\frac k2} (u) - 1\bigg|\, d \mu^\infty_\frac k2(u)\\
& = 2 C_{k,K,N} 
\lim_{\dl\to\infty} \dtv\big(\mu^\dl_\frac k2,  \mu^\infty_\frac k2\big)
= 0
\end{split}
\label{DT5}
\end{align}

\noi
for any $N \in \N$.
Therefore, \eqref{DT2} follows from \eqref{DT3}, \eqref{DT4}, 
and \eqref{DT5}.
\end{proof}

\section{Shallow-water \GGMs}
\label{SEC:SGM1}

In this section, 
we go over the construction of the \GGMs~ in the shallow-water regime
and study their convergence properties in the shallow-water limit
(Theorem~\ref{THM:4}\,(i) and (ii)).
As in the deep-water regime studied in the previous section, 
there are two main tasks:

 \smallskip
 
 \begin{itemize}
 \item[(i)]
 We need to establish
 an
$L^p$-integrability of 
the truncated density $\wt F^\dl_\frac k2 (\P_N v)$, 
which is 
uniform in 
 {\it both} $N \in \N$ and $0 < \dl \ll 1$. 
As in Section \ref{SEC:DGM1}, 
we employ a variational approach.
Thanks to our detailed analysis 
in Section \ref{SEC:cons2}
on the structure
of the shallow-water conservation laws, 
this part of analysis via the variational formula
follows closely to that in the deep-water regime.
See Subsection \ref{SUBSEC:SGM2}.

\smallskip

\item[(ii)] 
We also need to establish weak convergence
of 
the base Gaussian measures
$\wt \mu^\dl_\frac k2$ in \eqref{gauss3}
(see also \eqref{gauss4}).
See Proposition \ref{PROP:gauss2}.

 \end{itemize}

 \noi
 Unlike the deep-water case
 (Proposition \ref{PROP:gauss1}), 
 the convergence of the base Gaussian measures
 holds only weakly.
In Proposition \ref{PROP:gauss2}, 
we also establish
singularity of 
the base Gaussian measures
$\wt \mu^\dl_\frac k2$ 
and the limiting 
base Gaussian measure
$\wt \mu^0_{\lceil \frac{k}{2}\rceil}$, 
which also exhibits a sharp contrast 
to the deep-water case.
Furthermore, 
we establish 
an interesting dichotomy
on singularity\,/\,equivalence
for   the odd-order and 
 even-order Gaussian measures 
with different depth parameters $\dl$;
see Proposition \ref{PROP:gauss2}.

In Subsection \ref{SUBSEC:SGM3}, 
we present a proof of Theorem \ref{THM:4}\,(i)
on the construction of the shallow-water \GGM~ for fixed $0 \le \dl <  \infty$, 
while we  present 
 a proof of Theorem~\ref{THM:4}\,(ii) and (iii)
in Subsection \ref{SUBSEC:SGM4}.

\subsection{Singularity and equivalence  of the base Gaussian measures}
\label{SUBSEC:SGM1}

Given
$0\le  \dl < \infty$ and $k \in \N$, 
let $\wt{\mu}^\dl_{\frac k 2}$ be 
 the Gaussian measure defined   in \eqref{gauss3}
and \eqref{gauss4}. 
 Recall that 
 $\wt \mu^\dl_{\frac{k}{2}}$ is the induced probability measure 
 under the map $\o \in \O \mapsto \wt X^\dl_{\frac{k}{2}}(\o)$, 
 where $\wt X^\dl_{\frac{k}{2}}$ is as in~\eqref{Xdl2}
 and \eqref{Xdl3}.

The following proposition 
generalizes
Proposition 4.1 in \cite{LOZ}
(see also Proposition 4.2 in~\cite{LOZ}), 
where 
the $k = 1$ case was treated.
In this proposition, 
we in particular establish a 2-to-1 collapse
of the Gaussian measure $\wt \mu^\dl_\frac k2$
in the shallow-water limit.
Furthermore, we point out that
the Gaussian measure
$\wt \mu^\dl_\frac k2$ and its shallow-water limit
$\wt \mu^0_\frac k2$
are singular
and that 
 convergence 
of~$\wt \mu^\dl_\frac k2$
holds only weakly, 
which presents a stark contrast to the deep-water case
(Proposition \ref{PROP:gauss1}).

\begin{proposition}\label{PROP:gauss2}
Let $k\in\N$.  Then, the following statements hold.

\smallskip

\noi{\rm(i)}
Given any $0 \le  \dl <  \infty$, 
  finite $p \ge 1$,  and $1 \le r \le \infty$, 
the sequence $\{\P_N \wt X^\dl_{\frac k 2}  \}_{N\in\N}$ 
converges
to   $\wt  X^\dl_{\frac{k}{2}}$ in $L^p(\Omega; W^{\frac{k-1}{2}-\eps, r}(\T))$
as $N \to \infty$.
 Moreover, given any finite $p \ge 1$  
 and $1 \le r \le \infty$, 
 there exist  $C_\dl = C_\dl( p, r) > 0$ and $\ta > 0$ such that 
\begin{align}
\sup_{N\in\N} \sup_{0<  \dl \leq 1} 
\Big\|  \| \P_N \wt X^\dl_{\frac{k}{2}} \|_{W^{\frac{k-1}{2}-\eps, r}_x} \Big\|_{L^p(\Omega)} &< C_\dl  <\infty, 
 \label{SX2}\\
\sup_{0<  \dl \leq 1} \Big\| \| \P_M \wt  X^\dl_{\frac{k}{2}} - \P_N \wt X^\dl_{\frac{k}{2}} 
\|_{W^{\frac{k-1}{2}-\eps, r}_x} \Big\|_{L^p(\Omega) } & \leq \frac{C_\dl}{N^\ta}
 \label{SX3}
\end{align}

\noi
for any $M\geq N\ge 1$,
where the constant $C_{\dl}$ is independent of $0 <  \dl \leq 1$.
 In particular, the rate of convergence is uniform in $0<  \dl \leq 1$.
When $k$ is even, \eqref{SX2} and \eqref{SX3} hold
including $\dl = 0$.

\smallskip

\noi{\rm(ii)}
Given any finite $p \ge 1$ and $ 1\le r \le \infty$,  
 $\wt X^\dl_{\frac{k}{2}}$ converges to 
$\wt X^0_{\lceil \frac k 2 \rceil }$ in $L^p(\Omega;W^{\frac{k-1}{2}-\eps, r}(\T))$
 as $\dl\to0$, 
where  $\lceil x \rceil$ denotes
the smallest integer greater than or equal to $x$.
  In particular, $\wt \mu^\dl_{\frac{k}{2}}$ converges weakly to $\wt \mu^0_{\lceil \frac{k}{2}\rceil}$ as $\dl\to \infty$.

\smallskip

\noi{\rm(iii)}
Let   $0<\dl<\infty$.
Then,  the Gaussian measures $\wt \mu^{\dl}_{\frac k2}$ and
$\wt{\mu}^0_{\lceil \frac k2 \rceil}$
are singular.

\smallskip

\noi{\rm(iv)}
Let  $\kk \in \N$ and  $0<\dl_1 <  \dl_2<\infty$.
Then,  the odd-order Gaussian measures $\wt \mu^{\dl_1}_{\kk-\frac12}$ and
$\wt{\mu}^{\dl_2}_{\kk-\frac12}$ are  singular, 
while the even-order Gaussian measures $\wt \mu^{\dl_1}_{\kk}$ and
$\wt{\mu}^{\dl_2}_{\kk}$ are equivalent.

\end{proposition}

Before proceeding to a proof of Proposition~\ref{PROP:gauss2}, we 
state a preliminary lemma on 
the multiplier
$\wt T_{\dl, \frac k2}$ defined in \eqref{T2}.

\begin{lemma}\label{LEM:Tdl}
Let $\kk\in\N$ and $0<\dl<\infty$.
 Then, we have 
\begin{align}\label{T4}
\lim_{\dl\to0} \wt{T}_{\dl,\kk-\frac12}(n) = \lim_{\dl\to0} \wt{T}_{\dl, \kk}(n)
= n^{2\kk} = \wt{T}_{0,\kk}(n) 
\end{align}

\noi
for any $n \in \Z^*$.
More precisely, given any small $\eps_0 > 0$, 
there exists $C> 0$ such that 
\begin{align}
\big|\wt{T}_{\dl,\kk-\frac12}(n) - n^{2\kk}\big|
&    \le n^{2\kk} |\hf(\dl, n)|
+ C \dl^{\eps_0} |n|^{2\kk+\eps_0},
\label{T4a}\\
\big| \wt{T}_{\dl, \kk}(n)- n^{2\kk}\big|
&  \les
 \dl^{\eps_0} |n|^{2\kk+\eps_0}
\label{T4b}
\end{align}

\noi
for any $0 < \dl \le \infty$ and $n\in\Z^*$, 
where $\hf(\dl, n)$ is as in Lemma \ref{LEM:L1}\,(iv).
Moreover, given $\kk \in \N$ and $0 <  \dl < \infty$, we have
\begin{align}
 |n|^{2\kk - 1}
 \les_\dl \wt{T}_{\dl,\kk-\frac12}(n)
\les 
  \min (\dl^{-1} |n|^{2\kk-1},n^{2\kk})
  \qquad \text{and}
  \qquad 
\wt{T}_{\dl, \kk}(n)
\sim |n|^k, 
\label{T6}
\end{align}

\noi
uniformly in $n\in \Z^*$, 
where the implicit constant in the first bound  is independent
of $0 < \dl \le 1$.

\end{lemma}

\begin{proof}
From \eqref{T2} and \eqref{sym2}, we have
\begin{align}
\begin{split}
\wt T_{\dl, \kk - \frac{1}{2}}(n) & =
\sum_{\substack{\l=1\\\text{odd}}}^{2\kk-1}
\wt a_{2\kk-1, \l}\,
\dl^{\l-1}
\Ldl^\l(n)
 n^{2\kk- \l-1}, \\
\wt T_{\dl, \kk }(n) & =
 \sum_{\substack{\l=0\\ \text{even}}}^{2\kk}
\wt a_{2\kk, \l}\,
 \dl^{\l}
\Ldl^\l(n)
 n^{2\kk-\l}.
\end{split}
\label{T5}
\end{align}

\noi
As for $\wt T_{\dl, \kk }(n)$, 
the contribution from $\l = 0$ is 
the only term which survives in the limit as $\dl \to 0$, 
which yields \eqref{T4} in view of  
the normalization \eqref{E5a}.
As for $\wt T_{\dl, \kk- \frac 12  }(n)$, 
the contribution from $\l = 1$ is 
the only term which survives in the limit as $\dl \to 0$, 
which yields \eqref{T4} in view of  
Lemma \ref{LEM:L1} and 
the normalization \eqref{E5a}.
The bounds \eqref{T4a} and \eqref{T4b} follow
from \eqref{T5}, 
\eqref{E5a}, and Lemma \ref{LEM:L1}.

The upper bounds in \eqref{T6} follows
from \eqref{T5} and Lemma \ref{LEM:L1}\,(i).
As for the lower bounds in \eqref{T6}, 
it follows from \eqref{T5}, \eqref{E5a}, the positivity of the constants
$\wt a_{k, \l}$ (see \eqref{DEs3a}), 
 and Lemma~\ref{LEM:L1}\,(iii) that 
\begin{align*}
\wt{T}_{\dl, \kk - \frac{1}{2}}(n)  \ge 3 \Ldl(n) n^{2\kk -2} &\ge
\begin{cases}
\frac1\dl |n|^{2\kk-1}, &  \dl|n| \gg 1, \\
n^{2\kk}, & \dl|n| \les 1
\end{cases}\\
& \ges_{\dl} |n|^{2\kk-1}, 
\end{align*}

\noi
uniformly in $n\in\Z^*$, 
where the implicit constant is independent
of $0 < \dl \le 1$.
On the other hand, from \eqref{T5}, we have 
$\wt T_{\dl, \kk} (n) \ge n^{2\kk}$. 
Therefore, we obtain \eqref{T6}.
\end{proof}

We now present a proof of Proposition \ref{PROP:gauss2}.
While 
 it follows from the lines in  the proof of Proposition 4.1 in \cite{LOZ}, 
 the argument on singularity\,/\,equivalence presented below is more involved due to the
 more complex nature of the multiplier $\wt T_{\dl, \frac k2}$ in \eqref{T2}.

\begin{proof}[Proof of Proposition~\ref{PROP:gauss2}]
(i) 
With Lemma \ref{LEM:Tdl}, 
the bounds \eqref{SX2} and \eqref{SX3} follow
from a slight modification of   the proof of Proposition \ref{PROP:gauss1}\,(i).

\medskip

\noi
(ii)
From  Lemma \ref{LEM:Tdl}
(with $\eps_0 < 2\eps$)
and Lemma \ref{LEM:L1}\,(iv), we have 
\begin{align}
\jb{n}^{2\kk - 2-2\eps} \frac{|n^{2\kk} - \wt T_{\dl,\kk - \frac{1}{2}}(n)|}{\wt T_{\dl, \kk - \frac 12}(n) 
n^{2\kk}} 
\les
\frac{ \hf(\dl, n)}{\jb{n}^{1+2\eps}} + 
\frac{\dl^{\eps_0}}{\jb{n}^{1+2\eps-\eps_0}}  
\too 0,  
\label{T7}
\end{align}

\noi
as $\dl \to 0$, 
for each $n \in \Z^*$.
Moreover, we have 
$\eqref{T7} \les \jb{n}^{-1-2\eps + \eps_0}$
for any $0 < \dl \le 1$ and $n \in \Z^*$.
Hence, from the dominated convergence theorem, we have 
\begin{align}
\lim_{\dl \to 0}\sum_{n \in \Z^*} \eqref{T7}
= 0.
\label{T7a}
\end{align}

\noi
Similarly, 
we have 
\begin{align}
\jb{n}^{2\kk - 1-2\eps} \frac{|n^{2\kk} - \wt T_{\dl,\kk }(n)|}{\wt T_{\dl, \kk}(n) 
n^{2\kk}} 
\les
\frac{\dl^{\eps_0}}{\jb{n}^{1+2\eps-\eps_0}}  
\too 0,  
\label{T8}
\end{align}

\noi
as $\dl \to 0$, 
for each $n \in \Z^*$.
Hence, from the dominated convergence theorem, we have 
\begin{align}
\lim_{\dl \to 0}\sum_{n \in \Z^*} \eqref{T8}
= 0.
\label{T8a}
\end{align}

\noi
Then, proceeding as in \eqref{DX4} with \eqref{T7a} and \eqref{T8a}, 
we obtain the desired $L^p$-convergence of 
 $\wt X^\dl_{\frac{k}{2}}$ to $\wt X^0_{\lceil \frac k 2 \rceil }$ as $\dl \to 0$
 in $W^{\frac{k-1}{2}-\eps, r}(\T)$, 
 which in turn implies weak convergence of 
   $\wt \mu^\dl_{\frac{k}{2}}$
   to $\wt \mu^0_{\lceil \frac{k}{2}\rceil}$.

\medskip

\noi
(iii)
As in the proof of Proposition \ref{PROP:gauss1}\,(iii),
given $0 < \dl < \infty$ and $n \in \N$, set
\begin{align*}
A_n^\dl &= \frac{\Re g_n}{(\wt T_{\dl,  \frac{k}{2}}(n))^\frac12},
\quad  A^\dl_{-n} = - \frac{\Im g_n}{(\wt T_{\dl,  \frac{k}{2}}(n))^\frac12}, \quad
B_n = \frac{\Re g_n}{|n|^{ \lceil \frac{k}{2}  \rceil}}, \quad
B_{-n}= - \frac{\Im g_n}{|n|^{ \lceil\frac{k}{2} \rceil}}
\end{align*}

\noi
such that 
\begin{align*}
a^\dl_{\pm n} = \E[(A^\dl_{\pm n})^2] = (\wt T_{\dl,  \frac{k}{2}}(n))^{-1}
\qquad \text{and} \qquad
b_{\pm n} = \E[B^2_{\pm n}] = |n|^{-2\lceil \frac k2\rceil}.
\end{align*}

\noi
In the following, we consider the cases $k = 2\kk - 1$ and $k = 2\kk$
such that $\lceil \frac k2\rceil = \kk$.

Let  $k = 2\kk - 1$.
From \eqref{T5}, \eqref{E5a},
and Lemma \ref{LEM:L1}, we have
\begin{align*}
\frac{b_n}{a^\dl_n}- 1
& = \frac{\wt T_{\dl,  \kk -\frac 12}(n) - n^{2\kk}}{n^{2\kk}}
 = - \hf(\dl, n)
+ \sum_{\substack{\l=3\\\text{odd}}}^{2\kk-1}
\wt a_{2\kk-1, \l}\,
\dl^{\l-1}
\Ldl^\l(n)
 n^{- \l-1}\\
& =  - \hf(\dl, n) + O(\dl^{-1}|n|^{-1}).
\end{align*}

\noi
Thus, from \eqref{var6a} and
Lemma \ref{LEM:L1}\,(iv), we have
\begin{align*}
\sum_{n \in \Z^*}\bigg(\frac{b_n}{a^\dl_n}- 1 \bigg)^2
\ge \frac 12
\sum_{n \in \Z^*}\hf^2(\dl, n) - C\sum_{n \in \Z^*}\frac{1}{\dl^{2}|n|^{2}}
= \infty.
\end{align*}

\noi
Hence, from Kakutani's theorem (Lemma \ref{LEM:kak}),
we conclude that $\wt \mu^\dl_{\kk- \frac 12}$
and $\wt \mu^0_\kk$ are singular.

Next, we consider the case
$k = 2\kk$.
From \eqref{T5},
Lemma \ref{LEM:L1}\,(iii),
and the positivity of the coefficients $\wt a_{2\kk, \l}$,
we have
\begin{align*}
\frac{b_n}{a^\dl_n}- 1 = \frac{\wt T_{\dl, \kk}(n) - n^{2\kk}}{n^{2\kk}}
=  \sum_{\substack{\l=2\\ \text{even}}}^{2\kk}
\wt a_{2\kk, \l}\,
 \dl^{\l}
\Ldl^\l(n)
 n^{-\l}
 \ges 1
\end{align*}

\noi
for $|n| \gg \dl^{-1}$.
Hence, from Kakutani's theorem (Lemma \ref{LEM:kak}),
we conclude that $\wt \mu^\dl_\kk$
and $\wt \mu^0_\kk$ are singular.

\medskip

\noi (iv) 
Fix  $\kk \in \N$ and $0<\dl_1<\dl_2<\infty$.
In view of Kakutani's theorem (Lemma \ref{LEM:kak}),
the claimed 
 singularity of $\wt{\mu}^{\dl_1}_{\kk-\frac12}$ and $\wt{\mu}^{\dl_2}_{\kk-\frac12}$ and equivalence of $\wt{\mu}^{\dl_1}_{\kk}$ and $\wt{\mu}^{\dl_2}_{\kk}$
 follows once we prove

\begin{align}
 \sum_{n\in\Z^*} \bigg(  \frac{\wt T_{\dl_1, \kk-\frac12}(n) - \wt T_{\dl_2,  \kk - \frac12}(n)}
{\wt T_{\dl_1,  \kk-\frac12}(n)} \bigg)^2 = \infty
 \label{T9a}
\end{align}

\noi
and 
\begin{align}
 \sum_{n\in\Z^*} \bigg(  \frac{\wt T_{\dl_1, \kk}(n) - \wt T_{\dl_2,  \kk}(n)}
{\wt T_{\dl_1,  \kk}(n)} \bigg)^2 < \infty.\label{T9b}
\end{align}

We first show \eqref{T9a}. From \eqref{T5} and \eqref{sym2}, for any $\dl>0$, we have
\begin{align*}
 \wt{T}_{\dl, \kk-\frac12}(n) 
& = \dl^{-1} n^{2\kk-1} \sum_{\substack{\l=1\\\text{odd}}}^{2\kk-1} \wt{a}_{2\kk-1, \l}  (\coth(\dl n) - (\dl n)^{-1})^\l, 
\end{align*}

\noi
from which  we obtain
\begin{align*}
\lim_{n\to \infty} \frac{\wt{T}_{\dl, \kk-\frac12} (n)}{n^{2\kk-1}} = \dl^{-1} \sum_{\substack{\l=1\\\text{odd}}}^{2\kk-1} \wt{a}_{2\kk-1, \l}.
\end{align*}

\noi
This  in turn implies
\begin{align*}
\lim_{n\to \infty} \bigg(  \frac{\wt T_{\dl_1, \kk-\frac12}(n) - \wt T_{\dl_2,  \kk - \frac12}(n)}
{\wt T_{\dl_1,  \kk-\frac12}(n)} \bigg)^2 = \frac{(\dl_2 - \dl_1)^2}{\dl_2^2} \ne 0,
\end{align*}

\noi
yielding \eqref{T9a}.

Next, we show \eqref{T9b}. 
From \eqref{T5} and \eqref{sym2}, for any $\dl>0$, we have
\begin{align}
\begin{split}
\wt{T}_{\dl, \kk}(n)
& =  n^{2\kk}  \sum_{\substack{\l=0\\\text{even}}}^{2\kk} \wt{a}_{2\kk, \l} 
\big( \coth(\dl n) - (\dl n)^{-1}\big)^\l .
\end{split}
\label{T10b}
\end{align}

\noi 
Then, from \eqref{T10b} with \eqref{E5a}, we have 
\begin{align}
\begin{split}
& n^{-2\kk} \big( \wt{T}_{\dl_1, \kk}(n) - \wt{T}_{\dl_2, \kk}(n) \big)\\
&\quad = 
\bigg( \coth(\dl_1 n)  - \coth(\dl_2 n) - \frac{\dl_2 - \dl_1}{\dl_1\dl_2n }  \bigg) \\
& \quad \quad \times \sum_{\substack{\l=2\\\text{even}}}^{2\kk} \wt{a}_{2\kk, \l} \sum_{j=0}^{\l-1} 
\big(\coth(\dl_1 n) - (\dl_1 n)^{-1}\big)^{\l-1-j}  \big(\coth(\dl_2 n) - (\dl_2 n)^{-1} \big)^{j} .
\end{split}
\label{T10bb}
\end{align}

\noi
Combining \eqref{T10b} and \eqref{T10bb}, we have 
\begin{align}
\bigg(  \frac{\wt T_{\dl_1, \kk}(n) - \wt T_{\dl_2,  \kk}(n)}
{\wt T_{\dl_1,  \kk}(n)} \bigg)^2
= 
 C_n D^2_n,
\label{T10c}
\end{align}

\noi
where $C_n$ and $D_n$ are given by 
\begin{align}
\begin{split}
C_n & = \big(\coth(\dl_1 n) - \coth(\dl_2 n)\big)^2 
+ \frac{(\dl_2-\dl_1)^2}{\dl_1^2 \dl_2^2n^2}  \\
& \quad -  \frac{2(\dl_2-\dl_1)\big(\coth(\dl_1n) - \coth(\dl_2 n)\big)}{\dl_1 \dl_2 n}, \\
D_n & = \frac{ \displaystyle\sum_{\substack{\l=2\\\text{even}}}^{2\kk} \wt{a}_{2\kk, \l} \sum_{j=0}^{\l-1}  \big(\coth(\dl_1n) - (\dl_1 n)^{-1}\big)^{\l-1-j} \big(\coth(\dl_2 n) - (\dl_2 n)^{-1}\big)^j   }
{\displaystyle \sum_{\substack{\l=0\\\text{even}}}^{2\kk} \wt{a}_{2\kk, \l}
\big(\coth(\dl_1n) - (\dl_1 n)^{-1}\big)^{\l}    }  .
\end{split}
\label{T10cc}
\end{align}

\noi
From \eqref{T10cc},  the positivity of the constants $\wt{a}_{2\kk, \l}$, 
\eqref{sym1}, and Lemma \ref{LEM:K1}, 
we have 
\begin{align}
 D_n 
\sim 1
\label{T10d}
\end{align}

\noi
for any $|n| \gg \dl_1^{-1}$.
On the other hand, from Cauchy's inequality with   $\dl_2 > \dl_1 > 0$, we have 
\begin{align}
\begin{split}
|C_n|
&  \les_{\dl_1, \dl_2}
\frac{\max \big(e^{2(-\dl_1 + \dl_2)n}, e^{2(\dl_1 - \dl_2) n}\big)}
{(e^{\dl_1 n} - e^{-\dl_1 n})^2 (e^{\dl_2 n} - e^{-\dl_2 n})^2 }
+ \frac1{n^2}
\\
& \les_{\dl_1,\dl_2}  \frac{1}{e^{2\dl_1 |n|}} + \frac{1}{n^2}, 
\end{split}
\label{T10e}
\end{align}

\noi
which is summable in $n \in \Z^*$.
Hence, \eqref{T9b} follows
from \eqref{T10c}, \eqref{T10d}, and \eqref{T10e}. 
\end{proof}

\begin{remark}\rm
We point out that 
for the summability of 
the right-hand side of  \eqref{T10c}, 
it is crucial to have $\dl_1>0$;
see the right-hand side of \eqref{T10e}.
As we proved in Proposition \ref{PROP:gauss2}\,(ii), 
 \eqref{T9b} indeed fails
when $\dl_1 = 0$.
\end{remark}

\subsection{Uniform integrability
 of the shallow-water densities}
\label{SUBSEC:SGM2}

In this subsection, 
we establish the  uniform (in $N$ and also in $0<  \dl \le 1$)
$L^p$-integrability of  the truncated density $
\wt F^\dl_{\frac k2}(\P_N v)$, 
 where $\wt F^\dl_{\frac{k}{2}}(v)$ is as in \eqref{rho5}:
\begin{align}
\begin{split}
\wt F^\dl_{\frac{k}{2}}(v)  =  \eta_K\big(\| v\|_{L^2}\big)
\exp\Big(-\wt R^\dl_{\frac k 2}(v)\Big) 
\end{split}
\label{SV0}
\end{align}

\noi
for a fixed constant  $K>0$.

\begin{proposition}\label{PROP:SV1}
Let $k \ge 2$ be an integer and $K > 0$.
Then, given any  $0<\dl <  \infty$ and finite  $p\ge 1$,
there exists $ C_{\dl, p}> 0$ such that 
\begin{align}
\sup_{N\in\N}  \| \wt F^\dl_{\frac{k}{2}}( \P_N \wt X^\dl_{\frac{k}{2}}) \|_{L^p(\Omega)} 
& = \sup_{N\in\N} \|\wt  F^\dl_{\frac{k}{2}}( \P_N v) \|_{L^p(d\wt \mu^\dl_{\frac{k}{2}})}
\leq C_{\dl, p} < \infty, 
\label{SV1}
\end{align}

\noi
where the constant $C_{\dl, p}$ is  independent of 
$0< \dl \le 1$.
Here, 
$\wt X^\dl_\frac k2$ and 
$\wt \mu^\dl_\frac k2$ are as in \eqref{Xdl2} and \eqref{gauss3}, respectively;
see also \eqref{Xdl3} and \eqref{gauss4}.
When $k$ is even, \eqref{SV1} also holds
for $\dl = 0$.

\end{proposition}

As in the deep-water case, 
by defining
\begin{align*}
\wt \F^\dl_{ \frac{k}{2}}(v)
=\exp\Big(-\wt R^\dl_{\frac{k}{2}}(v) - A \| v\|_{L^2}^{\al} \Big), 
\end{align*}

\noi
we 
see that Proposition \ref{PROP:SV1} follows once we prove the following
bound on $\wt \F^\dl_{ \frac{k}{2}}(\P_N v)$.

\begin{proposition}\label{PROP:SV2}
Let $k \ge 2$ be an integer.
Then, given  any $0<\dl <  \infty$ and finite  $p\ge 1$,
there exist $A = A(p) > 0$,  $\al > 0$,  and $C_{\dl, p} > 0$ such that 
\begin{align}
&\sup_{N\in \N} \|\wt \F^\dl_\frac k2 (\P_N \wt X^\dl_{\frac{k}{2}})  \|_{L^p(\O)}
=
\sup_{N \in \N}\|\wt  \F^\dl_\frac k2 (\P_N v) \|_{L^p(d\wt \mu^\dl_{\frac{k}{2}})}
\le C_{\dl, p} < \infty, 
\label{SV10}
\end{align}

\noi
where the constant $C_{\dl, p}$ is  independent of 
$0<  \dl \le 1$.
Here, $A = A(p)$ is independent of $0 < \dl \le 1$, while
 $\al$ is independent of $p$ and $0 < \dl < \infty$.
When $k$ is even, \eqref{SV10} also holds
for $\dl = 0$.

\end{proposition}

The remaining part of this subsection
is devoted to a proof of Proposition \ref{PROP:SV2}.
Let us introduce some notations.
Let  $\wt{\mathcal{T}}_{\dl, \frac k 2}$
 the Fourier multiplier operator
with multiplier $\wt{T}_{\dl, \frac k 2}(n)$ defined  in \eqref{T2}.
Then, 
we define a centered Gaussian process $\wt \cX^\dl_{\frac k2}$
by setting
\begin{align}
\wt  \cX^\dl_\frac k2  (t)
= \wt  \TT_{\dl, \frac{k}{2}}^{-\frac 12}    W(t)
=\frac{1}{\sqrt{2\pi}} \sum_{n\in \Z^*} \frac{B_n(t)}{\big(\wt T_{\dl,\frac{k}{2}}(n)\big)^\frac12} e_n, 
\label{P3}
\end{align}

\noi
where $W$ is as in \eqref{P1}.
By comparing \eqref{P3} with \eqref{Xdl2}, 
we have 
$\Law \big( \wt \cX^\dl_\frac k2  (1)\big)
=\Law \big(\wt X^\dl_\frac k2 \big)$.
In the current shallow-water regime, 
the Bou\'e-Dupuis variational formula reads as follows.

\begin{lemma}\label{LEM:Svar0}
Let $k \ge 2$ be an integer,   $0 <   \dl <  \infty$, 
and $N \in \N$.
Suppose that  $G:C^\infty(\T) \to \R$
is measurable such that $\E\big[|G(\P_N \wt X^\dl_{\frac k2})|^p\big] < \infty$
and $\E\big[\big|\exp  \big( G(\P_N\wt X^\dl_{\frac k2} )\big)\big|^q \big] < \infty$ 
for some $1 < p, q < \infty$ with $\frac 1p + \frac 1q = 1$.
Then, we have
\begin{align}
\begin{split}
&  \log \E\Big[\exp\big(G(\P_N \wt X^\dl_{\frac k2})\big)\Big]\\
& \quad = \sup_{\dr \in  \Ha}
\E\bigg[ G\big( \P_N \wt \cX^\dl_{\frac k2} (1) + \P_N \wt I_{\dl, \frac k 2}(\dr)(1)\big) - \frac{1}{2} \int_0^1 \| \dr(t) \|_{L^2_x}^2 dt \bigg],
\label{Svar0}
\end{split}
\end{align}

\noi
where 
$\wt I_{\dl, \frac k 2}(\dr)$ is  defined by
\begin{align*}
\wt I_{\dl, \frac k 2}  (\dr)(t) = \int_0^t \wt \TT_{\dl, \frac{k}{2}}^{-\frac 12}     \dr(t') dt'
\end{align*}

\noi
and the expectation $\E = \E_\PP$
is taken with respect to the underlying probability measure~$\PP$.
When $k$ is even,~\eqref{Svar0} also holds
for $\dl = 0$
with  $\wt{T}_{0,\frac k2}(n) 
= n^{k}$ as in~\eqref{T4}.

\end{lemma}

As in Subsection \ref{SUBSEC:DGM2}, 
with a slight abuse of notation, 
we use
$\wt X^\dl_\frac k2$ to denote 
$\wt \cX^\dl_\frac k2  (1)$
in applying the variational formula \eqref{Svar0}
 in the remaining part
of this subsection.

As in the deep-water case (see \eqref{var2}), we have
\begin{align}
\| \wt I_{\dl, \frac k 2}(\theta)(1) \|_{H^\frac{k}{2}_x}^2 \leq \wt a_\dl \int_0^t \| \theta (t) \|^2_{L^2_x}  dt
\label{Svar2}
\end{align}

\noi
 for any $\theta \in \Ha$, 
where the constant $\wt a_\dl>0$ is independent of  $0 <  \dl \leq 1$.
When $k$ is even,~\eqref{Svar2} also holds
for $\dl = 0$
with the constant $\wt a_\dl$ independent of 
$0\le \dl \le 1$.

As in Subsection \ref{SUBSEC:DGM2}, 
our main task is to prove the following lemma on the interaction potential 
$\wt R^\dl_{\frac k2 }(v)$, 
which plays a role of 
Lemma \ref{LEM:var2}
in the deep-water regime.


\begin{lemma}\label{LEM:Svar2}
Let $k\ge 2$ be an integer.
Given $0<\dl< \infty$ and small $\eps_0 > 0$,  
 there exist $\al \ge 1$,  independent of $\dl$ and $\eps_0$,  and 
  $C_{\dl, \eps_0}>0$ such that
\begin{align}
\big|\E \big[ \wt R^\dl_{\frac k2 }(\wt X^\dl_N+\Dr_N) \big] \big|
 \leq 
 C_{\dl, \eps_0}+ 
 \E\Big[ C_{\dl, \eps_0} \|\Dr_N\|_{L^2}^{\al}   + \eps_0 \|\Dr_N\|_{H^\frac{k}{2}}^2 \Big], 
\label{SDV1}
\end{align}

\noi
uniformly in $N \in \N$, 
where the constant   $C_{\dl, \eps_0}>0$ 
is independent of $0<  \dl \le 1$.
Here, 
\begin{align}
\wt X^\dl_N :=  \P_N \wt X^\dl_{\frac k 2}
\qquad\text{and}\qquad 
\Dr_N = \P_N \Dr
\label{RRY1}
\end{align}

\noi
 as in \eqref{var4}
and, in particular, they have mean zero on $\T$.
When $k$ is even, \eqref{SDV1} also holds
for $\dl = 0$.

\end{lemma}

Once we prove Lemma \ref{LEM:Svar2}, 
arguing as in the proof of Proposition \ref{PROP:V2} in Subsection \ref{SUBSEC:DGM2}
with~\eqref{Svar2} and \eqref{SDV1}, 
we can prove Proposition \ref{PROP:SV2}
(and hence Proposition \ref{PROP:SV1}).
More precisely, 
as in \eqref{var3}, 
define
$\wt{ \mathcal{M}}_{\dl,\frac k2 , p, N}(\Dr)$
by setting 
\begin{align}
\begin{split}
\wt{ \mathcal{M}}_{\dl,\frac k2 , p, N}(\Dr) 
& =\E \bigg[ -p \wt R^\dl_{\frac{k}{2}}( 
\wt X^\dl_N +\Dr_N) \\
& \hphantom{XXX}- p A\|  \wt X^\dl_N + \Dr_N \|_{L^2}^{\al} 
- \frac12 \wt{a}_\dl^{-1}  \|\Dr\|^2_{H^{\frac{k}{2}}} \bigg]
\end{split}
\label{Svar3}
\end{align}

\noi
for $\Dr\in H^{\frac{k}{2}}_0(\T)$, where we used the notations in \eqref{RRY1}.
Then, 
Proposition~\ref{PROP:SV2} follows once
we establish an 
upper bound on 
$\wt{ \mathcal{M}}_{\dl,\frac k2 , p, N}(\Dr)$
for each $0 < \dl <  \infty$, 
uniformly in $N\in\N$  and $\Dr\in H^{\frac{k}{2}}_0(\T)$ 
and also 
uniformly in $0< \dl \le 1$.
It follows from  \eqref{Svar3}, Lemma \ref{LEM:Svar2}, 
\eqref{var6} (with $X^\dl_N$ replaced by $\wt X^\dl_N$), 
and Proposition \ref{PROP:gauss2}\,(i) that given any small $\eps_0 > 0$,  we have 
\begin{align}
\begin{split}
\wt {\mathcal{M}}_{\dl,\frac k 2,  p, N}(\Dr) 
& \leq p (1+A)\wt C_{\dl, \eps_0, \al}+
\E\Big[\big(Cp \eps_0 - \tfrac12 \wt a_\dl^{-1}\big) \|\Dr_N\|_{H^\frac{k}{2}}^{2}\\
& \quad +\big(p C_{\dl, \eps_0} - \tfrac 12 pA\big) \|\Dr_N\|_{L^2}^\al\Big], 
\end{split}
\label{Svar5}
\end{align}

\noi
uniformly in $N\in \N$ 
and $\Dr\in H^{\frac{k}{2}}(\T)$, 
where $\al \ge 1$ is as in Lemma \ref{Svar2}
and  the constants
$\wt C_{\dl, \eps_0, \al}$
and 
$C_{\dl, \eps_0}$ are independent of 
 $0 <  \dl \le 1$ (but depend on $\eps_0 > 0$).
Hence, the desired bound on 
$\wt{ \mathcal{M}}_{\dl,\frac k2 , p, N}(\Dr)$
(and hence 
Proposition \ref{PROP:SV2})
follows once we choose $\eps_0, A > 0$  in~\eqref{Svar5}
such that 
\[ Cp \eps_0 - \tfrac12 \wt a_\dl^{-1}  < 0
\qquad \text{and}\qquad 
p C_{\dl, \eps_0} - \tfrac 12 pA < 0.\]

\noi
Since 
 $\wt a_\dl$ in \eqref{Svar2} is independent of 
 $0 <  \dl \leq 1$, 
 we can first choose $\eps_0 = \eps_0(p)> 0$, 
 independently of 
 $0 <  \dl \leq 1$, 
which in turn allow us to choose 
$A = A(p) > 0$ 
 independently of 
 $0 <  \dl \leq 1$.

\medskip

It remains to prove Lemma \ref{LEM:Svar2}.
Before proceeding to a proof of Lemma \ref{LEM:Svar2}, 
we  first state  a preliminary lemma
which is a shallow-water analogue of Lemma \ref{LEM:var0a}.

\begin{lemma}\label{LEM:Svar0a}

Let $k \ge 2$ be an integer.
Let $0\le\dl < \infty$
and
$s, \s_1, \s_2 \in \R$ with $\s_1 \le \s_2 \le \frac k2$
such that
$2k - 2\s_1 - 2\s_2 > 1$ and 
\eqref{VL3} holds.
Then, given any finite $p \ge 1$, $1\le r \le \infty$, and $\al_1, \al_2 \in \Z_{\ge 0}$, we have 
\begin{align}
\E\bigg[ \big\|\P_{\ne 0} \big((\jb{\dx}^{\s_1}(\dl \Gd)^{\al_1} \P_N \wt X^\dl_{\frac k 2})(\jb{\dx}^{\s_2}
(\dl \Gd)^{\al_2}\P_N  \wt X^\dl_{\frac k 2})\big)\big\|_{W^{s, r}}^p  \bigg]
< C_\dl<\infty,
\label{SVL2}
\end{align}

\noi
uniformly in $N \in \N$, 
where the constant $C_\dl>0$ is independent of  $0\leq \dl \leq 1$.
The bound \eqref{SVL2} also holds
even if we replace
$\jb{\dx}^{\s_j}$ by $\jb{\dx}^{\s_j - \kk_j}\dx^{\kk_j} $
for any $\kk_j \in \N$ with $\kk_j \le \s_j$, $j = 1, 2$.

\end{lemma}

\begin{proof}
When $\al _1 = \al_2 = 0$, 
the bound \eqref{SVL2} follows 
from a slight modification of the proof of Lemma~\ref{LEM:var0a}
with~\eqref{Xdl2} and \eqref{T6} in 
Lemma \ref{LEM:Tdl} (which gives the same lower bound $|n|^k$ as in the deep-water case).
Otherwise, we proceed as in 
 the proof of Lemma~\ref{LEM:var0a}
and then use the bound
$|\dl \ft \Gd(n)| \les1$
which follows from  \eqref{sym2} and Lemma \ref{LEM:L1}.
\end{proof}

We now present a proof of Lemma \ref{LEM:Svar2}.

\begin{proof}[Proof of Lemma \ref{LEM:Svar2}]

We only consider $0 < \dl < \infty$.
For the $\dl = 0$ case with even $k$, see Remark \ref{REM:kdv1}.
We separately treat the cases $k = 2\kk$ and $k = 2\kk - 1$, $\kk \in \N$.
In the following, all the estimates hold uniformly in $N \in \N$
and
 all the constants are understood to be uniform in $0 <  \dl \le 1$.

\medskip

\noi
$\bullet$
{\bf Case 1:}
 $k = 2\kk-1$ is odd.\\
\indent
From  Proposition \ref{PROP:gauss2}
(see also Remark \ref{REM:gauss1}), 
we know that  $\wt X^\dl_N$ belongs
to $W^{\kk - 1-\eps, \infty}(\T)\setminus
W^{\kk - 1, 1}(\T)$, 
almost surely (in the limiting sense as $N \to \infty$).
Namely,  if a monomial $p(v)$
appearing on the right-hand side of~\eqref{Scons2} 
has two factors with $\kk - 1$ derivatives, 
we need to rely on a probabilistic estimate
(as in Lemma \ref{LEM:var2}).
Otherwise, by (fractional) integration by parts, 
we may assume that all the factors in $p(v)$
has at most $\kk - \frac 32$ derivatives
and we can proceed  deterministically
(as in Lemma \ref{LEM:var1}).

Let $\wt R_i$, $i = 1, 2$, be the $i$th term on the right-hand side of \eqref{Scons2}.
We  consider the first term $\wt R_1$.
From Lemma \ref{LEM:T1}\,(ii), we have 
\begin{align}
|\dl \ft \Gd(n)| \les1
\qquad \text{and}\qquad
| \ft \Gd(n)| \les|n|.
\label{SV2}
\end{align}

\noi
Moreover, the first bound 
is equivalence for 
$|n| \gg \dl^{-1}$, 
while 
 the second bound 
is equivalence for 
$|n| \les \dl^{-1}$.
See also 
Lemma \ref{LEM:T1}\,(iii).
Any monomial  in $\wt R_1$ satisfies
$\#\Gd = \#\dl + 1$, namely, 
there is exactly one factor of $\Gd$ which is not paired with $\dl$, 
causing a derivative loss  in view of~\eqref{SV2}.
Hence, although we have $\#\dx \le 2\kk +3 - 2j$
for $\wt R_1$,
by accounting for this derivative loss, 
we virtually need to work with a new constraint 
$\#\dx \le 2\kk +4 - 2j$.\footnote{For simplicity of the presentation, 
we simply replace such unpaired $\Gd$ by $\dx$ in the following.}
In particular, when $j = \# v = 3$, we have $\#\dx = 2\kk - 2$
such that there can be two factors with $\kk - 1$ derivatives
(which is 
 also the case for any cubic monomial in $\wt R_2$
under $\#\dx \le 2\kk + 1 - j$ (with $j = 3$)).

Consequently, in order to prove \eqref{SDV1}
for $\wt R^\dl_{\kk - \frac 12} = \wt R_1 + \wt R_2$, 
it suffices to consider a monomial $p(v)$ whose fundamental form $\wt p(v)$ is given by 
\begin{align}
\text{(a)\ $\wt p(v) = v \dx^{\kk-1} v \dx^{\kk-1} v$ 
\quad 
and \quad (b)\ $\wt p(v)= \prod_{i =1}^j \dx^{\al_i}v$}
\label{SV3}
\end{align}

\noi
for  
\begin{align}
\begin{split}
& j=3, \ldots, 2\kk+1, \quad 0 \le \al_j \le \cdots \le \al_1 \le \kk -1\quad 
\text{such that}\\
& (\al_1,\al_2)\ne(\kk-1,\kk-1), 
\quad 
\text{and}\quad  \\
& \al_{1\cdots j} \le \max(2\kk+4-2j, 2\kk+1-j)
= 2\kk+1-j.
\end{split}
\label{SV3x}
\end{align}

\medskip

\noi
$\bullet$ {\bf Subcase 1.a:}
$\wt p(v) = v \dx^{\kk-1} v \dx^{\kk-1} v$.\\
\indent
Since $k = 2\kk - 1$, we have
$\wt p(v) = v \dx^{\frac{k-1}{2}} v \dx^{\frac{k-1}{2}} v$.
Recall that,  
in Case 2 of the proof of Lemma \ref{LEM:var2}, 
 we treated
$p(u) = u  \dx^m u  \dx^m u$ with $m = \frac{k-1}{2}$
(see  \eqref{DDV2}-\eqref{DDV6}).
Suppose that $ p(v) = v \dx^{\frac{k-1}{2}} v \dx^{\frac{k-1}{2}} v$, 
namely, $\#\Gd = 0$.
Then,  proceeding as in 
\eqref{DDV2}-\eqref{DDV6}
but with Lemma~\ref{LEM:Svar0a} instead of 
Lemma~\ref{LEM:var0a}, 
we obtain \eqref{SDV1}.

In general, in view of the $L^r$-boundedness of $\dl \Gd$
 for $ 1 < r < \infty$ (Lemma \ref{LEM:T1}\,(iii)), 
 we can simply drop
any occurrence of  $\dl \Gd$, 
just as we dropped the Hilbert transform $\H$
and the perturbation operator $\Qdl$
in the proofs of Lemmas \ref{LEM:var1} and \ref{LEM:var2}. 
Then, proceeding as above, 
we obtain \eqref{SDV1}.

\medskip

\noi
$\bullet$ {\bf Subcase 1.b:}
$\wt p(v)= \prod_{i =1}^j \dx^{\al_i}v$ as in (b) of \eqref{SV3}, satisfying 
 \eqref{SV3x}.\\
\indent
In this case, from H\"older's inequality and \eqref{Q3} in Lemma \ref{LEM:T1}\,(iii), we have\footnote{As mentioned in the proofs
of Lemmas \ref{LEM:var1}  and \ref{LEM:var2}, the $W^{\al_i, \infty}$-norm should really be
the $W^{\al_i, r}$-norm for $r \gg1 $ which can then be controlled by 
the $W^{\al_i+\eps, \infty}$-norm. We, however, suppress this point in the following.}
 \begin{align}
 \bigg| \int_\T p(v)  dx \bigg| 
   \le
 C_\dl 
  \| u\|_{H^{\al_1}} \|v\|_{H^{\al_2}} 
 \prod_{i =3}^j \|v\|_{W^{\al_i, \infty}}, 
\label{SV3a}
 \end{align}

\noi
where $C_\dl $ is independent of $ 0<  \dl \le 1$.
Then, 
we can proceed as in Case 2 of the proof of Lemma \ref{LEM:var1}
by noting the following two points:
(i) after partial integration by parts, we may assume that 
$\al_i \le \kk - \frac 32 = \frac{k-2}{2}$ for any $i = 1, \dots, j$
such that \eqref{DV6} holds, 
and 
(ii) under the condition 
$ \al_{1\cdots j} \le 2\kk+1-j$ with $j \ge 3$, we have 
\begin{align}
\begin{split}
\g  
 = \frac{2}{k}\bigg(
\al_1 +   \al_2 + \sum_{i = 3}^j \Big(\al_i + \frac 12 + \eps\Big)\bigg)
 \le \frac{2}{k}
 \bigg(
k+1 - \frac 12j + (j-2)\eps\bigg) < 2, 
\end{split}
\label{SV4}
\end{align}

\noi
and thus the bound \eqref{DV7} holds with $\g$  in \eqref{SV4}.
Hence, the bound \eqref{SDV1} in this case follows from substituting 
 $v = \wt X^\dl_N+\Dr_N$ into
the left-hand side of  \eqref{SV3a}
and applying~\eqref{DV6},  \eqref{DV7},  Young's inequality, 
and Proposition \ref{PROP:gauss2}\,(i).

\medskip

\noi
$\bullet$
{\bf Case 2:}
 $k = 2\kk$ is even.\\
\indent
We now estimate $\wt R^\dl_\kk(v)$ in \eqref{Scons2}.
In this case, we have $\#\Gd \le \#\dl$, 
namely, every $\Gd$ is paired with $\dl$ in each monomial.
In view of the discussion in Case 1, 
we may simply drop occurrences of $\dl \Gd$.
Hence, under the condition $\#\dx\le 2\kk+2 -j$, 
$ |p(v)| \le \kk$, and $j \ge 3$, 
it suffices to consider a monomial $p(v)$ whose fundamental form $\wt p(v)$ is given by 
\begin{align}
\text{(a)\ $\wt p(v) = v \dx^{\kk-1} v \dx^{\kk} v$ 
\quad 
and \quad (b)\ $\wt p(v)= \prod_{i =1}^j \dx^{\al_i}v$}
\label{SV5}
\end{align}

\noi
for 
\begin{align}
 j=3, \ldots, 2\kk+2, \quad 0 \le \al_i\le \kk -1, i = 1, \dots, j, 
 \quad 
\text{such that}
\quad  \al_{1\cdots j} \le  2\kk+2-j.
\label{SV5x}
\end{align}

\medskip

\noi
$\bullet$ {\bf Subcase 2.a:}
$\wt p(v) = v \dx^{\kk-1} v \dx^{\kk} v$.\\
\indent
Since $k = 2\kk $, we have
$\wt p(v) = v \dx^{\frac{k}{2} - 1} v \dx^{\frac{k}{2}} v$.
Thus, we can proceed as in 
in Case 1 of the proof of Lemma~\ref{LEM:var2}, 
where we treated
$p(u) = u  \dx^{m -1}u  \dx^m u$ with $m = \frac{k}{2}$
(see  \eqref{DDV11a}-\eqref{DDV17}), 
and 
we obtain \eqref{SDV1}.

\medskip

\noi
$\bullet$ {\bf Subcase 2.b:}
$\wt p(v)= \prod_{i =1}^j \dx^{\al_i}v$ as in (b) of \eqref{SV5}, 
satisfying 
\eqref{SV5x}.\\
\indent
In this case,
we have 
(i) 
$\al_i \le \kk - 1 = \frac{k-2}{2}$ for any $i = 1, \dots, j$
such that \eqref{DV6} holds, 
and 
(ii)~under the condition 
$ \al_{1\cdots j} \le  2\kk+2-j$ with $j \ge 3$,
\eqref{DV6a} with $2m = k$
 holds.
Hence, by proceeding as in Subcase 1.b, we obtain
\eqref{SDV1}.

This concludes the proof of Lemma \ref{LEM:Svar2}.
\end{proof}

\subsection{Construction of the shallow-water \GGMs}
\label{SUBSEC:SGM3}

Arguing as in Subsection \ref{SUBSEC:DGM3}, 
Theorem~\ref{THM:4}\,(i)
on the construction of the shallow-water \GGMs~
follows once we prove
the following  lemma
on 
uniform (in $N$ and $\dl$)
convergence of the truncated interaction potential 
$\wt R^\dl_{\frac{k}{2}} (\P_N \wt X^\dl_{\frac{k}{2}})$, 
where $\wt R^\dl_{\frac{k}{2}}(v)$ 
and $\wt X^\dl_{\frac{k}{2}}$ are as in 
\eqref{Scons2}
 and~\eqref{Xdl2}, respectively.

\begin{lemma}\label{LEM:SR1}
Let $k \ge 2$ be an integer.
Given any $0 <   \dl  < \infty$
and  finite $p \ge 1$,  
the sequence $\{\wt R^\dl_{\frac{k}{2}} (\P_N \wt X^\dl_{\frac{k}{2}})\}_{N\in\N}$
converges
to   $\wt R^\dl_{\frac{k}{2}} (\wt X^\dl_{\frac{k}{2}})$
 in $L^p(\Omega)$ as $N \to \infty$.
 Moreover, given any finite $p \ge 1$, 
 there exist    $C_{\dl} = C_{\dl}(p) > 0$  and $\ta > 0$ such that 
\begin{align}
\sup_{N\in\N} \sup_{0 <\dl \leq 1} \|\wt R^\dl_{\frac{k}{2}} (\P_N \wt X^\dl_{\frac{k}{2}}) \|_{L^p(\Omega) } &
\le C_\dl < \infty,  \label{SRN1}\\
\| \wt R^\dl_{\frac{k}{2}} (\P_M \wt X^\dl_{\frac{k}{2}}) - \wt R^\dl_{\frac{k}{2}} (\P_N \wt X^\dl_{\frac{k}{2}}) \|_{L^p(\Omega) } & \leq \frac{C_{\dl}}{N^\theta}
 \label{SRN2}
\end{align}

\noi
for any $M\geq N \geq 1$, 
where the constant $C_{\dl}$ is  independent of $0 <  \dl \leq 1$.
In particular, the rate of convergence is uniform in $0 <  \dl \leq 1$.
When $k$ is even, \eqref{SRN1} and~\eqref{SRN2} also hold
for $\dl = 0$.

\end{lemma}

In the previous subsection, 
together with the $L^r$-boundedness of $\dl \Gd$ 
and the boundedness of $\Gd$ from 
$W^{1, r}(\T)$ from $L^r(\T)$
for $1 < r< \infty$
(see~\eqref{Q3} and \eqref{Q3a} in Lemma \ref{LEM:T1}), 
we saw how Lemma \ref{LEM:Svar2}
followed from a slight modification 
of the proofs of Lemmas \ref{LEM:var1} and \ref{LEM:var2}.
Lemma \ref{LEM:SR1} also follows
from a straightforward modification 
of the proof of the corresponding lemma in the deep-water case
(Lemma \ref{LEM:R1}). 
We omit details.

\begin{remark}\rm\label{REM:kdv1} 

Let $k\ge 2$ be an even integer.
Recall from \eqref{EE0} that 
$ \wt E^\KDV_{\kk}(v) = \wt{E}^\dl_{\kk, 0}(v)$, 
where 
$ \wt{E}^\dl_{\kk, 0}(v)$
is  the $\dl$-free part of 
the shallow-water conservation law
$\wt{E}^\dl_{\kk}(v)$  in \eqref{DEs},
defined as  the collection of all terms that do not depend on $\dl$ explicitly;
see Footnote \ref{FT:2}.
Then, 
from~\eqref{E6}, we see that 
$\wt R^\KDV_\kk(v) =  \wt R^\dl_{\kk, 0}(v)$, 
where 
$ \wt R^\dl_{\kk, 0}(v)$
denotes   the $\dl$-free part of 
$\wt R^\dl_{\kk }(v)$  in~\eqref{Scons2}.
Namely, all the monomials in 
$ \wt R^\dl_{\kk, 0}(v)$
appear in 
$\wt R^\dl_{\kk }(v)$.
Hence, 
 all the results established in Subsections~\ref{SUBSEC:SGM2}
 and~\ref{SUBSEC:SGM3} for even-order quantities extend to
 the $\dl=0$ case.

\end{remark}

\subsection{Shallow-water convergence of the \GGMs}
\label{SUBSEC:SGM4}

As in the deep-water case, 
we have the following lemma on 
  the $L^p(\O)$-convergence properties of the truncated density
$\wt F^\dl_{\frac{k}{2}} (\P_N v)$ defined  in \eqref{SV0}.

\begin{lemma}\label{LEM:SFN1}
Let $k \ge 2$ be an integer and $K > 0$.
Then, given any finite $p \ge 1$ and $N\in\N$, we have 
\begin{align*}
\lim_{\dl\to0} \|\wt F^\dl_{\frac{k}{2}} (\P_N \wt X^\dl_{\frac{k}{2}}) - \wt  F^0_{\lceil \frac{k}{2}\rceil} (\P_N 
\wt X^0_{\lceil \frac k 2\rceil}) \|_{L^p(\O)} & = 0, \\
\lim_{\dl\to0} \|\wt F^\dl_{\frac{k}{2}} (\P_N\wt  X^0_{\frac k 2}) - 
\wt F^0_{\frac{k}{2}} (\P_N\wt  X^0_{\frac k 2}) \|_{L^p(\O)} & = 0. 
\end{align*}

\end{lemma}


Arguing as in the proof of  
Lemma \ref{LEM:FN2}, 
we have 
the following 
deep-water analogue of Lemma \ref{LEM:FN2}
on an 
$N$-dependent 
(but uniform in $0 <  \dl\le 1$)
deterministic bound:
\begin{align}\label{Sunif1}
\sup_{0 < \dl \leq 1}
\big|\wt F^\dl_{\frac{k}{2}}(\P_N v) \big| \leq C_{k, K,  N} < \infty
\end{align}

\noi
for any $v \in L^2(\T)$, 
which 
holds including $\dl = 0$, when $k$ is even.
Then, 
Lemma \ref{LEM:SFN1}
follows from a straightforward modification 
of the proof of 
Lemma \ref{LEM:FN1}
with Lemma \ref{LEM:Tdl} and 
\eqref{Sunif1}.
We omit details.

We conclude this section by presenting a proof of  Theorem~\ref{THM:4}\,(ii) and~(iii).
\begin{proof}[Proof of Theorem~\ref{THM:4}\,(ii)  and~(iii)]
Let $k \ge 2$ be an integer.
The claim on singularity\,/\,equivalence 
follows from the corresponding statement 
for the base Gaussian measures (Proposition \ref{PROP:gauss2}\,(iii) and (iv))
and the equivalence of the shallow-water \GGMs~
and the base Gaussian measures with an $L^2$-cutoff
(Theorem \ref{THM:4}\,(i)).

Next, we show that 
 the measure $\wt{\rho}^\dl_{\frac k2}$  converges weakly to $\wt{\rho}^0_{\lceil \frac k2 \rceil}$
 as $\dl \to 0$.
 While we follow closely the argument in Subsection 4.3 of \cite{LOZ}, 
 we present details for readers' convenience.
Fix 
a Borel subset $A$ of $H^{\frac{k-1}{2}-\eps}(\T)$ 
with 
$\wt{\rho}^0_{\lceil \frac k2 \rceil} (\dd A) =0$, where $\dd A$ denotes the boundary of the set $A$. 
By the portmanteau lemma, weak convergence of 
$\wt{\rho}^\dl_{\frac k2}$ to 
$\wt{\rho}^0_{\lceil \frac k2 \rceil}$
 follows once we show that
\begin{equation}\label{SV6}
\lim_{\dl \to 0} 
\wt{\rho}^\dl_{\frac k2}(A) = \wt{\rho}^0_{\lceil \frac k2 \rceil}(A)
\end{equation}

\noi 
for every such $A$.
By the triangle inequality, we have 
\begin{equation}\label{SV7}
\begin{aligned}
\big|   \wt \rho^\dl_\frac k2 (A)- \wt{\rho}^0_{\lceil \frac k2 \rceil}(A) \big| 
& \leq  
 \sup_{0<  \dl\le 1}\big|  \wt{\rho}^\dl_{\frac k2}(A)  
- \wt{\rho}^\dl_{\frac k2, N}(A)  \big| \\
& \quad +  \big| \wt{\rho}^\dl_{\frac k2, N}(A) - \wt{\rho}^0_{\lceil \frac k2 \rceil, N}(A)\big|
+ \big|\wt{\rho}^0_{\lceil \frac k2 \rceil, N}(A)- \wt \rho^0_{\lceil \frac k2 \rceil}(A)\big|.
\end{aligned}
\end{equation}

\noi
In view of Theorem \ref{THM:4}\,(i), 
the first and third terms
on the right-hand side of \eqref{SV7} vanish
as $N \to \infty$.
Hence, once we show that
\begin{equation}\label{SV8a}
\lim_{\dl\to0} 
 \wt{\rho}^\dl_{\frac k2, N}(A) = \wt{\rho}^0_{\lceil \frac k2 \rceil, N}(A)
\end{equation}

\noi
for some $N \in \N$, 
then
 \eqref{SV6} follows from first taking a limit as $\dl\to0$ 
 and then  a limit as $N\to\infty$ in \eqref{SV7},

It remains to show \eqref{SV8a}.
As in \cite[(4.18)]{LOZ}, it suffices to prove
\begin{equation}\label{SV9}
\lim_{\dl\to0} 
 \E \Big[ \wt F^\dl_\frac k2(\P_N \wt{X}^\dl_\frac k2) \ind_A( \wt{X}^\dl_\frac k2)\Big]
 = 
\E\Big[ \wt{F}^0_{\lceil \frac k2\rceil}(\P_N \wt{X}^0_{\lceil \frac k2\rceil}) 
\ind_A(\wt{X}^0_{\lceil \frac k2\rceil}) \Big].
\end{equation}

\noi
By the triangle inequality and Cauchy-Schwarz's inequality, we have
\begin{align}
\begin{split}
\bigg| \E & \Big[ \wt F^\dl_\frac k2(\P_N \wt{X}^\dl_\frac k2) \ind_A( \wt{X}^\dl_\frac k2)\Big]
- 
\E\Big[ \wt{F}^0_{\lceil \frac k2\rceil}(\P_N \wt{X}^0_{\lceil \frac k2\rceil}) 
\ind_A(\wt{X}^0_{\lceil \frac k2\rceil}) \Big]\bigg|\\
& \le 
\big\|\wt F^\dl_{\frac{k}{2}} (\P_N \wt X^\dl_{\frac{k}{2}}) - \wt  F^0_{\lceil \frac{k}{2}\rceil} (\P_N 
\wt X^0_{\lceil \frac k 2\rceil}) \big\|_{L^1(\O)}\\
& \quad 
 + 
\|\wt{F}^0_{\lceil \frac k2\rceil}(\P_N \wt{X}^0_{\lceil \frac k2\rceil})\|_{L^2(\O)}
\E\Big[\big|\ind_A( \wt{X}^\dl_\frac k2)- \ind_A(\wt{X}^0_{\lceil \frac k2\rceil}) \big|
\ind_B ( \wt{X}^0_{\lceil \frac k2\rceil})
\Big], 
\end{split}
\label{SV10a}
\end{align}

\noi
where 
\begin{align*}
B =  \big\{ v \in H^{\frac{k-1}{2}-\eps}(\T): 
 \eta_K\big(\|v\|_{L^2}\big) >0\big\}.
\end{align*}

\noi
Note that by the continuity of $\eta_K$, 
we see that $B$ is open in $H^{\frac{k-1}{2}-\eps}(\T)$.
Since
$\wt{\rho}^0_{\lceil \frac k2 \rceil} (\dd A) =0$, we have 
$\wt{\mu}^0_{\lceil \frac k2 \rceil}(\partial A\cap B) =0$.
Then, we have
\begin{align}
\begin{split}
 \E& \Big[\big|\ind_A( \wt{X}^\dl_\frac k2)- \ind_A(\wt{X}^0_{\lceil \frac k2\rceil}) \big|
\ind_B ( \wt{X}^0_{\lceil \frac k2\rceil})\Big]\\
& = 
\E\Big[\ind_{(\textup{int} A)\cap B}(\wt{X}^0_{\lceil \frac k2\rceil})
\big|\ind_A( \wt{X}^\dl_\frac k2)- \ind_A(\wt{X}^0_{\lceil \frac k2\rceil}) \big|\Big]\\
& \quad + \E\Big[\ind_{(\textup{int} A^c)\cap B}(\wt{X}^0_{\lceil \frac k2\rceil})
\big|\ind_A( \wt{X}^\dl_\frac k2)- \ind_A(\wt{X}^0_{\lceil \frac k2\rceil}) \big|\Big], 
\end{split}
\label{SV11}
\end{align}

\noi
where $\text{int} A$ denotes the interior of $A$
given by $\text{int} A = A \setminus \dd A$.
From 
Proposition \ref{PROP:gauss2}\,(ii) 
and  the openness of $(\text{int} A)\cap B$ and $(\text{int} A^c)\cap B$, 
we see that the integrands of the terms on the right-hand side
of \eqref{SV11} tend to $0$ as $\dl \to 0$.
Hence, by the bounded convergence theorem, 
we obtain
\begin{align}
\lim_{\dl \to 0}\E\Big[\big|\ind_A( \wt{X}^\dl_\frac k2)- \ind_A(\wt{X}^0_{\lceil \frac k2\rceil}) \big|
\ind_B ( \wt{X}^0_{\lceil \frac k2\rceil})\Big] = 0.
\label{SV12}
\end{align}

Finally, the desired identity \eqref{SV9}
follows from \eqref{SV10a}, 
 Lemma \ref{LEM:SFN1}, 
Proposition \ref{PROP:SV1}, 
and~\eqref{SV12}.
This concludes
the proof of Theorem 
\ref{THM:4}\,(ii) and~(iii).
\end{proof}

\begin{remark}\label{REM:cutoff3}
\rm
If we consider the \GGM~ 
with 
a sharp $L^2$-cutoff $\ind_{\{\|v \|_{L^2} \le K\}}$
in place of the smooth $L^2$-cutoff
$ \eta_K\big(\| v\|_{L^2}\big)$, 
we can not exploit the continuity 
of $\eta_K$ as above.
Instead, we simply can use 
\begin{align*}
\PP\Big(\| \wt X^0_{\lceil \frac{k}{2}\rceil}\|_{L^2} = K\Big) = 0.
\end{align*}

\noi
See \eqref{RN3}.
Then, we can exploit the openness
of the sets
$\big\{ v \in H^{\frac{k-1}{2}-\eps}(\T): \|\wt X^0_{\lceil \frac{k}{2}\rceil}\|_{L^2} < K\big\}$.

\end{remark}

\section{Probabilistic  asymptotic conservation}
\label{SEC:AC}

In this section, 
we establish the probabilistic asymptotic 
conservation of the form \eqref{alc1}
for both the deep-water and shallow-water regimes
(Propositions \ref{PROP:AAS1} and \ref{PROP:AAS2}), 
which is a key ingredient for proving invariance
of the \GGMs.
In the next section, 
we prove  invariance for each {\it fixed} $\dl$.
For this reason, there is no need to establish
probabilistic  asymptotic conservation
which is uniform in $\dl \gg1$ (or $\dl \ll1$).

\subsection{Main statements}
\label{SUBSEC:AC0}

%

We first consider the deep-water regime.
Fix  $0 < \dl \le \infty$.
Given $N \in \N$, 
consider the truncated ILW~\eqref{ILW2}.
It is easy to check that the $L^2$-norm 
is conserved for the truncated dynamics~\eqref{ILW2}
and thus~\eqref{ILW2} is globally well-posed in $L^2(\T)$.
Moreover, \eqref{ILW2} is a Hamiltonian PDE
with the Hamiltonian:
\begin{align*}
E^\dl_{\frac k2, N}(u) = \frac12 \| \Gdl^\frac12 u\|^2_{\dot{H}^\frac12}+ 
R^\dl_\frac 12(\P_Nu)
= 
\frac12 \| \Gdl^\frac12 u\|^2_{\dot{H}^\frac12}+ \frac13  \int_\T  (\P_N u)^3 dx, 
\end{align*}

\noi
which is also conserved under the truncated ILW \eqref{ILW2}.
As observed in \cite{TV0, TV1, TV2}, 
the main source of difficulty in proving invariance 
of the \GGMs~ $\rho^\dl_\frac k2$ in~\eqref{rho3} under ILW \eqref{ILW}
comes 
from the non-conservation of 
the truncated $k$th energy:
\begin{align}
E^\dl_{\frac{k}{2},N }(u)
&  = \frac12\sum_{\substack{\l=0\\ \text{even}}}^{k} a_{k, \l} \| \Gdl^{\frac{k-\l}{2}} u \|^2_{\dot{H}^{\frac{k}{2}}} + R^\dl_{\frac{k}{2}}(\P_N u)
\label{AE2}
\end{align}

\noi
for $k \ge 2$, 
 which causes non-invariance of the truncated \GGMs~
 $\rho^\dl_{\frac k2, N}$ in~\eqref{rho4}
 under the truncated ILW dynamics \eqref{ILW2}.

Let us take a closer look.
By writing $u = \P_N u + \P_{>N} u$, where 
$\P_{>N} = \Id - \P_N$, 
we can decouple a solution $u$ to~\eqref{ILW2}
into 
 the high frequency part $\P_{>N}u$, satisfying  the linear ILW equation:
 \begin{equation}
\partial_t \P_{>N}u - \Gdl \partial_x^2 \P_{>N}u = 0, 
\label{ILW3}
\end{equation}

 \noi
 and 
 the  low frequency part $\P_N u$, satisfying
 \begin{equation}
\partial_t \P_N u - \Gdl \partial_x^2 \P_N u = \P_N \dx ( \P_N u)^2, 
\label{ILW4}
\end{equation}

 \noi
 which is  a finite-dimensional system of ODEs  when viewed on the Fourier side, 
 preserving the Hamiltonian 
 $E^\dl_{\frac 12}(\P_N u)$.
On the one hand, the linear equation \eqref{ILW3}
 preserves 
  \begin{align}
E^\dl_{\frac{k}{2},> N }(u)
&  = \frac12\sum_{\substack{\l=0\\ \text{even}}}^{k} a_{k, \l} \| \Gdl^{\frac{k-\l}{2}} \P_{>N} u \|^2_{\dot{H}^{\frac{k}{2}}}
\label{AE3}
\end{align}

 \noi
for any $k \in \Z_{\ge 0}$.
On the other hand,  
 as pointed out in the introduction, 
 $E^\dl_{\frac k2}(\P_N u)$
 is not conserved under \eqref{ILW4} for any $k \ge 2$, 
 which is the source of the non-conservation of the truncated $k$th energy
$E^\dl_{\frac{k}{2},N }(u)$ in \eqref{AE2}
under the truncated ILW \eqref{ILW2}.

As proposed in~\cite{TV1, TV2}, 
we aim to prove almost invariance of 
the truncated \GGMs~
 $\rho^\dl_{\frac k2, N}$
 (Proposition \ref{PROP:main3})
 which, in the limit as $N \to \infty$, 
 yields invariance of the limiting \GGM~$\rho^\dl_\frac k2$.
A key ingredient is 
the following probabilistic asymptotic conservation
of the truncated energy
$E^\dl_{ \frac k 2} (\P_N u)$ 
under  \eqref{ILW4}.
Before stating a result, let us introduce a notation;
given $0 <  \dl\leq \infty$ and $N \in \N$, 
let
$\Phi_N(t) = \Phi_N(t;\dl)$ be the solution map for the truncated ILW dynamics \eqref{ILW2},
sending initial data $u(0)$ at time $0$ to a solution $u(t) = \Phi_N(t) (u(0))$ at time $t$.

\begin{proposition}[deep-water regime]\label{PROP:AAS1}
Let $k \ge 2$ be an integer and $0 <  \dl\leq \infty$.
Then, given any finite $p \ge 1$, 
we have 
\begin{align}
\lim_{N\to\infty} \bigg\| \frac{d}{dt} E^\dl_{ \frac k 2} (\P_N \Phi_N(t) (u) ) \Big\vert_{t=0} 
\bigg\|_{L^p(d\mu^\dl_{\frac{k}{2}})} =0.
\label{AAS1}
\end{align}
\end{proposition}

We present a proof of 
 Proposition~\ref{PROP:AAS1} in Subsection~\ref{SUBSEC:AC1}.
As we explain below, thanks to~\eqref{DX1} and \eqref{Q1}, 
most of the terms can be treated 
by a slight modification
of the argument in \cite{TV0, TV1, TV2}.
Thus, we keep our presentation brief.

\begin{remark} \label{REM:div}\rm
As it is clear from the proof, 
the rate of convergence in \eqref{AAS1} is uniform in $ 2\le \dl \le \infty$.
On the other hand, 
our argument for 
 \eqref{AAS2} 
 in the shallow-water regime
 does not 
provide a uniform  rate of convergence 
 for $ 0< \dl \le 1$
 and in fact it diverges as $\dl \to 0$.
 We provide a separate argument for $\dl = 0$.
\end{remark}

\medskip

Next, we consider the shallow-water regime.
Given  $0 \le  \dl < \infty$ and $N \in \N$, 
we consider the following truncated scaled ILW:
\begin{equation}
\label{ILW5}
\partial_t v - \Gd \partial_x^2 v= \P_N \dx(\P_Nv)^2,
\end{equation}

\noi
where a solution  $v$ decouples into 
 the high frequency part $\P_{>N}v$, satisfying   the linear scaled ILW equation:
 \begin{equation*}
\partial_t \P_{>N}v - \Gd \partial_x^2 \P_{>N}v = 0, 
\end{equation*}

 \noi
 and 
 the  low frequency part $\P_N v$, satisfying
 \begin{equation}
\partial_t \P_N v - \Gd \partial_x^2 \P_N v = \P_N \dx ( \P_N v)^2.
\label{ILW7}
\end{equation}

\noi
Let 
$\wt \Phi_N(t) = \wt \Phi_N(t;\dl)$ be the solution map for the truncated scaled ILW~\eqref{ILW5},
sending initial data $v(0)$ at time $0$ to a solution $v(t) = \wt \Phi_N(t) (v(0))$ at time $t$.
Then, our main goal is
to establish 
the following probabilistic asymptotic conservation
of the truncated energy
$\wt E^\dl_{ \frac k 2} (\P_N v)$ 
under  \eqref{ILW7}.

\begin{proposition}[shallow-water regime]\label{PROP:AAS2}
Let $k \ge 2$ be an integer and $0 <   \dl < \infty$.
Then, given any finite $p \ge 1$, 
we have 
\begin{align}
\lim_{N\to\infty}  \bigg\| \frac{d}{dt} \wt{E}^\dl_{ \frac k 2}(\P_N \wt{\Phi}_N(t) (v)) \Big|_{t=0} 
\bigg\|_{L^p(d\wt{\mu}^\dl_{ \frac k 2})} =0.
\label{AAS2}
\end{align}

\noi
When $k$ is even, \eqref{AAS2} also holds for $\dl = 0$.

\end{proposition}

Proposition \ref{PROP:AAS2} follows 
from a slight modification of the proof of Proposition \ref{PROP:AAS1}
in the deep-water regime.
Due to the non-uniformity in $0< \dl \le 1$
of the rate of convergence, we need to provide a separate
argument for $\dl = 0$.
See 
Subsection~\ref{SUBSEC:AC2}
for a further discussion.

\medskip

\noi
{\bf Notations.}
Before proceeding further, 
let us introduce some notations, 
following Section~8 in  \cite{TV0}.
Let $p(u)$ be a monomial in $ \Pc_j(u)$
for some $j \in \N$
with the fundamental form
\begin{align*}
\wt{p}(u) = \prod_{i=1}^j \dx^{\al_i} u
\end{align*}

\noi
for some $\al_i \in \Z_{\ge 0}$;
see Definition \ref{DEF:mono1}.
Then, given $N \in \N$, by viewing 
$p(u)
= p(\dx^{\al_1} u, \dots, \dx^{\al_j} u)$
as a $j$-linear function, 
we define 
a new monomial  $p_{i,N}^*(u)$, 
$i  = 1, \dots, j$, 
 by setting
\begin{align*}
p_{i,N}^*(u) = p(u)\big|_{\partial_x^{\al_i} u=
-2 \partial_x^{\al_i} (\P_{>N} (u\partial_x u))},
\end{align*}

\noi
namely, 
$p_{i,N}^*(u)$
is the monomial  obtained from $p(u)$ by 
replacing $\dx^{\al_i} u$  in the $i$th factor  by $-2\dx^{\al_i}(\P_{>N}(u \dx u))$.
Furthermore, we define $p_N^*(u)$ as
\begin{align} \label{AC2}
p^*_N(u)=\sum_{i=1}^j p_{i,N}^*(u).
\end{align}

\noi
Note that 
\begin{align*}
\#u(p^*_N) = \#u(p) + 1
\qquad \text{and}
\qquad 
\#\dx(p^*_N) = \#\dx(p) + 1
\end{align*}

\noi
in the sense of Definition \ref{DEF:ord1}.
The notation $p^*_N(u)$ in \eqref{AC2}
 will become useful in expressing $\dt E^\dl_{\frac k 2}(\P_N \Phi_N(t)(u))$
 in a concise manner.
We extend these definitions
to monomials in $\wt \Pc_j(v)$ defined in Definition \ref{DEF:mono2}.

We use $u^+$ (and $u^-$, respectively)
to denote the projection of $u$
onto positive frequencies (and negative frequencies, respectively).
Then, we have
\begin{align}
\H u^+ = -i u^+\qquad \text{and}
\qquad \H u^- = i u^-
\label{AC4}
\end{align}

\noi
and 
\begin{align}
\P_{>N} (\P_N u_1\P_N u_2)
= \P_{>N} (\P_N u_1^+\P_N u_2^+)
+ \P_{>N} (\P_N u_1^-\P_N u_2^-).
\label{AC5}
\end{align}

\noi
In particular, from \eqref{AC4} and \eqref{AC5}, we have 
\begin{align}
\P_{>N} (\P_N u_1 \H \P_N u_2)
= \H \P_{>N} (\P_N u_1  \P_N u_2).
\label{AC6}
\end{align}

\noi
Given $N \in \N$, we use the following short-hand notation:
\begin{align}
u_N = \P_N u.
\label{AC3}
\end{align}

\subsection{Deep-water regime} \label{SUBSEC:AC1}

Fix an integer $k\geq 2$, $0 < \dl \le \infty$, and $N\in\N$.
Let $u$ be a solution to the truncated ILW \eqref{ILW2}.
Then, the low-frequency part $\P_Nu$ satisfies \eqref{ILW4}, 
which can be rewritten as
\begin{equation*}
\partial_t \P_N u - \Gdl \partial_x^2 \P_N
u =
2(\P_N u)\partial_x( \P_N u)-
2\P_{>N} \big((\P_{N}u)\partial_x(\P_{N} u)\big).
\end{equation*}

\noi
Then, from the conservation of $E^\dl_\frac k2(u)$
under the (untruncated) ILW \eqref{ILW}
(Theorem \ref{THM:1}), 
\eqref{cons1},  and \eqref{R1}, we have 
\begin{align*}
\frac{d}{dt} E^\dl_{ \frac k 2}(\P_N u )
& = \sum_{\l=1}^{k} \frac{1}{\dl^{k-\l} } A_{\frac k 2 , \frac{\l}{2}, N}^*(\P_N u),
\end{align*}

\noi
where $A_{\frac k 2, \frac{\l}{2}, N}^*(u)$ is 
defined by replacing each monomial  $p(u)$ in
$A^\dl_{\frac k 2, \frac{\l}{2}}(u)$, appearing in 
 \eqref{Aeven}, \eqref{Aodd}, 
 and  \eqref{Aodd1}, 
  by $p_N^*(u)$ defined in \eqref{AC2}.
Then, Proposition \ref{PROP:AAS1} follows once we prove the following proposition
on 
$A^*_{\frac k 2, \frac{\l}{2}, N} (\P_N u) $.

\begin{proposition}\label{PROP:AAS1a}
Let $k \ge 2$ be an integer and $0 <  \dl\leq \infty$.
Then, given any  finite $p \ge 1$, 
we have 
\begin{align}
\lim_{N\to\infty} \Big\| A^*_{\frac k 2, \frac{\l}{2}, N} (\P_N u) \Big\|_{L^p(d\mu^\dl_{\frac{k}{2}})} = 0
\label{decay1}
\end{align}

\noi
for $\l = 1, \dots, k$.
\end{proposition}

We first prove \eqref{decay1} when $\l = 1$.

\begin{proof}[Proof of Proposition \ref{PROP:AAS1a} when $\l = 1$]

Let $u =  X^\dl_\frac k2$.
Then, 
from \eqref{Aodd1}, 
\eqref{AC2}, \eqref{AC3}, 
and \eqref{Xdl1} we have 
\begin{align}
\begin{split}
  A^*_{\frac k 2 , \frac12,N}(u_N) 
& = -6 c_0
 \int_\T  u_N^2 \P_{>N} ( u_N \dx  u_N)  dx \\
& 
= -\frac{6c_0}{(2\pi)^\frac 32}
\sum_{\substack{n_{1\cdots 4}=0\\ 0<|n_j| \leq N \\ |n_{34}|>N}} 
\frac{\prod_{j = 1}^4 g_{n_j} (in_4)}{ \prod_{j = 1}^4 (T_{\dl, \frac k 2}(n_j))^\frac 12}\\
& =:\I + \II,  
\end{split}
\label{decay1a}
\end{align}

\noi
where $\I$
denotes the contribution from frequencies with no pair
and $\II$
 denotes the contribution from frequencies with two pairs
in the sense of Definition \ref{DEF:pair}.
Then, by 
the  Wiener chaos estimate (Lemma~\ref{LEM:hyp}), 
Jensen's inequality as in \eqref{RN5a}-\eqref{RN6},  \eqref{DX1}, 
and Lemma \ref{LEM:SUM} (recall that $k\ge 2$), 
we have 
\begin{align}
\| \I\|_{L^p(d\mu^\dl_{\frac k 2})}^2 
\les_p \sum_{\substack{n_{1\cdots 4}=0\\ 0<|n_j| \leq N \\ |n_{34}|>N}} 
\frac{1}{|n_1|^k |n_2|^k |n_3|^k |n_4|^{k-2}}
\les N^{-1+\eps} \too 0, 
\label{decay1b}
\end{align}

\noi
as $N\to \infty$, where we used the fact that $|n_{12}| = |n_{34}| > N$.

Next, we consider the case when 
 there are  two pairs in \eqref{decay1a}
(note that $n_{12} = -n_{34} \ne 0$ due to $\P_{> N}$).  
Without loss of generality, we suppose that $n_{13} = n_{24} = 0$.
We note from~\eqref{AC5}  that  $n_1$ and $n_2$  have the same sign in \eqref{decay1a}.
Then, by splitting the sum into two sums, 
 where $n_1, n_2 > 0$ and $n_1, n_2 < 0$, respectively, 
and relabeling (with $|g_{-n}| = |g_n|$ and $T_{\dl, \frac k 2}(-n) = T_{\dl, \frac k 2}(n)$), 
we have 
\begin{align}
\begin{split}
\II 
& = -\frac{6c_0}{(2\pi)^\frac 32}
\sum_{\substack{ 0<|n_1|, |n_2| \leq N \\ |n_{12}|>N}} 
\frac{|g_{n_1}|^2|g_{n_2}|^2 (-in_2)}{ T_{\dl, \frac k 2}(n_1) T_{\dl, \frac k 2}(n_2)}\\
& = -\frac{6c_0}{(2\pi)^\frac 32}
\sum_{\substack{ 0<n_1, n_2 \leq N \\ n_{12}>N}} 
\frac{|g_{n_1}|^2|g_{n_2}|^2 (-in_2 + i n_2)}{ T_{\dl, \frac k 2}(n_1) T_{\dl, \frac k 2}(n_2)}
= 0.
\end{split}
\label{decay1c}
\end{align}

Hence, \eqref{decay1} follows from 
\eqref{decay1a}, \eqref{decay1b}, and \eqref{decay1c}
\end{proof}

In the remaining part of this subsection, 
we discuss a proof of  \eqref{decay1}
for $\l = 2, \dots, k$.
In view of the structure of $A_{\frac k2, \frac \l2}(u)$
with \eqref{Xdl1} and \eqref{DX1}, 
it suffices to prove \eqref{decay1} for $\l = k$;
see \cite[Remark 8.3]{TV2}.
Once we prove Lemmas \ref{LEM:AAS3}
and \ref{LEM:AAS4} below, 
Proposition \ref{PROP:AAS1a}
for $\l = k \ge 2$ follows.

Let us first consider the leading order contributions.
It follows from  Lemma \ref{LEM:cub2} that 
the leading order cubic  terms are given by 
\begin{align}
p_0(u) = u (\H\dx^{m-1} u ) (\dx^{m} u)
\label{dec2}
\end{align}

\noi
when $k = 2m$ is even, 
and 
\begin{equation}\label{dec3}
p_{\vec \al} (u)  = 
(\H^{\al_1} u ) (\H^{\al_2} \dx^{m} u) (\H^{\al_3} \dx^m u )
\end{equation}

\noi
for $\vec \al
=  (\al_1, \al_2, \al_3)
\in 
\mathcal A = \big\{(0, 0, 0), (0, 0, 1),  (0, 1, 1)\big\}$
when $k = 2m+1$ is odd.

\begin{lemma}\label{LEM:AAS3}
Let $k \ge 2$ be an integer and $0 <  \dl\leq \infty$.

\smallskip

\noi
\textup{(i)}
Let $k = 2m$, $m \in \N$.
 Then, given any finite $p \ge 1$, 
we have 
\begin{align}
\lim_{N\to\infty} \bigg\| 
\int_\T p_{0, N}^*(\P_N u)  dx 
\bigg\|_{L^p(d\mu^\dl_{\frac{k}{2}})} =0, 
\label{AAS3}
\end{align}

\noi
where $p_0(u)$ is as in \eqref{dec2}
and $p_{0, N}^*(u)$ is as in \eqref{AC2}.

\smallskip

\noi
\textup{(ii)}
Let $k = 2m+1$, $m \in \N$, 
and $\mathcal A = \big\{(0, 0, 0), (0, 0, 1),  (0, 1, 1)\big\}$
be as in \eqref{DF6}.
 Then, given any finite $p \ge 1$, 
we have 
\begin{align}
\lim_{N\to\infty} \bigg\| 
\int_\T p_{\vec \al, N}^*(\P_N u)  dx 
\bigg\|_{L^p(d\mu^\dl_{\frac{k}{2}})} =0
\label{AAS4}
\end{align}

\noi
for
any $\vec \al
=  (\al_1, \al_2, \al_3)
\in 
\mathcal A$, 
where $p_{\vec \al}(u)$ is as in \eqref{dec3}
and  $p_{\vec \al, N}^*(u)$ is as in \eqref{AC2}.

\end{lemma}

\begin{proof}
When $\dl = \infty$, 
Tzvetkov and Visciglia proved
\eqref{AAS3}
in 
\cite[Lemma 11.1]{TV0}
(see also 
\cite[Proposition 3.1, in particular (3.10)]{TV1})
for even $k \ge 6$, 
\cite[Subsection 6.1]{TV2} for $k = 4$, 
and 
\cite[Proposition 4.1]{TV2} for $k = 2$, 
and also proved
\eqref{AAS4}
in  \cite[Lemma 7.2]{TV2}
for any odd $k \ge 3$.
 In view of  \eqref{Xdl1}, \eqref{DX1}, and $\#\Qdl = 0$, 
the  bounds in \cite{TV0, TV1, TV2}
in proving \eqref{AAS3} and \eqref{AAS4} when $\dl = \infty$
also hold for $0 < \dl < \infty$
with uniform constants for $2\le \dl \le \infty$, 
thus yielding~\eqref{AAS3} and \eqref{AAS4}
with  a  rate of convergence uniform in $2 \le \dl \le \infty$.
\end{proof}

The next lemma treats the remaining  terms.

\begin{lemma}\label{LEM:AAS4}
Let $k \ge 2$ be an integer and $0 <  \dl\leq \infty$.
 Then, given any finite $p \ge 1$, 
we have 
\begin{align}
\lim_{N\to\infty} \bigg\| 
\int_\T p_{N}^*(\P_N u)  dx 
\bigg\|_{L^p(d\mu^\dl_{\frac{k}{2}})} =0
\label{AAS5}
\end{align}

\noi
for any monomial $p(u)$ appearing in 
\begin{itemize}
\item
the second term on the right-hand side of \eqref{Aeven}
when $k = 2m$ is even, 

\smallskip

\item 
the second, third, and fourth terms on the right-hand side of \eqref{Aodd}
when $k = 2m+1$ is odd.

\end{itemize}

\end{lemma}

\begin{proof}
When $\dl = \infty$, 
Lemma \ref{LEM:AAS4} was established
in 
\cite[Lemmas 11.2, 11.3, and 11.4]{TV0}
(see also 
\cite[Proposition 3.1, in particular (3.11)]{TV1})
for even $k \ge 6$, 
\cite[Subsection 6.1, in particular Second case and Third case]{TV2} for $k = 4$, 
\cite[Proposition 4.1]{TV2} for $k = 2$, 
and 
 \cite[Lemmas 7.3, 7.4, and 7.5]{TV2}
 for odd $k \ge 3$.
We point out that, when $k \ge 3$, 
\eqref{AAS5} for $0 < \dl < \infty$ follows from 
the  bounds in \cite{TV0, TV1, TV2} for the $\dl = \infty$ case
with  \eqref{Xdl1} with \eqref{DX1} and
\eqref{Q1}, giving a  rate of convergence uniform in $2 \le \dl \le \infty$.

It remains to consider the case $k = 2$.
In this case, we have $p(u) = u^4$ or $p(u) = u^2 \Qdl u$ (see Lemma \ref{LEM:cub1}).
When $p(u) = u^4$, \eqref{AAS5}
follows from 
\cite[Proposition 4.1]{TV2} 
with \eqref{DX1}, 
giving a  rate of convergence uniform in $2 \le \dl \le \infty$.

We consider the remaining case:
$p(u) = u^2 \Qdl u$.
Let $u = X^\dl_1$.
Then, with \eqref{AC3}, 
we have
\begin{align}
\begin{split}
\int_\T p^*_N(u_N) dx
& = 
- 2  \int_\T  u_N^2  \P_{>N} \Qdl ( u_N \dx  u_N)  dx \\
&  \quad 
- 4  \int_\T  u_N \Qdl  u_N \P_{>N} ( u_N \dx u_N)  dx \\
& 
= - \frac2{(2\pi)^\frac 32}
\sum_{\substack{n_{1\cdots 4}=0\\ 0<|n_j| \leq N \\ |n_{34}|>N}} 
\frac{\prod_{j = 1}^4 g_{n_j} (\ft \Qdl(n_{34})
+ 2 \ft \Qdl(n_{2})) (in_4)}{ \prod_{j = 1}^4 (T_{\dl, 1}(n_j))^\frac 12}\\
& =:\I + \II,  
\end{split}
\label{AX1}
\end{align}

\noi
where $\I$
denotes the contribution from frequencies with no pair
and $\II$
 denotes the contribution from frequencies with two pairs
in the sense of Definition \ref{DEF:pair},  
Then, by 
the  Wiener chaos estimate (Lemma~\ref{LEM:hyp}), 
Jensen's inequality as in \eqref{RN5a}-\eqref{RN6},  \eqref{DX1}, \eqref{Q1}, 
and Lemma \ref{LEM:SUM} (recall that $k= 2$), 
we have 
\begin{align}
\| \I\|_{L^p(d\mu^\dl_{1})}^2 
\les_{\dl, p} \sum_{\substack{n_{1\cdots 4}=0\\ 0<|n_j| \leq N \\ |n_{12}|>N}} 
\frac{1}{|n_1|^2 |n_2|^2 |n_3|^2}
\les N^{-1} \too 0, 
\label{AX2}
\end{align}

\noi
as $N\to \infty$, 
where the implicit constant is independent of $2 \le \dl \le \infty$.

Next, we consider the case when 
 there are  two pairs in \eqref{AX1}
(note that $n_{12} \ne 0$ due to $\P_{> N}$).  
Without loss of generality, we suppose that $n_{13} = n_{24} = 0$.
Then, by applying Minkowski's inequality with \eqref{DX1} and \eqref{Q1}, 
 splitting the sum into two sums, 
where $|n_1|\ges N$ and $|n_2|\ges N$, respectively, 
and using $N - |n_1|\le |n_2|\le N $ in the second sum, we have
\begin{align}
\begin{split}
\| \II\|_{L^p(d\mu^\dl_{1})}
& \les_{\dl, p} \sum_{\substack{ 0<|n_1|, |n_2| \leq N \\ |n_{12}|>N}} 
\frac{1}{|n_1|^2 |n_2|}\\
& \les \sum_{\substack{ 0<|n_1|, |n_2| \leq N \\ |n_{1}|\ges N}} 
\frac{1}{|n_1|^2 |n_2|}
+ \sum_{\substack{ 0<|n_1| \leq N \\ 
N - |n_1|\le |n_2|\le N}} 
\frac{1}{|n_1|^2N}\\
& \les \frac{\log N}{N} \too 0, 
\end{split}
\label{AX3}
\end{align}

\noi
as $N\to \infty$, 
where the implicit constant is independent of $2 \le \dl \le \infty$;
see also \cite[Lemma~10.1]{TV0}.

Hence, \eqref{AAS5} in this case follows from \eqref{AX1}, 
\eqref{AX2}, and \eqref{AX3}. 
\end{proof}

\begin{remark}\label{REM:cancel}\rm
In \cite{TV0, TV1, TV2}, 
Tzvetkov and Visciglia proved
Lemma \ref{LEM:AAS3} for $\dl = \infty$, 
from which the general case $0 < \dl < \infty$ follows as we explained.
We point out that in treating these 
 leading order contributions, 
 it is crucial to exploit certain cancellations, using the identity~\eqref{AC6}
 and the anti self-adjointness of $\H$.
In the following, we provide an example.
See \cite{TV0, TV1, TV2} for further details.

 Let $k = 2m$, $m \in \N$, 
 and consider $p_0(v)$ in \eqref{dec2}.
We first consider the case $m = 1$.
Namely, $p_0(u) = u \H u \dx u$.
Then, from \eqref{AC2} with \eqref{AC3}, we have
\begin{align*}
\int_\T p_{0, N}^*(u_N) dx
& = - 2 \int_\T \P_{>N} (u_N \dx u_N) (\H u_N)( \dx u_N)dx\\
& \quad 
- 2 \int_\T u_N \big(\H\P_{>N} (u_N \dx u_N)\big) (\dx u_N)dx \\
& \quad 
-2 \int_\T u_N (\H u_N)\dx \P_{>N} (u_N \dx u_N)dx.
\end{align*}

\noi
There are terms with two factors having $m =1$ derivatives
(even after integration by parts), which 
can not be handled directly, even with a probabilistic argument.
In this case, by integration by parts (for the third term),~\eqref{AC6},  and the anti self-adjointness of $\H$, 
we obtain
\begin{align*}
\int_\T p_{0, N}^*(u_N) dx
 = 0. 
\end{align*}

Next, we consider the case $m \ge 2$.
From  \eqref{dec2} and \eqref{AC2}, we have 
\begin{align*}
\int_\T p_{0,N}^*(u)  dx 
& = - 2\int_\T \big(\P_{>N}(u\dx u)\big)(\H\dx^{m-1} u)(\dx^{m} u)  dx\\
& \quad - 2\int_\T u \big(\P_{>N} \dx^{m-1} \H(u \dx u)\big)(\dx^{m} u) dx  \\
& \quad -2 \int_\T u (\H\dx^{m-1} u)\P_{>N}\dx^{m} (u\dx u)  dx \\
& =: \text{I}_N(u) + \II_{N}(u) + \III_N(u).
\end{align*}

\noi
The first term $\text{I}_N(u)$
can be handled as in \eqref{AX2} and \eqref{AX3}.
On the other hand, 
as for $\II_{N}(u)$ and $\III_N(u)$, 
there are terms with two factors having $m$ derivatives
(even after integration by parts), which 
can cause a potential issue.
By the Leibniz rule   and the anti self-adjointness of $\H$, 
we have 
\begin{align}
\II_N(u) = \sum_{j = 1}^{m-1} c_{m, j}
 \int_\T u \dx^{m} u  \P_{>N}  \H(\dx^j u \dx^{m-j} u)dx.
\label{dec4b}
\end{align}

\noi
Namely, the contribution from the worst term $u \dx^{m} u  \P_{>N}  \H( u \dx^{m} u)$ vanishes.
Similarly,
by integration by parts, the Leibniz rule, \eqref{AC6},  and the anti self-adjointness of $\H$, 
 we have
\begin{align}
\begin{split}
\III_N(u_N)
& = 2 \int_\T \dx u_N \H \dx^{m-1} u_N \P_{>N}\dx^{m-1} (u_N\dx u_N)  dx \\
& \quad +2  \int_\T  u_N \H \dx^{m} u_N \P_{>N}\dx^{m-1} (u_N\dx u_N)  dx\\ 
& = \sum_{j = 0}^{m-1}
c_{m, j}
 \int_\T \dx u_N \H \dx^{m-1} u_N \P_{>N} (\dx^j u_N\dx^{m-j} u_N)  dx\\
& \quad  +  \sum_{j = 1}^{m-1}
c_{m, j}
 \int_\T  u_N \H \dx^{m} u_N 
\P_{>N} (\dx^j u_N\dx^{m-j} u_N)  dx.
\end{split}
\label{dec4c}
\end{align}

\noi
Namely, the worst contribution from $j = 0$ in the second sum vanishes.
Once
we reduce
$\II_{N}(u)$ and $\III_N(u)$ to the forms
on the right-hand sides of \eqref{dec4b} and \eqref{dec4c}, 
we can proceed as in 
 as in \eqref{AX2} and \eqref{AX3}.
We omit details.

\end{remark}

\subsection{Shallow-water regime}
\label{SUBSEC:AC2}

Proposition \ref{PROP:AAS2} follows
from a slight modification of the proof of 
Proposition \ref{PROP:AAS1}.
We sketch the argument in the following.

Let $v$ be a solution to the truncated scaled ILW \eqref{ILW5}.
Then, the low-frequency part $\P_Nv$ satisfies \eqref{ILW7}, 
which can be rewritten as
\begin{equation*}
\partial_t \P_N v - \Gd \partial_x^2 \P_Nv =
2(\P_N v)\partial_x( \P_N v)-
2\P_{>N} \big((\P_{N}v)\partial_x(\P_{N} v)\big).
\end{equation*}

\noi
Then, from the conservation of $\wt E^\dl_\frac k2(v)$
under the (untruncated) scaled ILW \eqref{sILW}
(Theorem~\ref{THM:2}), 
we have 
\begin{align*}
\frac{d}{dt} \wt{E}^\dl_{ \frac k 2}(\P_N v)  = \wt{R}^{\dl, *}_{\frac k 2, N}(\P_N v), 
\end{align*}

\noi
where $\wt{R}^{\dl, *}_{\frac k 2, N}( v)$  is 
defined by replacing each monomial  $p(v)$ in
$\wt R^\dl_\frac k2$, appearing in \eqref{Scons2}, 
  by $p_N^*(v)$ defined in \eqref{AC2}.
Therefore, Proposition~\ref{PROP:AAS2} 
follows once we prove
\begin{align*}
\lim_{N\to\infty} \Big\| \wt{R}^{\dl, *}_{\frac k 2, N}(\P_N v) \Big\|_{L^p(d\wt{\mu}^\dl_{\frac k2})} =0 .
\end{align*}

Let us first consider the contribution from 
the lower order terms.
Here, the lower order terms are given as follows:

\smallskip
\begin{itemize}
\item[(i)] when $k = 2\kk - 1$ is odd, 
a monomial $p(v)$ whose fundamental form $\wt p(v)$ is of the form 
(b) in \eqref{SV3}, satisfying \eqref{SV3x}, and

\smallskip
\item[(ii)]
 when $k = 2\kk $ is even, 
a monomial $p(v)$ whose fundamental form $\wt p(v)$ is of the form 
(b) in \eqref{SV5}, satisfying \eqref{SV5x}.

\end{itemize}

\smallskip

\noi
In this case, there is no need to exploit
cancellation.
In particular,
 as mentioned in 
 the proof of Lemma \ref{LEM:Svar2}, 
we drop  any occurrence of $\dl \Gd$ 
thanks to 
 the $L^r$-boundedness of $\dl \Gd$ for $1 < r< \infty$
(see~\eqref{Q3} in Lemma \ref{LEM:T1});
see also the discussion after Lemma \ref{LEM:SR1}.
More concretely, 
 by  noting that 
 the condition~\eqref{SV3x}
with $m = \kk-1$ when $k = 2\kk -1$ is odd,
and 
the condition~\eqref{SV5x}
with $m = \kk$ 
 when $k = 2\kk $ is even
 both satisfy 
the hypothesis of Lemma~\ref{LEM:AAS4}, 
we see that  
 a slight modification of the proof of Lemma \ref{LEM:AAS4} 
with Lemma \ref{LEM:T1}
(to bound $\dl \Gd$)
yields
\begin{align}
\lim_{N\to\infty} \bigg\| 
\int_\T p_{N}^*(\P_N v)  dx 
\bigg\|_{L^p(d\wt \mu^\dl_{\frac{k}{2}})} =0
\label{SAC3a}
\end{align}

\noi
in this case, 
where
$p_{N}^*(v)$  is as in \eqref{AC2}.
This part of analysis extends to 
the  $\dl = 0$ case 
in view of  \eqref{EE0} and  Remark~\ref{REM:kdv1}.

\medskip

Hence, it remains to consider the leading order terms:
\smallskip
\begin{itemize}
\item[(iii)] when $k = 2\kk - 1$ is odd, 
a monomial $p(v)$ whose fundamental form $\wt p(v)$ is of the form 
(a)~$\wt p(v) = v \dx^{\kk-1} v \dx^{\kk-1} v$  in \eqref{SV3}
with the understanding that, when 
$\#\Gd (p)=  \#\dl(p) + 1$
(namely, there is a factor of  $\Gd$ unpaired with $\dl$), 
one of the $\dx$-factors is replaced by $\Gd$.
See the proof of Lemma \ref{LEM:Svar2}
for a further discussion.

\smallskip
\item[(iv)]
 when $k = 2\kk $ is even, 
a monomial $p(v)$ whose fundamental form $\wt p(v)$ is of the form 
(a)~$\wt p(v) = v \dx^{\kk-1} v \dx^{\kk} v$  in \eqref{SV5}.

\end{itemize}

\noi
For these leading order terms, we need to exploit cancellations
as explained in Remark \ref{REM:cancel}.
However, the identity \eqref{AC6}
does not hold for $\Gd$ and $\dl \Gd$ and thus we need to proceed with care.

\medskip

\noi
$\bullet$ {\bf Case 1:} $\dl > 0$.\\
\indent
Fix a monomial $p(v)$ as in (iii) or (iv) above.
Suppose that $\#\Gd (p)\le \#\dl(p)$.
In this case, we can pair 
every $\Gd$ in $p(v)$  with $\dl$. 
Recall from 
\eqref{h1} with \eqref{BTX1aa}
that (before integration by parts
to reduce the monomial to the form 
appearing in~\eqref{Scons2}) 
every $\Gd$ must come with $\dx$.
Given $p(v)$ in \eqref{Scons2}, 
let $p_0(v)$ be its original form before integration by parts.
By substituting 
\begin{align}
\dl \Gd\dx = \Gdl\dx = \H\dx + \Qdl, 
\label{SAC3}
\end{align}

\noi
which follows from~\eqref{sILW} and~\eqref{Qdl1}, 
into  
every occurrence $\dl \Gd\dx $ in $p_0(v)$
and expanding via the multilinearity, 
we end up with new monomials,
involving $\H\dx$ and 
$\Qdl$.
Except for one monomial (call it $p_1(v)$), 
all the monomials involve 
at least one factor of 
$\Qdl$.
In view of the $L^r$-boundedness of $\Qdl$
(Lemma \ref{LEM:T1}), in particular, without any derivative loss, 
these monomials are now of lower order, belonging
to the classes (i) or~(ii) described above, 
and thus can be handled as in 
(the proof of) Lemma \ref{LEM:AAS4}
(but with constants diverging as $\dl \to 0$).
On the other hand, 
$p_1(v)$ involves only $\H\dx$
and thus must be of the form described in Lemma~\ref{LEM:cub2}, 
which then can be treated by Lemma \ref{LEM:AAS3}.
As a conclusion,  we obtain \eqref{SAC3a}.

Next, we consider the case 
$\#\Gd (p) =  \#\dl(p) + 1$ which happens only in (iii).
In this case, the semi-fundamental form $\cj p(v)$ of $p(v)$
(see Definition \ref{DEF:mono2})
 is of the form:
$\cj p(v) = v (\dx^{\kk-1} v)( \Gd \dx^{\kk-2} v)$. 
Then, by artificially inserting $\dl^{-1}\cdot \dl$ and using \eqref{SAC3}, 
we have 
\begin{align}
p(v) = \dl^{-1} p_2(v) +  \dl^{-1} p_3(v) 
\label{SACX}
\end{align}

\noi
such that the semi-fundamental forms of $p_2(v)$ and $p_3(v)$
are given by 
we have
\begin{align}
\cj p_2(v)  =   v \dx^{\kk-1} v \H \dx^{\kk-1} v
\qquad \text{and}
\qquad \cj p_3(v) =v  \dx^{\kk-1} v \Qdl  \dx^{\kk-2} v
\label{SAC3b}
\end{align}

\noi
We see that $p_3(v)$ is a monomial
of lower order
belonging to the class (i)  above, 
which can be treated as in 
(the proof of) Lemma \ref{LEM:AAS4}.
As for $p_2(v)$, 
we substitute \eqref{SAC3}
in every occurrence of $\dl  \Gd$ (if any)
and treat those monomials involving $\Qdl$
 as in 
(the proof of) Lemma \ref{LEM:AAS4},  since they are of lower order.
Lastly, 
by arguing as in the proofs of Lemmas \ref{LEM:cub1} and \ref{LEM:cub2}, 
the remaining monomial, only involving $\H$, must be of the form 
$p_{\vec \al}(v)$ in \eqref{dec3} (with $m = \kk -1$)
which is already treated in  Lemma \ref{LEM:AAS3}.

This concludes the proof of Proposition \ref{PROP:AAS2}
for each {\it fixed} $0 < \dl < \infty$.
In view of the use of Lemma \ref{LEM:T1} on $\Qdl$
and \eqref{SACX} which artificially introduced a factor of $\dl^{-1}$, 
the rate of convergence in \eqref{AAS2} 
diverges as $\dl \to 0$, and thus the argument presented above
does {\it not} extend to the $\dl = 0$ case.

\medskip

\noi
$\bullet$ {\bf Case 2:} $\dl = 0$ and $k$ is even.\\
\indent
We only need to consider monomials in the class (iv).
In this case, we have $\#\Gd = \#\dl = 0$.
Thus, we have 
$\wt p(v) = v \dx^{\kk-1} v \dx^{\kk} v$,  
which is  treated in Lemma \ref{LEM:AAS3}, 
yielding~\eqref{SAC3b}.
This proves  Proposition \ref{PROP:AAS2}
when $\dl = 0$.

\section{Dynamical problem}\label{SEC:INV}
In this section, 
by studying the associated dynamical problem, 
we prove invariance of the \GGMs~
and deep-water (and shallow-water) convergence of the invariant dynamics
(Theorems~\ref{THM:5} and~\ref{THM:6}).

We first recall the following well-posedness
and convergence results for the (scaled) ILW~\eqref{ILW} and \eqref{sILW};
see  \cite[Theorem~1.9 and Corollary~1.10]{MV15}
and  \cite[Theorems 1.3 and 1.4 and Corollary~1.5]{Li2022}.
See also \cite{BO93, Moli1, Moli2, MP, CLOP}.
While the well-posedness result in \cite{MV15}
is stated for the full dynamics, 
it also applies 
to the truncated (scaled) ILW equations \eqref{ILW2} and \eqref{ILW5}
in view of the uniform (in $N$) boundedness of 
the frequency projector $\P_N$.

\begin{lemma}[deep-water regime]
\label{LEM:GWP1}

\noi{\rm(i)} 
Let  $0 < \dl \le \infty$.
Then, the ILW equation \eqref{ILW} 
\textup{(}and the BO equation \eqref{BO} when $\dl = \infty$\textup{)} is globally well-posed in $H^s(\T)$ for ${s\ge \frac12}$.
In particular, 
given $u_0 \in H^s(\T)$, 
there exists $T = T\big(\dl, \|u_0 \|_{H^{\frac 12}}\big)> 0$
such that there exists a unique solution $u(t) = \Phi(t) (u_0)
= \Phi(t;\dl) (u_0)$ to \eqref{ILW}
with $u|_{t = 0} = u_0$
in the class $C(\R; H^s(\T))$, satisfying the bound\textup{:}
\begin{equation*}
\|\Phi(t)(u_0) \|_{L^\infty([-T,T];H^s)}, 
\, \|\Phi_N(t)(u_0) \|_{L^\infty([-T,T];H^s)}  \les_\dl \|u_0\|_{H^s}, 
\end{equation*}

\noi
uniformly in $N \in \N$, 
where
 $\Phi_N(t) = \Phi_N(t;\dl)$ denotes  the solution map for the truncated ILW~\eqref{ILW2}.

\smallskip

\noi{\rm(ii)} 
Let $ s> \frac 12$.
Suppose that
 $u_{0}^\dl \in H^s(\T)$, $2 \le \dl < \infty$, 
 converges to 
some $u_0 \in H^s(\T)$ as $\dl \to \infty$.
Then, as $\dl \to \infty$, 
the corresponding global-in-time solution $u^\dl$ to ILW \eqref{ILW}
\textup{(}with the depth parameter $\dl$\textup{)}
with $u^\dl|_{t = 0} = u^\dl_{0}$
converges 
to the solution  $u$ to BO \eqref{BO} with $u|_{t = 0} = u_0$
in $C(\R; H^s(\T))$, 
endowed with the compact-open topology in time.

\end{lemma}

\begin{lemma}[shallow-water regime]
\label{LEM:GWP2}

\noi{\rm(i)} 
Let  $0 \le  \dl \le 1$.
Then, the scaled ILW equation~\eqref{sILW} 
\textup{(}and the KdV equation \eqref{kdv} when $\dl = 0$\textup{)} is globally well-posed in $H^s(\T)$ for ${s\ge \frac12}$.
In particular, 
given $v_0 \in H^s(\T)$, 
there exists $T = T\big(\dl, \|v_0 \|_{H^{\frac 12}}\big)> 0$
such that there exists a unique solution $v(t) = \wt \Phi(t) (v_0)
= \wt \Phi(t;\dl) (v_0)
$ to \eqref{sILW}
with $v|_{t = 0} = v_0$
in the class $C(\R; H^s(\T))$, satisfying the bound\textup{:}
\begin{equation*}
\|\wt \Phi(t)(v_0) \|_{L^\infty([-T,T];H^s)}, 
\, \|\wt \Phi_N(t)(v_0) \|_{L^\infty([-T,T];H^s)} \les_\dl \|v_0\|_{H^s}, 
\end{equation*}

\noi
uniformly in $N \in \N$, 
where
 $\wt \Phi_N(t) = \wt \Phi_N(t;\dl)$ denotes  the solution map for 
the truncated scaled ILW \eqref{ILW5}

\smallskip

\noi{\rm(ii)} 
Let $ s> \frac 12$.
Suppose that
 $v_{0}^\dl \in H^s(\T)$, $0 <  \dl \le 1$, 
 converges to 
some $v_0 \in H^s(\T)$ as $\dl \to 0$.
Then, as $\dl \to 0$, 
the corresponding global-in-time solution $v^\dl$ to 
the scaled ILW \eqref{sILW} \textup{(}with the depth parameter $\dl$\textup{)}
with $v^\dl|_{t = 0} = v^\dl_{0}$
converges 
to the solution  $v$ to KdV \eqref{kdv} with $v|_{t = 0} = v_0$
in $C(\R; H^s(\T))$, 
endowed with the compact-open topology in time.

\end{lemma}

\begin{remark}\rm
(i) 
Lemmas \ref{LEM:GWP1} and \ref{LEM:GWP2}
are based on an energy method without using a gauge transform.
In \cite{CLOP}, we extended Lemma \ref{LEM:GWP1}
to the range  $s \ge 0$.  However, the argument in~\cite{CLOP}
uses the gauge transform for BO \eqref{BO}
and thus is not suitable for our purpose.
We point out that the solutions in Lemmas \ref{LEM:GWP1} and \ref{LEM:GWP2}
are unconditionally unique
in $C(\R;  H^s(\T))$.

\smallskip

\noi
(ii) In view of the $H^\frac 12$-well-posedness 
stated in Lemmas \ref{LEM:GWP1}\,(i) and \ref{LEM:GWP2}\,(i), 
it is likely that the convergence claims in 
Lemmas \ref{LEM:GWP1}\,(ii) and \ref{LEM:GWP2}\,(ii)
extend to the endpoint $s = \frac 12$.
We, however, do not pursue this issue, 
and simply use the convergence results from \cite{Li2022}.

\end{remark}

In the remaining part of this paper, we restrict our attention to the deep-water regime, 
since the same argument applies to the shallow-water regime with a straightforward modification.

\subsection{Invariance of the \GGMs}
\label{SUBSEC:INV1}

In this subsection, 
we prove invariance of the deep-water \GGMs~ (Theorem \ref{THM:5}\,(i)).
With the tools developed in the previous sections, 
Theorem \ref{THM:5}\,(i) 
 follows from the lines in
\cite{TV1, TV2}, 
and thus 
we will be very brief here.
Let us fix notations.

\smallskip

\begin{itemize}
\item
Fix an integer $k \ge 3$, 
 $0 < \dl \le \infty$, and $K > 0$ in the following.
Set $\frac 12 < s < \s < \frac {k-1}2$.

\smallskip
\item 
$F^\dl_{\frac k 2}(u)$ denotes the density of the \GGM~ $\rho^\dl_\frac k2$ as in \eqref{rho3}.
Recall that they depend on an $L^2$-cutoff size $K > 0$ but we suppress the $K$-dependence
from our notation.

\smallskip
\item 
Given $N \in \N$, 
 $\Phi_N(t) = \Phi_N(t;\dl)$ denotes  the solution map for the truncated ILW~\eqref{ILW2}.
We use $\Phi(t)$ to denote
 the solution map for  ILW~\eqref{ILW}.

\smallskip
\item 
Given $s \in \R$ and $R >0$, 
we use $B^s(R)$ to denote the ball in $H^s(\T)$ of radius $R > 0$
centered at the origin.

\end{itemize}


\smallskip

As in \cite{TV1, TV2}, 
the main goal is to prove the following
almost invariance of the 
\GGM, which can be regarded as a probabilistic counterpart
of the method of almost conservation laws \cite{CKSTT}.

\begin{proposition}
\label{PROP:main3}
Given any $T > 0$, we have 
\begin{equation}
\lim_{N\rightarrow \infty}
\sup_{\substack{t\in [0, T]\\A\in {\mathcal B}(H^\sigma)}}
\bigg |\frac d{dt} \int_{\Phi_N(t)(A)} F^\dl_{\frac k 2} (\P_N u) d\mu^\dl_{\frac k 2}(u)\bigg |=0, 
\label{iv0}
\end{equation}

\noi
where
${\mathcal B}(H^\s)$ denotes the Borel sets in $H^\s(\T)$.

\end{proposition}

See 
\cite[Proposition 5.4]{TV1}, 
where such a result was first introduced.
Once we have Proposition~\ref{PROP:main3}, 
the rest of the argument follows the lines in \cite{TV1, TV2}
by using Lemma~\ref{LEM:inv1} below
(see \cite[pp.\,4710-4713]{TV1}), 
yielding invariance of the \GGM~ $\rho^\dl_\frac k2$.
(Theorem~\ref{THM:5}\,(i)).  We omit details.

As the end of this subsection, 
we present a  proof of Proposition \ref{PROP:main3}, 
following the strategy developed in \cite{TV1, TV2}, 
where we  show how our simplified construction of the \GGMs~
simplifies the proof of Proposition \ref{PROP:main3}; 
see
Remark \ref{REM:TV}.

We state several preliminary lemmas.
The first lemma is on the approximation 
property of the truncated ILW dynamics \eqref{ILW2};
see \cite[Proposition 4.1]{TV1}.
This lemma is not needed for proving Proposition \ref{PROP:main3}, 
but is needed for proving an analogue of \cite[Lemma 5.6]{TV1}.

\begin{lemma}
\label{LEM:inv1}

Given $R > 0$,  there exists
$T=T(R)>0$ such that
for every $\eps_0>0$ there exists $N_0 = N_0(\eps_0) \in \N$ 
such that 
\begin{align*}
\Phi_N(t) (A) & \subset \Phi(t)(A) + B^s(\eps_0)
\end{align*}

\noi
for any $N \ge N_0$, $t \in [-T, T]$, 
and $A\in B^\s(R)$.

\end{lemma}

Lemma \ref{LEM:inv1}
follows from a slight modification of the proof of 
 \cite[Proposition~3.5]{MV15} 
 on the difference estimate.
 The only additional term 
 is $ \P_{>N} \dx ( \P_N u)^2$
which is in fact $O(N^{s-\s})$
 for 
  $\s > s > \frac 12$
  thanks to the 
projection $ \P_{>N}$ onto high frequencies
and the uniform (in $N$)
bound on the (higher) $H^\s$-norm of 
a solution to the truncated ILW \eqref{ILW2}, 
coming from Lemma~\ref{LEM:GWP1}\,(i).

Given $N \in \N$, 
we write $\mu^\dl_\frac k2 = \mu^{\dl, \le N}_\frac k2\otimes \mu^{\dl, > N}_\frac k2$, 
where 
$\mu^{\dl, \le N}_\frac k2 $
and 
$ \mu^{\dl, > N}_\frac k2$
are defined as the pushforward measures
of $\mu^\dl_\frac k2$ under the maps
$\P_N$ and $\P_{< N}$, respectively.
In particular, we have
$d \mu^{\dl, > N}_\frac k2
= Z^{-1}_{\dl, \frac k2, >N} 
\exp\big(- E^\dl_{\frac{k}{2},> N }(u)\big)d\P_{>N} u$, 
where
$E^\dl_{\frac{k}{2},> N }(u)$ is as in 
\eqref{AE3}.
Let 
$dL_N = \prod_{0<n \le N} d \ft u(n)$ 
denotes the Lebesgue measure on $E_N = \text{span} \big\{ \cos (nx), \sin (nx)\big\}_{0< n \le N}
\cong \C^N$.
Then, we can write the truncated \GGM~ $\rho^\dl_{\frac k2, N}$ in \eqref{rho4}
as
\begin{align}
\begin{split}
\rho^\dl_{\frac{k}{2}, N}(du) & = Z^{-1}_{\dl, \frac{k}{2}, N}  \,
F^\dl_{\frac{k}{2}}(\P_N u)  d\mu^\dl_{\frac{k}{2}}(u)\\
 & 
 = \ft Z_{\dl, \frac k 2,  N}^{-1}  
\eta_K \big(\| \P_N u\|_{L^2}\big) 
\exp\Big(-E^\dl_{\frac k 2}(\P_Nu) \Big)  
dL_N  \otimes d\mu^{\dl, > N}_\frac k2.
\end{split}
\label{rhoZ}
\end{align}

Next, we state a 
a change-of-variables formula; 
\cite[Proposition 5.1]{TV1}.
In recent years, such a change-of-variables formula 
has also played an important role
in the  study of
quasi-invariant Gaussian measures for Hamiltonian PDEs;
see, for example,  \cite{Tz15, OTz1, OTz2, GOTW}.

\begin{lemma}
\label{LEM:inv2}
Let $N \in \N$. Then, we have 
\begin{align*}
\rho^\dl_\frac k2 (\Phi_N(t)(A)) 
& = Z_{\dl,\frac k2,  N}^{-1} \int_{\Phi_N(t)(A)} F^\dl_\frac k2 (\P_N u)  d\mu^\dl_\frac k2 (u)\\
&
= \ft Z_{\dl, \frac k 2,  N}^{-1}  \int_{A} 
\eta_K \big(\| \P_N \Phi_N(t)(u)\|_{L^2}\big) \\
& \hphantom{XXXXL}\times 
\exp\Big(-E^\dl_{\frac k 2}(\P_N(\Phi_N(t)(u)) \Big)  
dL_N  \otimes d\mu^{\dl, > N}_\frac k2
\end{align*}

\noi
for  any $A\in\mathcal{B}(H^\s)$, 
where 
$Z_{\dl,\frac k2,  N}$ and $\ft Z_{\dl,\frac k2,  N}$
are as in \eqref{rhoZ}.

\end{lemma}

We conclude this subsection 
by presenting a proof of 
 Proposition \ref{PROP:main3}.

\begin{proof}[Proof of Proposition \ref{PROP:main3}]

We proceed as in the proof of \cite[Proposition~5.4]{TV1}.
 We first show that
\begin{equation}
\label{iv1}
\lim_{N\rightarrow \infty} \sup_{A\in {\mathcal B}(H^\sigma)}
\bigg |
\frac d{dt} \int_{\Phi_N(t)(A)} F^\dl_{\frac k 2} (\P_N u) d\mu^\dl_{\frac k 2}(u)\Big|_{t = 0}
\bigg |=0.
\end{equation}

\noi
From the change-of-variable formula 
(Lemma~\ref{LEM:inv2}) and the $L^2$-conservation under the truncated dynamics \eqref{ILW4}, 
we have
\begin{align}
\begin{split}
& \frac d{dt} \int_{\Phi_N(t)(A)} F^\dl_{\frac k 2} (\P_N u) d\mu^\dl_{\frac k 2}(u)
\Big|_{t = 0}\\
& \quad = 
-\int_A \frac d{dt} E^\dl_{\frac k 2}(\P_N\Phi_N(t)(u))\Big|_{t = 0}
F^\dl_{\frac k 2} (\P_N u) d\mu^\dl_{\frac k 2}(u).
\end{split}
\label{invX}
\end{align}

\noi
Then, by Cauchy-Schwarz's inequality
and Proposition~\ref{PROP:V1}, we have 
\begin{align}
\begin{split}
&  \sup_{A\in {\mathcal B}(H^\sigma)}\bigg |
\frac d{dt} \int_{\Phi_N(t)(A)} F^\dl_{\frac k 2} (\P_N u) d\mu^\dl_{\frac k 2}(u)\Big|_{t = 0}
\bigg |\\
& \quad \le 
\sup_{N\in\N} \| F^\dl_{\frac{k}{2}}( \P_N u) \|_{L^2(d\mu^\dl_{\frac{k}{2}})}
\cdot 
\bigg\| \frac{d}{dt} E^\dl_{ \frac k 2} (\P_N \Phi_N(t) (u) ) \Big\vert_{t=0} 
\bigg\|_{L^2(d\mu^\dl_{\frac{k}{2}})}\\
& \quad \les 
\bigg\| \frac{d}{dt} E^\dl_{ \frac k 2} (\P_N \Phi_N(t) (u) ) \Big\vert_{t=0} 
\bigg\|_{L^2(d\mu^\dl_{\frac{k}{2}})}.
\end{split}
\label{iv2}
\end{align}

\noi
Then, \eqref{iv1} follows from 
Proposition~\ref{PROP:AAS1}.

Let $0 < \tau \le T$.
Since we have
\begin{align*}
\frac d{dt} \int_{\Phi_N(t)(A)} F^\dl_{\frac k 2} (\P_N u) d\mu^\dl_{\frac k 2}(u)
\Big|_{t = \tau}
= \frac d{dt} \int_{\Phi_N(t)(\Phi_N(\tau)(A))} F^\dl_{\frac k 2} (\P_N u) d\mu^\dl_{\frac k 2}(u)
\Big|_{t = 0}, 
\end{align*}

\noi
we can repeat the computation in \eqref{invX} and \eqref{iv2}, 
where we replace $A$ by $\Phi_N(\tau)(A)$.
In particular, 
\eqref{iv2} holds uniformly in $0 \le \tau \le T$, 
allowing us to insert a supremum over $0 \le \tau \le T$
in \eqref{iv2}.
This proves \eqref{iv0}.
\end{proof}

\begin{remark}\label{REM:TV} \rm

Thanks to our simplified construction of the \GGMs~
only with an $L^2$-cutoff, 
the step \eqref{invX} is simpler,  since it does not involve
the time derivatives
of lower order conservations laws.
Compare \eqref{invX} with 
the bottom math display  on \cite[p.\,4709]{TV1}.
\end{remark}

\subsection{Dynamical convergence}
\label{SUBSEC:INV2}

We conclude this paper by presenting a proof of Theorem~\ref{THM:5}\,(ii)
on the deep-water convergence of the invariant ILW dynamics
constructed in the previous subsection.
Theorem \ref{THM:6}\,(iii) and (iv)
on shallow-water convergence follow from a similar discussion
which we omit.

Fix $k \ge 3$.
From Theorem \ref{THM:3}\,(ii), 
we know that, as $\dl \to \infty$,  $\rho^\dl_\frac k2$ converges to 
$\rho^\infty_\frac k2$ in total variation and hence weakly.
Then, 
by  the 
Skorokhod representation theorem
along a continuous parameter (Lemma \ref{LEM:Sk}), 
there exist random functions
$\{u^\dl_0\}_{2\le \dl \le  \infty}\subset H^{\frac{k-1}{2}-\eps}(\T)$
such that $\Law (u^\dl_0) = \rho^\dl_\frac k2$ for each $2\le \dl \le \infty$
and $u^\dl_0$ converges almost surely to $u^\infty_0$ in $H^{\frac{k-1}{2}-\eps}(\T)$
as $\dl \to \infty$.

Let $u^\dl$ be the unique global-in-time solution to ILW \eqref{ILW}
(with the depth parameter $\dl$; BO~\eqref{BO} when $\dl = \infty$) with $u^\dl|_{t = 0} = u^\dl _0$.
Then, it follows from 
Lemma \ref{LEM:GWP1}\,(ii)
that $u^\dl$ converges almost surely to $u^\infty$
in $C(\R; H^s(\T))$, 
endowed with the compact-open topology in time.
This proves Theorem~\ref{THM:5}\,(ii).

\begin{ackno}\rm

The authors would like to thank 
Philippe Sosoe and Guangqu Zheng
for helpful discussions on the 
Skorokhod representation theorem
along a continuous parameter (Lemma \ref{LEM:Sk}).
 The authors would 
also like to express their gratitude to the anonymous referee for the helpful comments which improved the quality of the paper.
This material is based upon work supported by the Swedish
Research Council under grant no.~2016-06596
while
the first and  third authors were in residence at
Institut  Mittag-Leffler in Djursholm, Sweden
during the
program ``Order and Randomness in Partial Differential Equations''
in Fall, 2023.
The third author is grateful
for the hospitality and support provided by 
the Erwin Schr\"odinger International Institute for Mathematics and Physics 
during 
the thematic program ``Nonlinear Waves and Relativity''
in June, 2024.
A.C., G.L, and T.O.~were supported by the European Research Council
(grant no.~864138 ``SingStochDispDyn'').
A.C.~also received support from CNRS-INSMI through a grant ``PEPS Jeunes chercheurs et jeunes chercheuses 2025''.
G.L. was also supported by the EPSRC New Investigator Award (grant no.~EP/S033157/1), 
 by the NSFC (grant no.~12501181), 
  and by the Fundamental Research Funds for the Central Universities. 
T.O. was also supported by the EPSRC Mathematical Sciences Small
Grant (grant no.~EP/Y033507/1)
and 
acknowledges support from  
the NSFC (grant no.~W2531005).

\end{ackno}


\end{document}